\documentclass{article}

\usepackage{amscd}
\usepackage{appendix} 
\usepackage[normalem]{ulem} %for underlining things, used 
\usepackage{color}

\usepackage[hyphens]{url}
\usepackage[hidelinks, pdfpagelabels]{hyperref}
\usepackage{authblk}

\usepackage[english]{babel}
\usepackage{amsmath}
\usepackage{bbm}
\usepackage{amsfonts}
\usepackage{dsfont}
\usepackage{amsthm}
\usepackage{thmtools}
\usepackage{mathtools}
\usepackage{amssymb}
\usepackage{graphicx}
\usepackage{epstopdf}

\usepackage{cite}

\usepackage{bm}
\usepackage[all,cmtip]{xy}
\usepackage{enumitem}
\usepackage{mathrsfs}
\usepackage{vmargin}
\usepackage{verbatim}
\usepackage{cleveref}
\usepackage{stackrel}
\usepackage{epigraph}

\usepackage{chngcntr}
\usepackage{etoolbox}

\AtBeginEnvironment{subappendices}{%
\chapter*{Appendix}
\addcontentsline{toc}{chapter}{Appendices}
\counterwithin{figure}{section}
\counterwithin{table}{section}
}

\newtheorem{theo}{Theorem}[section]

\newtheorem{cor}[theo]{Corollary}

\newtheorem{prop}[theo]{Proposition}

\newtheorem{lemma}[theo]{Lemma}

\newtheorem{conj}[theo]{Conjecture}

\theoremstyle{definition}

\newtheorem{defi}[theo]{Definition}

\theoremstyle{remark}

\newtheorem{remark}[theo]{Remark}

\newtheorem{notation}[theo]{Notation}

\newcommand{\BA}{{\mathbb{A}}}

\newcommand{\BC}{{\mathbb{C}}}
\newcommand{\BD}{{\mathbb{D}}}
\newcommand{\BE}{{\mathbb{E}}}
\newcommand{\BF}{{\mathbb{F}}}
\newcommand{\BG}{{\mathbb{G}}}

\newcommand{\BM}{{\mathbb{M}}}
\newcommand{\BN}{{\mathbb{N}}}

\newcommand{\BP}{{\mathbb{P}}}
\newcommand{\BQ}{{\mathbb{Q}}}

\newcommand{\BZ}{{\mathbb{Z}}}

\newcommand{\CC}{{\mathcal{C}}}

\newcommand{\CG}{{\mathcal{G}}}

\newcommand{\CI}{{\mathcal{I}}}

\newcommand{\CL}{{\mathcal{L}}}
\newcommand{\CM}{{\mathcal{M}}}
\newcommand{\CN}{{\mathcal{N}}}
\newcommand{\CO}{{\mathcal{O}}}

\newcommand{\CR}{{\mathcal{R}}}
\newcommand{\CS}{{\mathcal{S}}}
\newcommand{\CT}{{\mathcal{T}}}
\newcommand{\CU}{{\mathcal{U}}}
\newcommand{\CV}{{\mathcal{V}}}

\newcommand{\Hom}{\mathop{\rm Hom}\nolimits}
\newcommand{\Homint}{\underline{\mathsf{Hom}}}

\newcommand{\Ext}{\mathop{\rm Ext}\nolimits}

\newcommand{\Lim}{{\rm Lim}}
\newcommand{\Colim}{{\rm Colim}}

\newcommand{\ra}{\rightarrow}
\newcommand{\lra}{\longrightarrow}

\newcommand{\rap}{\stackrel{+}{\rightarrow}}

\newcommand{\Spec}{\mathop{{\bf Spec}}\nolimits}

\newcommand{\Sch}{\mathsf{Sch}}
\newcommand{\Sm}{\mathsf{Sm}}

\newcommand{\un}{\mathbbm{1}}

\newcommand{\Gm}{\mathbb{G}_\mathrm{m}}
\newcommand{\Ga}{\mathbb{G}_\mathrm{a}}

\newcommand{\charact}{\mathop{\rm char}\nolimits}

\newcommand{\Gal}{\mathop{\rm Gal}\nolimits}

\newcommand{\SH}{\mathop{\mathbf{SH}}\nolimits}
\newcommand{\DM}{\mathop{\mathbf{DM}}\nolimits}
\newcommand{\DA}{\mathop{\mathbf{DA}}\nolimits}

\newcommand{\HI}{\mathop{\mathbf{HI}}\nolimits}
\newcommand{\MM}{\mathop{\mathbf{MM}}\nolimits}
\newcommand{\Cpl}{\mathop{\mathbf{Cpl}}\nolimits}

\newcommand{\Sh}{\mathop{\mathbf{Sh}}\nolimits}

\newcommand{\id}{{\rm id}}

\newcommand{\Shv}{\mathbf{Sh}}
\newcommand{\D}{\mathsf{D}}

\newcommand{\Pic}{{\rm Pic}}
\newcommand{\sPic}{\mathcal{P}ic}
\newcommand{\NS}{{\rm NS}}

\newcommand{\Aut}{{\rm Aut}}
\newcommand{\Nis}{{\rm Nis}}
\newcommand{\Et}{{\rm Et}}

\newcommand{\Gr}{{\rm Gr}}

\newcommand{\Res}{\mathop{\rm Res}\nolimits}

\newcommand{\Sus}{\mathop{\rm Sus}\nolimits}

\newcommand{\Ev}{\mathop{\rm Ev}\nolimits}

\newcommand{\Coker}{\mathop{\rm Coker}\nolimits}
\newcommand{\Ker}{\mathop{\rm Ker}\nolimits}

\newcommand{\Jac}{\mathop{\rm Jac}\nolimits}

\newcommand{\JG}{\mathcal{JG}}
\newcommand{\DG}{\mathcal{DG}}

\newcommand{\sNS}{\mathcal{NS}}

\newcommand{\adj}{\mathop{\rm adj}\nolimits}

\newcommand{\Cone}{{\rm Cone}}
\newcommand{\LAlb}{\mathop{\rm LAlb}\nolimits}
\newcommand{\RPic}{\mathop{\rm RPic}\nolimits}

\newcommand{\Set}{\mathop{\rm Set}\nolimits}

\newcommand{\et}{\mathrm{\acute{e}t}}
\newcommand{\eff}{\mathrm{eff}}

\newcommand{\gm}{\mathrm{gm}}
\newcommand{\op}{\mathrm{op}}

\newcommand{\coh}{\mathrm{coh}}
\newcommand{\homo}{\mathrm{hom}}

\newcommand{\red}{\mathrm{red}}
\newcommand{\sm}{\mathrm{sm}}

\newcommand{\gsm}{\mathrm{gsm}}
\newcommand{\sgsm}{\mathrm{sgsm}}

\newcommand{\Ex}{{\rm Ex}}
\newcommand{\tr}{{\rm tr}}
\newcommand{\perf}{{\rm perf}}

\newcommand{\an}{{\rm an}}
\newcommand{\psh}{{\rm psh}}

\newcommand{\sh}{{\rm sh}}

\input cyracc.def
\DeclareFontFamily{U}{russian}{}
\DeclareFontShape{U}{russian}{m}{n}
        { <5><6> wncyr5
        <7><8><9> wncyr7
        <10><10.95><12><14.4><17.28><20.74><24.88> wncyr10 }{}
\DeclareSymbolFont{Russian}{U}{russian}{m}{n}
\DeclareSymbolFontAlphabet{\mathcyr}{Russian}
\makeatletter
\let\@math@cyr\mathcyr
\renewcommand{\mathcyr}[1]{\@math@cyr{\cyracc #1}}

\newcommand{\Bei}{\mathcyr{B}}

\newcommand{\std}{{\rm std}}

\makeatletter
\DeclareFontFamily{OMX}{MnSymbolE}{}
\DeclareSymbolFont{MnLargeSymbols}{OMX}{MnSymbolE}{m}{n}
\SetSymbolFont{MnLargeSymbols}{bold}{OMX}{MnSymbolE}{b}{n}
\DeclareFontShape{OMX}{MnSymbolE}{m}{n}{
    <-6>  MnSymbolE5
   <6-7>  MnSymbolE6
   <7-8>  MnSymbolE7
   <8-9>  MnSymbolE8
   <9-10> MnSymbolE9
  <10-12> MnSymbolE10
  <12->   MnSymbolE12
}{}
\DeclareFontShape{OMX}{MnSymbolE}{b}{n}{
    <-6>  MnSymbolE-Bold5
   <6-7>  MnSymbolE-Bold6
   <7-8>  MnSymbolE-Bold7
   <8-9>  MnSymbolE-Bold8
   <9-10> MnSymbolE-Bold9
  <10-12> MnSymbolE-Bold10
  <12->   MnSymbolE-Bold12
}{}

\let\llangle\@undefined
\let\rrangle\@undefined
\DeclareMathDelimiter{\llangle}{\mathopen}%
                     {MnLargeSymbols}{'164}{MnLargeSymbols}{'164}
\DeclareMathDelimiter{\rrangle}{\mathclose}%
                     {MnLargeSymbols}{'171}{MnLargeSymbols}{'171}
\makeatother

\title{Triangulated categories of relative $1$-motives}
\author{Simon Pepin Lehalleur\thanks{The author thanks the University of Z\"urich, where he conducted his PhD thesis which gave rise to this paper and the Freie Universit\"at Berlin, where he finished the write-up of the work as a post-doctoral student. The author also acknowledges the support of the Einstein Foundation, through the Einstein visiting Fellowship 0419745104 ``Algebraic Entropy, Algebraic Cycles'' of Professor V. Srinivas.}}
\affil{Freie Universit\"at Berlin}
\begin{document}

\maketitle

\begin{abstract}
The category $\DA^{1}(S,\BQ)$ of relative cohomological $1$-motives is the localising subcategory of the triangulated category $\DA(S,\BQ)$ of relative Voevodsky motives with rational coefficients over a scheme $S$ which is generated by cohomological motives of curves over $S$. We construct and study a candidate for the standard motivic $t$-structure on $\DA^1(S,\BQ)$ (for $S$ noetherian, finite-dimensional and excellent). We show this $t$-structure is non-degenerate and relate its heart $\MM^1(S)$ with Deligne $1$-motives over $S$; in particular, when $S$ is regular, the category of Deligne $1$-motives embeds in $\MM^1(S)$ fully faithfully. We also study the inclusion of $\DA^1(S)$ into the larger category $\DA^{\coh}(S)$ of relative cohomological motives on $S$, and prove that its right adjoint $\omega^1$ preserves compact objects. 
\end{abstract}

\tableofcontents

\section*{Introduction}
\addcontentsline{toc}{section}{Introduction}

We now have at our disposal a mature theory of triangulated categories of motivic sheaves with rational coefficients over general base schemes. Here are some of its highlights.
\begin{itemize}
\item Given a noetherian finite-dimensional scheme $S$, there is a tensor triangulated $\BQ$-linear category $\DA(S)$.
\item There are \emph{realisation functors} from $\DA(S)$ to classical triangulated categories of coefficients: derived categories of abelian sheaves for the classical topology in the Betti setting \cite{Ayoub_Betti} and derived categories of $\ell$-adic sheaves in the $\ell$-adic setting \cite{Ayoub_Etale} \cite{CDEt}.
\item The assignement $S\mapsto \DA(S)$ has a rich functoriality leading to a ``formalism of Grothendieck operations'' \cite{Ayoub_these_1} \cite{Ayoub_these_2} (including nearby and vanishing cycles) which is compatible via realisation functors with the classical Grothendieck operations for constructible sheaves in the Betti setting and for $\ell$-adic sheaves in the $\ell$-adic setting \cite[Theoreme 3.19]{Ayoub_Betti}, \cite[Theoreme 9.7]{Ayoub_Etale}.
\item Morphisms groups in $\DA(S)$ are related to rational algebraic $K$-theory for $S$ regular \cite[Corollary 14.2.14]{Cisinski_Deglise_BluePreprint} and to Bloch's higher Chow groups when $S$ is smooth over a field \cite[Example 11.23.]{Cisinski_Deglise_BluePreprint}.
\item A natural finiteness condition leads to a subcategory $\DA_c(S)$ of ``constructible'' motivic sheaves, which is stable under Grothendieck operations and maps to the constructible derived categories of classical coefficients via realisations functors \cite[\S 8]{Ayoub_Etale} \cite[\S 15]{Cisinski_Deglise_BluePreprint}.
\item The category $\DA(S)$ can be constructed in several ways, each of which captures important aspects of the theory: motives without transfers \cite{Ayoub_Etale}, Beilinson motives $\DM_{\Bei}(S)$ \cite{Cisinski_Deglise_BluePreprint}, motives with transfers $\DM(S)$ \cite{Cisinski_Deglise_BluePreprint} (in the case $S$ is geometrically unibranch), $h$-motives $\DM_h(S)$ \cite{VHS} \cite{Cisinski_Deglise_BluePreprint} \cite{CDEt}, etc. In each of those cases, it is constructed as the homotopy category of a stable combinatorial dg-model category, hence $\DA(S)$ admits natural enhancements as a stable dg-category, a stable derivator and a stable $(\infty,1)$-category.
\item When $S$ is the spectrum of a perfect field $k$, the category $\DA(k)$ is in particular equivalent to $\DM(k)$, which gives access to Voevodsky's cancellation theorem \cite{Voevodsky_Cancellation} and to the theory of homotopy invariant sheaves with transfers \cite{RedBook} \cite{MVW}.
\end{itemize}
 
In view of those achievements, a major open question is the existence of the motivic $t$-structure on $\DA(S)$, whose heart would provide an abelian category of mixed motivic sheaves realising the conjectures of Beilinson \cite{Jannsen_t}. Here is one possible statement in terms of the $\ell$-adic realisation. 

\begin{conj}\label{conj:mot_t_struct}
Let $S$ be a noetherian finite-dimensional scheme and $\ell$ a prime invertible on $S$.
\begin{itemize}
\item The $\ell$-adic realisation functor $R_{\ell}:\DA_c(S)\ra D^b_c(S_\et,\BQ_{\ell})$ is conservative.
\item There exists a non-degenerate $t$-structure $t_{\MM}$ on $\DA_c(S)$ such that if we equip $D^b_c(S_\et,\BQ_{\ell})$ with its standard $t$-structure, the functor $R_{\ell}$ is $t$-exact.  
\end{itemize}
\end{conj}
The $t$-structure $t_{\MM}$ is uniquely determined by the compatibility with $R_{\ell}$ if it exists, because we include conservativity of realisations in the statement.

The case where $S$ is the spectrum of a field (say of characteristic $0$) is already extremely interesting; the conjecture in that case implies the Beilinson-Soul\'e vanishing conjecture  for $K$-theory, Grothendieck's standard conjectures on algebraic cycles \cite{Beilinson_standard} and the Bloch-Beilinson-Murre conjectures on the structure of Chow groups of smooth projective varieties \cite{Jannsen_t}. Moreover, a theorem of Bondarko \cite[Theorem 3.1.4]{Bondarko_t} shows that for a large class of schemes, if $t_{\MM}(K)$ exists for any residue field $K$ of such a scheme $S$, then the perverse analogue $^p t_{\MM}(S)$ of $t_{\MM}(S)$ exists, and one can then presumably reconstruct the standard motivic $t$-structure from the perverse one as in \cite[4.6.2]{Saito_MHM}. 

Since the general conjecture seems inaccessible, one looks for subcategories of $\DA(S)$ where one can hope to construct the restriction of the conjectural $t$-structure. For $n\in \BN$, we introduce the subcategory $\DA_n(S)$ of homological $n$-motives, i.e., the subcategory generated by homological motives of smooth $S$-schemes of relative dimension less than or equal to $n$. It seems reasonable to conjecture further that $t_{\MM}$ should restrict to a $t$-structure $t_{\MM,n}$ on $\DA_{n,c}(S)$. For $n\geq 2$, we have no idea how to construct $t_{\MM,n}$ even when $S$ is a field. Our goal is to provide a reasonable candidate for $t_{\MM,0}$ and $t_{\MM,1}$. 

For a perfect field $k$, the structure of $\DA_1(k)$ and $t_{\MM,1}$ have already been extensively studied. Here is a summary of the main results, transferred from the set-up of $\DM^{\eff}$ in the original papers to $\DA$ via the cancellation theorem and the comparison theorem of \cite[Corollary 16.2.22]{Cisinski_Deglise_BluePreprint} (for details on these results, we refer the reader to Sections \ref{sec:omega_field} and \ref{t_struct_field}).

\begin{theo}[Voevodsky, Orgogozo \cite{Orgogozo}, Barbieri-Viale-Kahn \cite{BVK}, Ayoub-Barbieri-Viale \cite{Ayoub_Barbieri-Viale}, Ayoub \cite{Ayoub_2-mot}]\label{theo:main_theo_k}
Let $k$ be a perfect field and $\ell$ a prime different from $\mathrm{char}(k)$. 
\begin{enumerate}[label={\upshape(\roman*)}]
\item There exists a non-degenerate $t$-structure $t_{\MM,1}$ on $\DA_1(k)$ which restricts to $\DA_{1,c}(k)$.  
\item There is an $t$-exact equivalence of triangulated categories
\[
\DA_{1,c}(k)\simeq D^b(\CM_1(k))
\]
where $\CM_1(k)$ is the abelian category of Deligne $1$-motives over $k$ with rational coefficients \cite{Deligne_HodgeIII}.
\item The $\ell$-adic realisation functor $R_\ell :\DA_{1,c}(k)\ra D(k_{\et},\BQ_l)$ is conservative and $t$-exact.
\item The inclusion of $\DA_1(k)$ into the category $\DA_{\homo}(k)$ of all homological motives admits a left adjoint, the ``motivic Albanese functor'' $\LAlb:\DA_{\homo}(k)\ra \DA_1(k)$, which sends constructible objects to constructible objects, and whose value on the motive of a smooth $k$-variety $X$ is closely related to its semi-abelian Albanese variety.
\end{enumerate}
\end{theo}

Our work builds on these results and the six operations formalism to produce a similar picture for $\DA_1(S)$. 

The most natural approach to a motivic $t$-structure on $\DA_1(S)$ would proceed by combining the $t$-structures on $\DA_1(s)$ provided by the previous theorem for all points $s$ of $S$ to a $t$-structure on $\DA_{1}(S)$, i.e., proving that the subcategories $\DA_{1}(S)_{\geq 0} := \{M\in \DA_{1}(S)|\forall s\in S,\ s^*M\in \DA_{1}(s)_{\geq 0}\}$ and $\DA_{1}(S)_{\leq 0} := \{M\in \DA_{1}(S)|\forall s\in S,\ s^*M\in \DA_{1}(s)_{\leq 0}\}$ form a $t$-structure, which would then automatically be compatible with standard $t$-structures on target categories of realisation functors when they are defined. We do not know how to prove this in general, even when restricting to subcategories of compact objects; the gluing arguments of \cite[\S 3.2]{Bondarko_t} are tailored for ``perverse'' $t$-structures and cannot be applied directly. We refer however to \cite{Vaish_IC3} for a different approach to the motivic $t$-structure for $1$-motives via gluing.

We thus implement an alternative approach, which is inspired by another description \cite[Proposition 3.7]{Ayoub_2-mot} of $t_{\MM,1}(k)$ for a perfect field $k$, as a generated $t$-structure in the sense of \cite[Definition 2.1.71]{Ayoub_these_1}. This leads to a $t$-structure $t_{\MM,1}(S)$ on $\DA_1(S)$ (Definition \ref{main_def_1}). Let us write $\MM_1(S)$ for its heart.

\begin{theo}[\ref{prop:t_non_degenerate}, \ref{theo:gen_in_heart}, \ref{theo:deligne_full}, \ref{cor:motive_group_heart}]
Let $S$ be a noetherian finite-dimensional excellent scheme.
\begin{enumerate}[label={\upshape(\roman*)}]
\item If $S$ is the spectrum of a perfect field $k$, the $t$-structure $t_{\MM,1}(k)$ coincides with the $t$-structure of Theorem~\ref{theo:main_theo_k}.
\item The $t$-structure $t_{\MM,1}(S)$ is non-degenerate.
\item Write $\CM_1(S)$ for the $\BQ$-linear category of Deligne $1$-motives over $S$ with rational coefficients. The natural functor $\Sigma^\infty:\CM_1(S)\ra \DA(S)$ factors through $\MM_1(S)$, and is fully faithful if $S$ is regular.
\item Let $G$ be a smooth commutative group scheme with connected fibres. Then the motive $\Sigma^\infty G_\BQ[-1]$ is in $\MM_1(S)$.
\end{enumerate}
\end{theo}

% The problem with the definition of $t_{\MM,1}$ as generated $t$-stucture is that it is not easy to show that certain functors expected to be $t$-exact are $t$-exact, and that it restricts to $\DA_{1,c}(S)$. In the case when $S$ is of characteristic $0$, using resolution of singularities, we can establish much stronger results, which in particular show that the gluing approach sketched above works at least in this case.

% \begin{theo}[\ref{}]
% Let $S$ be a noetherian finite-dimensional quasi-excellent $\BQ$-scheme.
% \begin{enumerate}[label={\upshape(\roman*)}]
% \item The $t$-structure $t_{\MM,1}$ restricts to $\DA_{1,c}(S)$.
% \item For $f:T\ra S$, the functor $f^*:\DA_{1,c}(S)\ra \DA_{1,c}(T)$ is $t$-exact.
% \item Let $M\in \DA_{1,c}(S)$. Then
% \[
% M\in \DA_{1,c}(S)\in \DA_{1,c}(S)_{\geq 0} (resp. \DA_{1,c}(S)_{\leq 0}) \simeq \forall s\in S,\ s^*M\in  \DA_{1,c}(s)_{\geq 0} (resp. \DA_{1,c}(s)_{\leq 0}) 
% \]
% \item Let $\ell$ be a prime number invertible on $S$. Then $R_{\ell}:\DA_{1,c}\ra D_c(S_\et,\BQ_\ell)$ is $t$-exact. 
% \end{enumerate}
% \end{theo}

The result (iv) on $\Sigma^\infty G_{\BQ}[-1]$ was announced in \cite{AHPL}; there, this motive appeared as the $H_1$ piece in a "K\"unneth-type" decomposition of the homological motive $M_S(G)$ \cite[Theorem 3.3]{AHPL}.

In the relative situation, it is unclear whether the left adjoint $\LAlb$ of the inclusion $\DA_1(S)\ra \DA_{\homo}(S)$ actually exists. We can however define a motivic analogue of the \emph{Picard scheme}. We have a category $\DA^1(S)$ of \emph{cohomological $1$-motives} (resp. $\DA^{\coh}(S)$ of cohomological motives) and it turns out that $\DA^1(S)=\DA_1(S)(-1)$ (Proposition~\ref{prop:hom_cohom_twists}), so that $\DA^1(S)$ also has a motivic $t$-structure $t^1_{\MM}=t_{\MM,1}(-1)$ which satisfies analogues of the theorems above. The inclusion $\DA^1(S)\ra \DA^{\coh}(S)$ admits a right adjoint $\omega^1:\DA^{\coh}(S)\ra \DA^1(S)$, as a corollary of Neeman's version of Brown representability for compactly generated categories; however, unlike most right adjoints constructed this way, $\omega^1$ satisfies a strong finiteness property.

\begin{theo}[\ref{thm:finiteness_omega1}]
Let $S$ be a noetherian finite-dimensional excellent scheme satifying resolution of singularities by alterations. Then $\omega^1$ sends compact objects to compact objects.  
\end{theo}

The main step of the proof is to compute $\omega^1$ in a special case, namely $\omega^1(f_*\BQ_X)$ with $S$ regular and $f:X\ra S$ smooth projective ``Pic-smooth'' (Definition~\ref{def:pic_smooth}). In this case, Theorem~\ref{theo:omega_1_Pic_smooth} shows that
\[
\omega^1 (f_*\BQ_X)\simeq \Sigma^\infty P(X/S)(-1)[-2]
\]
where $P(X/S)$ is the \emph{Picard complex} of $f$, an object closely related to the Picard scheme of $X$ over $S$. Another proof of Theorem~\ref{thm:finiteness_omega1} is given in \cite{Vaish_IC3} (for schemes of finite type over a field, but the arguments carry through in our context).

The results above on $t_{\MM,1}$ and $\omega^1$ also have (simpler) counterparts for the category $\DA_0(S)$ of $0$-motives and for the functor $\omega^0$, which we establish along the way.

The two main questions which this work leaves open are whether the $t$-structure $t_{\MM,1}$ restricts to compact objects and whether the resulting $t$-structure on $\DA_{1,c}(S)$ satisfies the analogue of Conjecture~\ref{conj:mot_t_struct}, i.e., whether the $\ell$-adic realisation on $\DA_{1,c}(S)$ is then $t$-exact. In \cite{constructible_1-mot}, we answer both questions positively. The fact that $t_{\MM,1}$ restricts to compact objects has also been established, with a different argument, in the preprint~\cite{Vaish_IC3}. 

We have chosen to work with motives with rational coefficients. The theory of triangulated categories of motives with integral coefficients naturally splits in two: a ``Nisnevich'' version and an ``étale'' version, depending on what topology we want to have descent for. Voevodsky already observed that the Nisnevich category $\DM_{\Nis}(S,\BZ)$ does not admit a motivic $t$-structure, even when $S$ is the spectrum of a field. Let $\Lambda$ be a ring of coefficients. Then, if we make the assumption that every prime is invertible either in $\Lambda$ or in $\mathcal{O}_S$, the category $\DA_{\et}(S,\Lambda)$ is rather well understood. The key statement is the relative rigidity theorem of Ayoub \cite[Theoreme 4.1]{Ayoub_Etale} which roughly tells us that the category of \'etale motives with torsion coefficients coincides with the derived category of torsion \'etale sheaves. Building on this, one can show that with these hypotheses, the motivic $t$-structure on $\DA_{\et}(S,\Lambda)$ exists if and only if it exists for $\DA_{\et}(S,\Lambda\otimes\BQ)$. More specifically for $1$-motives over a perfect field $k$ with exponential characteristic $p$, the motivic $t$-structure on $\DA_{\et,1}(k,\BZ[\frac{1}{p}])$ was constructed in \cite[Remark 2.7.2]{BVK}. It seems likely that the ideas of \cite{BVK} on $1$-motives with torsion and the relative rigidity theorem can be combined with the methods of this paper to give a satisfactory theory of $t_{\MM,1}$ and $\omega^1$ for relative \'etale motives with integral coefficients.

\subsection*{Structure of this paper}
Let $S$ be a finite-dimensional noetherian scheme. In Section~\ref{sec:subcats}, we introduce the categories $\DA_n(S)$ of homological $n$-motives (resp. $\DA^n(S)$ of cohomological $n$-motives) which are full subcategories of $\DA(S)$ generated as triangulated categories with small sums by homological (resp. cohomological) motives of smooth (resp. proper) $S$-schemes of relative dimension less or equal to $n$ (Definition \ref{def:subcats}). We then study their permanence properties under Grothendieck operations (Propositions \ref{prop:n_monoidal} to \ref{prop:permanence_hom_n}) and prove that the homological and cohomological variants are closely related (Proposition \ref{prop:hom_cohom_twists}).

In Section~\ref{sec:albanese}, we study the motives associated to smooth commutative group schemes over $S$ and prove that they live in $\DA_1(S)$ (Proposition~\ref{prop:smooth_complex_compact}). We also study motives attached to Deligne $1$-motives. Finally, we introduce a motive attached to what we call the Picard complex $P(X/S)$ of a morphism of schemes $f:X\ra S$. It is an object in a derived category of sheaves which packages together information about the relative connected components of $f$ and the Picard scheme of $X/S$; in some cases, $P(X/S)$ yields a motive in $\DA_{1,c}(S)$ (Corollary~\ref{cor:Picard_smooth}).

In Section~\ref{sec:mot_picard}, we introduce and study the right adjoint $\omega^1:\DA^\coh(S)\ra \DA^1(S)$ to the embedding of cohomological $1$-motives into cohomological motives. We first establish a number of relatively formal results involving its commutation properties with the six operations (Proposition \ref{prop:omega_basics}). The main result is then that $\omega^1$ preserves constructibility (Theorem~\ref{thm:finiteness_omega1}). This relies on combining techniques from \cite{Ayoub_Zucker} with a computation of $\omega^1(f_*\BQ_X)$ in a favorable situation: the precise statement involves the motive of the Picard complex from the previous section.

In Section~\ref{sec:t_structure}, we finally introduce a candidate for the motivic $t$-structure on $\DA_1(S)$ and $\DA^1(S)$, using the formalism of generated $t$-structures. A number of equivalent generating families can be used for this purpose (see Definition \ref{def:generators}). We prove some basic exactness properties for the six operations. The main result we show is that motives attached to Deligne $1$-motives lie in the heart $\MM_1(S)$ of the $t$-structure on $\DA_1(S)$, and that the category $\CM_1(S)$ embeds fully faithfully into $\MM_1(S)$ for $S$ regular. Most results in this section require the additional assumption that $S$ is excellent.

Appendix \ref{sec:app_deligne} provide technical results about Deligne $1$-motives over a general base. Appendix \ref{sec:app_mot_coh} gathers some computations of motivic cohomology groups for $\BQ(0)$ and $\BQ(1)$ which are used at several places in the text.

\section*{Acknowledgements}
This work is based on the main part of my PhD thesis, done under the supervision of Joseph Ayoub at the University of Z\"urich. I would like to thank him dearly for his constant support, both mathematical and personal. 

My thesis was reviewed by Annette Huber and Mikhail Bondarko, and I thank them for their comments. I would also like to thank Giuseppe Ancona, Luca Barbieri-Viale, Javier Fresan, Annette Huber, Shane Kelly, Andrew Kresch, Michel Raynaud, Vaibhav Vaish and Alberto Vezzani for discussions or e-mail exchanges on the topic of this paper. Finally, I thank the reviewer of the first submitted version for his or her extremely thorough reading, for catching several mistakes and providing many suggestions for improvement.

\section*{Background, conventions and notations}
\label{sec:conv_not}
\addcontentsline{toc}{section}{Conventions and notations}
We collect here several conventions and pieces of notation which will be used throughout this paper.

When considering several variants of a category in parallel, distinguished by a decoration, we put the decoration in parenthesis. For instance, when considering categories of both compact and non-compact motives, we write $\DA_{(c)}(S)$. 

\subsection*{Homological algebra in abelian and triangulated categories}

When discussing complexes in abelian categories and t-structures on triangulated categories, we consistently use homological indexing conventions. 

Let $F:\CT\ra \CT'$ be a triangulated functor between triangulated categories with $t$-structures. We say that $F$ is $t$-positive or right $t$-exact (resp. $t$-negative or left $t$-exact) if $F(\CT_{\geq 0})\subset \CT'_{\geq 0}$ (resp. $F(\CT_{\leq 0}) \subset \CT'_{\leq 0})$. 

Let $\CT$ be a triangulated category, and $\CG$ be a family of objects of $\CT$. We introduce a number of subcategories of $\CT$ generated in various ways by $\CG$. Recall that a triangulated subcategory is said to be thick (resp. localising) if it is stable by direct factors (resp. small sums).

We denote by $\langle \CG\rangle$ (resp. $\langle \CG\rangle_+$, $\langle \CG\rangle_-$)  the smallest thick triangulated subcategory of $\CT$ (resp. the smallest subcategory stable by extensions, positive shifts and direct factors, the smallest subcategory stable by extensions, negative shifts and direct factors) containing $\CG$. Assume now that $\CT$ admits small sums; by convention, this includes the hypothesis that small sums of distinguished triangles are distinguished triangles. We denote by $\llangle \CG\rrangle$ (resp. $\llangle \CG\rrangle_+$, $\llangle \CG\rrangle_-$)  the smallest localising triangulated subcategory of $\CT$ (resp. the smallest subcategory stable by extensions, small sums and $[+1]$, the smallest subcategory stable by extensions, small sums and $[-1]$) containing $\CG$. Note that $\langle \CG\rangle \subset \llangle \CG \rrangle$ by \cite[Lemme 2.1.17]{Ayoub_these_1}.

In the constructions above, we refer informally to $\CG$ as the \emph{generating family} and to objects of $\CG$ as \emph{generators}. In each case, these subcategories can be defined by an induction (transfinite in the $\llangle -\rrangle$ cases): start with the full subcategory with objects $\CG[\BZ]$; to pass to a successor ordinal, introduce, depending on the case, cones of all morphisms and direct factors of all objects, just the cones and direct factors, just the cocones and direct factors, the cones, direct factors and small sums, etc.; finally, to pass to a limit ordinal, take the union over all previous subcategories. These subcategories do not change if one replaces $\CT$ by a triangulated subcategory containing $\CG$ and which is stable by direct factors in the $\langle-\rangle$ case and stable under small sums in the $\llangle -\rrangle$ case. In practice, this means we do not need to specify the ambient triangulated category.

We adopt the notational convention that functors between triangulated categories are triangulated by default, i.e., we write $f_*$ for $Rf_*$, $f^*$ for $Lf^*$, $\otimes$ for $\otimes^L$, $a_\tr$ for $La_\tr$, etc. In the few cases where we need to refer to the ``underived functor'', that is, the underlying Quillen functor at the level of model categories, we underline the notation, i.e., we write $\underline{\smash{f}}_*$, $\underline{\smash{f}}^*$, $\underline{\otimes}$, $\underline{a}_{\tr}$, etc.

\subsection*{Schemes and group schemes}
Unless specified, all schemes are noetherian and finite-dimensional. The notation $\Sm/S$ (resp. $\Sch/S$) denotes the category of all smooth $S$-schemes (resp. all separated, locally of finite type $S$-schemes), usually considered as a site with the \'etale topology.

A geometric point of a scheme $S$ is a morphism $\bar{s}:\Spec(k)\ra S$ with $k$ an algebraically closed field.

\begin{defi}
A morphism $f:X\ra S$ between noetherian schemes is an alteration if it is proper, surjective, generically finite, and if the union of the fibers of $f$ above the finitely many generic points of $S$ is dense in $X$ (this is implied by the first three conditions if $X$ and $S$ are integral).

We say that a scheme $S$ admits the resolution of singularities by alterations if for any separated $S$-scheme $X$ of finite type and any nowhere dense closed subset $Z\subset X$, there is a projective alteration $g:X'\ra X$ with $X'$ regular and such that $g^{-1}(Z)$ is a strict normal crossing divisor.  
\end{defi}
 The best result available in this direction is due to Temkin \cite[Theorem 1.2.4]{Temkin_distillation}: any $S$ which is of finite type over a quasi-excellent scheme of dimension $\leq 3$ satisfies resolution of singularities by alterations.

Let us recall basic terminology and facts about exact sequences of group schemes. Let 
\[(C):0\ra G'\stackrel{i}{\ra} G\stackrel{p}{\ra} G''\ra 0\]
 be a sequence of commutative group schemes over a scheme $S$. We say that $(C)$ is exact if it induces an exact sequence of fppf sheaves on $\Sch/S$. If $(C)$ is exact, then $G'$ is the scheme-theoretic kernel of $p$ and $p$ is a surjective morphism of schemes. In the other direction, if $p$ is an fppf epimorphism and $G'$ is its scheme-theoretic kernel, then $(C)$ is exact. Moreover, if the group scheme $G'$ is smooth over $S$, then one obtains an equivalent definition by replacing the fppf topology with the \'etale topology. Indeed, $G\ra G''$ is an $G'$-torsor, because the action of $G'$ on $G$ is free, and an fppf torsor under a smooth group is also an étale torsor, because a smooth surjective morphism has local sections in the étale topology \cite[17.16.3 (ii)]{EGAIV_4}.

\subsection*{Triangulated categories of motives}

We work in the context of the stable homotopical $2$-functor $\DA^\et(-,\BQ)$ considered in \cite[\S 3]{Ayoub_Etale}. Most results in the paper are still valid, with the same proofs, for $\DA^{\et}(-,R)$ with $R$ a $\BQ$-algebra; however, we stick to $R=\BQ$ for simplicity.

 Since we only consider the \'etale topology and rational coefficients, we simplify the notation and write $\DA(S)$ for $\DA^\et(S,\BQ)$. The category $\DA(S)$ is equivalent to several other triangulated categories of motives with rational coefficients, e.g. Beilinson motives \cite{Cisinski_Deglise_BluePreprint}: see \cite[\S 16]{Cisinski_Deglise_BluePreprint} for various comparison theorems. 

By \cite{Ayoub_these_1}, the system of categories $\DA(-)$ admits the functoriality of the Grothendieck six operations. In particular, for any quasi-projective morphism $f:S\ra T$ of schemes, Ayoub constructs adjoint pairs 
\[f^*:\DA(T)\leftrightarrows \DA(S):f_*\]
\[f_!:\DA(S)\leftrightarrows \DA(T):f^!\] 
and when $f$ is smooth 
\[f_\sharp:\DA(S)\leftrightarrows \DA(T): f^*.\] 
There is a morphism of functors $f_!\ra f_*$, which is an isomorphism for $f$ projective. Given a smooth $S$-scheme $f:X\ra S$, we also write $M_{S}(X)$ for the homological motive $f_{\sharp}f^{*}\BQ_{S}\in\DA(S)$.

Note that for those operations, as well as for the pullbacks and pushforwards functors on derived categories of sheaves on $\Sm/-$, the notation $f^*$, $f_*$, $\ldots$ stands for the triangulated or derived functors. When we want to use the underived functor, we underline the functor: $\underline{f}^*, \underline{f} _*,\ldots$

In the definitions of the Grothendieck operations, one can relax the condition that $f$ is quasi-projective in the following ways.

\begin{enumerate}[label={\upshape(\roman*)}]
\item As observed in \cite[Appendice 1.A]{Ayoub_rigide}, one can define $f^*$ and $f_*$ for any morphism $f$ (without any finiteness hypothesis), and prove for instance that proper base change \cite[Proposition~3.5]{Ayoub_Etale}, the $\Ex_\sharp^*$ isomorphism \cite[Proposition~3.6]{Ayoub_Etale} and ``regular base change'' \cite[Corollaire 1.A.4]{Ayoub_rigide} still hold.
\item As observed in \cite[Theorem~2.2.14]{Cisinski_Deglise_BluePreprint}, one can define the exceptional functors $f_!$ and $f^!$ for any $f$ separated of finite type, and prove that all the properties in \cite{Ayoub_these_1} still hold (in particular with $f_!\simeq f_*$ for any $f$ proper). 
\end{enumerate}

We freely use these more general constructions and results.

The six operations for $\DA(-)$ satisfy a large number of properties and compatibilities (see \cite[Proposition~3.2]{Ayoub_Etale}, \cite[Scholie 1.4.2]{Ayoub_these_1}). For results which come up repeatedly in this paper, we introduce the following terminology. Let
\[
\xymatrix{
Z\ar[r]^{\tilde{g}} \ar[d]_{\tilde{f}} & X \ar[d]^{f}\\
W \ar[r]_{g} & Y
}
\]
be a cartesian square of morphisms of schemes.
\begin{itemize}
\item By the $\Ex^*_\sharp$ isomorphism (resp. the $\Ex^!_*$ isomorphism, the $\Ex^*_!$ isomorphism), we mean the natural isomorphism $\tilde{f}_\sharp \tilde{g}^*\stackrel{\sim}{\lra} g^*f_\sharp$ for $f$ smooth (resp. the natural isomorphism $\tilde{f}_*\tilde{g}^!\stackrel{\sim}{\lra} g^! f_*$, the natural isomorphism $g^*f_!\stackrel{\sim}{\lra} \tilde{f}_! \tilde{g}^*$).  
\item By ``smooth base change'', we mean the natural isomorphism $g^* f_*\stackrel{\sim}{\lra} \tilde{f}_*\tilde{g}^*$ for $g$ smooth.
\item By ``proper base change'', we mean the natural isomorphism $g^*f_*\stackrel{\sim}{\lra} \tilde{f}_* \tilde{g}^*$ for $f$ proper, and its generalisations $g^*f_!\stackrel{\sim}{\lra} \tilde{f}_! \tilde{g}^*$ and $f^{!}g_{*}\stackrel{\sim}{\lra}\tilde{g}_{*}\tilde{f}^{!}$ for $f$ separated of finite type.
\item Let $i:Z\rightarrow
X$ be a closed immersion and $j:U\rightarrow X$ be the complementary open
immersion. When we write ``by localisation'', we mean the use of
the distinguished triangle of functors 
\[
j_\sharp j^*\rightarrow \id\rightarrow i_*i^*\stackrel{+}{\rightarrow}.
\]
Dually, when we write ``by colocalisation'', we mean the use
of the distinguished triangle of functors
\[
i_!i^!\rightarrow \id\rightarrow j_*j^*\stackrel{+}{\rightarrow}.
\]
\item By ``relative purity'', we mean the fact that for any smooth morphism $f:S\ra T$ of pure relative dimension $d$, there are  isomorphisms of functors $f_!\simeq f_{\sharp}(d)[2d]$ and $f^!\simeq f^*(-d)[-2d]$.
\item By ``the separation property for $\DA$'', we mean the fact that for any surjective morphism of finite type (resp. any finite surjective radicial morphism) $f:S\ra T$, the functor $f^*:\DA(T)\ra \DA(S)$ is conservative (resp. an equivalence of categories) \cite[Th\'eor\`eme 3.9]{Ayoub_Etale}.
\item By ``absolute purity'', we mean the fact that for any regular immersion $i:S\ra T$ of pure codimension $d$, we have $i^!\BQ_T\simeq \BQ_S(-d)[-2d]$ (\cite[Corollaire 7.5]{Ayoub_Etale} and \cite[Remarque 11.2]{Ayoub_Etale}).
\item By ``cohomological $h$-descent'', we mean the fact that for any finite type morphism $f:S\ra T$ of quasi-excellent schemes and any hypercover $\pi_\bullet:S_{\bullet}\ra S$ in Voevodsky's $h$-topology, the natural morphism of functors
\[
f_*f^*(-) \ra f_*(\pi_\bullet)_*\pi_\bullet^*f^*(-)
\]
(which is part of the algebraic derivator structure for $\DA(-)$) is an isomorphism \cite[Theorem 14.3.4]{Cisinski_Deglise_BluePreprint}. In particular, we apply this in the case $f=\id$ and through the induced descent spectral sequence for morphisms groups in $\DA(-)$; namely, for such an hypercover $\pi_\bullet$ and for any motives $M,N\in \DA(S)$, there is a cohomological spectral sequence
\[
E^{p,q}_1=\DA(S_p)(\pi_p^*M,\pi_p^*N[q])\Rightarrow \DA(S)(M,N[p+q]).
\]
Note that this spectral sequence is only contained a priori in the right half-plane and so is not guaranteed to converge in general.

\end{itemize}

We also need some functoriality properties for categories of (effective) motives with transfers. For any noetherian finite-dimensional scheme $S$, we have tensor triangulated $\BQ$-linear categories $\DM^{(\eff)}(S)$. By \cite[\S 11.1.a.]{Cisinski_Deglise_BluePreprint}, when $S$ vary, these acquire the structure of a ``premotivic category'' in the sense of loc. cit.; in particular, for any morphism $f:T\ra S$, there are adjunctions
\[
f^*:\DM^{(\eff)}(S) \leftrightarrows \DM^{(\eff)}(T):f_*
\]
and, when $f$ is smooth, there are adjunctions
\[
f_{\sharp}:\DM^{(\eff)}(T) \leftrightarrows \DM^{(\eff)}(S):f^*
\]
These satisfy a smooth base change and a smooth projection formula. We write $\BQ_{S}^{\tr}$ for the monoidal unit of $\DM^{(\eff)}(S)$ and, for $f:X\ra S$ a smooth morphism, we write $M^{(\eff),\tr}_{S}(X)$ for the homological motives $f_{\sharp}f^{*}\BQ^{\tr}_{S}\in\DM^{(\eff)}(S)$.

\section{Triangulated categories of \texorpdfstring{$n$}a-motives}
\label{sec:subcats}

Categories of motives are naturally filtered by the dimension of ``geometric generators'', and such filtrations have been studied in various motivic contexts \cite{Beilinson_generic} \cite{Ayoub_Barbieri-Viale} \cite{Ayoub_2-mot}. We give definitions in the context of $\DA(-)$ and prove a number of basic results. Since such a treatment does not appear in the literature, we study a more general situation than is necessary for the rest of the paper; outside of this section, we are concerned with the special case of (co)homological $0$- and $1$-motives. Note that some of our results on the operations for cohomological motives are also discussed in \cite[\S 3.1]{Vaish_EM}.

\subsection{Definitions}
\label{sec:n_defs}

We fix a (noetherian, finite-dimensional) base scheme $S$ and an integer $n\geq 0$ for the remainder of this section.

\begin{defi}
\label{def:subcats}
  The category $\DA^{\coh}(S)$
  (resp. $\DA_{\homo}(S)$) of \emph{cohomological motives} (resp.
  \emph{homological motives}) is the full subcategory of $\DA(S)$ defined as
\[
\DA^{\coh}(S)=\ \llangle f_*\BQ_X|\mbox{ $f:X\rightarrow S$ proper morphism}\rrangle
\]
(resp.
\[
\DA_{\homo}(S)=\ \llangle f_{\sharp}\BQ_X|\mbox{ $f:X\rightarrow S$ smooth morphism}\rrangle).
\]
The category $\DA^n(S)$ (resp. $\DA_n(S)$) of \emph{cohomological $n$-motives} (resp. \emph{homological $n$-motives}) is the full subcategory of $\DA(S)$ defined as
\[
\DA^n(S)=\ \llangle f_*\BQ_X|\mbox{ $f:X\rightarrow S$ proper morphism of relative dimension $\leq n$}\rrangle
\]
(resp. 
\[
\DA_n(S)=\ \llangle f_\sharp\BQ_X|\mbox{ $f:X\rightarrow S$ smooth morphism of relative dimension $\leq n$}\rrangle).
\]
\end{defi}

\begin{remark}
\label{rmk:hom-n_coh-n}
As we will see in Proposition~\ref{prop:hom_cohom_twists}, the categories $\DA_{n}(S)$ and $\DA^n(S)$ are in fact equivalent as triangulated categories, so that many questions about $\DA^n(S)$ can be reduced to $\DA_n(S)$. In the special cases $n=0,1$, this is a crucial ingredient for several results in this paper. However to establish Proposition~\ref{prop:hom_cohom_twists} we need to study $\DA_n$ and $\DA^n$ in parallel.
\end{remark}

We have subcategories of smooth and geometrically smooth objects. Recall that an object $X$ in a symmetric monoidal category is said to be strongly dualisable if there exists an object $X^\vee$ together with morphisms $\epsilon:\un\ra X\otimes X^\vee$ and $\eta: X^{\vee}\otimes X\ra \un$ satisfying the classical adjunction triangle laws.

\begin{defi}
\label{def:sm_subcats}
The category $\DA^{\sm}(S)$ (resp, $\DA^{\coh}_{\sm}(S)$, $\DA_\homo^\sm(S)$) of \emph{smooth motives} (resp. \emph{smooth cohomological motives}, \emph{smooth homological motives}) is defined as
\[
\DA^{\sm}(S)=\ \llangle M\in\DA(S)| \mbox{ $M$ strongly dualisable }\rrangle
\]
(resp.
\[
\DA^{\coh}_{\sm}(S)=\ \llangle M\in \DA^{\coh}(S)| \mbox{ $M$ strongly dualisable in $\DA(S)$}\rrangle,
\]
\[
\DA_\homo^\sm(S)=\ \llangle M\in \DA_{\homo}(S)| \mbox{ $M$ strongly dualisable in $\DA(S)$}\rrangle).
\]
The category $\DA^{\gsm}(S)$ (resp. $\DA^{\coh}_{\gsm}(S)$, $\DA_\homo^\gsm(S)$) of \emph{geometrically smooth motives} (resp. \emph{geometrically smooth cohomological motives}
  resp. of \emph{geometrically smooth homological motives}) is the full subcategory of $\DA(S)$ defined as  
\[
\DA^{\gsm}(S)=\ \llangle f_\sharp\BQ_X(-n)|\mbox{ $f:X\rightarrow S$ proper smooth morphism},\ n\in\BZ\rrangle
\]
(resp.
\[
\DA^{\coh}_{\gsm}(S)=\ \llangle f_*\BQ_X|\mbox{ $f:X\rightarrow S$ proper smooth morphism}\rrangle,
\]
\[
\DA_\homo^\gsm(S)=\ \llangle f_\sharp\BQ_X|\mbox{ $f:X\rightarrow S$ proper smooth morphism}\rrangle).
\]
We then define their subcategories of $n$-motives as
\[\DA^n_{\gsm}(S)=\ \llangle \DA^n(S)\cap \DA^{\coh}_{\gsm,c}(S)\rrangle \]
\[\DA^{\gsm}_n(S)=\ \llangle \DA_n(S)\cap \DA^{\gsm}_{\homo,c}(S)\rrangle\]
\[\DA^n_{\sm}(S)=\ \llangle \DA_n(S)\cap \DA^{\coh}_{\sm,c}(S)\rrangle \]
\[\text{ and }\DA^{\sm}_n(S)= \llangle \DA_n(S)\cap \DA^{\sm}_{\homo,c}(S)\rrangle.\]
We also have categories of \emph{strongly geometrically smooth $n$-motives}
\[
\DA^{n}_{\sgsm}(S)=\ \llangle f_*\BQ_X|\mbox{ $f:X\rightarrow S$ proper smooth morphism of relative dimension }\leq n\rrangle,
\]
\[
\DA_n^\sgsm(S)=\ \llangle f_\sharp\BQ_X|\mbox{ $f:X\rightarrow S$ proper smooth morphism of relative dimension }\leq n\rrangle).
\]

\end{defi}

\begin{remark}
We have $\DA_{n}^{\sgsm}(S)\subset \DA_{n}^{\gsm}(S)$. Deciding whether this is an equality seems difficult, although we can prove this when $S$ is the spectrum of a field, see Proposition~\ref{prop:subcats_field}. Our motivation for introducing geometrically smooth $1$-motives is that the notion of strongly geometrically smooth $1$-motives is too strong for what we can actually establish about motives attached to Deligne $1$-motives, as the proof of Corollary~\ref{cor:Deligne_da_1} below shows. 
\end{remark}

\begin{lemma}
\label{lem:gsm_sm}
Geometrically smooth objects are smooth: we have $\DA^\gsm(S)\subset \DA^\sm(S)$, \\ $\DA^\gsm_\homo(S)\subset \DA^\sm_\homo(S)$, etc.
\end{lemma}
\begin{proof}
This result is due to Riou \cite{Riou} in the case of the stable motivic homotopy category, and the same proof applies to $\DA$. One can also look at \cite[Lemma 4.2.8]{CDEt}. 
\end{proof}

\begin{remark}
Proposition~\ref{prop:subcats_field} below shows that when $S$ is the spectrum of a field, any motive is geometrically smooth.
  
It is not clear if one should expect $\DA^{\sm}(S)$ (resp. $\DA^{\sm}_{\homo}$, $\DA^{\sm}_{n}$, etc.) to be generated by motives coming from smooth projective morphisms. Informally, when $S$ is a discrete valuation ring, it would mean that a ``motive with good reduction'' is always realisable in the cohomology of a variety with good reduction.

 There is a further reasonable definition of a smooth-like object in $\DA_c(S)$, namely a motive whose realisations have cohomology sheaves which are local systems (in the appropriate sense, e.g. lisse $\ell$-adic sheaves). This is conjecturally equivalent to being strongly dualisable; this equivalence would follow from the conservativity of realisation functors.
\end{remark}

An important property of smooth compact objects is that they satisfy a form of absolute purity.

\begin{prop}\label{prop:sm_abs_purity}
 Let $i:Z\ra S$ be an immersion. For $M\in \DA^{\sm}(S)$ and any $N\in \DA(S)$, there is an isomorphism 
\[
i^*M\otimes i^!N \simeq i^!(M\otimes N)
\]
which is functorial in $M$ and $N$, so that in particular, for any $f:M\ra M'\in \DA^{\sm}_c(S)$ the diagram
\[
\xymatrixcolsep{4pc}
\xymatrix{
i^* M\otimes i^!N \ar[r]_{i^*(f)} \ar[d] & i^* M' \otimes i^!N\ar[d] \\
i^! (M\otimes N) \ar[r]_{i^!(f)} & i^! (M'\otimes N)
}
\]
commutes. If $i$ is a regular immersion of codimension $c$, we have a functorial purity isomorphism
\[
i^*M\simeq i^!M(c)[2c].
\]
\end{prop}
\begin{proof}
  We first reduce to the case of closed immersions. Since we work with noetherian schemes, $i$ is quasi-compact, so that we can write $i=\bar{\imath}\circ j$ with $j$ an open immersion and $\bar{\imath}$ a closed immersion \cite[01QV]{stacks-project}. We then have a natural isomorphism $j^{*}\simeq j^{!}$ and $j^{*}$ is monoidal, so that
  \[
i^{*}M\otimes i^{!}N\simeq j^{*}(\bar{\imath}^{*}M\otimes \bar{\imath}^{!}N)\text{  and  }i^{!}(M\otimes N)\simeq j^{*}\bar{\imath}^{!}(M\otimes N).
    \]
We can thus assume that $i$ is a closed immersion. In this case, by \cite[Lemme 2.3.12]{Ayoub_these_1}, there exists for any $M,N\in \DA(S)$ a map
  \[
i^*M\otimes i^{!}N\ra i^{!}(M\otimes N)
    \]
which by \cite[Lemme 2.3.10, Proposition 2.1.103]{Ayoub_these_1} is defined as the composition
    \[
i^*M\otimes i^!N\stackrel{\sim}\ra i^!i_*(i^*M\otimes i^!N)\stackrel{\sim}{\leftarrow}i^!(M\otimes i_*i^!N)\ra i^!(M\otimes N)
      \]
where the first arrow is induced by the unit of the adjunction $(i_!=i_*,i^!)$ (invertible because $i_*$ is fully faithful), the second arrow is the invertible map $q_d$ of \cite[Lemme 2.3.10]{Ayoub_these_1}, and the third arrow is induced by the counit of the adjunction $(i^*,i_*)$.

This map is functorial in $M,N$, and the functors $i^*,i^!,\otimes$ commute with small sums (the proof is easy and recalled in Lemma~\ref{lem:preservation} below), hence it suffices to show that it is an isomorphism for $M$ strongly dualisable. By construction, it suffices to show that the map $i^!(M\otimes i_*i^!N)\ra i^!(M\otimes N)$ is then an isomorphism, or equivalently, by localisation, that its cone $i^!(M\otimes j_*N)$ vanishes (where $j:S\setminus Z\ra S$ is the complementary open immersion). Let $P\in \DA(Z)$. We have

\begin{eqnarray*}
  \DA(Z)(P,i^!(M\otimes j_*N)) & \simeq & \DA(S)(i_*P,M\otimes j_*N) \\
                             & \simeq & \DA(S)(i_*P\otimes \Homint(M,\BQ),j_*N) \\
                             & \simeq & \DA(S\setminus Z)(j^*(i_*P\otimes \Homint(M,\BQ),N) \\
  & \simeq & \DA(S\setminus Z)(j^*i_*P\otimes j^*\Homint(M,\BQ),N)
\end{eqnarray*}
where we have used adjunctions, the monoidality of $j^*$ and the biduality property for the strongly dualisable object $M$. Since $i^{*}j_{*}=0$, we deduce by Yoneda that $i^{!}(M\otimes j_{*}N)=0$.

In the case of a regular immersion, we combine the result with the absolute purity isomorphism $i^!\BQ\simeq \BQ(-c)[-2c]$.
\end{proof}

\begin{lemma}
\label{lem:subcats_comp}
Let $\CT$ be one of $\DA_{\hom}(S),\ \DA^\coh(S),\ \DA_n(S),\ \DA^n(S)$ or their subcategories of smooth or (strongly) geometrically smooth objects. Then the triangulated category $\CT$ is compactly generated by its generating family, and an object of $\CT$ is compact if and only if it is compact in $\DA(S)$.
\end{lemma}
\begin{proof}
Write $\CG$ for the generating family of $\CT$. By the fact that strongly dualisable objects in a symmetric monoidal triangulated category with compact unit are automatically compact (for the $\DA^{\sm}(S)$ case) and \cite[Proposition~3.20, Proposition~8.5]{Ayoub_Etale} (for the other cases), we see that all objects of $\CG$ are compact. This means that $\CT$ is compactly generated by $\CG$. Write $\CT_c$ for the full subcategory of objects of $\CT$ which are compact in $\CT$. By \cite[Lemma 4.4.5]{Neeman_book}, $\CT_c=\langle \CG\rangle$. In particular any object of $\CT_c$ is compact in $\DA(S)$; the converse implication is clear.
\end{proof}

\begin{defi}
\label{def:subcats_comp}
We write $\DA^{\coh}_c(S)$, $\DA_{\homo,c}(S),\mathrm{etc}.$ for the full subcategories of compact objects of $\DA^{\coh}(S)$, $\DA_{\homo}(S),\mathrm{etc}$.
\end{defi}

\subsection{Permanence properties}
\label{sec:n_permanence}

The subcategories we have introduced are each stable under certain Grothendieck operations. We start with the compatibilities with the monoidal structure.

\begin{prop}
\label{prop:n_monoidal}
  Let $S$ be a base scheme.
  \begin{enumerate}[label={\upshape(\roman*)}]
  \item \label{coh_tensor} $\DA^{\coh}_{(c)}(S)$ is stable by tensor products and negative Tate twists.
  \item \label{coh_n_tensor} For all $m,n\geq 0$, we have $\DA^m_{(c)}(S)\otimes \DA^n_{(c)}(S) \subset \DA^{m+n}_{(c)}(S)$.
  \item \label{coh_twists} For all $m,n\geq 0$, we have $\DA^m_{(c)}(S)(-n) \subset \DA^{m+n}_{(c)}(S)$.
  \item \label{hom_tensor} $\DA_{\homo,(c)}(S)$ is stable by tensor products and positive Tate twists.
  \item \label{hom_n_tensor} For all $m,n\geq 0$, we have $\DA_{m,(c)}(S)\otimes \DA_{n,(c)}(S) \subset \DA_{m+n,(c)}(S)$.
  \item \label{hom_twists} For all $m,n\geq 0$, we have $\DA_{m,(c)}(S)(n) \subset \DA_{m+n,(c)}(S)$.
  \end{enumerate}
The same properties hold for the smooth and (strongly) geometrically smooth versions of those subcategories.
\end{prop}
\begin{proof}
First, note that $\otimes$ commutes with small sums in both variables, being a left adjoint. This reduces the proof to checking the result for generators. 

Let us prove point \ref{coh_tensor}. Recall that we have a projection formula for $f_!$ and $f^*$ from \cite[Theoreme 2.3.40]{Ayoub_these_1}, i.e., for any finite type separated morphism $f:S\ra T$ and any $M\in \DA(S), N\in \DA(T)$, we have a natural isomorphism 
\[
f_!(M\otimes f^*N)\simeq f_! M\otimes N.
\]
Let $g:X\ra S$ and $h:Y\ra S$ be proper morphisms. Let $Z=X\times_S Y$ and let $g':Z\ra Y$ and $h':Z\ra X$ be the two projections. We have a sequence of isomorphisms

\begin{eqnarray*}
g_*\BQ_X\otimes h_*\BQ_Y & \simeq & g_!\BQ_X\otimes h_!\BQ_Y\\
& \simeq & g_!(\BQ_X\otimes g^*h_! \BQ_Y)\\
& \simeq & g_! h'_! {(g')}^*\BQ_Y\\
& \simeq & g_* h'_* \BQ_Z
\end{eqnarray*}
where the first and fourth isomorphisms follows from properness, the second is the projection formula and the third is the $\Ex^*_!$ isomorphism. This shows that $g_*\BQ_X\otimes h_*\BQ_Y$ is cohomological. The negative Tate twist $\BQ_S(-n)$ is cohomological, as it is a direct factor of $(\BP^n_S\ra S)_*\BQ$. This finishes the proof of \ref{coh_tensor}. The same proof, combined with the fact that relative dimension is stable by base change and adds up in compositions, gives \ref{coh_n_tensor} and \ref{coh_twists}.

For the proof of point \ref{hom_tensor}, we use a parallel argument; we combine the projection formula for $f_\sharp$ and $f^*$ of \cite[Proposition~4.5.17]{Ayoub_these_2} with the $\Ex^*_\sharp$ isomorphism and the fact that $\BQ_S(n)$ is a direct factor of $(\BP^n_S\ra S)_\sharp\BQ$ by the projective bundle formula. The same proof, combined with the fact that relative dimension is stable by base change and adds up in compositions, gives \ref{hom_n_tensor} and \ref{hom_twists}.

Finally, the analoguous statement for smooth and (strongly) geometrically smooth versions follow from similar arguments together with the fact that a tensor product of strongly dualisable objects is strongly dualisable.
\end{proof}

\begin{lemma}\label{lem:preservation}
Let $f:S\ra T $ be a morphism of schemes (resp. a finite type separated morphism of schemes).
  \begin{enumerate}[label={\upshape(\roman*)}]
  \item \label{pre_sums} The operations $f^*,f_*$, (resp. $f_!,f^!$) commute with small sums.
  \item \label{pre_left} The operations $f^*$ (resp. $f_!$) preserve compact objects.
  \item \label{pre_right} Assume $T$ is quasi-excellent. Then the operations  $f_*$ (resp. $f^!$) preserve compact objects.
  \item \label{pre_sharp_sums} If $f$ is smooth, the operation $f_\sharp$ commutes with small sums.
  \item \label{pre_sharp_compact} If $f$ is smooth, the operation $f_\sharp$ preserves compact objects.
    \end{enumerate}
\end{lemma}  
\begin{proof}
By \cite[Proposition 3.19]{Ayoub_Etale}, compact objects in $\DA(S)$ coincide with constructible objects. This immediately implies Statement \ref{pre_sharp_compact}. Statement \ref{pre_left} then follows from \cite[Proposition 8.5]{Ayoub_Etale} and Statement \ref{pre_right} from \cite[Theoreme 8.10]{Ayoub_Etale} (the result, stated for excellent schemes, actually holds for quasi-excellent schemes since Gabber's local uniformisation theorem holds in that generality). 

Statement \ref{pre_sums} is immediate for $f^*,f_!$ since they are left adjoints. The same holds for \ref{pre_sharp_sums}. For $f_*, f^!$, by \cite[Lemme 2.1.28]{Ayoub_these_1} it is enough to see that their left adjoints preserve compact objects, which is the already established Statement \ref{pre_left}.
\end{proof}

\begin{prop}
  \label{prop:permanence_coh}
  Let $f:S\ra T$  be a morphism of schemes. The following operations preserve the subcategories $\DA^{\coh}(-)$.
  \begin{enumerate}[label={\upshape(\roman*)}]
  \item \label{coh^*} $f^*$ for any $f$.
  \item \label{coh_*} $f_*$ when $f$ is separated of finite type and $T$ admits the resolution of singularities by alterations.
  \item \label{coh_!} $f_!$ when $f$ is separated of finite type.
  \item \label{coh^!} $f^!$ when $f$ is quasi-finite separated and $T$ admits the resolution of singularities by alterations.
  \end{enumerate}
Moreover, they also preserve $\DA^{\coh}_{c}$, with the additional assumption that $T$ is quasi-excellent for points \ref{coh_*} and \ref{coh^!}.
\end{prop}
\begin{proof}
By Lemma \ref{lem:subcats_comp} and Lemma \ref{lem:preservation}, we see that to prove the result for both $\DA^{\coh}(-)$ and for $\DA^{\coh}_c$ it is enough to show that in each case \ref{coh^*}-\ref{coh^!} the operation sends generators of $\DA^{\coh}_c$ to $\DA^{\coh}$. 

  We prove the results in a slightly different order than in the statement: we first establish \ref{coh^*}, \ref{coh_!} (which contains the special case of \ref{coh_*} for proper morphisms), \ref{coh^!} for closed immersions, \ref{coh_*} and finally \ref{coh^!} in all generality. 

\textit{Proof of \ref{coh^*}:} Proper base change implies that $f^*$ sends generators of $\DA^{\coh}(T)$ to generators of $\DA^\coh(S)$. 

\textit{Proof of \ref{coh_!}:} Let $g:X\ra S$ be a proper morphism. We need to show that $f_!g_*\BQ_X\simeq (f\circ g)_!\BQ_X$ is in $\DA^\coh(T)$. Since $f$ is assumed to be separated of finite type, the same holds for $f\circ g$. Nagata's theorem \cite{Nagata_comp} \cite{Conrad_Nagata} implies that $f\circ g$ admits a compactification, i.e., that there exists a factorisation $f\circ g=\bar{f}\circ j$ with $j:X\ra \overline{X}$ an open immersion and $\bar{f}:\overline{X}\ra T$ a proper morphism. Let $i:Z\ra \overline{X}$ be a complementary closed immersion to $j$. By localisation, we have a distinguished triangle
\[
j_!\BQ_X\ra \BQ_{\overline{X}}\ra i_!\BQ_Z\rap
\]
which after applying $\bar{f}_*\simeq \bar{f}_!$ yields
\[
\bar{f}_*j_!\BQ_X\ra \bar{f}_!\BQ_{\overline{X}}\ra (\bar{f}i)_!\BQ_Z\rap.
\]
By definition, the second and third terms in this triangle are in $\DA^\coh(T)$. This implies that the first, which is isomorphic to $f_!g_*\BQ_X$, is as well.

\textit{Proof of \ref{coh^!} for $f=i$ a closed immersion:}

The blueprint for this proof is taken from Section~2.2.2 of \cite{Ayoub_these_1}. 

Lemma~\ref{lemma:coh_gen} below, applied to $i:S\ra T$, shows that it is enough to prove that, for any $g:X\ra T$ with $X$ connected regular and $g^{-1}(S)$ equal to either $X$ or a normal crossing divisor, the motive $i^!g_* \BQ_X$ is compact. Form the cartesian square
\[
\xymatrix{
Y \ar[d]_{g'} \ar[r]^{i'} & X \ar[d]^g\\
S \ar[r]_{i} & T.
}
\]
We have an $\Ex^!_*$ isomorphism $i^!g_*\BQ_X\simeq g'_*i'^!\BQ_X$. By point \ref{coh_!}, it is enough to show that $i'^!\BQ_X$ is in $\DA^\coh(X)$. By assumption, $Y$ is either equal to $X$ or is a normal crossing divisor; only the second case requires a proof. By \cite[Lemme 2.2.31]{Ayoub_these_1} applied to the branches and point \ref{coh_!} for closed immersions, we reduce to the case of a regular immersion, which then follows from absolute purity and Proposition~\ref{prop:n_monoidal} \ref{coh_tensor}. 

% \textbf{DETAILS}

\textit{Proof of \ref{coh_*}:}

Using Nagata's theorem and the proper case of point \ref{coh_!}, it suffices to show that $j_*\BQ_S$ is in $\DA^\coh(T)$ for $j:S\ra T$ an open immersion. This now follows from colocalisation and point \ref{coh^!} for the complementary closed immersion.

\textit{Proof of \ref{coh^!} for $f$ quasi-finite general:}

By the same argument as above, using the $\Ex^!_*$ isomorphism, it is enough to show that $f^!\BQ_T$ is in $\DA^\coh(S)$. Using Zariski's main theorem~\cite[Th\'eor\`eme 8.12.6]{EGAIV_3}, the fact that $j^!\simeq j^*$ for $j$ open immersion and point \ref{coh^*}, we are reduced to the case of finite morphisms. 

If $f$ is finite \'etale, then $f^!\simeq f^*$ again and we are done by point \ref{coh^*}. If $f$ is finite and purely inseparable, then a corollary of the separation property of $\DA$ is that $f^!\simeq f^*$ is an equivalence of categories \cite[Corollaire 2.1.164]{Ayoub_these_1}. In general, we proceed by induction on the dimension of $T$. The proof for the $0$-dimensional case follows the same pattern as the inductive step, so we treat both in parallel. If $T$ is $0$-dimensional, or generically on $T$, say above the image of a dense open immersion $j:U\ra T$, the morphism $f$ is the composite of a finite \'etale morphism followed by a finite purely inseparable morphism. Let $l:V\ra S$ be $j\times_T S$ and $k:W\ra S$ be a complementary closed immersion (take $W$ empty in the $0$-dimensional case). Then $l^!f^!\BQ_T\simeq f_U^!\BQ_U$ is in $\DA^\coh(V)$ by combining the arguments for finite \'etale and finite purely inseparable morphisms above. By point \ref{coh_*}, we get that $l_*l^*f^!\BQ_T$ is in $\DA^\coh(S)$. This concludes the proof for $\dim(T)=0$. In general, by the inductive hypothesis and point \ref{coh_!}, we get that $k_!k^!f^!\BQ_T$ lies in $\DA^\coh(S)$. The colocalisation triangle then shows that $f^!\BQ_T$ lies in $\DA^\coh(S)$. This completes the proof.
\end{proof}

\begin{lemma}\label{lemma:coh_gen}
Let $S$ be a scheme admitting the resolution of singularities by alterations, $f:X\ra S$ a finite type morphism and $T\subset X$ closed. Then $\DA^\coh(X)$ is compactly generated by motives of the form $g_*\BQ_{X'}$ with $g:X'\ra X$ a projective morphism and $X'$ connected regular and $g^{-1}(T)$ equal either to $X'$ or to a normal crossing divisor.   
\end{lemma}
\begin{proof}
The reference \cite[Proposition~2.2.27]{Ayoub_these_1} specialized to the $\BQ$-linear, separated, homotopical $2$-functor $\DA(-)$ proves a similar statement for the category of constructible objects $\DA_c(S)$ (with added positive Tate twists of the generators, and restriction to quasi-projective morphisms). Once one removes the Tate twists and the restriction to quasi-projective morphisms, one notices that using Statement \ref{coh_!} of Proposition~ above instead of Corollaire 2.2.21 in loc.cit, the proof of loc.cit \cite[Proposition~2.2.27]{Ayoub_these_1} then applies verbatim.
\end{proof}

\begin{prop}
  \label{prop:permanence_hom}
Let $f:S\ra T$ be a morphism of schemes. The following operations preserve the subcategories $\DA_{\homo}(-)$
  and $\DA_{\homo,c}(-)$.
  \begin{enumerate}[label={\upshape(\roman*)}]
  \item \label{hom^*} $f^*$ for any $f$.
  \item \label{hom_sharp} $f_\sharp$ when $f$ is smooth.
  \item \label{hom_!} $f_!$ for any quasi-finite separated morphism $f$.
  \end{enumerate}

\end{prop}

\begin{remark}
In the proof of point \ref{hom_!}, we use results from Sections \ref{sec:n_continuity} and \ref{sec:n_field}. The reader can check that we do not use the reference \ref{prop:permanence_hom} \ref{hom_!} in between. We feel this break from logical order is justified by the commodity of stating these properties together.
\end{remark}

\begin{proof}
By Lemma \ref{lem:subcats_comp} and Lemma \ref{lem:preservation}, we see that to prove the result for both $\DA_{\homo}(-)$ and for $\DA_{\homo,c}$ it is enough to show that in each case \ref{hom^*}-\ref{hom_!} the operation sends generators of $\DA_{\homo,c}$ to $\DA_{\homo}$. 

\textit{Proof for \ref{hom^*}:}
The $\mathrm{Ex}^*_{\sharp}$ isomorphism implies that $f^*$ sends generators of $\DA_{\homo}(T)$ to generators of $\DA_\homo(S)$.

\textit{Proof for \ref{hom_sharp}:} The fact that generators are sent to homological motives follows directly from the definition.

\textit{Proof for \ref{hom_!}:} Using Zariski's Main theorem~\cite[Th\'eor\`eme 8.12.6]{EGAIV_3} and \ref{hom_sharp}, we see that it is enough to treat the
case of $f$ finite. 
    
We first do the case of closed immersions. The next lemma is proved using Mayer-Vietoris distinguished triangles.
\begin{lemma}
\label{lemm:hom_Z_loc}
Let $T$ be a scheme and $\CU=\{j_k:U_k\hookrightarrow T\}_{k=1}^n$ be a finite Zariski open covering of $T$. Let $M\in\DA(T)$ Then
\[
M\in \DA_\homo(T)\Longleftrightarrow \text{ for all }1\leq k\leq n,\text{ we have }j_k^*M\in \DA_\homo(T). 
\] \qed
\end{lemma}

Let $i:Z\ra X$ be a closed immersion and $g:U\ra Z$ be a smooth morphism. We need to show that $i_*g_\sharp \BQ_U\in \DA_{\homo}(X)$. There exists a finite open affine cover $\{U_k=\Spec(A_k)\}_{1\leq k\leq n}$ of $U$ and a finite open affine cover $\{Z_k=\Spec(R_k)\}_{1\leq k\leq n}$ of $Z$ with $g(U_k)\subset Z_k$ and such that via $g_k:=g_{|U_k}^{|Z_k}$, the ring $A_k$ takes the form:
\[
A_k=R_k[x_1,\ldots,x_{n_k}]/(f^k_1,\ldots,f^k_{c_k})
\]
with $\left(\det(\frac{\partial f^k_j}{\partial x_i})_{1\leq i,j\leq c_{k}}\right)$ invertible in $A_k$ (i.e. $g_k$ is a standard smooth map). We can choose an open affine cover $\{W_k\}$ of $X$ such that $W_k\cap Z=Z_k$. Applying Lemma \ref{lemm:hom_Z_loc} to the open cover $W_k$ and using base change for closed immersions and smooth base change, we can suppose that $g$ itself is a standard smooth map and that $X=\Spec(R)$ is affine.

In this situation, we can lift the equations $f_j$ to $\tilde{f}_j\in R[x_1,\ldots,x_n]$. The open set $W$ of $X$ over which the resulting map $\tilde{g}:\Spec(R[x_1,\ldots,x_n]/(\tilde{f}_1,\ldots,\tilde{f}_n))\rightarrow X$ is standard smooth contains $Z$, and $\tilde{g}$ extends $g$. We have a localisation triangle
\[
(W\setminus Z\ra W)_\sharp\tilde{g}_\sharp\BQ \ra \tilde{g}_{\sharp}\BQ\ra i_*g_{\sharp}\BQ_U\stackrel{+}{\ra}
\]
where the first two terms are in $\DA_{\hom}(X)$. We deduce that $i_*g_{\sharp}\BQ_U\in\DA_{\hom}(X)$ as wanted.

For a general quasi-finite $f:T\ra S$, using localisation, the case of closed immersions and an induction on the dimension of $S$, we see that we can replace $S$ by any everywhere dense open subset. The case of closed immersions also ensures that we can assume $S$ is reduced. By continuity for $\DA_{\homo}(-)$ (proven in Proposition~\ref{prop:subcats_continuity} below; the proof does not use permanence properties of $\DA_\homo(-)$ besides \ref{hom^*}), we see that we can even replace $S$ by any of its generic points. We are thus reduced to the case of a finite field extension, which follows from the more precise Lemma \ref{lemm:hom_!_fields} below.
\end{proof}

\begin{prop}
  \label{prop:permanence_coh_n}\ 
  \begin{enumerate}[label={\upshape(\roman*)}]
  \item \label{coh_n_pullback} Let $f$ be any morphism of schemes. Then $f^*$ preserves the subcategories $\DA^n(-)$ and
  $\DA^n_c(-)$.
  \item \label{coh_n_!} Let $f:S\ra T$ be separated of finite type and of relative dimension $m$. Then $f_!$ sends $\DA^n(S)$ (resp. $\DA^n_c(S)$) to $\DA^{n+m}(T)$ (resp. $\DA^{n+m}_c(T)$). In particular, if $f$ is quasi-finite, then $f_!$ preserves the subcategories $\DA^n(-)$ and $\DA^n_c(-)$.
  \end{enumerate}
\end{prop}
\begin{proof}
By Lemma \ref{lem:subcats_comp} and Lemma \ref{lem:preservation}, we see that to prove the result for both $\DA^n(-)$ and for $\DA^n_c$ it is enough to show that in each case \ref{coh_n_pullback} and \ref{coh_n_!} the operation sends generators of $\DA^n_c$ to $\DA^n$. 

The case of \ref{coh_n_pullback} follows proper base change and the fact that being of relative
    dimension $\leq n$ is stable by base change.

The proof in the case of \ref{coh_n_!} is the same as that of Proposition~\ref{prop:permanence_coh} \ref{coh_!}, keeping track of the relative dimensions involved.
\end{proof}

\begin{prop}
  \label{prop:permanence_hom_n}\ 
  \begin{enumerate}[label={\upshape(\roman*)}]
  \item \label{hom_n_pullback} Let $f$ be any morphism of schemes. Then $f^*$ preserves the subcategories $\DA_n(-)$ and
  $\DA_{n,c}(-)$.
  \item \label{hom_n_!} Let $f:S\ra T$ be separated and quasi-finite. Then $f_!$ preserves the subcategories $\DA_n(-)$ and $\DA_{n,c}(-)$.
  \end{enumerate}
\end{prop}
\begin{proof}
By Lemma \ref{lem:subcats_comp} and Lemma \ref{lem:preservation}, we see that to prove the result for both $\DA_{n}(-)$ and for $\DA_{n,c}$ it is enough to show that in each case \ref{hom_n_pullback} and \ref{hom_n_!} the operation sends generators of $\DA_{n,c}$ to $\DA_n$. 
 
The case of \ref{hom_n_pullback} follows from the 
    $\mathrm{Ex}^*_{\sharp}$ isomorphism and the fact that being of relative
    dimension $\leq n$ is stable by base change.

The proof in the case of \ref{hom_n_!} is the same as that of Proposition~\ref{prop:permanence_hom} \ref{hom_!}, keeping track of the relative dimensions involved.
\end{proof}

We list some useful corollaries of the results above.

\begin{cor}\label{cor:localisation_subcats}
Let $\CT(-)$ be one of $\DA^\coh(-)$, $\DA_\homo(-)$, $\DA^n(-)$, $\DA_n(-)$ or one of their subcategories of compact objects.
\begin{enumerate}[label={\upshape(\roman*)}]
\item \label{loc_sub} The system $\CT(-)$ localises in the following sense: for $M\in \DA(S)$ and $i:Z\ra S$ and $j:U\ra S$ complementary closed and open immersions, we have $M\in\CT(S)$ if and only if $i^*M\in \CT(Z)$ and $j^*M\in \CT(U)$.
\item \label{rad_sub} Let $f:T\ra S$ be a finite surjective purely inseparable morphism (e.g. a nil-immersion), $M\in \DA(S)$, $N\in \DA(T)$. Then we have $M\in \CT(S)$ if and only if $f^*M\in \CT(T)$, and we have $N\in \CT(T)$ if and only if $f_*N\in \CT(S)$.
\end{enumerate}
\end{cor}
\begin{proof}
Statement \ref{loc_sub} follows directly from localisation and the permanence properties above. Similarly, statement \ref{rad_sub} follows directly from \cite[Proposition~2.1.163]{Ayoub_these_1} (which applies because $\DA(-)$ is separated) and the permanence properties. 
\end{proof}

Finally, let us discuss what happens with internal Homs and duality. 

\begin{cor}\label{cor:homint}
The internal Hom satisfies $\Homint(\DA_{\homo,c}(S),\DA^{\coh}_{(c)}(S))\subset \DA^{\coh}_{(c)}(S)$. In particular, if $S$ is regular and we take $\BQ_S$ as dualising object, then Verdier duality $\BD_S:=\Homint(-,\BQ_S)$ sends compact homological motives to compact cohomological motives. 
\end{cor}
\begin{proof}
If $M\in \DA(S)$ is compact, then $\Homint(M,-)$ commutes with small sums. This shows that we can restrict to generators of $\DA^\coh(S)$ in the second variable. Using \cite[Lemma 4.4.5]{Neeman_book}, we see that we can restrict to generators of $\DA_{\homo,c}(S)$ in the first variable. The result then follows from \cite[Proposition~2.3.51-52]{Ayoub_these_1}, the $\Ex^*_\sharp$ isomorphism and Proposition~\ref{prop:permanence_coh} (ii). 
\end{proof}

\begin{lemma}\label{lemm:duality_field}
Let $S$ be a regular scheme. Write $\BD_S:=\Homint(-,\BQ_S):\DA(S)^\op\ra \DA(S)$ for the Verdier duality functor. We have 
\[\BD_S (\DA_{\homo,c}(S))\subset \DA^\coh_c(S)\]
and $\BD_S$ restricts to anti-equivalences of categories
\[ \BD_S:\DA^{\gsm}_{\homo,c}(S)\stackrel{\sim}{\lra} \DA^{\coh}_{\gsm,c}(S)\text{ and }\]
\[ \BD_S:\DA^\sgsm_{n,c}(S)\stackrel{\sim}{\lra} \DA^n_{\sgsm,c}(S).\]
\end{lemma}

\begin{proof}
 For  a separated scheme $X$ of finite type over $S$, consider the more general Verdier duality functor $\BD_{X/S}:=\Homint(-,\pi_X^!\BQ_S):\DA(X)^\op\ra \DA(X)$. By \cite[Th\'eor\`emes 8.12-8.14]{Ayoub_Etale}, this functor preserves compact objects and its restriction to $\DA_c(X)$ is involutive, i.e. an anti-autoequivalence which is its own quasi-inverse.

The first inclusion is a special case of Corollary \ref{cor:homint} but we provide an argument since the same computation is used in the rest of the proof. The behaviour of $\BD_{X/S}$ with respect to the four operations is explained in \cite[Th\'eor\`eme 2.3.75]{Ayoub_these_1}: informally, Verdier duality exchanges $f_*$ and $f_!$, and $f^*$ and $f^!$. Moreover, recall that, for $f$ smooth, relative purity provides an isomorphism $f_\sharp f^* \simeq f_! f^!$. This allows to compute the action of $\BD_{X/S}$ on generating families. For instance, we have, for any $f$ smooth, $\BD_S(f_\sharp f^*\BQ_X)\simeq \BD_S(f_!f^!\BQ_S)\simeq f_* f^* \BD_S(\BQ_S)\simeq f_* f^*\BQ_S$ which is in $\DA^{\coh}(S)$ by Proposition~\ref{prop:permanence_coh} \ref{coh_*}. This proves the first inclusion.

For the equalities for (strongly) geometrically smooth subcategories, note that if $f$ is smooth projective (resp. smooth projective of relative dimension $\leq n$), the same computation shows that $\BD_S(f_\sharp f^*\BQ_X)$ is in $\DA^{\coh}_{\gsm}(S)$ (resp. $\DA^n_{\sgsm}(S)$). This proves one inclusion of the equalities, and the other follows by the involutivity of $\BD$.
\end{proof}

\begin{remark}
  Even on a regular
  scheme, the categories of constructible homological and cohomological motives are
  not anti-equivalent through Verdier duality with dualising object
  $\BQ_{S}$ (see, however, Proposition~\ref{prop:subcats_field} below for the field case). Indeed, assume $S$ regular of dimension $d>0$, let $i:x\ra
  S$ be the inclusion of a closed point $x$ and $j:U\ra S$ be the
  complementary open immersion. Then by colocalisation and absolute purity, $j_*\BQ_U\in \DA^{\coh}(S)$ sits in a triangle
\[
i_*\BQ(-d)[-2d]\ra \BQ_S\ra j_*\BQ_U\rap.
\]
On the other hand, we have $\BD_S(\BQ_S)\simeq \BQ_S\in \DA^{\coh}(S)$ and $\BD_S(i_!i^!\BQ_S)\simeq i_*\BQ_S\in
  \DA^{\coh}(S)$, so that by taking the Verdier dual of the triangle above, we have $\BD_S(j_*\BQ_U)\in\DA^{\coh}(S)$. 

If Verdier duality did exchange homological and cohomological motives, we would have $j_*\BQ_U\in \DA_{\hom}(S)\cap \DA^{\coh}(S)$ which is equal to $\DA_0(S)$ by Corollary
  \ref{coro:omega_0} \ref{hom_coh_intersect} below. We would then also have $i_*\BQ(-d)[-2d]\in \DA_0(S)$; hence, $i^*i_*\BQ(-d)\simeq \BQ(-d)\in \DA_0(x)$. This is not the case, as can be seen in a number of ways; for instance, in the proof of Corollary~\ref{coro:omega_0} \ref{omega_1_twists} we will show that for all $M\in \DA_0(x)$ and $d>0$, we have $\Hom(M,\BQ(-d))=0$.
\end{remark}

\subsection{Continuity}
\label{sec:n_continuity}

We have a continuity result for subcategories of compact objects.

\begin{prop}
 \label{prop:subcats_continuity}
  Let $I$ be a cofiltering small category and $(X_i)_{i\in I}\in
  \Sch^I$ with affine transition morphisms. Let $X=\varprojlim_{i\in
    I}X_i$, which we assume to be noetherian and finite-dimensional. Then $\DA^{\coh}_c(X)$
  (resp. $\DA_{\homo,c}(X)$, $\DA^n_c(X)$, $\DA_{n,c}(X)$) is equal to the $2$-colimit of the
  $\DA^{\coh}_c(X_i)$ (resp. $\DA_{\homo,c}(X_i)$, $\DA^n_c(X_i)$,
  $\DA_{n,c}(X_i)$) via the pullback functors $(X\ra X_i)^*$.
\end{prop}
\begin{proof}
  Using the continuity result for morphisms in $\DA$ from
  \cite[Proposition~3.19]{Ayoub_Etale} and the arguments from
  \cite[Corollaire 1.A.3]{Ayoub_rigide}, it is enough
  to prove the following lemma (which extends \cite[Lemme
  1.A.2]{Ayoub_rigide}).

\begin{lemma}\label{lem:limits}
  With the notation of the proposition, let $Y$ be an $X$-scheme of
  finite presentation. Then there exists an $i\in I$ and an
  $X_i$-scheme $Y_i$ of finite presentation such that $Y\simeq
  Y_i\times_{X_i}X$. Moreover, if $Y/X$ is smooth (resp. proper, of relative
  dimension $\leq n$, smooth of relative dimension $\leq n$), then
  $Y_i$ can be chosen smooth (resp. proper, of relative dimension $\leq n$,
  smooth of relative dimension $\leq n$).
\end{lemma}
\begin{proof}\let\qed\relax
  The first part is \cite[Th\'eor\`eme 8.8.2.(ii)]{EGAIV_3}. For the second part, the case of $Y/X$ proper is \cite[Th\'eor\`eme 8.10.5.(xii)]{EGAIV_3} and \cite[Lemme 1.A.2]{Ayoub_rigide} and its proof cover
  the case of smooth and smooth of relative dimension $\leq n$. The case of morphisms of relative dimension $\leq n$ (without smoothness assumption) is \cite[Tag 05M5]{stacks-project}.
\end{proof}
\end{proof}
We deduce a useful punctual characterisation of compact $n$-motives:

\begin{prop}
\label{prop:punctual_car}
  Let $S$ be a scheme and $M\in\DA_c(S)$. Then the following are
  equivalent.
  \begin{enumerate}[label={\upshape(\roman*)}]
  \item \label{global_n} $M\in \DA_c^{\coh}(S)$ (resp. $\DA_{\homo,c}(S),\
    \DA^n_c(S),\ \DA_{n,c}(S)$).
  \item \label{local_n} For all $s\in S$, we have $s^*M\in \DA_c^{\coh}(s)$
    (resp. $ \DA_{\homo,c}(s),\ \DA^n_c(s),\ \DA_{n,c}(s)$).
  \end{enumerate}
\end{prop}
\begin{proof}
The direction \ref{global_n}$\Rightarrow$\ref{local_n} follows from the stability by pullbacks for all these subcategories established above. In the other direction, we can assume $S$ is reduced by Corollary~\ref{cor:localisation_subcats} \ref{rad_sub}. We then proceed by noetherian induction. The case of generic points is settled by the hypothesis, we then use Proposition~\ref{prop:subcats_continuity} to spread-out the property to an open set. We conclude by using Corollary~\ref{cor:localisation_subcats} \ref{loc_sub} and the induction hypothesis. 
\end{proof}

\subsection{Over a field}
\label{sec:n_field}

When the base is the spectrum of a field, several of the notions we have introduced coincide.

\begin{prop}
  \label{prop:subcats_field}
Let $k$ be any field; then we have the following equalities.
\[ \DA_{\homo}(k)=\DA_\homo^\sm(k) =\DA_\homo^\gsm(k). \]
\[ \DA^\coh(k)=\DA_\sm^\coh(k) =\DA^\coh_\gsm(k). \]
\[ \DA_n(k)=\DA_n^\sm(k) =\DA_n^\gsm(k)=\DA_n^{\sgsm}(k). \]
\[ \DA^n(k)=\DA^n_\sm(k) =\DA^n_\gsm(k)=\DA^n_{\sgsm}(k). \]
The same equalities hold for the subcategories of compact objects, and $\BD_k$ restricts to anti-equivalences of categories:
\[ \BD_k:\DA_{\homo,c}(k)\stackrel{\sim}{\lra} \DA^{\coh}_c(k):\BD_k\]
\[ \BD_k:\DA_{n,c}(k)\stackrel{\sim}{\lra} \DA^n_c(k):\BD_k\]
\end{prop}

\begin{proof}
 In each case, we prove equality by showing that the generating family on each side lies in the other. The generating families used in the definitions of these categories are formed of compact objects, hence it suffices to prove the equalities for the subcategories of compact objects. By Lemma \ref{lem:gsm_sm}, we need only prove the inclusions
\[ \DA_{\homo,c}(k)\subset\DA_{\homo,c}^\gsm(k), \]
\[ \DA^\coh_c(k)\subset\DA^\coh_{\gsm,c}(k), \]
\[ \DA_{n,c}(k)\subset\DA_{n,c}^{\sgsm}(k) \text{ and } \]
\[ \DA^n_c(k)\subset\DA^n_{\sgsm,c}(k). \]

The key is to prove the following claim
\[
\text{For all $n\in \BN$, we have }\BD_k(\DA^n_c(k))\subset \DA^\sgsm_{n,c}(k).\tag{*}\label{claim}  
\]
Indeed, assume Claim \eqref{claim} for the next three paragraphs. Then by looking at generators we also get $\BD_k(\DA^\coh_c(k))\subset \DA^\gsm_{\homo,c}(k)$. By applying $\BD_k$ again and the equivalence of categories of Lemma \ref{lemm:duality_field}, we get inclusions $\DA^n_c(k)\subset \DA^n_{\sgsm,c}(k)$ and $\DA^{\coh}_c(k)\subset \DA^{\coh}_{\gsm,c}(k)$. By applying $\BD_k$ to the inclusion $\BD_k (\DA_{\homo,c}(k))\subset \DA^\coh_c(k)$ of Lemma \ref{lemm:duality_field}, we also obtain $\DA_{\homo,c}(k)\subset \DA^{\gsm}_{\homo,c}(k)$. It remains to see that $\DA^n_c(k)\subset\DA^n_{\sgsm,c}(k)$, which is slightly less clear. 

Let $f:X\ra k$ smooth of relative dimension $i\leq n$ (we can reduce to this case by considering connected components of $X$). By relative purity, we have $f_{\sharp}\BQ_X(-n) \simeq f_!\BQ_X(i-n)[2i]$ which is in $\DA^n_c(k)$ by Proposition~\ref{prop:permanence_coh_n} and \ref{prop:n_monoidal}. This shows that $\DA_{n,c}(k)(-n)\subset \DA^n_c(k)=\DA^n_{\sgsm,c}(k)$ (the last equality having just been established in the previous paragraph). Applying Verdier duality, we get $\BD_k(\DA_{n,c}(k))(n)\subset \BD_k(\DA^n_{\gsm,c}(k))$. By Claim \eqref{claim}, this implies that $\BD_k(\DA_{n,c}(k))(n)\subset\DA_{n,c}^{\sgsm}(k)$. 

Another application of relative purity shows that $\DA_{n,c}^{\sgsm}(k)(-n)=\DA^{n}_{\sgsm,c}(k)$. 
Putting everything together, we have $\BD_k(\DA_{n,c}(k))\subset \DA^n_{\sgsm,c}(k)=\BD_k(\DA^{\sgsm}_{n,c}(k))$ so by the involutivity of $\BD$ we get the missing inclusion $\DA^n_c(k)\subset\DA^n_{\sgsm,c}(k)$. This finishes the proof of the proposition modulo the claim.

The rest of the proof is devoted to show Claim \eqref{claim}. To simplify notations, we write $\pi_Y:Y\ra k$ for the structural morphism of a $k$-scheme $Y$. Using the generating families, we reformulate the claim as follows: for $\pi_X:X\ra k$ proper of relative dimension $\leq n$, we have $\BD_k(\pi_{X*}\BQ_X)\simeq \pi_{X!}\pi_X^!\BQ_k$ in $\DA^\sgsm_{n,c}(k)$. Let $i:X_{\red}\rightarrow X$. Then by localisation we have $\pi_{X!}\pi_X^!\BQ_k\simeq \pi_{X!}i_!i^!\pi_X^!\BQ_k\simeq \pi_{X_\red !}\pi_{X_\red}^!\BQ_k$. Consequently, we can assume that $X$ is reduced.
 
We first treat the case of a perfect field $k$. We proceed by induction on the dimension of $X$. When $X$ is $0$-dimensional, we see that $\pi_X$ is finite \'etale because $k$ is perfect and $X$ is reduced, so that $\pi_{X!}\pi_X^!\simeq \pi_{X\sharp}\pi_X^*$ and we are done. For the induction step, we apply De Jong's resolution of singularities by alterations \cite[Theorem~4.1 and following remark]{De_Jong_alterations} (more precisely, since that reference requires $X$ to be integral, we apply it to every connected component of the normalisation of $X$, and then compose with the normalisation morphism). We obtain an alteration $h:\widetilde{X}\rightarrow X$ with $\widetilde{X}/k$ a smooth projective variety (smoothness can be achieved because $k$ is perfect). Recall that $h$ is proper surjective and generically finite. We choose a diagram of schemes with cartesian squares
\[
\xymatrix{
V \ar[r]_{\tilde{\jmath}} \ar[d]^{h_U} & \widetilde{X} \ar[d]^h & Z \ar[l]^{\tilde{\imath}} \ar[d]^{h_T}\\
U \ar[r]_j & X & T \ar[l]^i
}
\] 
with the following properties.
\begin{itemize}
\item $T$ is a nowhere dense closed subset of $X$ and $U$ is its open complement.
\item $h_U$ can be written as the composite of a purely inseparable finite morphism followed by a finite \'etale morphism.
\end{itemize}

Starting from the distinguished colocalisation triangle for the pair $(X,U)$ and applying $\pi_{X!}$, we obtain a triangle
\[
\pi_{X!}i_*i^!\pi_X^!\BQ_k\ra \pi_{X!}\pi_X^!\BQ_k\ra \pi_{X!}j_*j^!\pi_X^!\BQ_k\stackrel{+}{\ra} 
\]
that we rewrite as
\[
\pi_{T!}\pi_T^!\BQ_k\ra \pi_{X!}\pi_X^!\BQ_k\ra \pi_{X!}j_*\pi_U^!\BQ_k\stackrel{+}{\ra}.
\]
The left-hand term of this last triangle is in $\DA_{n,c}^\sgsm(k)$ by induction. To prove that the middle term is in $\DA_{n,c}^\sgsm(k)$, it remains to prove the same for the right-hand term. Since $h_U$ is finite and the composite of a purely inseparable morphism followed by an \'etale morphism, the separation property of $\DA$ \cite[Theorem~3.9]{Ayoub_Etale} together with \cite[Corollaire 2.1.164]{Ayoub_these_1} implies that there is a natural isomorphism of functors:
\[
h_{U!}h_U^!\simeq h_{U*}h_U^*.
\]
Now, \cite[Lemma 2.1.165]{Ayoub_these_1} implies that $\pi^!_U\BQ_k$ is a direct factor of $(h_U)_*h_U^*\pi_X^!\BQ_k$. Applying the isomorphism just above, we conclude that $\pi^!_U\BQ_k$ is a direct factor of $(h_U)_!h_U^!\pi^!_U\BQ_k$. This last motive is isomorphic to $(h_U)_* \pi_V^! \BQ_k\simeq (h_U)_{*} \tilde{\jmath}^{*} \pi_{\widetilde{X}}^! \BQ_k$ because $h_U$ is proper and $\tilde{\jmath}$ is \'etale. We get that $\pi_{X!}j_*\pi_U^!\BQ_k$ is a direct factor of $\pi_{X!} j_*h_{U*}\tilde{\jmath}^* \pi^!_{\widetilde{X}}\BQ_k\simeq \pi_{X!} h_*\tilde{\jmath}_* \tilde{\jmath}^* \pi^!_{\widetilde{X}}\BQ_k$. We have $\pi_{X!}h_*\simeq \pi_{\widetilde{X}!}$ since $h$ is proper, hence we deduce that $\pi_{X!}j_*\pi_U^!\BQ_k$ is a direct factor of $\pi_{\widetilde{X}!}\tilde{\jmath}_* \tilde{\jmath}^* \pi_{\widetilde{X}}^!\BQ_k$. Applying localisation to the pair $(\widetilde{X},V)$, the fact that $\widetilde{X}/k$ is smooth projective  and the induction hypothesis for $Z$ shows that $\pi_{X!}j_*\pi_U^!\BQ_k$ is in $\DA^\sgsm_{n,c}(k)$. This concludes the proof when $k$ is perfect.

We now treat the case of a general field $k$. By the perfect field case and continuity for $\DA^\sgsm_{n,c}(-)$ (Proposition~\ref{prop:subcats_continuity}) applied to the spectrum of the perfect closure of $k$, we see that there exists a finite purely inseparable extension $l/k$ with $(l/k)^*\pi_{X!}\pi_X^!\BQ_k$ in $\DA^\sgsm_{n,c}(l)$. By the separation property, we have an isomorphism of functors $\id\simeq (l/k)_*(l/k)^*$, so that it is enough to show Lemma \ref{lemm:hom_!_fields} below. This completes the proof of Claim \eqref{claim}, hence of the chains of equalities in the proposition.

Finally, the Verdier duality statement is just a restatement of Lemma \ref{lemm:duality_field} in the light of these chains of equalities.
\end{proof}

\begin{lemma}\label{lemm:hom_!_fields}
For a finite field extension $l/k$ and $g:Y\ra \Spec(l)$ a smooth projective morphism of relative dimension $\leq n$, there exists a smooth projective variety $g':Y'\ra k$ of dimension $\leq n$ such that $(l/k)_* g_\sharp \BQ_Y\simeq g'_\sharp \BQ_{Y'}\in \DA^\sgsm_{n,c}(k)$.    
\end{lemma}
\begin{proof}
  Let $\tilde{k}$ be the separable closure of $k$ in $l$, so that $\tilde{k}/k$ is finite separable and $l/\tilde{k}$ is finite purely inseparable. Assume the conclusion holds for $l/\tilde{k}$, i.e. that there exists $\tilde{g}:\widetilde{Y}\ra\Spec(\tilde{k})$ smooth projective of dimension $\leq n$ such that $(l/\tilde{k})_* g_\sharp \BQ_Y\simeq \tilde{g}_\sharp \BQ_{\widetilde{Y}}\in \DA^\sgsm_{n,c}(\tilde{k})$. We have $(l/k)_*g_\sharp \BQ_Y\simeq (k'/k)_*\tilde{g}_\sharp \BQ_{\widetilde{Y}}\simeq h_\sharp\BQ_{\widetilde{Y}}$ with $h:\widetilde{Y}\ra \Spec(k)$ since $k'/k$ is finite \'etale, and $\widetilde{Y}$ is smooth projective over $k$, of dimension $\leq n$. This shows that we can that assume $l/k$ is finite purely inseparable.
  
By treating separately the connected components of $Y$, we can assume that $Y$ is of pure dimension $n$. Let $F:\Spec(l)\ra \Spec(l)$ be a high enough power of the Frobenius of $l$ that factors through $k$. We denote by $\overline{F}$ the induced morphism $\Spec(k)\ra \Spec(l)$ and its natural lift $\Spec(k)\ra \Spec(k)$ (the corresponding power of $\mathrm{Fr}_k$). We have the following diagram of schemes, where the upper square is cartesian:
\[
\xymatrix{
Y' \ar[r]^{G} \ar[d]_{\pi_{Y'}} & Y \ar[d]^{\pi_Y}\\
\Spec(k) \ar[r]^{\overline{F}} \ar[rd]_F & \Spec(l) \ar[d]^{(l/k)}\\
  &  \Spec(k).
}
\]
By base change, the $k$-scheme $Y'$ is smooth projective and the morphism $G$ is finite purely inseparable, so that $Y'$ is of dimension $\leq n$. By the separation property of $\DA$, we have
\[
(l/k)_*(\pi_Y)_*\BQ_Y\simeq (l/k)_* (\pi_Y)_* G_*\BQ_{Y'}\simeq (l/k)_* \overline{F}_* (\pi_{Y'})_*\BQ_{Y'}\simeq F_*(\pi_{Y'})_*\BQ_{Y'}.
\]
Let $\mathrm{Fr}_{Y'}$ be the corresponding power of the absolute Frobenius on $Y'$. By naturality of the absolute Frobenius, we have $\pi_{Y'}\circ \mathrm{Fr}_{Y'}=F\circ \pi_{Y'}:Y'\ra \Spec(k)$. We deduce that
\[
F_*(\pi_{Y'})_*\BQ_{Y'}\simeq (\pi_{Y'})_*(\mathrm{Fr}_{Y'})_*\BQ_{Y'}\simeq (\pi_{Y'})_*\BQ_{Y'}\in \DA^n_{\sgsm}(k),
\]
where the last isomorphism follows by separation. By relative purity and the projection formula, we deduce that
\begin{eqnarray*}
(l/k)_* (\pi_Y)_{\sharp}\BQ_{Y} & \simeq & (l/k)_* ((\pi_Y)_* \BQ_Y\otimes \BQ_{l}(n)[2n])\\
& \simeq & (l/k)_*((\pi_Y)_*\BQ_Y)\otimes\BQ_k(n)[2n]\\
& \simeq & (\pi_{Y'})_*\BQ_{Y'}\otimes\BQ_k(n)[2n]\\
& \simeq & (\pi_{Y'})_{\sharp}\BQ_{Y'}.
\end{eqnarray*}
This completes the proof of the lemma.
\end{proof}

\subsection{Homological vs cohomological motives}
\label{sec:n_hom_coh}

\begin{prop}
\label{prop:hom_cohom_twists}
Let $S$ be a scheme and $n\geq 0$. We have
\[
\DA^n_{(c)}(S) = \DA_{n,(c)}(S)(-n)
\]
In particular, we have $\DA^0_{(c)}(S)=\DA_{0,(c)}(S)$.
\end{prop}
\begin{proof}
In both directions, it is enough to check the inclusion for a family of compact generators.

  Let $f:X\ra S$ be a smooth morphism of relative dimension $i\leq n$ (we can reduce to this case by considering connected components of $S$ and $X$). By relative purity, we have
\[
f_{\sharp}\BQ_X(-n) \simeq f_!\BQ_X(i-n)[2i]
\]
which is in $\DA^n(S)$ by Propositions~\ref{prop:permanence_coh_n} and \ref{prop:n_monoidal}.

The other inclusion is true for smooth cohomological $n$-motives by the same relative purity argument. For general compact cohomological $n$-motives (which include the generating family), we argue as follows. By Corollary~\ref{cor:localisation_subcats} \ref{rad_sub}, we can assume $S$ reduced. We then proceed by noetherian induction. Let $M\in \DA^n(S)$. The restriction of $M$ to any generic point of $S$ is smooth by Proposition~\ref{prop:subcats_field}. There we can apply the smooth case and see that $\eta^* M\in \DA_{n,c}(\eta)(-n)$ for any generic point $\eta$ of $S$. Then we apply continuity for compact homological $n$-motives (Proposition~\ref{prop:subcats_continuity}) to find a dense open immersion $j:U\ra S$ with $j^*M\in \DA_{n,c}(U)(-n)$. Applying the induction hypothesis, localisation and the fact that $i_*$ preserves homological $n$-motives for $i$ closed immersion (Proposition~\ref{prop:permanence_hom_n} (ii)) completes the proof.
\end{proof}

\section{Commutative group schemes and motives}
\label{sec:albanese}

Several motives of interest for this paper are obtained from group schemes or complexes of group schemes. The main examples we are interested in are smooth commutative group schemes (Section \ref{sec:mot_group_schemes}), Deligne $1$-motives (Section \ref{sec:mot_del_mot}), and the smooth Picard complex (Section~\ref{sec:smooth_picard}).

\subsection{Motives of commutative group schemes}
\label{sec:mot_group_schemes}
 In this section, we introduce the relevant definitions and reformulate and extend results from \cite{AHPL} and \cite{Orgogozo} in this language. For the rest of the section, fix a noetherian finite-dimensional scheme $S$. 

In \cite[Thm D.1]{AHPL}, we constructed a functorial cofibrant resolution of the sheaf $G\otimes \BQ$ for $G$ a smooth (locally of finite type) commutative group scheme over $S$. Let us recall the statement, minus the statement about functoriality which we discuss in details immediately afterwards.

\begin{lemma}\cite[Thm D.1]{AHPL}
  \label{lemma:resolution}
Let $(\mathcal{S},\tau)$ be a Grothendieck site. 
We denote $\BZ(-)$ the ``free abelian sheaf'' functor (the sheafification of the sectionwise free abelian group functor).

There is a functor:
\[
A:\Shv_\tau(\CS,\BZ)\ra \mathbf{Cpl}_{\geq 0}\Shv_\tau(\CS,\BZ)
\] 
together with a natural transformation
\[
r:A\rightarrow (-)[0]
\]
satisfying the following properties.
\begin{enumerate}
\item \label{explicit_form_2} For any $\CG\in \Sh_\tau(\CS,\BZ)$ and $i\geq 0$, the sheaf $A(\CG)_i$ is of the form $\bigoplus_{j=0}^{d(i)}\BZ(\CG^{a(i,j)})$ for some $d(i),a(i,j)\in\BN$. 
% \item \label{lifting_2} There is a natural transformation $\tilde{a}:\BZ(-)[0]\ra A$ which lifts the addition map $a:\BZ(-)\rightarrow \id$; that is, one has $a[0]=r\tilde{a}$.
% \item \label{site_functoriality_2} The functor $A$ and the transformations $r$ and $\tilde{a}$ are compatible with pullbacks by morphisms of sites. 
\item \label{Q-resolution_2} The map $r\otimes \BQ$ is a quasi-isomorphism.
\end{enumerate}
\end{lemma}

Theorem D.1 in \cite{AHPL} also contained a statement about the functoriality with respect to morphisms of sites. As was pointed out by the reviewer, functoriality for morphisms of sites is insufficient for the main application of this theorem to our setting (and indeed for the similar application in \cite{AHPL}), which is to compute pullbacks of certain associated motives. We set out to repair this problem.

Let us first recall a bit of notation. Given a continuous map of sites $u:(\mathcal{S},\tau)\ra (\mathcal{S}',\tau')$, we have an adjunction
\[
u^{s}:\Sh_{\tau}(\mathcal{S},\Set) \leftrightarrow \Sh_{\tau'}(\mathcal{S}',\Set):u_{s}
  \]
and that when $u$ defines a morphism of sites $F:(\mathcal{S}',\tau')\ra(\mathcal{S},\tau)$ (i.e., when $u^{s}$ commutes with finite limits; note that the morphism of sites goes in the inverse direction), we have by definition $F^{-1}:=u^{s}$ and $F_{*}:=u_{s}$.

The functoriality for sheaves of abelian groups (resp. $\BQ$-vector spaces) with respect to continuous morphisms of sites requires some care. We have similarly defined adjunctions
\[
u_{\BZ}^{s}:\Sh_{\tau}(\mathcal{S},\BZ) \leftrightarrow \Sh_{\tau'}(\mathcal{S}',\BZ):u_{s}^{\BZ}
\]
and
\[
u_{\BQ}^{s}:\Sh_{\tau}(\mathcal{S},\BQ) \leftrightarrow \Sh_{\tau'}(\mathcal{S}',\BQ):u_{s}^{\BQ}
\]
for sheaves of abelian groups and $\BQ$-vector spaces \cite[Expos\'e III Proposition 1.7]{SGA4_1}.

\begin{lemma}\label{lem:representables}
  Let $u:(\mathcal{S},\tau)\ra (\mathcal{S}',\tau')$ be a continuous map of subcanonical sites. Let $U\in \CS$, which we identify with the corresponding representable sheaf of sets on $\CS$. Then there are natural isomorphisms
  \[
u^s(U)\simeq u(U)\text{  and  }u_\BZ^s(\BZ(U))\simeq \BZ(u(U))\text{  and  }u_\BQ^s(\BQ(U))\simeq \BQ(u(U))
\]
\end{lemma}  
\begin{proof}
The result for sheaves of sets is clear by adjunction and the Yoneda lemma, and the analoguous result for sheaves of abelian groups and $\BQ$-vector spaces then follows from \cite[Expos\'e III Proposition 1.7.5)]{SGA4_1}.
\end{proof}  

Note that, because $u^s$ does not commute with finite limits in general, $u_\BZ^s$ and $u_\BQ^s$ do not always coincide with $u^s$ after forgetting the algebraic structure. The following result seems well-known and implicit in some places in the literature but I could not find a good reference. 

\begin{prop}\label{prop:pullback_finite_products}
  Let $u:(\mathcal{S},\tau)\ra (\mathcal{S}',\tau')$ be a continuous map of sites such that $u^{s}$ commutes which finite products (and in particular final objects). Then the functors $u_\BZ^s$ and $u_\BQ^s$ coincide with $u^s$ after forgetting the algebraic structure. Moreover, we have a natural isomorphism
  \[
u_\BQ^s(-\otimes\BQ)\simeq u_\BZ^s(-)\otimes\BQ.
\]
\end{prop}  
\begin{proof}
  We first treat the case of $u_\BZ^s$. The idea is to construct another left adjoint $u^{s'}$ to $u^s_{\BZ}$ (resp. $u^s_\BQ$) for which the property is clearly true, and then use the uniqueness of left adjoints up to natural isomorphism. Let $F\in \Sh_\tau(\CS,\BZ)$. The underlying sheaf of sets of $u^{'s}(F)$ is simply $u^s(F)$. The structure of group object of $u^{'s}(F)$ is transported from the one of $F$ by using the isomorphisms
  \[
u^s(F\times F)\simeq u^s(F)\times u^s(F)\text{  and  }u^s(*)\simeq *
  \]
  where $*$ denotes the final objects. It is clear that the corresponding structure maps satisfy the abelian group axioms (using functoriality and associativity of direct products). Now we have, for all $F\in \Sh_\tau(\CS,\BZ)$ and $G\in \Sh_{\tau'}(\CS',\BZ)$,
  \[
\Hom_{\Sh_{\tau'}(\CS',\BZ)}(u^{s'}(F),G)\subset \Hom_{\Sh_{\tau'}(\CS',\Set)}(u^{s}(F),G)
  \]
is the subset of homomorphisms of sheaves of sets compatible with the abelian group structures on $u_s(F)$ and $G$, which by adjunction (including naturality properties of adjunctions) is naturally in bijection with the subset
  \[
\Hom_{\Sh_{\tau}(\CS,\BZ)}(F,u^s_{\BZ}(G))\subset \Hom_{\Sh_{\tau}(\CS,\Set)}(F,u_s(G))
  \]
of homomorphisms of sheaves of sets compatible with the abelian group structures on $F$ and $u_s(G)$, that is, the group of homomorphisms of sheaves of abelian groups. This proves that $u^{s'}$ is a left adjoint to $u_s^{\BZ}$ and completes the proof for $u^s_\BZ$. The same argument, replacing abelian groups by $\BQ$-vector spaces, gives the result for $u_\BQ^s$.
  Finally, write $V^{\BQ,\BZ}_\CS:\Sh_{\tau}(\CS,\BQ)\ra \Sh_{\tau}(\CS,\BZ)$ for the natural forgetful functor. Then we clearly have a natural isomorphism $V^{\BQ,\BZ}_\CS u_s^\BQ\simeq u_s^\BZ V^{\BQ,\BZ}_{\CS'}$ and the natural isomorphism involving tensor products by $\BQ$ follows by adjunction.
\end{proof}  

\begin{prop}\label{prop:EM_functoriality}
 Let $u:(\mathcal{S},\tau)\ra (\mathcal{S}',\tau')$ be a continuous map of sites. Assume that the functor $u^{s}$ on sheaves of sets commutes with finite direct products.
Let $\CG\in \Sh_{\tau}(\mathcal{S},\tau)$ be a sheaf of abelian groups. Then there exists an isomorphism of complexes $b_{u,\CG}:u^s_\BZ(A(\CG))\ra A(u^s_\BZ\CG)$ which is termwise compatible with the isomorphisms $u^s_\BZ(\BZ(\CG^{a(i,j)}))\simeq \BZ(u^s_\BZ(\CG^{a(i,j)}))$ of \cite[Expos\'e III Proposition 1.7.5)]{SGA4_1} and which makes the diagram
\[
\xymatrixcolsep{4pc}
\xymatrix{
u^s_\BZ(A(\CG)) \ar[r]^{u^s_\BZ (r(\CG))} \ar[d]_{b_{u,\CG}}^{\sim}& u^s_\BZ(\CG) \\
A(u_\BZ^s(\CG)) \ar[ru]_{r(u(\CG)} &  
}
\]
commute.
\end{prop}
\begin{proof}
  We go back to the construction of the complex $A(\CG)$ from \cite[Chapter I \S 1]{Breen}. There is an unfortunate shift by $1$ between the definitions of $A$ of \cite[Appendix D]{AHPL} and \cite[Chapter I \S 1]{Breen}, which was not pointed out explicitely in \cite{AHPL}; we follow the convention of \cite[Appendix D]{AHPL}. We do not need to present the whole construction, only the description given after equation (1.8) in loc.cit. For every $q\geq 0$, consider the finite set $\CI_q$ of all tuples $(k_1,\ldots,k_{r-1})$, with all $k_i>0$ and $\sum_{i=1}^{k_{r-1}}k_i=q$ (\cite{Breen} has $q-1$ here; this is where the mismatch of a shift by $1$ occurs). For such a tuple $I$, write $A_I(\CG)=\BZ(\CG^{\times (|I|+1)})$. Then
  \[
A(\CG)_q := \prod_{I\in \CI_q}A_I(\CG) 
\]
and the differentials are constructed in a complicated but purely combinatorial fashion from the addition map of $\CG$. Since the description above leave ambiguous what happens for $q=0$, let's be more explicit: $A(\CG)$ starts with
\[
\BZ(\CG^{\times 2})\times \BZ(\CG^{\times 3})^{\times 2}\times \BZ(\CG^{\times 4})\ra \BZ(\CG^{\times 2})\times \BZ(\CG^{\times 3})\ra \BZ(\CG^\times 2)\ra \BZ(\CG) \ra 0
\]
and the first non-zero differential is $\BZ(\CG^{\times 2})\ra \BZ(\CG), [(g,h)]\mapsto [g]+[h]-[g+h]$.

Since $u^s_\BZ$ commutes with direct products, and applying the isomorphisms of \cite[Expos\'e III Proposition 1.7.5)]{SGA4_1}, the description makes it clear that we have a natural isomorphism
\[
b_{u,\CG}:u^s_\BZ(A(\CG))\simeq A(u^s_\BZ(\CG)).
\]
We also recall the construction of the map $r:A(\CG)\ra \CG$. By the above, we have $A(\CG)_0=\BZ(\CG)$. Moreover, the morphism $A(\CG)_1\ra A(\CG)_0$ composed with the addition map $a_\CG:\BZ(\CG)\ra \CG$ vanishes, so that we get a morphism $r:A(\CG)\ra \CG[0]$. The commutation of the last diagram of the proposition follows easily from this definition.
\end{proof}

We can now come to the application relevant to motives.

\begin{lemma}\label{lem:pullbacks_products}
Let $f:T\ra S$ be a morphism of schemes. We consider $\Sm/S$, $\Sm/T$ as equipped with the \'etale topology. The continuous functor
  \[
f^{-1}:\Sm/S \ra \Sm/T,\ X\mapsto X\times_S T
  \]
and the associated functors $(f^{-1})^s$, $(f^{-1})_\BZ^s$ and $(f^{-1})_\BQ^s$ all commute with finite products. The functors $(f^{-1})_\BZ^s$ and $(f^{-1})_\BQ^s$ coincide with $(f^{-1})^s$ after forgetting the algebraic structures.
\end{lemma}  
\begin{proof}
The fact that $f^{-1}$ commutes with finite products is clear. By \cite[Proposition 2.2.36]{Olsson_stacks_book}, we deduce that $(f^{-1})^s$ commutes with finite products. Since the forgetful functor from abelian groups or $\BQ$-vector spaces to sets preserves limits, we conclude from Proposition \ref{prop:pullback_finite_products} that $(f^{-1})_\BZ^s$ and $(f^{-1})_\BQ^s$ both coincide with $(f^{-1})^s$ after forgetting the algebraic structures and commute with finite products.
\end{proof}  

Note that by \cite[Example 3.1.19]{Morel_Voevodsky}, the functor $(f^{-1})^{s}$ does not preserve finite limits in general.

In the following, we make a slight change of notations to be compatible with our notations for pullback functors between categories of complexes of sheaves of $\BQ$-vector spaces: given a morphism of schemes $f:T\ra S$, we write $\underline{f}^{-1}$ for the functor $(f^{-1})^s_\BQ$ of Lemma \ref{lem:pullbacks_products}, and we write $f^{-1}$ for its (left) derived functor.

\begin{prop}
\label{prop:pullback_complex}
Let $K_*$ be a bounded complex of smooth commutative group schemes over $S$ and $f:T\ra S$ a morphism of schemes. Write $K_{*}\times_{S}T$ for the bounded complex of smooth commutative group schemes over $T$ obtained by base change. We have an isomorphism
\[
R_f:f^{-1} (K_*\otimes \BQ) \stackrel{\sim}{\lra} (K_{*}\times_{S}T)\otimes\BQ
\]
in $D(\Sm/S)$ which is natural in $K_*$. Moreover, $R_f$ is compatible with further pullbacks: for $g:U\ra T$, the diagram
\[
\xymatrix{
g^{-1} f^{-1} (K_*\otimes\BQ) \ar[r]_{\sim} \ar[d]_{R_f}^{\sim} & (fg)^{-1} (K_*\otimes\BQ) \ar[r]_{\sim}^{R_{fg}} & (K_*\times_{S}U)\otimes \BQ \ar[d]^{\sim}\\
g^{-1}((K_*\times_{S}T)\otimes\BQ) \ar[rr]_{R_g}^{\sim} & & (K_*\times_{S}T\times_{T}U)\otimes\BQ
}
\]
commutes.
\end{prop}
\begin{proof}
We apply Lemma \ref{lemma:resolution} to the individual sheaves $K_n$, and use the functoriality of the construction deduced from Proposition \ref{prop:EM_functoriality} and \ref{lem:pullbacks_products}, keeping in mind the change in notation explained before the proposition. This yields a double chain complex $A_{\bullet}(K_*)$ such that $A_{n}(K_{*})=0$ for all $n<0$ together with a map $A_{0}(K_*)\stackrel{r(K_*)}{\ra} K_*$ . For every fixed $n\in\BZ$, the induced morphism $A_{\bullet}(K_{n})\otimes\BQ\stackrel{r_\BQ(K_n)}{\ra} K_{n}[0]\otimes\BQ$ of chain complexes is a quasi-isomorphism. Because $K_{*}$ is bounded, for every $m\in\BZ$ there are finitely many pairs $(p,q)\in \BZ^{2}$ with $A_{p}(K_{q})\neq 0$.

  Let $B_{\BQ}(K_{*})$ be the $\oplus$-total complex of $A_{\bullet}(K_{*})$ and $r_\BQ(K_*):B_{\BQ}(K)_*\ra K_*$ be the morphism of chain complexes induced by the map $A_{0}(K_{*})\ra K_{*}$. By Lemma \ref{lemma:resolution}, Proposition \ref{prop:EM_functoriality} and (the dual of) Lemma \cite[Tag 0133]{stacks-project}, we have the following properties.

\begin{enumerate}[label={\upshape(\roman*)}]
\item The map $r_\BQ(K_*)$ is a quasi-isomorphism. 
\item For all $i\in \BZ$, the sheaf $B_\BQ(K)_i$ is of the form $\BQ(H_i)$ for some smooth commutative group scheme $H_i$ over $S$ (a fibre product of various copies of the $K_n$'s); therefore, $B_{\BQ}(K_*)$ is a cofibrant object in the projective model structure on  $\Cpl(\Shv_{\et}(\Sm/S,\BQ))$, and by the first point $r_{\BQ}(K_{*})$ is a cofibrant resolution of $K_{*}$.
\item Let $f:T\ra S$ be a morphism of schemes. The formation of $B_\BQ(K_*)$ and $r_\BQ(K_*)$ is compatible with pullback, in the sense that there exists an isomorphism of complexes $b_{f,K_*}:\underline{f}^{-1}(B_{\BQ}(K_*))\stackrel{\sim}{\ra} B_{\BQ}(K_*\times_{S}T)$ which makes the following diagram in $\Cpl(\Shv_{\et}(\Sm/T,\BQ))$
\[
\xymatrixcolsep{4pc}
\xymatrix{
\underline{f}^{-1}(B_{\BQ}(K_*)) \ar[r]^{\underline{f}^{-1} (r_\BQ(K_*))} \ar[d]_{b_{f,K_*}}^{\sim}& (K_{*}\times_{S}T)\otimes\BQ \\
B_{\BQ}(K_*\times_{S}T) \ar[ru]_{r_{\BQ}(K_{*}\times_{S}T)} & 
}
\]
commute. Here, we use the fact that the pullback functor on sheaves of $\BQ$-vector spaces coincides with one on sheaves of sets by Proposition \ref{prop:pullback_finite_products}, and the fact that on representable the pullback functor on sheaves of sets is simply the extension of $-\times_{S}T$.
\end{enumerate}
As $r_{\BQ}(K_{*})$ is a cofibrant resolution in the projective model structure, we have an isomorphism in $D(\Sm/S)$ given by
\[
f^{-1} (K_*\otimes\BQ)\stackrel[\sim]{f^{-1}(r_\BQ(K_*))}{\xleftarrow{\hspace*{1cm}}} f^{-1}(B_{\BQ}(K_*))\stackrel{\sim}{\leftarrow} \underline{f}^{-1}(B_\BQ(K_*)).
\]
We define $R_f$ as the composition
\[
f^{-1} (K_*\otimes\BQ)\stackrel[\sim]{f^{-1}(r_\BQ(K_*))^{-1}}{\xrightarrow{\hspace*{1cm}}} \underline{f}^{-1}(B_\BQ(K_*))\stackrel[\sim]{b_{f,K_{*}}}{\xrightarrow{\hspace*{1cm}}} B_{\BQ}(K_{*}\times_{S}T) \stackrel[\sim]{r_{\BQ}(K_{*}\times_{S}T)}{\xrightarrow{\hspace*{1cm}}} (K_*\times_{S}T)\otimes\BQ
\]
where we use the isomorphism $b_{f,K_{*}}$ and the quasi-isomorphism $r_{\BQ}(K_{*}\times_{S}T)$.

The proof of compatibility with further pullbacks is a long exercise in commutative diagrams. The key point is the commutation of
\[\label{diagram}\tag{**}
  \xymatrix{
    \underline{g}^{-1}\underline{f}^{-1}B_{\BQ}(K_{*}) \ar[r]_{\sim} \ar[d]_{b_{f},K_{*}} & \underline{fg}^{-1}B_{\BQ}(K_{*}) \ar[r]^{b_{fg,K_{*}}}_{\sim} & B_{\BQ}(K_{*}\times_{S}U) \ar[d]_{\sim} \\
    \underline{g}^{-1}B_{\BQ}(K_{*}\times_{S}T) \ar[rr]^{b_{g,K_{*}\times_{S}T}} & & B_{\BQ}((K_{*}\times_{S}T)\times_{T}U)
  }
  \]
  which after unraveling the definitions comes down to the fact that for two composable continuous morphisms of sites $u,v$ such that both $u^{s}$ and $v_{s}$ commute with finite products, we have a natural isomorphism $u^{s}_{\BQ}v^{s}_{\BQ}\simeq (uv)^{s}_{\BQ}$, as is clear from the fact that the same is obvious for the right adjoints $u^{s}$ and $v^{s}$.
  
Besides Diagram \eqref{diagram}, the proof of the compatibility then consists in iterated applications of the naturality of isomorphisms of the form $h^{-1}(C)\ra \underline{h}^{-1}(C)$ for $C$ cofibrant and of the naturality of natural isomorphisms $f^{-1}g^{-1}\simeq (gf)^{-1}$ (both derived and underived).

\end{proof}

\begin{cor}
\label{coro:pullback_complex_DA}
Let $K_*$ be a bounded complex of smooth commutative group schemes over $S$ and $f:T\ra S$ be a morphism of schemes. We have natural isomorphisms
\[
R_f:f^*(K_*\otimes \BQ) \stackrel{\sim}{\lra} (K_*\times_{S}T)\otimes\BQ
\]
in $\DA^{\eff}(S)$ and
\[
R_f:f^* \Sigma^\infty (K_*\otimes \BQ)\stackrel{\sim}{\lra} \Sigma^\infty (K_*\times_{S}T)\otimes\BQ
\]
in $\DA(S)$.
These isomorphisms are compatible with further pullbacks in the same way as in the previous proposition.
\end{cor}
\begin{proof}
The first isomorphism follows directly from Proposition~\ref{prop:pullback_complex}. The second follows from the first together with the commutation of $f^*$ and $\Sigma^\infty$. 
\end{proof} 

For some arguments, we need to use motives with transfers of commutative group schemes. Let $S$ be a scheme and $G$ a smooth (locally of finite type) commutative group scheme over $S$. Recall that the \'etale sheaf $G\otimes\BQ$ on $\Sm/S$ admits a canonical structure of sheaf with transfers, which is functorial in $G$. We write $G^\tr_\BQ$ for the resulting sheaf with transfers. Recall that there are adjunctions 
\[
a_{\tr}:\DA^{(\eff)}(S) \leftrightarrows \DM^{(\eff)}(S):o^{\tr}
\]
which relate motives with and without transfers. By construction, we have \(\underline{o}^\tr G^\tr_\BQ = G_{\BQ}\), and $\underline{o}^\tr$ preserves $\BA^1$-equivalences \cite[Lemme 2.111]{Ayoub_Galois_1}.

\begin{prop}\cite[Proposition 3.10]{AHPL}\label{prop:with_without_tr}
Let $S$ be an excellent scheme and $G$ a smooth commutative group scheme over $S$. Then the counit morphisms
\[
a_{\tr} o^{\tr}G^\tr_\BQ \stackrel{\sim}{\ra} G^\tr_\BQ
\]
in $\DM^\eff(S)$ and 
\[
a_{\tr} o^{\tr} \Sigma^\infty_\tr G^\tr_\BQ \stackrel{\sim}{\ra} \Sigma^\infty_\tr G^\tr_\BQ
\]
in $\DM(S)$ are isomorphisms.
\end{prop}

An important consequence for us is the following computation, which consists of translating a classical result of Voevodsky to our context, and which we will generalize later on.

\begin{prop}\label{prop:motive_curve_field}
Let $k$ be a field and $C/k$ be a smooth projective geometrically connected curve. There exists a direct sum decomposition
\[
M(C)\simeq \BQ\oplus \Sigma^\infty \Jac(C)_\BQ \oplus \BQ(1)[2]
\]
in $\DA(k)$.
\end{prop}
\begin{proof}
We first assume that $k$ is perfect. For a smooth projective connected curve $C$ over $k$ with a rational point, Voevodsky has computed the motive $M^{\eff}_{\tr}(C)\in \DM^\eff(k)$ (see e.g. \cite[Proposition~2.5.5]{BVK}) and shown that
\[
M^{\eff}_{\tr}(C)\simeq \BQ\oplus (\Jac(C)^\tr_\BQ) \oplus \BQ(1)[2].
\]
The role of the rational point in this argument can be played by a $0$-cycle of degree $1$ as long as $C$ is geometrically connected; such a cycle exists with rational coefficients on any geometrically connected smooth projective curve. By Proposition~\ref{prop:with_without_tr} and the remarks preceding it, we have
\[\Jac(C)^\tr\simeq a^\tr o^\tr \Jac(C)^\tr\simeq a^\tr \underline{o}^\tr \Jac(C)^\tr \simeq a^\tr \Jac(C).\]
Applying $\Sigma^\infty_\tr$ and using that $a^\tr$ commutes with suspension, we get an isomorphism
\[
M_{\tr}(C)\simeq \BQ\oplus a^\tr \Sigma^\infty (\Jac(C)_\BQ) \oplus \BQ(1)[2]
\]
in $\DM(k)$. The adjunction $a^\tr:\DA(k)\leftrightarrows \DM(k):o^\tr$ is an equivalence of categories by \cite[Corollary~16.2.22]{Cisinski_Deglise_BluePreprint}. This implies that $o^\tr M_{\tr}(C)\simeq o^\tr a^\tr M(C)\simeq M(C)$ and similarly $o^\tr\BQ\simeq\BQ$ and $o^\tr \BQ(1)[2]\simeq \BQ(1)[2]$. Applying $o^\tr$ to the isomorphism above, we thus get an isomorphism
\[
M(C)\simeq \BQ\oplus \Sigma^\infty(\Jac(C)_\BQ)\oplus\BQ(1)[2]
\]
as required. 

Let $k$ be an arbitrary field, and $k^\perf$ a perfect closure of $k$. Write $\xi:\Spec(k^\perf)\ra \Spec(k)$. The field extension $k^\perf/k$ is a filtered union of finite purely inseparable field extensions. By the separation and continuity properties of $\DA_{c}(-)$, the pullback functor $\xi^*:\DA_c(k)\ra \DA_c(k^\perf)$ is an equivalence of categories with inverse $\xi_*$. We thus have
\begin{eqnarray*}
  M(C) & \simeq & \xi_* \xi^* M(C)\\
       & \simeq & \xi_* M(C_{k^\perf}) \\
       & \simeq & \xi_* \BQ \oplus \xi_*\Sigma^\infty \Jac(C_{k^\perf})_\BQ \oplus \xi_*\BQ(1)[2] \\
       & \simeq & \xi_*\xi^*\BQ \oplus \xi_*\xi^*\Sigma^\infty \Jac(C)_\BQ \oplus \xi_*\xi^*\BQ(1)[2] \\
       & \simeq & \BQ\oplus \Sigma^\infty \Jac(C)_\BQ \oplus \BQ(1)[2]
\end{eqnarray*}
where the third isomorphism follows from the perfect field case and the fourth isomorphism uses the base change property of the Jacobian of a curve and Proposition \ref{prop:pullback_complex}.
\end{proof}

We also need an alternative description of the motive $\Sigma^\infty (\Gm\otimes\BQ)$ (a relative, rational version of the standard description of the motivic complex $\BZ(1)$).

\begin{prop}\label{prop:Gm_Q1}
There is a canonical isomorphism
\[
u_S:\Sigma^\infty (\Gm\otimes \BQ) \stackrel{\sim}{\ra} \BQ_S(1)[1]
\] 
in $\DA(S)$. The isomorphism $u_S$ is compatible with pullbacks and the isomorphisms $R_f$ of Corollary~\ref{coro:pullback_complex_DA}: for $f:T\ra S$, the diagram 
\[
\xymatrix{
f^*\Sigma^\infty(\BG_{m,S}\otimes\BQ) \ar[r]^{R_f}_{\sim} \ar[d]_{u_S}^{\sim} & \Sigma^\infty (\BG_{m,T}\otimes\BQ) \ar[d]_{u_T} \\
f^*(\BQ_S(1)[1]) \ar[r]^{\sim} & \BQ_T(1)[1]
}
\]
commutes.
\end{prop}
\begin{proof}
By Theorem~\cite[Theorem 3.3]{AHPL} in the special case $G=\Gm$ (with the ``Kimura dimension'' $\mathrm{kd}(\Gm/S)$ of the statement equal to $1$), there is an isomorphism
\[
\Psi:=\Psi_{\Gm/S}:M_S(\Gm)\simeq \BQ\oplus \Sigma^\infty(\Gm\otimes\BQ).
\]
It is compatible with pullbacks and the isomorphisms $R_f$ of Corollary~\ref{coro:pullback_complex_DA} (This is the precise meaning of ``compatible with pullbacks'' in loc.cit.). By definition, $\BQ_S(1)[1]$ is the reduced motive of $M_S(\Gm)$ pointed at the unit section of $\Gm$, and it follows from the naturality of $\Psi_{G/S}$ applied to the neutral section in $G$ that the direct factor $\BQ_S(1)[1]$ corresponds to the direct factor $\Sigma^\infty(\Gm\otimes\BQ)$. This yields an isomorphism $\widetilde{\Psi}:\BQ_S(1)[1]\simeq \Sigma^\infty(\Gm\otimes\BQ)$, and we put 
$u_S:=\widetilde{\Psi}^{-1}$.
\end{proof}

\begin{remark}
Various results and constructions in this paper could be simplified if we knew the effective analogue of Proposition~\ref{prop:Gm_Q1}, i.e., that the natural map $\BQ(1)\ra \Gm[-1]\otimes\BQ$ in $\DA^{\eff}(S)$ is an isomorphism. The corresponding statement in $\DM^{\eff}(S)$ is known if $S$ is normal \cite[Proposition 11.2.11.]{Cisinski_Deglise_BluePreprint}, hence in $\DA^{\eff}(S)$ for $S$ normal scheme of finite type over a field of characteristic $0$ by \cite[Th\'eor\`eme B.1]{Ayoub_Galois_1}.
\end{remark}

We also need a version with transfers of this statement.

\begin{cor}\label{cor:Gm_Q1_tr}
Let $S$ be an excellent scheme. There is a canonical isomorphism
\[
u_S^\tr:\Sigma^\infty_\tr \Gm^\tr\otimes\BQ \stackrel{\sim}{\ra} \BQ_S(1)[1].
\]
It is compatible with pullbacks in the same way as in Proposition~\ref{prop:Gm_Q1}. Modulo the isomorphism of Proposition~\ref{prop:with_without_tr}, we have in fact
\[
a_{\tr}u_S = u_S^\tr.
\] 
\end{cor}
\begin{proof}
For our purposes, it is enough to \emph{define} $u_S^\tr$ as $a_{\tr}u_S$ modulo the isomorphism of Proposition~\ref{prop:with_without_tr}. The claim then follow from Proposition~\ref{prop:Gm_Q1}.   
\end{proof}

% \textbf{TODO:determine if a more precise description is necessary. If yes, not difficult using AHPL Proposition 2.11}

\begin{cor}\label{cor:motive_torus}
Let $T/S$ be a torus, and $X_*(T)$ its cocharacter lattice. There is an isomorphism
\[
\Sigma^\infty T_\BQ\simeq \Sigma^\infty X_*(T)_\BQ (1)[1].
\]
In particular, if $S$ is geometrically unibranch, the motive $\Sigma^\infty T_\BQ$ is in $\DA_{1,c}^{\sgsm}(S)$. 
\end{cor}
\begin{proof}
In this proof, we distinguish between derived and underived tensor products for clarity. There is a natural morphism $X_*(T)\underline{\otimes}\Gm\ra T$ of \'etale sheaves on $\Sm/S$, which is an isomorphism (this can be checked \'etale locally, hence for a split torus, where it is obvious). Since the functor $\Sigma^\infty$ is monoidal, we have $\Sigma^\infty (X_*(T)_\BQ\otimes (\Gm\otimes\BQ))\simeq \Sigma^\infty (X_*(T)_\BQ)\otimes \Sigma^\infty (\Gm\otimes\BQ)\simeq \Sigma^\infty X_*(T)_\BQ (1)[1]$ (by Proposition~\ref{prop:Gm_Q1}). It remains to check that the tensor product $X_*(T)\underline{\otimes}\Gm$ coincides with the derived tensor product; this follows from the fact that the lattice $X_*(T)$ is \'etale locally free, thus flat.

If $S$ is geometrically unibranch, $X_*(T)_\BQ$ is a direct factor of the sheaf $\BQ(V)$ for $V/S$ finite \'etale by Lemma~\ref{lemm:permutation_torus}, so it is (strongly) geometrically smooth.
\end{proof}

\begin{remark}
For more precise (integral) results on motives attached to tori over a field, see \cite[\S 7]{HK}. 
\end{remark}

We can now give a result which is our main source of compact homological 1-motives.

\begin{prop}\label{prop:smooth_complex_compact}
Let $G$ be a smooth (not necessarily of finite type) commutative group scheme over $S$. Then $\Sigma^\infty G_\BQ$ lies in $\DA_{1,c}(S)$.
\end{prop}
\begin{proof}
  Write $M=\Sigma^\infty G_\BQ$.  By \cite[Theorem~3.3.(3)]{AHPL}, $M$ is a compact motive. It remains to show that $M$ is an homological $1$-motive. The proof of \cite[Theorem~3.3.(3)]{AHPL} essentially establishes this as well, but we provide an argument for convenience. By compactness and Proposition~\ref{prop:punctual_car}, it is enough to show that for all $s\in S$, the motive $s^*M$ is in $\DA_1(s)$.

  Let $\kappa(s)^\perf$ a perfect closure of $\kappa(s)$, and write $\xi:\Spec(\kappa(s)^\perf)\ra \Spec(\kappa(s))$. The field extension $\kappa(s)^\perf/\kappa(s)$ is a filtered union of finite purely inseparable field extensions. By continuity for $\DA_{1,c}(-)$ and Corollary \ref{cor:localisation_subcats} \ref{rad_sub}, it is enough to show that $\xi^* s^*M\in \DA_{1}(\kappa(s)^\perf)$. By Proposition \ref{prop:pullback_complex}, we have $\xi^*s^*M\simeq \Sigma^\infty (G_{\kappa(s)^{\perf}})_{\BQ}$. We are thus reduced to the case where $S$ is the spectrum of a perfect field $k$.

  The group scheme $G$ over the field $k$ has a neutral component $G^\circ$ which is smooth and of finite type. The \'etale quotient sheaf $\pi_0(G)=G/G^\circ$ is representable by an \'etale group scheme (see Lemma \ref{lem:quotient_connected} below and the remark following it), hence can be written as a filtered colimit of \'etale locally constant finite type sheaves of abelian groups. Since we are working with rational coefficients, we can assume that those group schemes are in fact lattices. Using Lemma~\ref{lemm:permutation_torus}, we then conclude that the motive $\Sigma^\infty (G/G^\circ)_\BQ$ lies in $\DA_0(k)\subset \DA_1(k)$. In the case of a smooth commutative connected algebraic group, we reduce by a standard d\'evissage to the cases of unipotent algebraic groups, tori and abelian varieties.

  A unipotent algebraic group over a perfect field has a composition series with $\Ga$ factors, and the motive $\Sigma^\infty\Ga\otimes\BQ$ is trivial by \cite[Lemma 7.4.5]{AEH} (proved in $\DM^{\eff}(k)$, which yields the result in $\DA(k)$ by applying $\Sigma^\infty o_{\tr}$). The case of tori follows from the case of lattices above together with Corollary \ref{cor:motive_torus}. In the case of abelian varieties, using \cite[Theorem~11]{Katz_SpaceFill} reduces us to to the case of a Jacobian $\Jac(C)$ of a smooth projective curve $C/k$ with a rational point. The fact that $\Sigma^\infty (\Jac(C)\otimes\BQ)$ is in $\DA_1(k)$ then follows from Proposition~\ref{prop:motive_curve_field}.
\end{proof}

We now lay the groundwork for the study of the motivic Picard functor in Section~\ref{sec:picard_sm_pr}. Let $n\in \BN$. Recall that the adjunction ``suspension-evaluation'' at the level of spectra induces derived adjunctions
\[
\Sus^n :\DA^{\eff}(S) \leftrightarrows \DA(S):\Ev_n 
\]
with $\Sus^0=\Sigma^\infty$ and, for $M\in \DA^\eff(S)$ and $N\in \DA(S)$, canonical isomorphisms 
\[\Sus^n(M)\simeq \Sigma^\infty M(-n)[-2n]\in \DA(S),
\]
\[
\Ev_n(N)\simeq \Ev_0(M(n)[2n]).  
\]
Using the map $u_S:\Sigma^\infty (\Gm\otimes\BQ)\ra \BQ_S(1)[1]$, we get a map 
\[
\Sus^1(\Gm\otimes\BQ[1])\ra \BQ_S
\]
which by adjunction corresponds to a map
\[
w_S:\Gm\otimes\BQ [1] \ra \Ev_1(\BQ_S).
\]
Over an excellent scheme $S$, there is an analoguous construction for motives with transfers (using the map $u^\tr_S$ instead of $u_S$), resulting in a map
\[
w^\tr_S:\Gm^\tr\otimes\BQ[1]\ra \Ev^\tr_1(\BQ^\tr_S)
\]
in $\DM^\eff(S)$.

% \textbf{TODO: determine if necessary to spell out relation more precisely at this point.}

Let $f:X\ra S$ be a morphism of schemes. To state the compatibility of $w_S$ with base change, we introduce the  composition
\[
d_f:f^*\Ev_{1} \BQ_S\stackrel{\epsilon}{\lra} \Ev_1\Sus^1 f^*\Ev_1 \BQ_S \simeq \Ev_1 f^*\Sus^1\Ev_1 \BQ_S \stackrel{\eta}{\lra} \Ev_1 f^*\BQ_S\simeq \Ev_1 f^*\BQ_X
\]
where the isomorphism in the middle is the canonical isomorphism $\Sus^1 f^*\simeq f^* \Sus^1$.

\begin{lemma}\label{lemma:Ev_1}
Let $S$ be a noetherian finite-dimensional scheme. If $f:X\ra S$ is any morphism of finite type, the following diagram
\[
\xymatrix{
f^*(\Gm\otimes\BQ[1]) \ar[r]_{\sim}^{R_f} \ar[d]_{f^*w_S}^{\sim} & \Gm\otimes\BQ[1] \ar[d]^{w_X} \\
f^*\Ev_1 \BQ_S \ar[r]_{d_f} & \Ev_1 \BQ_X
}
\]
commutes.
\end{lemma}
\begin{proof}
Going through the definitions of $w_S$ and $d_f$, we see that the diagram in (i) is obtained from the commutative diagram of Proposition~\ref{prop:Gm_Q1} via the adjunction $\Sus^1 \dashv \Ev_1$ and the commutation of $\Sus^1$ and $f^*$.
\end{proof}

The following result is not used in the rest of the paper, but seems of independent interest.

\begin{prop}\label{prop:Ev_1}
\begin{enumerate}[label={\upshape(\roman*)}]
\item Assume $S$ is regular. Then the morphism $w_S$ is an isomorphism. 
\item If $f:X\ra S$ is a morphism of finite type with $X$ and $S$ regular, then $d_f$ is an isomorphism.
\end{enumerate}
\end{prop}
\begin{proof}
Statement (ii) follows from the combination of (i) and Lemma \ref{lemma:Ev_1}, so we are left with proving (i).

Since $\DA^{\eff}(S)$ is generated as a triangulated category by objects of the form $M^{\eff}_S(X)[n]$ for $f:X\ra S$ smooth morphism and $n\in\BZ$, it is enough to show that for such an object, the induced map
\[
\DA^{\eff}(S)(M^{\eff}_S(X)[n],\Gm\otimes\BQ[1]) \stackrel{w_{S*}}{\lra} \DA^{\eff}(S)(M^{\eff}_S(X)[n],\Ev_1(\BQ_S))
\]
is an isomorphism. The idea is to compare both sides to similar morphism groups in the derived category $D(\Sm/S)$. Consider the following diagram.
\[
\xymatrix@C=1em{
D(\Sm/S)(\BQ_S(X)[n],\Gm[1]) \ar[r]^{\mathbf{(\alpha)}} \ar[dd]^{\adj}_{\sim}  &  \DA^{\eff}(S)(M^{\eff}_S(X)[n],\Gm[1]) \ar[r]^{w_{S*}} \ar[d]_{\Sigma^\infty} \ar @{} [rd] |{\mathbf{(B)}} & \DA^{\eff}(S)(M^{\eff}_S(X)[n],\Ev_1(\BQ_S)) \ar[d]^{\sim}_{\adj} \\
\ar @{} [rd] |{\mathbf{(A)}} & \DA(S)(M_S(X)[n],\Sigma^\infty(\Gm)[1]) \ar[r]_{\sim}^{u_{S*}} \ar[d]_{\sim}^{\adj}  \ar @{} [rd] |{\mathbf{(C)}} &  \DA(S)(M_S(X)[n],\BQ_S(1)[2])\ar [d]^{\sim}_{\adj}  \\
D(\Sm/X)(\BQ_X[n],f^*\Gm[1]) \ar[d]_{\sim}^{R_{f*}} & \DA(X)(\BQ_X[n],f^*\Sigma^\infty \Gm[1]) \ar[r]^{(f^*(u_S))_*} \ar[d]^{R_{f*}}_{\sim}  \ar @{} [rd] |{\mathbf{(D)}} & \DA(X)(\BQ_X[n],\BQ_X(1)[2]) \ar@{=}[d]\\
D(\Sm/X)(\BQ_X[n],\Gm[1]) \ar[r] \ar@/_2pc/[rr]_{\mathbf{(\beta)}}& \DA(X)(\BQ_X[n],\Sigma^\infty \Gm[1]) \ar[r]_{\sim}^{u_{X*}} & \DA(X)(\BQ_X[n],\BQ_X(1)[2])
}
\]
The square $(A)$ commutes because the isomorphisms $R_f$ in the derived category and in $\DA$ are compatible by construction. The square $(B)$ commutes by construction of $w_S$ and $u_S$. The square $(C)$ commutes by naturality of the adjunction. Finally, the square $(D)$ commutes by Proposition~\ref{prop:Gm_Q1}. 

To complete the proof that $w_{S*}$ is an isomorphism, it remains to see that the maps $(\alpha)$ and $(\beta)$ are isomorphisms as well. For $(\beta)$, this is precisely the statement of Proposition~\ref{prop:mot_coh_1} \ref{mot_coh_1_1}-\ref{mot_coh_1_3}. Let us prove that $(\alpha)$ is an isomorphism.

Since $S$ is regular, all smooth $S$-schemes are regular. They are in particular reduced, which implies that $\Gm$ is $\BA^1$-invariant on $\Sm/S$, and normal, which implies that $\Pic=H^1(-,\Gm)$ is $\BA^1$-invariant. The higher cohomology groups $H^i(-,\Gm)$ for $i\geq 2$ are torsion on regular schemes by \cite[Proposition~1.4]{BrauerII}. Combined with Lemma \ref{lem:Q_coeffs} below, this shows that the sheaf $\Gm\otimes\BQ$ is $\BA^1$-local in the model category underlying $\DA^{\eff}(S)$. We deduce that the morphism $(\alpha):\D(\Sm/S)(\BQ_S(X)[n],\Gm\otimes\BQ[1])\ra \DA^{\eff}(M^{\eff}_S(X),\Gm\otimes\BQ)$ is an isomorphism. This completes the proof that $w_S$ is an isomorphism.
\end{proof}

\begin{lemma}\label{lem:Q_coeffs}
  Let $S$ be a scheme, and $F$ a sheaf of abelian groups on one of the sites $(\Sm/S)_{\et}$ or $(\Sch/S)_{\et}$. Then the canonical morphism
  \[
H^i_{\et}(S,F)\otimes\BQ\rightarrow H^i_{\et}(S,F\otimes\BQ)
\]
is an isomorphism.
\end{lemma}
\begin{proof}
Given our running assumption that schemes are noetherian finite dimension, this follows from \cite[Proposition 1.11]{CDEt}.
\end{proof}

\subsection{Motives of Deligne $1$-motives}
\label{sec:mot_del_mot}

We relate the category $\CM_1(S)$ of Deligne $1$-motives with rational coefficients (Appendix~\ref{sec:app_deligne}) to $\DA(S)$. Let $\BM=[L\ra G]\otimes\BQ$ be in $\CM_1(S)$. Then by viewing $\BM$ as a complex of \'etale sheaves of $\BQ$-vector spaces on $\Sm/S$, we can associate to $\BM$ an object in $\DA^{\eff}(S)$, which we also denote by $\BM$.

\begin{cor}\label{cor:Deligne_da_1}
  Let $\BM\in \CM_1(S)$. Then $\Sigma^\infty\BM$ lies in $\DA_{1,c}(S)$. If $S$ is moreover assumed to be geometrically unibranch, then the motive $\Sigma^\infty\BM$ is also geometrically smooth, thus lies in $\DA^{\gsm}_{1,c}(S)$.
\end{cor}
\begin{proof} 
Let $\BM=[L\ra G]\otimes\BQ$. We apply Proposition~\ref{prop:smooth_complex_compact} to the distinguished triangle
\[
\Sigma^\infty G_\BQ[-1] \ra \Sigma^\infty\BM\ra \Sigma^\infty L_\BQ\rap
\]
which proves the first part. Assume now that $S$ is geometrically unibranch. We have a further distinguished triangle
\[
\Sigma^\infty T_\BQ \ra \Sigma^\infty G_\BQ \ra \Sigma^\infty A_\BQ \rap
\]
where $T$ (resp. $A$) is the torus (resp. abelian) part of $G$. The motives $\Sigma^\infty T_\BQ$ and $\Sigma^\infty L_\BQ$ are geometrically smooth by Corollary~\ref{cor:motive_torus} and its proof. The motive $\Sigma^\infty A_\BQ$ is a direct factor of the motive of $A$ by \cite[Theorem 3.3]{AHPL}, so it is geometrically smooth. This completes the proof.
\end{proof}

\begin{remark}
In Corollary \ref{cor:Deligne_da_1}, we do not know if ``geometrically smooth $1$-motive'' can be replaced by ``strongly geometrically smooth $1$-motive''.
\end{remark}  

From Corollary~\ref{coro:pullback_complex_DA} and the definition of $\Sigma^\infty$, we deduce the following.

\begin{cor}
Let $f:T\ra S$ be a morphism of schemes. There is an isomorphism of functors
\[
R_f: f^* \Sigma^\infty \stackrel{\sim}{\lra} \Sigma^\infty f^{-1}:\CM_1(S)\ra \DA(T).
\]
which is compatible with further pullbacks.
\end{cor}

As explained in Section~\ref{sec:Weil_Deligne}, we have also a covariant functoriality for finite \'etale morphisms, coming from Weil restrictions of scalars. Here is how this relates to pushforwards of motives.

\begin{lemma}\label{lemm:Weil_pushforward}
Let $f:T\ra S$ be a finite \'etale morphism of schemes. There is an isomorphism of functors
\[
f_*\Sigma^\infty_S \stackrel{\sim}{\lra} \Sigma^\infty_T f_*:\CM_1(T)\ra \DA(S)
\]  
\end{lemma}
\begin{proof}
Because of the definition of pushforwards in $\CM_1(-)$ (Definition~\ref{defi:pushforward_deligne}), it is enough to show the following: for $G/T$ smooth (not necessarily of finite type) commutative group scheme, there is an isomorphism $f_*\Sigma^\infty G_\BQ\simeq \Sigma^\infty (\Res_f G)_\BQ$, functorial in $G$  (note that we do not claim that the sheaf $\Res_f G$ is representable in this generality). We have a sequence of functorial isomorphisms
\begin{eqnarray*}
f_*\Sigma^\infty G_\BQ & \simeq & f_\sharp \Sigma^\infty G_\BQ\\
& \simeq & \Sigma^\infty f_\sharp G_\BQ\\
& \simeq & \Sigma^\infty f_* G_\BQ\\
& \simeq & \Sigma^\infty \underline{f}_* G_\BQ\\
& \simeq & \Sigma^\infty (\Res_f G)_\BQ  
\end{eqnarray*}
where the first and third isomorphisms follow from the fact that $f$ is finite \'etale, the second comes from the commutation between $\Sigma^\infty$ and $f_\sharp$, the fourth follows from the fact that $\underline{f}_*$ in $\DA^{\eff}(-)$ preserves $(\BA^1,\et)$-equivalences for $f$ finite (an argument can be found in Part A of the proof of~\cite[Lemme B.7]{Ayoub_Galois_1}), and the last is the definition of Weil restriction. This completes the proof.
\end{proof}

\subsection{Picard complexes}
\label{sec:smooth_picard}

Classically the Picard functor of a morphism of schemes $f$ is defined as $R^1f_*\Gm$. We introduce a variant of this construction which includes information about relative connected components. 

\begin{defi}\label{def:Picard_complex}
Let $f:X\ra S$ be a finite type morphism of schemes. The Picard complex $\mathrm{P}(X/S)$ of $X$ over $S$ is the object $\tau_{\geq 0}f_*(\Gm\otimes\BQ[1])\in D_{[0,1]}(\Sm/S)$. 
\end{defi}

\begin{remark}
Recall from \cite[Expos\'e XVIII \S 1.4]{SGA4_3} that there is an equivalence of categories between the category of commutative group stacks over a site $\CS$ (with morphisms taken up to $2$-isomorphisms) and the category $D_{[0,1]}(\Sh(\CS,\BZ))$. The Picard complex corresponds via this equivalence to the smooth Picard stack, i.e., the version for $\Sm/S$ of the usual Picard stack (see e.g. \cite{Brochard_Picard}). This point of view will not be used in the rest of this paper.
\end{remark}

We will also need a version with transfers.

\begin{defi}\label{def:Picard_complex_tr}
  Let $S$ be a scheme, $f:X\ra S$ a finite type morphism of schemes. The Picard complex with transfers $\mathrm{P}^\tr(X/S)$ of $X$ over $S$ is the object $\tau_{\geq 0}f_*(\Gm^\tr\otimes\BQ[1])\in D_{[0,1]}(\mathrm{Cor}/S)$.
\end{defi}

The adjunction adding and forgetting transfers at the level of sheaves with transfers induces an 
\[
a_{\tr}:D(\mathrm{Cor}/S) \leftrightarrows D(\Sm/S):o^{\tr}
\]
where the functor $o^{\tr}$ at the level of sheaves with transfers derives trivially.

TODO: clarify the amplitude restriction and relationship with the motivic result! Because $a^{\tr}$ is left Quillen and $o^{\tr}$ derives trivially, this should be very easy for effective motives, need some extra thought to stabilize.

\begin{lemma}
  Let $S$ be an excellent scheme and $f:X\ra S$ a finite type morphism of schemes. There is a natural isomorphism
  \[
P(X/S)\simeq o^{\tr}P^{\tr}(X/S)
\]
hence by adjunction a natural morphism
\[
a^{\tr}P(X/S)\ra P^{\tr}(X/S)
  \]
\end{lemma}
\begin{proof}
  By \cite[Proposition 3.10.(i)]{AHPL}, we have an isomorphism $\Gm\otimes\BQ\simeq o^{\tr}\Gm^{\tr}\otimes\BQ$ of sheaves with transfers. By construction, $f_{*}$ commutes with $o^{\tr}$ at the presheaf level and $o^{\tr}$ preserves quasi-isomorphisms and $\et$-local equivalences, hence $f_{*}o^{\tr}\simeq o^{\tr}f_{*}$ and $\tau_{\geq 0}o^{\tr}\simeq o^{\tr}\tau_{\geq 0}$. All together, this provides the required isomorphism
  \[
\tau_{\geq 0}f_{*}(\Gm\otimes\BQ[1])\simeq \tau_{\geq 0}f_{*}(o^{\tr}\Gm^{\tr}\otimes\BQ[1])\simeq \tau_{\geq 0}o^{\tr}f_{*}(\Gm^{\tr}\otimes\BQ[1])\simeq o^{\tr}\tau_{\geq 0}f_{*}(\Gm^{\tr}\otimes\BQ[1]).
 \]
  \end{proof}

We note the following result which shows that the truncation in the definition of the Picard complex is sometimes unnecessary; this will not be used in the rest of the paper.
  
\begin{lemma}\label{lemm:Pic_truncation}
 Let $f:X\ra S$ be a smooth morphism with $S$ regular. Then for $i\geq 1$, the sheaf $R^i f_* (\Gm\otimes\BQ[1])\simeq R^{i+1}f_*(\Gm\otimes\BQ)$ is trivial. As a consequence, we have
\[
\mathrm{P}^{(\tr)}(X/S)\stackrel{\sim}{\lra} f_*(\Gm^{(\tr)}\otimes\BQ[1]).
\]  
\end{lemma}
\begin{proof}
  This follows from Lemma~\ref{lem:Q_coeffs} together with the fact that for a regular scheme $T$ and $i\geq 2$, the \'etale cohomology groups $H^i(T,\Gm)$ are torsion \cite[Proposition~1.4]{BrauerII}.
\end{proof}

We proceed to analyse the structure of $\mathrm{P}(X/S)$, following closely the standard structure theory for the Picard scheme \cite{Kleiman_Picard} and the Picard stack \cite{Brochard_Picard}. We will see that restricting to the smooth site leads to simpler results than in the classical case.

In the sequel, we consider \'etale sheaves of abelian groups and $\BQ$-vector spaces on the two sites $(\Sch/S)_{\et}$ and $(\Sm/S)_{\et}$. We have a continuous functor $\zeta:\Sch/S\ra \Sm/S$. The restriction functor $\zeta_*:\Sh(\Sch/S)\ra \Sh(\Sm/S)$ is exact, since an \'etale scheme over a smooth $S$-scheme is a smooth $S$-scheme. We have $\zeta_*\Gm\simeq \Gm$. The functor $\zeta_*$ commutes with $f_*$ and $\underline{f}_*$, in the sense that there are natural isomorphisms of functors $\zeta_* \underline{f}_*^{\Sch}\simeq \underline{f}_*^{\Sm}\zeta_*$ and $\zeta_* f_*^{\Sch}\simeq f_*^{\Sm}\zeta_*$. By abuse of terminology, we say that a sheaf of sets on $\Sm/S$ is representable if it is isomorphic to the functor $\zeta_* X$ for $X$ a (not necessarily smooth) $S$-scheme; such a scheme is then not uniquely determined up to isomorphism.

\begin{defi}
Let $f:X\rightarrow S$ be a morphism of schemes or algebraic spaces. We say that $f$ is cohomologically flat in degree $0$ if the construction of $\underline{f}_*\CO_X$ commutes with arbitrary base change.  
\end{defi}

Recall that by \cite[Corollaire 4.3.3]{EGAIII_1}, a smooth proper morphism $f$ has a Stein factorisation \[f:X\stackrel{f^\circ}{\ra} \pi_0(X/S):=\Spec_S(\underline{f}_*\CO_X)\stackrel{\pi_0(f)}{\ra} S.\] Moreover, $f$ is in this case cohomologically flat in degree $0$ by \cite[Proposition 7.8.6]{EGAIII_2}, so that the construction of $\pi_0(f)$ commutes with arbitrary base change, and $\pi_0(f)$ is finite \'etale \cite[Remarque 7.8.10.(i)]{EGAIII_2}.

\begin{lemma}\label{lem:stein_torus}
 Let $f:X\rightarrow S$ be a smooth proper morphism. The sheaf $\underline{f}_*\Gm$ is representable by a torus, the Weil restriction $\Res_{\pi_0(f)}\Gm$ (see Definition~\ref{def:Weil}).
\end{lemma}
\begin{proof}
  For any $U\ra S$ smooth, we have \[\underline{f}_*(\Gm)(U)=\CO^\times (X\times_S U) \simeq \CO^\times (\pi_0(X\times U/U)) \simeq \CO^\times (\pi_0(X/S)\times_S U).\]
  This implies the claim.
\end{proof}
 
Next, we look at the Picard sheaf $\sPic_{X/S}:=R^1 f_*\Gm\in \Sh((\Sch/S)_\et,\BZ)$ and its smooth analogue $\sPic^\sm_{X/S}\in \Sh((\Sm/S)_\et,\BZ)$ defined by the same formula on the smooth site. By exactness of $\zeta_*$, we have $\zeta_*\sPic_{X/S}\simeq \underline{\zeta}_*\sPic_{X/S}\simeq \sPic^\sm_{X/S}$.

% \begin{lemma} Let $f:X\ra S$ be a smooth projective morphism. Suppose . For any $T\in \Sm/S$, we have an exact sequence
% \[
% 0\ra \Pic(\pi_0(X_T/T))\ra \Pic(X_T)\ra \sPic^{(\sm)}_{X/S}(T)\ra 0.\tag{L}
% \]
% \end{lemma}

% \begin{proof}
% The Leray spectral sequence provides, for all $T\in \Sm/S$, a short exact sequence
% \[
% 0\ra H^1(T,f_{T*}\CO^\times_{X_T})\ra H^1(X_T,\CO^\times_{X_T})\ra H^0(T,R^1 f_{T*}\CO^\times_{X_T})\ra H^2(T,f_{T*}\CO^\times_{X_T})\ra H^2(X_T,\CO^\times_{X_T}) 
% \]
% By Lemma \ref{lem:stein_torus}, the proof is complete if we show that the last map in the sequence is injective. The morphism $f$ is both open and closed, so that its image $S'\subset S$ is a union of connected components. Write $T'=T\times_S S'$ for the corresponding open and closed set of $T$. Then we have $X_T\simeq X_{T'}$, and both sheaves $f_{T*}\CO^\times_{X_T}$ and $R^1 f_{T*}\CO^\times_{X_T}$ are supported on $T'$. Hence $H^2(T,f_{T*}\CO^\times_{X_T})\simeq H^2(T',f_{T'*}\CO^\times_{X_{T'}})$ and $H^2(X_T,\CO^\times_{X_T})\simeq H^2(X_{T'},\CO^\times_{X_{T'}})$. The morphism $f$, being smooth, has sections over its image $S'$ locally in the \'etale topology by~\cite[17.16.3 (ii)]{EGAIV_4}, and such a section provides a retraction of the map $H^2(T',f_{T'*}\CO^\times_{X_{T'}})\ra H^2(X_{T'},\CO^\times_{X_{T'}})$. 
% \end{proof}

\begin{lemma}\label{lem:base_change_Pic}
  Let $\pi:S'\ra S$ be a morphism of schemes.
  \begin{enumerate}[label={\upshape(\roman*)}]
  \item There are  natural isomorphisms
\[
v_\pi:\pi^{-1}\sPic_{X/S} \simeq \sPic_{X\times_S S'/S'}.
\]
\item There are natural morphisms
\[
v^{\sm}_\pi:\pi^{-1}\sPic^{sm}_{X/S} \ra \sPic^{sm}_{X\times_S S'/S'},
\]
which are isomorphisms if $\pi$ is smooth.
  \end{enumerate}
  \end{lemma}
\begin{proof}
The sheaf $\sPic_{X/S}$ is the \'etale sheaf associated with the ``naive'' Picard functor $\sPic^{\psh}_{X/S}:T\mapsto \Pic(X\times_S T)$. We have, for any $S'$-scheme $T'$: 
  \[
(\pi^{-1} \sPic^{\psh}_{X/S})(T') = \Pic(X\times_S T')= \Pic((X\times_S S')\times_{S'} T') = \sPic^{\psh}_{X\times_S S'/S'}(T')
  \]
  This equality is functorial in $T'$. After passing to associated sheaves, we get the isomorphism $v_{\pi}$. This concludes the proof of (i).

We now turn to $\sPic^{\sm}_{X/S}$. This sheaf is also the \'etale sheaf associated with the ``naive'' Picard functor $\sPic^{\psh,\sm}_{X/S}$ on $\Sm/S$. We have, for any smooth $S'$-scheme $T'$: 
  \[
(\pi^{-1} \sPic^{\psh,\sm}_{X/S})(T')=\Colim_{T\in (T'\backslash (\Sm/S))} \Pic(X\times_S T) \rightarrow \Pic(X\times_S T') = \sPic^{\psh,\sm}_{X\times_S S'/S'}(T')
  \]
and this defines the morphism $v_{\pi}$. If $\pi$ is smooth, then the category $T'\backslash (\Sm/S)$ has an initial object $T'\rightarrow S'\rightarrow S$ and we get isomorphisms. This proves (ii).
\end{proof}

Let us also define a natural base change map for $P(X/S)$. Consider a cartesian diagram
\[
\xymatrix{
X' \ar[d]_{\tilde{f}} \ar[r]^{\tilde{\pi}} & X \ar[d]^{f} \\
S' \ar[r]_{\pi}  & S
  }
\]
with $\pi$ any morphism of schemes. The following composition
\[
\pi^{-1}P(X/S)=\pi^{-1}\tau_{\geq 0} (f_*\Gm\otimes\BQ[1])\ra \pi^{-1}f_*\Gm\otimes\BQ[1]\ra \tilde{f}_*\tilde{\pi}^{-1}\Gm\otimes\BQ[1]\stackrel{R_{\pi}}{\simeq} \tilde{f}_*\Gm\otimes\BQ[1]
\]
factors through the truncation $\tau_{\geq 0}(\tilde{f}_*\Gm\otimes\BQ[1])=P(X_{S'}/S')$. We denote the resulting morphism by
\[
  V_\pi:\pi^{-1}P(X/S)\ra P(X_{S'}/S').
\]

In general, the construction of $\sPic^\sm$ and $P(X/S)$ does not commute with arbitrary base change, i.e., $v^{\sm}_\pi$ and $V_{\pi}$ are not always isomorphisms.

Representability results for $\sPic$ are subtle in general. Let $f:X\rightarrow S$ be a smooth projective morphism. It is in particular proper, flat and cohomologically flat in degree $0$. By \cite[8.3/1]{BLR}, $\sPic_{X/S}$ is represented by a group algebraic space $\Pic_{X/S}$ over $S$. Note that if $f$ has in addition geometrically connected fibres, then $\Pic_{X/S}$ is in fact a group scheme, separated and locally of finite presentation \cite[8.2/1]{BLR}; since we do not want to restrict to this case, we use algebraic spaces as a technical crutch. Finally, if $S$ is the spectrum of a field $k$, then $\sPic_{X/k}$ is represented by a group scheme which is locally of finite type over $k$, regardless of whether $f$ has geometricaly connected fibres or not \cite[8.2/3]{BLR}.

We want to discuss the identity component $\sPic^{0}_{X/S}$ of $\sPic_{X/S}$. Before we introduce it, let us recall some basic facts about connected components of locally of finite type group schemes over a field. Let $k$ be an arbitrary field and $G$ be a $k$-group scheme which is locally of finite type. Since $G$ is locally noetherian, its connected components are open and closed in $G$, and we denote by $G^0$ the connected component containing the identity.

\begin{defi}\label{def:identity_component}
  Let $f:X\ra S$ be a smooth projective morphism. The group algebraic space $\sPic_{X/S}$ comes with a natural subfunctor $\sPic^{0}_{X/S}$, its \emph{relative identity component}. Let $T\in\Sm/S$. Given a point $s\in S$, through Lemma~\ref{lem:base_change_Pic}, we can restrict a section in $\sPic_{X/S}(T)$ to a section of $\sPic_{X_s/s}(T_s)$. Since $\sPic_{X_s/s}$ is represented by a group scheme $\Pic_{X_s/s}$ over $s$, locally of finite type, it has therefore an identity component $\Pic^0_{X_s/s}$. By definition, a section in $\sPic_{X/S}(T)$ lies in $\sPic^{0}_{X/S}(T)$ if for all $s\in S$, its restriction to $\sPic_{X_s/s}(T_s)$ lies in $\Pic^0_{X_s/s}(T_s)$.
We then also define $\sPic^{\sm,0}_{X/S}$ as $\underline{\zeta}_* \sPic^{0}_{X/S}$.
\end{defi}

\begin{remark}\label{rmk:neutral_field_extension}
  Let $k$ be a field, $L$ be any field extension and $G$ be a $k$-group scheme. Let $T\in \Sm/k$ and $x\in G(T)$. Then since $G^0$ is open and closed in $G$, we see that $x\in G^0(T)$ if and only if for $x_L\in G_L^0(T_L)$. In particular, in Definition \ref{def:identity_component}, it is possible to check that a section lies in $\sPic^{0}_{X/S}(T)$ by pulling back to all geometric points instead.
\end{remark}

We have a further result on base change.

\begin{lemma}\label{lem:base_change_Pic_0}
  Let $\pi:S'\ra S$ a morphism of schemes.
  \begin{enumerate}[label={\upshape(\roman*)}]
  \item There are  natural isomorphisms
\[
v_\pi:\pi^{-1}\sPic^0_{X/S} \simeq \sPic^0_{X\times_S S'/S'}.
\]
\item There are natural morphisms
\[
v^{\sm}_\pi:\pi^{-1}\sPic^{sm,0}_{X/S} \ra \sPic^{sm,0}_{X\times_S S'/S'},
\]
which are isomorphisms if $\pi$ is smooth.
  \end{enumerate}
  \end{lemma}
\begin{proof}
  The maps in the lemma are obtained by restriction from Lemma~\ref{lem:base_change_Pic}. We thus only have to check that the subfunctor $\sPic^0_{X/S}$ is mapped into $\sPic^0_{X\times_S S'/S'}$. Let $T\in \Sch/S'$ and $\alpha\in \sPic^0_{X/S}(T)$. Let $s'$ be a point in $S'$, with corresponding morphism $i':s'\ra S'$. Write $i=\pi i:s'\ra S$. It is easy to see from the construction of the base change map that we have $v_i = v_{i'}\circ (i')^{-1}v_\pi:(i')^{-1}\pi^{-1}\sPic_{X/S}\simeq i^{-1}\sPic_{X/S}\ra \sPic_{X_s/s}$. By definition, we have $v_i(i^{-1}\alpha)\in \sPic^0_{X_{s'}/s'}(T_{s'})$, hence $v_{i'}\circ v_\pi(i^{'})^{-1}(\alpha)\in \sPic^0_{X_{s'}/s'}(T_{s'})$. Since this holds for all $s'\in S'$, we conclude that $v_{\pi}(\alpha)$ lies in $\sPic^0_{X\times_S S'/S'}$.
\end{proof}

\begin{lemma}\label{lem:repr_Pic_0}
Let $f:X\ra S$ be a smooth projective morphism. The functor $\sPic^0_{X/S}$ is representable by a proper group algebraic space $\Pic^0_{X/S}$ over $S$.
\end{lemma}

\begin{proof}
By \cite[8.3/1]{BLR}, it is enough to show that the functor $\sPic^0_{X/S}\rightarrow \sPic_{X/S}$ is relatively representable by a closed immersion and that the resulting group algebraic space $\Pic^0_{X/S}$ is of finite type. 

By \cite[Corollaire 2.3]{Grothendieck_Picard_VI}, this is the case for $\Pic^0_{X/S}$ under the additional assumptions that the geometric fibres of $f$ are integral (or equivalently, connected); note that in this case, $\Pic_{X/S}$ is representable by a scheme. We are going to reduce to this case by \'etale descent. 

By replacing $S$ by the image of $f$ (which is a disjoint union of connected components of $S$ since $f$ is open and closed), we can assume that $f$ is surjective. Write $f:X\stackrel{f^\circ}{\ra} \pi_0(X/S)\stackrel{\pi_0(f)}{\ra}S$ for the Stein factorisation of $f$. The morphism $\pi_0(f)$ is finite \'etale surjective. Let $p:\Pi\rightarrow S$ be an \'etale Galois covering which dominates every connected component of $\pi_0(X/S)$. We thus have that $\Pi':=\pi_0(X/S)\times_S \Pi\simeq \coprod_{i\in I} \Pi_i$ for some finite set $I$ and $\Pi_i\simeq \Pi$. Write $Y=X\times_S \Pi\stackrel{g}{\rightarrow}\Pi$. The Stein factorisation of the morphism $g$ is $Y\rightarrow \Pi'\rightarrow \Pi$. Since $Y\rightarrow \Pi'$ is smooth with connected geometric fibres, we have $Y=\coprod Y_i$ with $g^\circ(Y_i)\subset \Pi_i$, and each morphism $Y_i\rightarrow \Pi$ is smooth projective with geometrically connected fibres. We have
\[
p^{-1}\sPic^{0}_{X/S} \simeq \sPic^0_{Y/\Pi}\simeq \prod_{i\in I} \sPic^{0}_{Y_i/\Pi}
\]
and, by the beginning of the proof, each of the factors in this product is representable by a proper group algebraic space. The same argument applies over $\Pi\times_S \Pi$. Using that \'etale descent for algebraic spaces is effective, we conclude that $\sPic^0_{X/S}$ is representable by a proper group algebraic space. 
\end{proof}

\begin{prop}\label{prop:Pic_smooth_car_0}
Let $S$ be a $\BQ$-scheme and $f:X\ra S$ a smooth projective morphism. The algebraic group space $\Pic^0_{X/S}$ is in fact an abelian scheme over $S$.
\end{prop}
\begin{proof}
  Under these hypotheses, and assuming furthermore that $f$ has geometrically connected fibres, $\Pic^0_{X/S}$ (which is then a group scheme) is smooth, as explained in \cite[Remark 5.21]{Kleiman_Picard}. Let us now drop the assumption on the fibres of $f$. By the \'etale descent argument from Lemma~\ref{lem:repr_Pic_0} and the fact that smoothness can be checked \'etale locally, we deduce that the algebraic group space $\Pic^0_{X/S}$ is smooth and proper. Moreover, by definition of $\Pic^0_{X/S}$, it has geometrically connected fibres. But a smooth proper algebraic group space with geometrically connected fibres is an abelian scheme \cite[Theorem 1.9]{Faltings_Chai}.
\end{proof}

In general, when $S$ is not a $\BQ$-scheme, $\Pic^0_{X/S}$ can have non-reduced fibres. We need a condition under which we can ``extract'' an abelian scheme from $\Pic^0$. Here is a result in that direction.

\begin{prop}\label{prop:Brochard} \cite[Proposition~2.15]{Brochard_Duality}
  Let $S$ be a noetherian scheme, and $G$ a group $S$-algebraic space which is proper, flat and cohomologically flat in degree $0$. Then there exists an abelian scheme $A/S$ and a finite flat group scheme $F/S$ such that $G$ fits into a unique exact sequence
  \[
0\rightarrow A\rightarrow G\rightarrow F\rightarrow 0
  \]
of $S$-group schemes. In particular, $G$ is a scheme.
\end{prop}

The uniqueness in the previous proposition is not stated in \cite{Brochard_Duality}, but follows from the proof as $F$ is shown to be the affinisation $\Spec_{\CO_S}((G\ra S)_*\CO_G)$ of $G$.

This motivates the following definition.

\begin{defi}\label{def:pic_smooth}
Let $f$ be a smooth and proper morphism. We say that $f$ is Pic-smooth if the algebraic space $\Pic^0_{X/S}$ is flat and cohomologically flat in degree $0$. 
\end{defi}

By Proposition~\ref{prop:Pic_smooth_car_0}, the condition of being Pic-smooth is automatic if $S$ is of characteristic $0$.

\begin{lemma}\label{lem:bc_Pic_smooth}
Let $f:X\ra S$ a smooth proper Pic-smooth morphism, and $T\rightarrow S$ be any morphism of schemes. Then $f\times_S T$ is Pic-smooth.  
\end{lemma}
\begin{proof}
This follows from Lemma~\ref{lem:base_change_Pic_0} together with the facts that flatness and cohomological flatness in degree $0$ are stable by base change.
\end{proof}

\begin{prop}\label{prop:generic_pic_smooth}
Let $f:X\ra S$ be a smooth projective morphism. Assume $S$ is reduced. Then there is a dense open set $U\subset S$ such that $f\times_S U$ is Pic-smooth.
\end{prop}

\begin{proof}
  Recall that $\Pic^{0}_{X/S}$ is representable by a proper algebraic group space by Lemma~\ref{lem:repr_Pic_0}. We are going to show that, more generally, for any proper $S$-algebraic space $g:Q\ra S$, there exists a dense open set $U$ such that $g\times_S U$ is flat and cohomologically flat in degree $0$.

  Since $S$ is reduced, generic flatness for morphisms of algebraic spaces \cite[Tag 06QS]{stacks-project}\cite[Corollaire 6.9.3]{EGAIV_2} provides a dense open subset $V\subset U$ of $S$ over which $g\times_S V$ is flat.

  By restricting $V$ further, we can assume that $V$ is affine and disjoint union of its irreducible components, and thus reduce to the case $V=\Spec(A)$ affine and integral. Cohomological flatness in degree $0$ on a dense open set of $V$ then follows from \cite[Corollaire 7.3.9]{EGAIII_2}, applied to the homological functor $T_\bullet(-):\mathrm{Mod}(A)\ra \mathrm{Mod}(A),\ N\mapsto R^\bullet g_* g^{-1}\widetilde{N}$, which takes values in finitely generated $A$-modules by properness of $g$.
\end{proof}

\begin{prop}\label{prop:repr_pic_smooth}
  Let $f:X\ra S$ be a smooth projective Pic-smooth morphism. Then $\sPic^{\sm,0}_{X/S}$ is representable by an abelian scheme.
\end{prop}

\begin{proof}
By Proposition~\ref{prop:Brochard}, we have a short exact sequence of group schemes
\[
0\ra \Pic^{0,\red}_{X/S}\ra \Pic^{0}_{X/S}\ra F \ra 0
\]
with $F$ a finite flat group scheme with connected fibres (since the fibres of $\Pic^0_{X/S}$ are connected by definition), and $\Pic^0_{X/S}$ represents $\sPic^{0}_{X/S}$. Let us show, more generally, that if $G$ is a group scheme fitting in an exact sequence
\[
0\ra A\ra G\stackrel{\pi}{\ra} F\ra 0
  \]
  with $A$ an abelian scheme and $F$ a finite flat group scheme with connected fibres, then the restriction to $\Sm/S$ of the functor of points of $G$ is representable by $A$.

  Let $T\in \Sm/S$ and $h:T\ra G$ an $S$-morphism. Let $s\in S$. Then $F_s$ is a finite flat connected group scheme over $\kappa(s)$, hence we have $(F_s)_{\red}=\Spec(\kappa(s))$. Hence the morphism $T_s\ra F_s$ factors through the identity section of $F_s$. Using the smoothness of $S$, this implies that the morphism $\pi\circ h:T\ra F$ also factors through the identity section. By the exactness in the middle of
  \[
0\ra A(T)\ra G(T)\ra F(T)
    \]
we see that $h$ comes from $A(T)$. This proves the result.
  \end{proof}

\begin{defi}
Let $f:X\ra S$ be a smooth projective Pic-smooth morphism. We denote by $\Pic^{0,\red}_{X/S}$ the abelian scheme representing $\sPic^{\sm,0}_{X/S}$.
\end{defi}

\begin{lemma}\label{lem:vsm_iso}
  Let $f:X\ra S$ be a smooth projective Pic-smooth morphism. For any morphism $\pi:S'\ra S$, the morphism $v^{\sm}_{\pi}$ induces a isomorphism
  \[
v^{\sm}_{\pi}:\Pic^{0,\red}_{X/S}\times_{S}S' \simeq\Pic^{0,\red}_{X'/S'}.
\]  
\end{lemma}

\begin{proof}
  First, the morphism $X'\ra S'$ is still Pic-smooth by Lemma~\ref{lem:bc_Pic_smooth}; hence the statement makes sense. By combining Lemma~\ref{lem:base_change_Pic_0} and Proposition~\ref{prop:repr_pic_smooth}, we get a morphism $v^{\sm}_{\pi}:\Pic^{0,\red}_{X/S}\times_S S'\ra \Pic^{0,\red}_{X'/S'}$. By the uniqueness of the short exact sequence in Proposition~\ref{prop:Brochard}, we see that, on the other hand, we have a base change isomorphism $\Pic^{0,\red}_{X/S}\times_S S'\simeq \Pic^{0,\red}_{X'/S'}$, and it is not difficult to see that it coincides with $v^{\sm}_{\pi}$.
\end{proof}

\subsection{N\'eron-Severi sheaves}

We now turn to the study of the N\'eron-Severi groups in families. Let us start by recording some properties of N\'eron-Severi groups over algebraically closed fields.

\begin{defi}\label{def:NS_group}
Let $k$ be an algebraically closed field and $X/k$ be a smooth proper variety. Let $L,L'\in \Pic(X)$. We say that $L$ is \emph{algebraically equivalent} to $L'$ if there exists a smooth projective connected curve $C$, two points $x_0,x_1\in C(k)$ and a line bundle $\tilde{L}\in \Pic(X\times_k C)$ such that $x_0^*\tilde{L}\simeq L$ and $x_1^*\tilde{L}\simeq L'$. Algebraic equivalence is compatible with the tensor product of line bundles, and the quotient group is the \emph{N\'eron-Severi group} $\NS(X)$.
\end{defi}  

\begin{lemma}\label{lem:ns_prop}
Let $k$ be an algebraically closed field and $X/k$ be a smooth proper $k$-scheme. 
\begin{enumerate}[label={\upshape(\roman*)}]
\item \cite[Expos\'e XIII Th\'eor\`eme 5.1]{SGA6} \label{ns_fg} The group $\NS(X)$ is finitely generated.
\item \cite[Proposition 3.1]{Maulik_Poonen} \label{ns_invariance} Let $K/k$ be an extension of algebraically closed fields. Then the natural morphism $\NS(X)\ra \NS(X_{K})$ is an isomorphism.
\end{enumerate}    
\end{lemma}

For our purposes, the correct generalisation of N\'eron-Severi groups in families is the following.

\begin{defi}\label{def:NS}
Let $f:X\ra S$ be a smooth projective morphism. We define the \emph{N\'eron-Severi sheaf} as the quotient \'etale sheaf
\[\sNS^{(\sm)}_{X/S}:=\sPic^{(\sm)}_{X/S}/\sPic^{(\sm),0}_{X/S}.\]
\end{defi}

This definition works well over a general scheme $S$. Let us explain an equivalent, group-theoretic perspective, when $S$ is the spectrum of a perfect field $k$. The following result is \cite[II \S 5, Proposition 1.8]{Demazure_Gabriel}.

\begin{lemma}\label{lem:quotient_connected}
  Let $G$ be a group scheme locally of finite type over a field $k$. There exists an \'etale group scheme $\pi_0(G)$ together with a surjective morphism $G\ra \pi_0(G)$ which is the initial \'etale group scheme with a surjective morphism from $G$. It fits into an exact sequence
  \[
0\ra G^0\ra G\ra \pi_0(G)\ra 0
\]
where $G^0$ is the neutral component of $G^0$.
\end{lemma}  

Note that if $G$ is smooth, then $G^0$ is smooth as well, and by general results on quotients by smooth group schemes discussed in the Conventions section, $\pi_0(G)$ also represents the \'etale quotient sheaf $G/G^0$. If $G$ is not smooth, there is still something useful to say in our setting.

\begin{lemma}\label{lem:quotient_perfect}
Let $G$ be a group scheme locally of finite type over a perfect field $k$. Then $\pi_0(G)$ represents the \'etale quotient sheaf $G/G^0$ when $G,G^0$ are seen as \'etale sheaves over $\Sm/k$.
\end{lemma}
\begin{proof}
Since $k$ is perfect, $G_{\red}$ is a closed subgroup scheme of $G$. Then $G_{\red}$, being reduced over a perfect field, is geometrically reduced, and, begin a geometrically reduced group scheme, is smooth. We have $(G_{\red})^0= (G^0)_{\red}$, which we denote in this situation by $G^0_{\red}$. By the universal property of $\pi_0(G)$ as initial \'etale quotient of $G$, we get an isomorphism $\pi_0(G_{\red})\simeq \pi_0(G)$. So we get an exact sequence
\[
0\ra G^0_{\red} \ra G_{\red}\ra \pi_0(G)\ra 0
\]
and since $G^0_{\red}$ is smooth, this sequence gives rise to an exact sequence of \'etale sheaves (both on $\Sch/k$ and $\Sm/k$). But, as sheaves on $\Sm/k$, we have $G=G_{\red}$ and $G^0=G^0_{\red}$. Hence we get $G/G^0\simeq \pi_0(G)$ as \'etale sheaves on $\Sm/k$.
\end{proof}  

In the special case of $\Pic$, we get the following conclusion.

\begin{lemma}\label{lem:NS_pi_0}
Let $k$ be a perfect field and $X/k$ be a smooth projective variety. We have an isomorphism
  \[\sNS^{\sm}_{X/k}\simeq \pi_0(\Pic_{X/k})\]
  as \'etale sheaves on $\Sm/k$.
If $k$ is moreover algebraically closed, then we have an isomorphism
  \[
\pi_0(\Pic_{X/k})\simeq \NS(X).
  \]
\end{lemma}
\begin{proof}
The first part follows from Lemma \ref{lem:quotient_perfect}. The second follows immediately from the fact that $\Pic_{X/k}$ represents the Picard functor of $X$ over $k$ and the fact that over an algebraically closed field, two points of a variety are in the same connected component if and only if they can be linked through the image of a smooth projective connected curve.
\end{proof}

In the context of a proper variety $X$ over a non-necessarily closed field $k$, the \'etale group scheme $\pi_{0}(\Pic_{X/k})$ is sometimes called the N\'eron-Severi group scheme of $X$ over $k$. 

We now return to a general base scheme $S$. The following lemma comes out directly from the exactness of $\zeta_*$ and from Lemmas~\ref{lem:base_change_Pic} and \ref{lem:base_change_Pic_0}.

\begin{lemma}\label{lem:NS_pullback}
Let $f:X\ra S$ be a smooth projective Pic-smooth morphism. We have a canonical isomorphism $\zeta_*\sNS_{X/S}\simeq \sNS^{\sm}_{X/S}$, and the construction of $\sNS_{X/S}$ (resp. $\sNS^\sm_{X/S}$) commutes with base change by an arbitrary morphism (resp. by a smooth morphism).
\end{lemma}

We will not use the following observation, but it seems of independent interest in order to understand N\'eron-Severi sheaves.

\begin{lemma}\label{lem:NS_Pic_smooth}
Let $f:X\ra S$ be smooth projective Pic-smooth with $S$ regular. Then for all $T\in \Sm/S$, the natural map
\[
 (\sPic^{\sm}_{X/S}(T)/\sPic^{\sm,0}_{X/S}(T))\otimes\BQ\lra\sNS^{\sm}_{X/S}\otimes\BQ(T).
\]  
is an isomorphism.
\end{lemma}

\begin{proof}
It suffices to prove that in this situation, the cohomology group $H^1_{\et}(T,\sPic^{\sm,0}_{X/S}\otimes\BQ)$ vanishes. By Proposition~\ref{prop:Brochard}, we have a short exact sequence
\[
0\ra \Pic^{0,\red}_{X/S}\ra \Pic^{0}_{X/S}\ra F \ra 0
\]
where $\Pic^{0,\red}_{X/S}$ is an abelian scheme and $F$ is a finite flat commutative group scheme, thus $\Pic^{0}_{X/S}\otimes\BQ\simeq \Pic^{0,\red}_{X/S}\otimes\BQ$ as \'etale sheaves.

Since $T$ is noetherian and regular, \cite[Proposition~XIII 2.6.(ii)]{Raynaud_ample} and \cite[Proposition~XIII 2.3.(ii)]{Raynaud_ample} imply that torsors under the abelian scheme $\Pic^{0,\red}_{X/S}$ are torsion, which shows that $H^1_{\et}(T,\Pic^{0,\red}_{X/S})\otimes\BQ=0$. By Lemma~\ref{lem:Q_coeffs}, we deduce that $H^1_{\et}(T,\Pic^{0,\red}_{X/S}\otimes\BQ)=0$.
\end{proof}

We have a morphism of sites $\alpha:\Sm/S\ra \Et/S$ where $\Et/S$ is the small \'etale site of $S$. Put $\gamma=\alpha\circ \zeta:\Sch/S\ra \Et/S$ (note that $\gamma$ is only a continuous functor). We say that a sheaf $F$ on $\Sm/S$ (resp. $\Sch/S$) is constructible if it is in the essential image of the fully faithful functor $\alpha^{-1}$ (resp. $\gamma^{-1}$) (or equivalently if the counit morphism $\alpha^{-1} \alpha_* F\ra F$ (resp. $\gamma^{-1}\gamma_* F\ra F$) is an isomorphism) and $\alpha_* F$ (resp. $\gamma_* F$) is constructible (as a sheaf of $\BZ$-modules, i.e., we do not require fibres to be finite abelian groups, but only to be finitely generated).

It is well known that the sheaf $\sNS_{X/S}$ is far from being constructible, even for $f$ smooth projective and $S$ regular in characteristic $0$; in particular, the rank of the geometric fibres (which are finitely generated abelian groups by \cite[Exp XIII, Thm 5.1]{SGA6}) is not a constructible function in general \cite[8.4 Remark 8]{BLR}. We are going to see that $\sNS^{\sm}_{X/S}$ is better behaved in some cases.

Let $S$ be a regular scheme and $f:X\rightarrow S$ a smooth projective morphism. We want to define a locally constant \'etale sheaf over $S$; for this it is enough to work connected component by connected component, so that we can assume $S$ to be irreducible, with generic point $\eta$.

Fix a geometric point $\bar{\eta}$ over $\eta$. Recall that by geometric point, we always mean the spectrum of an algebraically closed extension of $\eta$. Since $S$ is regular, the \'etale fundamental group $\pi^\et_1(S,\bar{\eta})$ is isomorphic to the maximal quotient of $\Gal(\kappa(\bar{\eta})^{s}/\kappa(\eta))$ which is unramified at every codimension $1$ point (where we denote by $\kappa(\bar{\eta})^{s}$ the separable closure of $\kappa(\eta)$ in $\kappa(\bar{\eta})$) \cite[Expos\'e V Proposition 8.2]{SGA1}. Recall also that the restriction morphism $\Aut(\kappa(\bar{\eta})/\kappa(\eta))\ra \Gal(\kappa(\bar{\eta})^{s}/\kappa(\eta))$ is always surjective.

By transport of structure, the group $\NS(X_{\bar{\eta}})$ comes with a continuous action of the profinite Galois group $\Gal(\kappa(\bar{\eta})^s/\kappa(\eta))$.

\begin{lemma}\label{lem:unramified}
The induced action of $\Gal(\kappa(\bar{\eta})^s/\kappa(\eta))$ on $\NS(X_{\bar{\eta}})\otimes\BQ$ is unramified, i.e., it factors through an action of the \'etale fundamental group $\pi_1^\et(S,\bar{\eta})$. Moreover, this action factors through a finite quotient of $\pi_1^\et(S,\bar{\eta})$.
\end{lemma}

\begin{proof}
  The $\ell$-adic first Chern class yields a Galois-equivariant morphism $c_1:\NS(X_{\bar{\eta}})\ra H^2(X_{\bar{\eta}},\BQ_l(1))$. The kernel of this morphism consists, by definition, of those classes which are homologically trivial, hence in particular numerically trivial. By \cite{Matsusaka}, numerical equivalence coincides with algebraic equivalence up to torsion for divisors on smooth projective varieties over an algebraically closed field, hence the map is injective after tensoring with $\BQ$. Moreover, since $f$ is smooth and projective, the proper and smooth base change theorems for $\ell$-adic cohomology imply that, for any codimension 1 point $s\in S$, the Galois representation on $H^2(X_{\bar{\eta}},\BQ_l(1))$ is unramified at $s$. By \cite[Exp X, 7.13.10]{SGA6}, the construction of $c_1$ commutes with specialisation; this implies that $\NS(X_{\bar{\eta}})\otimes\BQ$ is also unramified at $s$.

By Lemma \ref{lem:ns_prop}, the group $\NS(X_{\bar{\eta}})$ is finitely generated. This implies that the continuous action of the \'etale fundamental group $\pi_1^\et(S,\bar{\eta})$ on the finite-dimensional $\BQ$-vector space $\NS(X_{\bar{\eta}})\otimes\BQ$ factors through a finite quotient.
\end{proof}

Recall that, given a connected scheme $S$ with a fixed geometric point $\bar{\eta}$ and an abelian group $M$ (considered equipped with the discrete topology) together with a continuous action of the \'etale fundamental group $\pi_1(S,\bar{\eta})$, there is an associated \'etale sheaf $\CM$ on $S$ whose sections on $V\in \Et/S$ are given as follows. Write $V=\coprod_{i=0}^{n} V_i$ for the connected components of $V$. Then $\CM(V):=\prod_{i=0}^n \CM(V_i)$, which reduces us to specify the sections on a connected \'etale $S$-scheme $V$. Choose a geometric point $\bar{v}$. Write $\bar{s}$ for its image in $S$. Then we have an homomorphism $\pi^\et_1(V,\bar{v})\ra \pi_1^\et(S,\bar{s})\simeq \pi^\et_1(S,\bar{\eta})$ (with the isomorphism well-defined up to inner conjugacy), and $\CM(V):= M^{\pi_{1}^{\et}(V,\bar{v}}$ via this homomorphism. This does not depend on the choice of the point $\bar{v}$ or on the choice of ``change of base point'' isomorphism, since two choices lead to conjuguate homomorphisms; in particular, we can and will often choose $\bar{v}$ to be above $\bar{\eta}$. Moreover, when the action of $\pi_1(S,\bar{\eta})$ on $M$ factors through a finite quotient, the corresponding \'etale sheaf $\CM$ is clearly locally constant (indeed, it becomes constant on the finite \'etale Galois cover corresponding to the finite quotient).

\begin{defi}
Let $S$ be a regular scheme and $f:X\rightarrow S$ a smooth projective morphism. The locally constant finitely generated \'etale sheaf of $\BQ$-vector spaces attached to $\NS(X_{\bar{\eta}})\otimes\BQ$ is the \emph{N\'eron-Severi lattice} $\mathcal{N}_{X/S}$ of $X$ over $S$. 
\end{defi}

Let $\pi:S'\ra S$ be a universally open morphism between regular schemes. Then any generic point of an irreducible component of $S'$ is sent to the generic point of an irreducible component of $S$. Let $f:X\ra S$ be a smooth projective morphism, and write $f':X'\ra S'$ for its base change along $\pi$. We are going to define a base change isomorphism $v^{N}_{\pi}:\pi^{-1}\CN_{X/S}\ra \CN_{X'/S'}$. Let us assume for simplicity that $S'$ and $S$ are irreducible (the general case is easily obtained from this, given that $S$, $S'$ are regular). Write $\eta'$, $\eta$ for the generic points; we have $\pi(\eta')=\eta$ by assumption. Choose compatible geometric points $\bar{\eta}'$, $\bar{\eta}$ above them. The morphism $\pi$ induces compatible morphisms of absolute Galois groups and \'etale fundamental groups:
\[
\xymatrix{
\Gal(\kappa(\bar{\eta}')/\kappa(\eta')) \ar[r]^{\pi_*} \ar@{>>}[d] & \Gal(\kappa(\bar{\eta})/\kappa(\eta))  \ar@{>>}[d] \\
\pi_*:\pi^{\et}_1(S',\bar{\eta}') \ar[r]^{\pi_*} & \rightarrow \pi^{\et}_1(S,\bar{\eta}).
}
\]
The locally constant sheaf $\pi^{-1}\CN_{X/S}$ is attached to the representation of $\pi^{\et}_1(S',\bar{\eta}')$ obtained from the one of $\pi^{\et}_1(S,\bar{\eta})$ on $\NS(X_{\bar{\eta}})$ by restriction along $\pi_*$. We also have an induced isomorphism
\[
\NS(X_{\bar{\eta}})\simeq \NS(X'_{\bar{\eta}'})
\]
which is equivariant for the Galois actions (via $\pi_*$); note that it is an isomorphism because the geometric N\'eron-Severi group is invariant under extension of algebraically closed field (Lemma \ref{lem:ns_prop} \ref{ns_invariance}). Combined with the previous observation, this provides the desired isomorphism
\[
v^{N}_{\pi}:\pi^{-1}\CN_{X/S}\stackrel{\sim}{\ra} \CN_{X'/S'}.
  \]
With the same hypotheses, let us now define a morphism
\[e_S:\alpha_*\sNS^{\sm}_{X/S}\ra \mathcal{N}_{X/S}.\]

We first define a morphism $\tilde{c}_S:\alpha_* \sPic^{\sm}_{X/S}\ra \mathcal{N}_{X/S}$. Recall that $\alpha_*\sPic^{\sm}_{X/S}$ is the \'etale sheaf associated to the presheaf $\sPic^{\sm,\mathrm{psh}}_{X/S}:V\in \Et/S\mapsto \Pic(X\times_S V)$. Since $\mathcal{N}_{X/S}$ is an \'etale sheaf, defining $\tilde{c}_S$ is equivalent to writing down a morphism $\sPic^{\sm,\mathrm{psh}}_{X/S}\ra \mathcal{N}_{X/S}$.

Let $V\in \Et/S$, which we can assume to be connected, and $\CL$ be a line bundle on $X\times_S V$. Choose a geometric point $\bar{v}\ra V$ such that there exists a factorisation $\bar{v}\ra \bar{\eta}\ra S$. By definition of $\mathcal{N}_{X/S}$, we have $\mathcal{N}_{X/S}(V)\simeq \NS(X_{\bar{\eta}})_\BQ^{\pi^\et_1(V,\bar{v})}$. We can pullback $\CL$ along $\bar{v}\ra V$ to get $[\CL_{\bar{v}}]\in \Pic(X_{\bar{v}})$. Since $\NS(X_{\bar{v}})\simeq \NS(X_{\bar{\eta}})$ by Lemma \ref{lem:ns_prop}, we get a class in $\NS(X_{\bar{\eta}})\otimes\BQ$. By construction, since it comes from a line bundle on $X\times_S V$, it is fixed by $\pi^\et_1(V,\bar{v})$, and thus defines a section $\tilde{c}_S([\CL])\in \mathcal{N}_{X/S}(V)$.

Let us show that the morphism $\tilde{c}_S$ is trivial on $\alpha_*\sPic_{X/S}^{\sm,0}$. Let $V\in \Et/S$ and $x\in \sPic_{X/S}^{\sm,0}(V)$. Since $\mathcal{N}_{X/S}$ is separated for the \'etale topology, to check that $\tilde{c}_S(x)=0$ we can pass to an \'etale covering of $V$ and assume that $x$ is represented by a line bundle $\CL\in \Pic(X\times_S V)$. By definition of $\sPic_{X/S}^{\sm,0}$ and Remark \ref{rmk:neutral_field_extension}, we have $[\CL_{\bar{v}}]\in Pic^0_{X_{\bar{v}}/\bar{v}}(\bar{v})$, hence the class of $\CL_{\bar{v}}$ in $\NS(X_{\bar{v}})$ vanishes, thus $\tilde{c}_S([\CL])=0$. We conclude that $\tilde{c}_S$ induces a morphism $e_{X/S}:\alpha_*\sNS^{\sm}_{X/S}\simeq \alpha_*\sPic_{X/S}^{\sm}/\alpha_*\sPic^{\sm,0}_{X/S}\ra \mathcal{N}_{X/S}$ as required.

The base change morphisms of Lemma~\ref{lem:base_change_Pic} induce a morphism
  \[
v^{NS}_{\pi}:\pi^{-1}\sNS^{\sm}_{X/S}\ra \sNS^{\sm}_{X'/S'}
\]
By going through the definitions of the various base change maps, one can show the following compatibility. 

\begin{lemma}\label{lem:comp_bc_NS}
Let $\pi:S'\ra S$ be a universally open morphism between regular schemes. The diagram of \'etale sheaves
  \[
\xymatrix{
\pi^{-1} \alpha_* \sNS^{\sm}_{X/S} \ar[d] \ar[r]^{\pi^{-1}e_{X/S}} & \pi^{-1}\CN_{X/S} \ar[dd]^{v^{N}_{\pi}} \\
\alpha_*\pi^{-1} \sNS^{\sm}_{X/S} \ar[d]_{\alpha_* v^{NS}_{\pi}} & \\
\alpha_*\sNS^{\sm}_{X'/S'} \ar[r]_{e_{X'/S'}}& \CN_{X'/S'}
    }
  \]
is commutative.  
\end{lemma}

We now come to the main result of this section.

\begin{prop}\label{prop:NS_smooth}
  Let $f:X\ra S$ be a smooth projective morphism with $S$ regular. Then 
\begin{enumerate}[label={\upshape(\roman*)}]
\item \label{ns_one} $e_{X/S}:\alpha_* \sNS^{\sm}_{X/S}\otimes\BQ\ra \mathcal{N}_{X/S}$ is an isomorphism, and
\item \label{ns_two} the counit morphism $\alpha^{-1}\alpha_* \sNS^{\sm}_{X/S}\otimes\BQ\lra \sNS^{\sm}_{X/S}\otimes\BQ$ is an isomorphism.
\end{enumerate}
In particular, the sheaf $\sNS^\sm_{X/S}\otimes\BQ$ is a locally constant finitely generated sheaf of $\BQ$-vector spaces.
\end{prop}
\begin{proof}

  The morphism in \ref{ns_one} (resp. \ref{ns_two}) is a morphism of \'etale sheaves on $\Et/S$ (resp. $\Sm/S$). To show that it is an isomorphism, it suffices to check on stalks at geometric points of $S$ (resp. at geometric points of all smooth $S$-schemes).

  We first compute the stalks of $\sNS^{\sm}_{X/S}$ on $\Et/S$. Let $\bar{s}$ be a geometric point of $S$. Write $\mathcal{U}$ for the projective system of all \'etale neighbourhoods of $\bar{s}$ in $S$, i.e. the system of all pairs $(U,\bar{u})$ with $U$ an \'etale $S$-scheme and $\bar{u}$ a lift of $\bar{s}$ to $U$. Write $S^{\sh}_{\bar{s}}$ for the projective limit of $\mathcal{U}$, the spectrum of a strict henselisation of the local ring $\mathcal{O}_{S,s}$, and $X^{\sh}_{\bar{s}}=X\times_S S^{\sh}_{\bar{s}}$. Let us write $\nu$ for the generic point of $S^{\sh}_{\bar{s}}$. By definition of $\sNS^{\sm}_{X/S}$, we have an exact sequence of stalks
  \[
0\ra (\sPic^{\sm,0}_{X/S})_{\bar{s}}\ra (\sPic^{\sm}_{X/S})_{\bar{s}}\ra (\sNS^{\sm}_{X/S})_{\bar{s}}\ra 0
    \]
The stalks of a higher direct image is easily computed; e.g., by \cite[Tag 03Q7]{stacks-project}, we have
\begin{eqnarray*}
        (\sPic^{\sm}_{X/S})_{\bar{s}} & \simeq & (R^1 f_*\Gm)_{\bar{s}} \\
                                      & \simeq & H^1(X^{\sh}_{\bar{s}},\Gm) \\
                                      & \simeq & \Pic(X^{\sh}_{\bar{s}}).                                                 
        \end{eqnarray*}
Moreover, for $(V,\bar{v})\in \CU$, it is easy to see that the composition
      \[
\Pic(X\times_S V)\rightarrow \sPic^{\sm}_{X/S}(V)\rightarrow \Pic(X^{\sh}_{\bar{s}})
        \]
coincides with the pullback map on Picard groups. Let us denote $\Pic^0(X^{\sh}_{\bar{s}})$ for the subgroup of $\Pic(X^{\sh}_{\bar{s}})$ consisting of isomorphism classes of line bundles $L$ on $X^{\sh}_{\bar{s}}$ which are such that for all geometric points $\bar{r}$ of $S^{\sh}_{\bar{s}}$, we have $\overline{L_{\bar{r}}}=0\in \NS(X_{\bar{r}})$. Then we have
\[
(\sPic^{\sm,0}_{X/S})_{\bar{s}}\simeq \Pic^0(X^{\sh}_{\bar{s}})
\]
as subgroups of $(\sPic^{\sm}_{X/S})_{\bar{s}}\simeq \Pic(X^{\sh}_{\bar{s}})$. We thus have
\[
(\sNS^{\sm}_{X/S})_{\bar{s}}\simeq \Pic(X^{\sh}_{\bar{s}})/\Pic^0(X^{\sh}_{\bar{s}}).
\]

We now compute the stalk of $\mathcal{N}_{X/S}$ at $\bar{s}$. Write $\pi:S^{h}_{\bar{s}}\ra S$ which is a universally open morphism between regular schemes. Using the isomorphism $v^{N}_{\pi}$, we see that
\[
(\pi^{-1}\mathcal{N}_{X/S})_{\bar{s}}\simeq (\mathcal{N}_{X^{\sh}_{\bar{s}}/S^{\sh}_{\bar{s}}})_{\bar{s}}.
  \]
  The sheaf $\mathcal{N}_{X^{\sh}_{\bar{s}}/S^{\sh}_{\bar{s}}}$ is locally constant on $S^{\sh}_{\bar{s}}$, which is a local stricly henselian scheme, thus has trivial \'etale fundamental group. This implies that the stalks of $\mathcal{N}_{X^{\sh}_{\bar{s}}/S^{\sh}_{\bar{s}}}$ at all generic points are canonically isomorphic. Hence, for any geometric generic point $\bar{\nu}$ of $S^{\sh}_{\bar{s}}$, we get a canonical isomorphism
  \[
(\mathcal{N}_{X/S})_{\bar{s}}\simeq \NS(X_{\bar{\nu}})\otimes\BQ.
    \]
  
Let us now prove that the map \ref{ns_one} is an isomorphism. It suffices to prove that the induced morphism on \'etale stalks at $\bar{s}$ geometric point of $S$ is an isomorphism. The morphism $(e_{X/S})_{\bar{s}}$ is a morphism
\[
(\Pic(X^{\sh}_{\bar{s}})/\Pic^0(X^{\sh}_{\bar{s}}))\otimes\BQ \ra \NS(X_{\bar{\nu}})\otimes\BQ
  \]
and by going through the definitions, it is easy to see that this map is induced by the composition
\[
\Pic(X^{\sh}_{\bar{s}})\ra \Pic(X^{\sh}_{\bar{s}}\times_{S^{\sh}_{\bar{s}}} \nu)=\Pic(X_{\nu})\ra \NS(X_{\bar{\nu}})
  \]
where the first map is the pullback map on Picard groups. Let us show that this map is an isomorphism, at least after tensoring by $\BQ$.

We first show the injectivity (before tensoring by $\BQ$). Let $[L]\in \Pic(X^{\sh}_{\bar{s}})$ with $\overline{L_{\bar{\nu}}}=0$ in  $\NS(X_{\bar{\nu}})$. We have to show that for all geometric points $\bar{r}$ of $S^{\sh}_{\bar{s}}$, we have $\overline{L_{\bar{r}}}=0$ in $\NS(X_{\bar{r}})$. This is the content of Lemma \ref{lem:spe_eq_alg}.
  
We now prove the surjectivity after tensoring by $\BQ$ (we thank the referee for pointing that the original argument for this was wrong). By Lemma \ref{lem:ext_line_bundles}, the morphism $\Pic(X^{\sh}_{\bar{s}})\ra \Pic(X^{\sh}_{\bar{s}}\times_{S^{\sh}_{\bar{s}}} \nu)$ is surjective, so that it is enough to show that the morphism $\Pic(X_{\nu})\ra \NS(X_{\bar{\nu}})$ is surjective after tensoring by $\BQ$.

Since $X^{\sh}_{\bar{s}}\ra S^{\sh}_{{\bar{s}}}$ is a smooth projective morphism, we have seen just before Lemma \ref{lem:stein_torus} that $\pi_{0}(X^{\sh}_{\bar{s}}/S^{\sh}_{{\bar{s}}})$ is a finite \'etale cover of $S^{\sh}_{\bar{s}}$ which is strictly henselian; this implies that $\pi_{0}(X^{\sh}_{\bar{s}}/S^{\sh}_{{\bar{s}}})$ is a trivial \'etale cover of $S^\sh_{\bar{s}}$. In particular, every geometric connected component of $X_{\nu}$ is defined over $\nu$. Since $\bar{s}$ is algebraically closed, every connected component of $X$ contains a rational point. Since $X^\sh_{\bar{s}}\ra S^\sh_{\bar{s}}$ is smooth and $S^\sh_{\bar{s}}$ is henselian, these extend to sections over the whole of $S^\sh_{\bar{s}}$. This implies by the above that every connected component of $X_{\nu}$ contains a rational point. Hence we only need to show the surjectivity of $\Pic(X_{\nu})\ra \NS(X_{\bar{\nu}})$ connected component by connected component and we can assume that $X_{\nu}$ is (geometrically) connected with a rational point.

Under this additional assumption, the Leray spectral sequence for the \'etale sheaf $\Gm$ yields an exact sequence 
\[
\Pic(X_{\nu})\ra \Pic(X_{\bar{\nu}})^{\Gal(\kappa(\bar{\nu})^{s}/\kappa(\nu))}\ra \mathrm{Br}(\nu)
  \]
  and since the Brauer group of a field is torsion, the morphism $\Pic(X_{\nu})\ra \Pic(X_{\bar{\nu}})^{\Gal(\kappa(\bar{\nu})^{s}/\kappa(\nu))}$ is surjective after tensoring with $\BQ$. So it remains to show that $\Pic(X_{\bar{\nu}})^{\Gal(\kappa(\bar{\nu})^{s}/\kappa(\nu))}\otimes\BQ\ra \NS(X_{\bar{\nu}})\otimes\BQ$ is surjective.

  By Lemma \ref{lem:Q_coeffs}, we have $\Pic(X_{\bar{\nu}})^{\Gal(\kappa(\bar{\nu})^{s}/\kappa(\nu))}\otimes\BQ\simeq (\Pic(X_{\bar{\nu}})\otimes\BQ)^{{\Gal(\kappa(\bar{\nu})^{s}/\kappa(\nu))}}$. The morphism $\Pic(X_{\bar{\nu}})\ra \NS(X_{\bar{\nu}})$ is surjective by definition and Galois-equivariant; moreover, its kernel is the group of $\bar{\nu}$-points of the abelian variety $\Pic^{0,\red}_{X_{\nu}/\nu}$. The Weil-Ch\^atelet group $H^{1}(\nu_{\et},\Pic^{0,\red}_{X_{\nu}/\nu})$ is a filtered colimit of torsion abelian groups (as clearly follows with its identification with Galois cohomology and the fact that the $n-$torsion of an abelian variety for any given $n\in\BN$ is finite), hence it is a torsion group. By Lemma \ref{lem:Q_coeffs}, we deduce that $H^{1}(\nu_{\et},\Pic^{0,\red}_{X_{\nu}/\nu}\otimes\BQ)=0$. We conclude that
  \[
    (\Pic(X_{\bar{\nu}})\otimes\BQ)^{{\Gal(\kappa(\bar{\nu})^{s}/\kappa(\nu))}}\ra (\NS(X_{\bar{\nu}})\otimes\BQ)^{{\Gal(\kappa(\bar{\nu})^{s}/\kappa(\nu))}}
    \]
is surjective.
  
By Lemma \ref{lem:unramified}, the Galois action on $\NS(X_{\bar{\nu}})\otimes\BQ$ is unramified, hence it factors through the unramified quotient $\pi_{1}^{\et}(S^{\sh}_{\bar{s}},\bar{\nu})$ \cite[Expos\'e V Proposition 8.2]{SGA1}, which vanishes since $ S^{\sh}_{\bar{s}}$ is strictly henselian. Hence we have $(\NS(X_{\bar{\nu}})\otimes\BQ)^{{\Gal(\kappa(\bar{\nu})^{s}/\kappa(\nu))}}\simeq \NS(X_{\bar{\nu}})\otimes\BQ$. This concludes the proof of surjectivity, and of \ref{ns_one}.

Let us finally prove that \ref{ns_two} is an isomorphism. Let $T\in \Sm/S$ and $\bar{t}$ a geometric point of $T$. We must show that the morphism
\[
(\alpha^{-1}\alpha_* \sNS^{\sm}_{X/S})_{\bar{t}}\otimes\BQ \ra (\sNS^{\sm}_{X/S})_{\bar{t}}\otimes\BQ
  \]
is an isomorphism. By replacing $T$ by a neighbourhood of the image of $\bar{t}$, we can assume that $T$ is integral. By composition, $\bar{t}$ determines a geometric point $\bar{s}$ of $S$, and we have $(\alpha^{-1}\alpha_* \sNS^{\sm}_{X/S})_{\bar{t}}\simeq (\sNS^{\sm}_{X/S})_{\bar{s}}$.
  
Since, for any \'etale $T$-scheme $W$, we have
\[
\sNS^{\sm}_{X/S}(W)=\sNS^{\sm}_{X\times_S T/T}(W)
  \]
  we conclude that
  \[
(\sNS^{\sm}_{X/S})_{\bar{t}}\simeq (\sNS^{\sm}_{X\times_S T})_{\bar{t}}
\]
In the proof of \ref{ns_one} above, we have seen that for any morphism $f:X\ra S$ satisfying the hypothesis of the proposition and with $S$ integral, we have $(\sNS^{\sm}_{X/S})_{\bar{t}}\otimes\BQ\simeq \NS(X_{\eta})$, with $\bar{\eta}$ any geometric point above the generic point $\eta$ of $S$. By Lemma~\ref{lem:bc_Pic_smooth}, the morphism $X\times_S T\ra T$ satisfies these assumptions. Choose $\bar{\theta}$ a geometric generic point of $T$, compatible with $\bar{\eta}$. We deduce that the morphism we are interested in coincides with the natural morphism
\[
\NS(X_{\bar{\eta}})\rightarrow \NS(X_{\bar{\theta}})
\]
which is an isomorphism by Lemma \ref{lem:ns_prop} \ref{ns_invariance}. This completes the proof of \ref{ns_two}, and of the proposition.

\end{proof}

\begin{lemma}\label{lem:ext_line_bundles}
Let $S$ be a regular integral scheme with generic point $\eta$ and $f:X\ra S$ a smooth projective morphism. Then the restriction morphism $\Pic(X)\ra \Pic(X_\eta)$ is surjective. If $S$ is moreover the spectrum of a discrete valuation ring, then it is bijective.
\end{lemma}
\begin{proof}
This is essentially \cite[9.4/Theorem 3]{BLR} (it follows directly from the arguments of the proof in loc.cit.)
  % The surjectivity follows immediately from the fact that for a regular scheme $Z$, we have $\Pic(Z)\simeq \CHo^1(Z)$, and that the restriction morphism on Weil divisors is split surjective with a section given by Zariski closure. Assume that $S$ is the spectrum of a discrete valuation ring. Let $\CL\in \Pic(X)$ with $\CL_{\eta}\simeq 0$. Then there exists a section $\sigma\in \Gamma(X_{\eta},\CL_{\eta})$ which generates $\CL_{\eta}$. Up to multiplying by a power of the uniformizer, we can assume that $\sigma$ ex
\end{proof}  

\begin{lemma}\label{lem:spe_eq_alg}
Let $f:X\ra S$ be a smooth projective morphism with $S$ integral. Let $\CL\in \Pic(X)$. Let $\bar{\eta}$ be a geometric generic point of $S$, and $\bar{s}$ any geometric point of $S$. If $\CL_{\bar{\eta}}$ is algebraically equivalent to $0$, then $\CL_{\bar{s}}$ is algebraically equivalent to $0$.
\end{lemma}
\begin{proof}
Write $\eta$ for the generic point of $S$ and $s$ for the image of $\bar{s}$. Since $S$ is integral, $s$ is a specialisation of $\eta$. By Proposition \cite[7.1.9]{EGAII}, we can find the spectrum $S'$ of a discrete valuation ring $R$ and a morphism $S'\ra S$ sending the generic point of $S'$ to $\eta$ and the closed point to $s$. Using Lemma \ref{lem:ns_prop} \ref{ns_invariance} if necessary to change the geometric points, this implies that we can assume that $S$ is the spectrum of a discrete valuation ring, which we can even assume to be excellent by pulling back further to the completion.

Let $C$ be a smooth projective connected curve over $\bar{\eta}$ with two points $x_0,x_1\in C(\bar{\eta})$ and a line bundle $\widetilde{\CL}\in \Pic(X_{\bar{\eta}}\times_{\bar{\eta}} C)$ such that $x_0^*\widetilde{\CL}\simeq \CL_{\bar{\eta}}$ and $x_1^*\widetilde{\CL}\simeq 0$. The curve $C$, the points $x_0,x_1$ and the line bundle $\widetilde{\CL}$ are defined over a finite extension $K'$ of $\kappa(\eta)$, so that after replacing $S$ by its normalisation in $K'$ we can assume that everything is defined over $\eta$. By Lipman's resolution of singularities for excellent $2$-dimensional schemes \cite{Lipman}, the curve $C$ extends to a regular proper flat $S$-scheme $\mathcal{C}$, and the special fiber $\mathcal{C}_{s}$ is still geometrically connected by Zariski's connectedness theorem. The points $x_0,x_1$ extend to sections $y_0,y_1$ of $\mathcal{C}\ra S$ by properness. By Lemma \ref{lem:ext_line_bundles}, the line bundle $\widetilde{\CL}$ extends to a line bundle $\widehat{\CL}$ on $X\times_S \mathcal{C}$ which satisfies $y_0^*\widehat{\CL}\simeq \CL$ and $y_1^*\widehat{\CL}\simeq 0$. This implies that $\CL_{\bar{s}}$ is algebraically equivalent to $0$ and finishes the proof.
\end{proof}

\subsection{Motivic applications of Picard complexes}

The results of the two previous sections can be applied to define and study some interesting $1$-motives.

\begin{cor}\label{cor:Picard_smooth}
  Assume $S$ is regular. Let $f:X\ra S$ be a smooth projective Pic-smooth morphism of schemes. Then we have natural distinguished triangles
\[
\Sigma^\infty (\underline{f}_* \Gm\otimes\BQ)[1] \ra \Sigma^\infty P(X/S) \ra \Sigma^\infty (\sPic^{\sm}_{X/S}\otimes\BQ)\rap
\]
and
\[
\Sigma^\infty(\sPic^{\sm,0}_{X/S}\otimes\BQ)\ra \Sigma^\infty (\sPic^{\sm}_{X/S}\otimes\BQ) \ra \Sigma^\infty \sNS^\sm_{X/S}\otimes\BQ\rap
\]
and the motive $\Sigma^\infty \mathrm{P}(X/S)$ lies in $\DA^{\gsm}_{1,c}(S)$. Moreover, these two distinguished triangles admit (non-canonical) splittings, so that we have
\[
\Sigma^\infty P(X/S)\simeq \Sigma^\infty (\underline{f}_* \Gm\otimes\BQ)[1] \oplus \Sigma^\infty(\sPic^{\sm,0}_{X/S}\otimes\BQ) \oplus \Sigma^\infty \sNS^\sm_{X/S}\otimes\BQ
\]
\end{cor}
\begin{proof}
  The first distinguished triangle is obtained from the truncation triangle for $P(X/S)$ for the standard t-stucture on $D((\Sm/S)_\et,\BQ)$. The second one follows from the short exact sequence of sheaves
  \[
0\rightarrow \sPic^{\sm,0}_{X/S}\otimes\BQ\rightarrow \sPic^{\sm}_{X/S}\otimes\BQ\rightarrow \sNS^\sm_{X/S}\otimes\BQ\rightarrow 0
  \]
By Lemma \ref{lem:stein_torus}, the sheaf $\underline{f}_*\Gm\simeq \Res_{\pi_0(f)}\Gm$ is representable by a torus. By Proposition \ref{prop:repr_pic_smooth}, the sheaf $\sPic^{\sm,0}_{X/S}$ is representable by the abelian scheme $\Pic^{0,\red}_{X/S}$. Finally, the sheaf $\sNS^{\sm}_{X/S}$ is representable by a lattice by Proposition~\ref{prop:NS_smooth}. From Corollary~\ref{cor:Deligne_da_1}, we conclude that $\Sigma^\infty \mathrm{P}(X/S)$ is in $\DA^{\gsm}_{1,c}(S)$.

To show that the triangles split, it is enough to show that the connecting morphisms vanish. Given the representability results for the various pieces, this follows from Lemma~\ref{lem:Deligne_decomposition} below.  
\end{proof}

\begin{lemma}\label{lem:Deligne_decomposition}
  Let $S$ be a regular scheme. Let $L$ be an  $S$-lattice, $T$ be an $S$-torus and $A$ be an $S$-abelian scheme.
  \begin{enumerate}[label={\upshape(\roman*)}]
  \item $\DA(S)(\Sigma^\infty L_\BQ,\Sigma^\infty T_\BQ[2])=0$
  \item $\DA(S)(\Sigma^\infty A_{\BQ},\Sigma^\infty T_\BQ[2])=0$
  \item $\DA(S)(\Sigma^\infty L_\BQ,\Sigma^\infty A_\BQ[1])=0$
  \end{enumerate}
\end{lemma}

\begin{proof}
  Since $S$ is regular, and in particular geometrically unibranch, Lemma~\ref{lemm:permutation_torus} together with Proposition~\label{prop:Gm_Q1} implies that there exists $e:T\rightarrow S$ finite \'etale such that $\Sigma^\infty L_\BQ$ is a direct factor of $e_\sharp \BQ_T$  and that $\Sigma^\infty T_\BQ$ is a direct factor of $e_*\BQ_T(1)[1]$. By adjunction and Proposition~\ref{prop:pullback_complex}, we then reduce to the case where $L$ is constant and $T$ splits. Point (ii) then says that $\DA(S)(\BQ,\BQ(1)[3])=0$, which is proved in Proposition~\ref{prop:mot_coh_1} \ref{mot_coh_1_3} (or in Proposition~\ref{mot_coh_gg}). By \cite[]{AHPL}, writing $\pi:A\ra S$ for the structure morphism, we see that $\Sigma^\infty A_{\BQ}$ is a direct factor of $\pi_\sharp \BQ_A$. By adjunction, we have
  \[
\DA(S)(\pi_\sharp \BQ_A,\BQ(1)[3])\simeq \DA(A)(\BQ_A,\BQ_A(1)[3])
\]
which vanishes, again by Proposition~\ref{prop:mot_coh_1} \ref{mot_coh_1_3}. This proves (iii). Write $d$ for the fibre dimension of $A/S$. We have $\pi_\sharp \BQ_A\simeq \pi_*\BQ_A(d)[2d]$, hence by adjunction
\[
\DA(S)(\BQ,\pi_\sharp A_{\BQ}[1])\simeq \DA(A)(\BQ,\BQ(d)[2d+1])
\]
This last group vanishes by Proposition~\ref{mot_coh_gg}.
\end{proof}

We also have an application to the question of base change for $P(X/S)$. Recall that a base change map for $P(X/S)$ was defined after the proof of Lemma~\ref{lem:base_change_Pic}.

\begin{cor}\label{cor:Picard_bc}
  Let $S,S'$ be regular schemes, and $f:X\ra S$ be a smooth projective Pic-smooth morphism. Let $\pi:S'\ra S$ be a universally open morphism of schemes. Then the base change map
  \[
V_{\pi}:\pi^{-1}P(X/S)\ra P(X'/S')    
  \]
  is an isomorphism.
\end{cor}
\begin{proof}
By Corollary~\ref{cor:Picard_smooth}, we know how to compute $P(X/S)$ and $P(X'/S')$. By the definition of $V_\pi$, the commutation of $\pi_0(X/S)$ with arbitrary base change and Lemma~\ref{lem:vsm_iso}, we see that $V_{\pi}$ is an isomorphism if and only if $v^{NS}_{\pi}:\pi^{-1}\sNS^{\sm}_{X/S}\ra \sNS^{\sm}_{X'/S'}$ is. By Lemma~\ref{lem:comp_bc_NS} and Proposition~\ref{prop:NS_smooth}, we see that this is the case since $v^N_{\pi}$ is an isomorphism by construction.
\end{proof}

Another important corollary is the comparison with the theory with transfers.

\begin{cor}\label{cor:Picard_tr}
Let $S$ be a regular excellent scheme and $f:X\ra S$ a smooth projective Pic-smooth morphism. We have distinguished natural triangles
\[
(\underline{f}_* \Gm^\tr)\otimes\BQ \ra P(X/S)^\tr_\BQ \ra (\sPic^{\sm,\tr}_{X/S}\otimes\BQ)\rap
\]
and
\[
\sPic^{\sm,0,\tr}_{X/S}\ra (\sPic^{\sm,\tr}_{X/S}\otimes\BQ) \ra \sNS^{\sm,\tr}_{X/S}\otimes\BQ\rap.
\]
and these triangles are (non-canonically) split. Moreover, the natural map 
\[
a^\tr P(X/S) \lra P^\tr(X/S)
\]
is an isomorphism.
\end{cor}
\begin{proof}
The distinguished triangles follow from the same arguments as for $P(X/S)$. For $G/S$ a smooth commutative group scheme, the natural map $a^{\tr}G\otimes\BQ\ra G^\tr\otimes\BQ$ is an isomorphism by  \cite[Proposition 3.10]{AHPL}. Since each term of the triangles is represented by a smooth commutative group scheme, we deduce that the map $a^\tr P(X/S)\lra P^\tr(X/S)$ is an isomorphism.
\end{proof}

Finally, we look more closely at the case of a relative smooth projective curve, where things are simpler.

\begin{prop}\label{prop:picard_curve}
Let $f:C\ra S$ be a smooth projective relative curve (with $S$ arbitrary). Then $f$ is Pic-smooth, and $\sNS^{\sm}_{C/S}$ is represented by a lattice canonically isomorphic to $\BQ[\pi_0(C/S)]$. In particular, for any $g:T\ra S$, the morphism $v_g:g^{-1}P(C/S)\ra P(C_T/T)$ is an isomorphism.
\end{prop}
\begin{proof}
When $f$ has connected fibres, this is contained in the computation of relative Picard schemes for smooth projective curves in \cite[Theorem 9.3.1]{BLR}. Since $\pi_0(C/S)$ is finite \'etale, the general case follows by \'etale descent. The addendum comes from Corollary~\ref{cor:Picard_bc} and the fact that the construction of $\pi_0(C/S)$ commutes with arbitrary base change.
\end{proof}

We adopt a special notation in this case.

\begin{notation}
Let $f:C\ra S$ be a smooth projective relative curve. We call the abelian scheme $\Pic^{0,\red}_{C/S}$ the \emph{(relative) Jacobian} of $C$ over $S$, and we denote it by $\Jac(C/S)$.  
\end{notation}

Let $f:X\ra S$ be a finite type morphism of schemes. We introduce a morphism $\Theta_f:\Sigma^\infty\mathrm{P}(X/S)(-1)[-2]\lra f_*\BQ_X$ which plays a key role in the computation of the motivic Picard functor in the next section. 

We start with the adjunction morphism
\[
\Sus^1\Ev_1 f_*\BQ_X\stackrel{\eta}{\lra} f_*\BQ_X.
\]
The functors $\Ev_1$ and $f_*$ commute, because they are right derived functors of right Quillen functors which commute at the model category level. We thus have a canonical isomorphism
\[ 
f_*\Ev_1\simeq \Ev_1 f_*:\DA^{\eff}(X)\lra \DA(S).
\]
By composition we obtain a map
\[
\Sus^1 f_* \Ev_1(\BQ_X)\lra f_*\BQ_X.
\]
We then use the morphism $w_S$ described at the end of Section~\ref{sec:mot_group_schemes} to construct a map
\[
\Sus^1 f_*(\Gm\otimes\BQ[1]) \lra f_*\BQ_X.
\] 
Recall that $\Sus^1\simeq \Sigma^\infty (-) (-1)[-2]$ so that we get a morphism
\[
\Sigma^\infty f_*(\Gm\otimes\BQ[1])(-1)[-2]\lra f_*\BQ_X.
\]
Then by composing with the counit morphism $\tau_{\geq 0}(-)\ra \id$, we obtain the desired morphism
\[
\Theta_f : \Sigma^\infty \mathrm{P}(X/S)(-1)[-2]\lra f_*\BQ_X.
\]
We can do the same construction in $\DM(-)$ using $w^\tr_S$ and the pushforward operations in $\DM(-)$, resulting in a morphism
\[
\Theta^\tr_f:\Sigma^\infty_\tr \mathrm{P}^\tr(X/S)(-1)[-2]\lra f_*\BQ_X^\tr
\]
in $\DM(S)$. Later on, we will need an alternative description of the map $\Theta^\tr_f$ at the effective level. Recall from the conventions section that $\DM^{\eff}(-)$ has its own functoriality, in the form of a premotivic category as in \cite[\S 11.1.a]{Cisinski_Deglise_BluePreprint}.

\begin{prop}\label{prop:theta_tr}
  Let $S$ be a regular scheme and $f$ be a smooth projective Pic-smooth morphism.
  \begin{enumerate}[label={\upshape(\roman*)}]
  \item The natural morphism $a^\tr\Sigma^\infty P(X/S)\simeq \Sigma^\infty_\tr a^\tr P(X/S)\ra \Sigma^\infty P^\tr(X/S)$ is an isomorphism by Corollary~\ref{cor:Picard_tr}, and the natural morphism  $a^\tr f_*\BQ_X\ra f_*\BQ_X^\tr$ in $\DM(S)$ is an isomorphism because of the comparison theorem between $\DA$ and $\DM$ on geometrically unibranch schemes \cite[16.2.22]{Cisinski_Deglise_BluePreprint}. Modulo these identifications, we have
\[
a^\tr\Theta_f =\Theta_f^\tr.
\]
\item The morphism $\Theta_f^\tr$ admits the following alternative description. The morphism $\alpha_{G}^{\eff,\tr}:\BQ^\tr(1)[1]\ra \Gm\otimes\BQ$ is an isomorphism in $\DM^{\eff}(X)$ by \cite[Proposition 11.2.1]{Cisinski_Deglise_BluePreprint} since $X$ is normal, and we denote by $u_X^{\eff,\tr}$ its inverse, so that we have $\Sigma^\infty u_X^{\eff,\tr}=u_X^{\tr}$ (they are inverses to the same map). Then $\Theta^\tr_f$ is the composition
\[
\Sus^1 \tau_{\geq 0}(f_* (\Gm\otimes\BQ[1])) \xrightarrow{\Sus^1_{\tr} f_* u_X^{\eff,\tr}} \Sus^1_{\tr} f_*(\BQ^\tr(1)[2])\simeq (\Sigma^\infty_{\tr} f_*(\BQ^\tr(1)[2])(-1)[-2]\ra f_*\BQ^\tr
\]
where the last morphism is induced by the natural transformation $\Sigma^\infty_{\tr} f_* \ra f_*\Sigma^\infty_{\tr}$ which is constructed by adjunction from the natural isomorphism $\Sigma^\infty_{\tr}f^*\simeq f^*\Sigma^\infty_{\tr}$.
  \end{enumerate}
\end{prop}
\begin{proof}
Statement (i) translates into proving the commutativity of the outer square in the following diagram.
\[
\xymatrix{
a^\tr \Sus^1 f_*\Gm\otimes\BQ[1] \ar[r]^{w_X} \ar[d]_{\sim} & a^\tr \Sus^1 f_* \Ev_1\BQ \ar[r]^{\sim} \ar[d]_{\sim} & a^\tr \Sus^1 \Ev_1 f_*\BQ \ar[r]^{\eta} \ar[d]_{\sim} & a^\tr f_*\BQ \ar[ddd]\\
\Sus^1_{\tr} a^\tr f_*\Gm\otimes\BQ[1] \ar[r]^{w_X} \ar[d] & \Sus^1_{\tr} a^\tr f_*\Ev_1\BQ \ar[r]_{\sim} \ar[d] & \Sus^1_{\tr} a^\tr \Ev_1 f_*\BQ \ar[d] & \\
\Sus^1_\tr f_* a^\tr \Gm\otimes\BQ[1] \ar[r]^{w_X} \ar[d]_{\sim} & \Sus^1_\tr f_* a^\tr \Ev_1\BQ \ar[d] & \Sus^1_\tr  \Ev_1^\tr a^\tr f_*\BQ \ar[d] & \\
\Sus^1_\tr f_* \Gm^\tr \otimes \BQ[1] \ar[r]^{w_X} & \Sus^1_\tr f_* \Ev^\tr_1\BQ^\tr \ar[r]_{\sim} & \Sus^1_\tr \Ev_1^\tr f_*\BQ^\tr \ar[r]^{\eta} & f_*\BQ^\tr 
}
\]
All squares in this diagram commute either by naturality of adjunctions or because of the commutation $\Sus^1_\tr a^\tr \simeq a^\tr\Sus^1$.

For the end of the proof, in order to fit the definitions on a line, let's write $E_{1}=\Ev_{1}$ and $S^{1}=\Sus^{1}$. For Statement (ii), we observe that $\Theta_f^\tr$ is defined as the composition
\[
S^1 f_*\Gm^\tr[1]\stackrel{\epsilon}{\ra} S^1 f_*E_1S^1\Gm^\tr[1]\xrightarrow{S^1 f_*E_1 (u^\tr_X(-1)[-2])} S^1 f_*E_1 \BQ^\tr \simeq S^1 E_1 f_*\BQ^\tr \stackrel{\eta}{\ra}f_*\BQ^\tr
\]
(we have expanded the definition of $w^{\tr}_X$), whereas the map of the statement is the composition
\[
S^1 f_*\Gm^\tr[1]\xrightarrow{S^1 f_* u^{\eff,\tr}_X} S^1 f_*\BQ^\tr(1)[2] \stackrel{\epsilon}{\ra} S^1 f_*E_1 S^1 \BQ^\tr(1)[2]\simeq S^1 E_1 f_* S^1 \BQ^\tr(1)[2] \stackrel{\eta}{\ra} f_*\BQ^\tr
\]
(we have expanded the definition of the map $\Sigma^\infty f_*\ra f_*\Sigma^\infty$). The equality of those two compositions follows from the naturality of the $(S^1,E_1)$ adjunction and the equality
\[S^1u_X^{\eff,\tr}=\Sigma^\infty u_x^{\eff,\tr}(-1)[-2]=u_X^\tr(-1)[-2].\]
\end{proof}

We finish with a study of the compatibility of the map $\Theta_f$ with base change.

\begin{prop}\label{prop:base_change_theta}
Let $f:X\ra S$ be a smooth projective Pic-smooth morphism of schemes. Let $g:T\ra S$ be any morphism. Let $f':X_T\ra T$ be the pullback (which is still smooth projective Pic-smooth by Lemma~\ref{lem:bc_Pic_smooth}). The diagram
\[
\xymatrix{
g^* \Sigma^\infty \mathrm{P}(X/S) (-1)[-2] \ar[r]^(0.7){g^*\Theta_f} \ar[d]_{V_g} & g^* f_* \BQ_X \ar[d]^{\Ex^*_*} \\
\Sigma^\infty \mathrm{P}(X_T/T)(-1)[-2] \ar[r]_(0.7){\Theta_{f'}} & f'_* \BQ_{X_T}
}
\]
 commutes in $\DA(S)$.
\end{prop}
\begin{proof}
The first observation is that, using the natural transformation $g^*\tau_{\geq 0}\ra \tau_{\geq 0}g^*$, we can reduce to the same commutation for the full $f_*\Gm\otimes\BQ[1]$ instead of $P(X/S)$. 

In the rest of the proof, we need notations for the natural transformations
\[
(\alpha_f): f^*\Sus^1\stackrel{\sim}{\lra} \Sus^1 f^*
\]
\[
(\beta_f): f_*\Ev_1\stackrel{\sim}{\lra} \Ev_1 f_*
\]
and 
\[
(\gamma_f):f^*\Ev_1 \lra \Ev_1 f^*. 
\]
The natural isomorphisms $(\alpha)$ and $(\beta)$ are derived versions of isomorphisms at the level of model categories of spectra. The natural transformation $(\gamma)$ can be defined in two different ways, one using $(\alpha)$ and one using $(\beta)$; namely, as the two equal compositions
\[
f^*\Ev_1 \stackrel{\epsilon}{\ra} \Ev_1 \Sus^1 f^*\Ev_1 \stackrel{(\alpha^{-1}_f)}{\ra} \Ev_1 f^* \Sus^1\Ev_1 \stackrel{\eta}{\ra} \Ev_1 f^* \label{GA}
\]
and
\[
f^*\Ev_1\stackrel{\epsilon}{\ra} f^* \Ev_1 f_* f^* \stackrel{(\beta^{-1}_f)}{\ra} f^*f_*\Ev_1 f^* \stackrel{\eta}{\ra}  \Ev_1 f^*  \label{GB}
\]
Writing down the definition of the maps in the square, we see that we have to show the commutation of the outer square in the following diagram (when an arrow is obtained from another one by a clear functoriality, we omit the functor from the notation as well; for instance the first vertical arrow in the top left should be named $g^*\Sus^1 f_* w_S$).
\[
\xymatrix{
g^* \Sus^1 f_* \Gm[1] \ar[r]^{w_S} \ar[d]_{\sim}^{(\alpha_g)} & g^* \Sus^1 f_* \Ev_1 \BQ \ar[r]^{(\beta_f)}_{\sim} \ar[d]_{\sim}^{(\alpha_g)} & g^* \Sus^1 \Ev_1 f_* \BQ \ar[r]^{\eta} \ar[d]^{(\alpha_g)}_{\sim} & g^* f_*\BQ \ar[ddd]^{\Ex^*_*}\\
\Sus^1 g^* f_* \Gm[1] \ar[r]^{w_S} \ar[d]_{\sim}^{\Ex^*_*} & \Sus^1 g^* f_* \Ev_1\BQ \ar[r]^{(\beta_f)}_{\sim} \ar[d]^{\Ex^*_*}  \ar @{} [rdd] |{\mathbf{(*)}} & \Sus^1 g^* \Ev_1 f_*\BQ \ar[d]_{(\gamma_g)} & \\
\Sus^1 f'_* g'^*\Gm[1] \ar[r]^{w_S} \ar[d]^{R_{g'}}_{\sim} & \Sus^1 f'_* g'^* \Ev_1\BQ \ar[d]^{(\gamma_{g'})} & \Sus^1 \Ev_1 g^*f_*\BQ \ar[uur]_{\eta} \ar[d]_{\Ex^*_*} & \\
\Sus^1 f'_*\Gm[1] \ar[r]^{w_X} & \Sus^1 f'_* \Ev_1 \BQ  \ar[r]^{(\beta_{f'})}_{\sim} & \Sus^1 \Ev_1 f'_*\BQ \ar[r]^{\eta} & f'_*\BQ
}
\]
The commutation of the three squares in the top left corner and of the bottom right corner follows directly by naturality of various natural transformations. The bottom left square commutes by Proposition~\ref{prop:Ev_1}. The top right square commutes by the first description of $(\gamma)$.

It remains to show the commutation of $(*)$. By expanding the second description of $(\gamma)$, we see that we have to show the commutativity of the outer square in the following diagram.
\[
\xymatrix{
g^* f_* \Ev_1\BQ \ar[rrr]^{(\beta_f)} \ar[d]_{\Ex^*_*} \ar[rd]^{\epsilon_{g'}} & & & g^*\Ev_1 f_*\BQ \ar[ld]_{\epsilon_{g'}} \ar[d]_{\epsilon_g} \\
f'_* g'^* \Ev_1 \BQ \ar[d]_{\epsilon_{g'}} & g^* f_* \Ev_1 g'_{*}\BQ \ar[ld]_{\Ex^*_*} \ar[r]^{(\beta_f)} \ar[d]^{(\beta^{-1}_{g'})} & g^* \Ev_1 f_* (g')_*\BQ \ar[d]_{\sim} & g^*\Ev_* g_*g^* f_*\BQ \ar[ld]_{\Ex^*_*} \ar[d]^{(\beta^{-1}_g)} \\
f'_* g'^* \Ev_1 g'_* \BQ \ar[d]_{(\beta^{-1}_{g'})} & g^* f_*g'_*\Ev_1\BQ \ar[d]_{\sim} \ar[ld]_{\Ex^*_*} & g^*\Ev_1 g_* f'_*\BQ  \ar[d]_{(\beta^{-1}_g)} & g^* g_* \Ev_1 g^* f_*\BQ \ar[d]_{\eta_g} \ar[ld]_{\Ex^*_*} \\
f'_* g'^* g'_*\Ev_*\BQ \ar[d]_{\eta_{g'}}& g^* g_* f'_*\Ev_1\BQ   \ar[ld]_{\eta_g} \ar[r]^{(\beta_{f'})} & g^* g_*\Ev_1 f'_*\BQ \ar[rd]_{\eta_g} & \Ev_1 g^*f_*\BQ \ar[d]_{\Ex^*_*} \\
f'_*\Ev_1 \BQ \ar[rrr]_{(\beta_{f'})} & & & \Ev_1 f'_*\BQ
}
\]
The commutation of each of the subdiagrams follow from naturality properties of various natural transformations and from the definition of the exchange maps $\Ex^*_*$. This completes the proof.
\end{proof}

\section{Motivic Picard functor}
\label{sec:mot_picard}

We introduce and study the motivic Picard functor $\omega^1$, which is a (mixed motivic, relative) generalisation of the Picard variety of a smooth projective variety over a field. We also study in parallel the $0$-motivic analogue $\omega^0$, which was first introduced in \cite{Ayoub_Zucker}. 

\subsection{Definition and elementary properties}
\label{sec:picard_def}

\begin{defi}
  Let $n\geq 0$. The full embedding $\iota^n:\DA^n(S)\hookrightarrow \DA^\coh(S)$ preserves small sums, thus by Neeman's version of Brown representability for compactly generated triangulated categories  (see e.g. \cite[Theorem 8.3.3]{Neeman_book}), $\iota^n$ admits a right adjoint $\omega^n:\DA^\coh(S)\ra \DA^n(S)$. We also write $\omega^n$ for the functor $\DA^\coh(S)\ra \DA^\coh(S)$ obtained by postcomposing with $\iota^n$. We write $\delta^n:\omega^n\ra \id$ for the natural transformation induced by the counit. 
\end{defi}

\begin{remark}
The definition above can be extended to the whole of $\DA(S)$, but the resulting functors are not well-behaved; in particular, they do not respect compactness. Here is the simplest example of this phenomenon. Let $k$ be an algebraically closed field. It is easy to see that the category $\DA_{0,c}(k)$ is equivalent to the bounded derived category of the category of finite-dimensional $\BQ$-vector spaces. In particular homomorphisms groups in this category are finite-dimensional. On the other hand, $\DA(k)(\BQ_k, \BQ_k(1)[1])\simeq k^\times\otimes\BQ$ (Proposition~\ref{prop:mot_coh_1}) is not finite-dimensional in general. This shows that $\omega^0(\BQ(1))$ is not compact. 
\end{remark}

We start by giving some general formal properties of all the $\omega^n$ functors.

\begin{prop}\label{prop:omega_basics}
Let $S$ be a noetherian finite-dimensional scheme.
  \begin{enumerate}[label={\upshape(\roman*)}]
  \item \label{omega_idemp} Let $M\in \DA^n(S)$. Then we have an isomorphism $\delta^n(M):\omega^n(M)\simeq M$ and the natural transformation $\delta^n(\omega^n):\omega^n\circ \omega^n \ra \omega^n$ is invertible.
  \item \label{omega^*} Let $f:T\ra S$ be any morphism of schemes. There is a natural transformation $\alpha^n_f:f^*\omega^n\ra \omega^n f^*$ making the triangles
\[
\xymatrixcolsep{4pc}
\xymatrix{
f^*\omega^n \ar[rd]_{f^*(\delta^n)} \ar[r]^{\alpha^n_f} & \omega^n f^* \ar[d]^{\delta^n(f^*)} & \text{and} & \omega^n f^* \omega^n \ar[rd]_{(\omega^n f^*)(\delta^n)} \ar[r]^{\delta^n(f^*\omega^n)} & f^*\omega^n \ar[d]^{\alpha^n_f} \\ 
& f^* & & & \omega^n f^*
}
\]
commutative.
\item \label{omega_*} Let $f:T\ra S$ be any morphism of schemes. The natural transformation $\omega^n f_*(\delta^n)$ is invertible. Moreover there is a natural transformation $\beta^n_f:\omega^n f_*\ra f_*\omega^n$ such that
  \begin{enumerate}[label={\upshape\alph*)}]
  \item the following triangles
\[
\xymatrixcolsep{4pc}
\xymatrix{
\omega^n f_* \ar[rd]_{\delta^n(f_*)} \ar[r]^{\beta^n_f} & f_* \omega^n \ar[d]^{f_*(\delta^n)} & \text{and} & \omega^n f_* \omega^n \ar[rd]_{\delta^n(f_* \omega^n)} \ar[r]^{\omega^n(f_* \delta^n)} & \omega^n f_* \ar[d]^{\beta^n_f} \\ 
& f_* & & & f_*\omega^n
}
\]
are commutative,
  \item $\omega^n(\beta^n_f)$ is invertible for any $f$, and
  \item $\beta^n_f$ is invertible for $f$ finite.
  \end{enumerate}
\item \label{omega_!} Let $e:T\ra S$ be a quasi-finite morphism of schemes. There exists a natural transformation $\eta^n_e:e_!\omega^n\ra \omega^n e_!$ such that
  \begin{enumerate}[label={\upshape\alph*)}]
  \item the following triangles
\[
\xymatrixcolsep{4pc}
\xymatrix{
e_!\omega^n \ar[rd]_{e_!(\delta^n)} \ar[r]^{\eta^n_e} & \omega^n e_! \ar[d]^{\delta^n(e_!)} & \text{and} & \omega^n e_! \omega^n \ar[rd]_{(\omega^n e_!)(\delta^n)} \ar[r]^{\delta^n(e_!\omega^n)} & e_!\omega^n \ar[d]^{\eta^n_e} \\ 
& e_! & & & \omega^n e_!
}
\]
commute, and
\item when $e$ is finite, $\eta^n_e$ is invertible and coincides with $\beta_e^{-1}$ modulo the natural isomorphism $e_!\simeq e_*$. 
  \end{enumerate}
\item \label{omega^!} Let $e:T\ra S$ be a quasi-finite morphism. The natural transformation $\omega^n e^!(\delta^n)$ is invertible. Moreover, there is a natural transformation $\gamma^n_e:\omega^n e^!\ra e^!\omega^n$ such that
  \begin{enumerate}[label={\upshape\alph*)}]
  \item the following triangles
\[
\xymatrixcolsep{4pc}
\xymatrix{
\omega^n e^! \ar[rd]_{\delta^n(e^!)} \ar[r]^{\gamma^n_e} & e^! \omega^n \ar[d]^{e^!(\delta^n)} & \text{and} & \omega^n e^! \omega^n \ar[rd]_{\delta^n(e^! \omega^n)} \ar[r]^{\omega^n(e^! \delta^n)} & \omega^n e^! \ar[d]^{\gamma^n_e} \\ 
& e^! & & & e^!\omega^n
}
\]
are commutative,
  \item $\omega^n(\gamma^n_e)$ is invertible for any $e$ quasi-finite, and
  \item $\gamma^n_e$ is invertible and coincides with $(\alpha^{n}_{e})^{-1}$ for $e$ \'etale.
  \end{enumerate}
\item \label{omega_key} Let $j:U\ra S$ and $i:Z\ra S$ be complementary open and closed immersions. Let $M\in \DA^\coh(S)$ with $j^*M\in \DA^n(S)$. Then the morphism $i^*\omega^n M\ra \omega^n i^* M$ is invertible.
  \end{enumerate}
\end{prop}

\begin{remark}
  The formulation of Proposition~\ref{prop:omega_basics} follows closely the one of \cite[Proposition~2.16]{Ayoub_Zucker} about $\omega^0$. More precisely, it is a direct generalisation to all $\omega^n$ and to more general base schemes of all statements of loc. cit., except the assertions that $\alpha^0_f$ is invertible for $f$ smooth and that $\omega^0$ preserves compact objects. Unlike the others, these properties of $\omega^0$ are not formal. We study their generalisation to more general base schemes and higher $n$'s below.
\end{remark}

\begin{proof}
We can apply verbatim the proof of \cite[Proposition~2.16]{Ayoub_Zucker} up to the sentence ``To complete the proof (...)'' on page 319. Notice that the rest of the proof after that sentence establishes the non-formal assertions described in the previous remark, which are exactly the points we are not claiming.

More precisely, up to that sentence, the proof of loc. cit. uses only general properties of $\DA$, the definition of $\omega^0$ as right adjoint, and the following permanence properties of cohomological $0$-motives under the six operations.
\begin{itemize}
\item For all morphisms $f$, the functor $f^*$ preserves $\DA^0$.
\item For all finite morphisms $f$, the functor $f_*$ preserves $\DA^0$.
\item For all quasi-finite morphism $e$, the functor $e_!$ preserves $\DA^0$.
\end{itemize}
The generalisation of these properties to $\DA^n$ are established in the necessary generality in Proposition~\ref{prop:permanence_coh_n}.
\end{proof}

Here are other useful common properties of the $\omega^n$'s.

\begin{lemma}\label{lemm:omega_sums}
Let $S$ be a noetherian finite-dimensional scheme and $n\in\BN$. The functor $\omega^n:\DA^{\coh}(S)\ra \DA^n(S)$ commutes with small sums.  
\end{lemma}
\begin{proof}
  The inclusion functor $\DA^n(S)\ra \DA^\coh(S)$ sends compact objects to compact objects by Lemma~\ref{lem:subcats_comp}; hence by \cite[Lemme 2.1.28]{Ayoub_these_1}, its right adjoint $\omega^n$ commutes with infinite sums. 
\end{proof}

\begin{lemma}
  Let $h:S'\ra S$ be a finite purely inseparable morphism of schemes, and $n\in \BN$. The natural transformation $\alpha_h:h^*\omega^n\ra \omega^n h^*$ is an isomorphism.
\end{lemma}
\begin{proof}
This follows directly from the separation property of $\DA(-)$ and Corollary~\ref{cor:localisation_subcats}~\ref{rad_sub}.  
\end{proof}

We now come to the less formal properties of $\omega^0$.

\begin{prop}\label{prop:omega_0_props}
Let $S$ be a noetherian finite-dimensional scheme.
\begin{enumerate}[label={\upshape(\roman*)}]
\item \label{omega_0_calc} Let $f:X\ra S$ be a smooth proper morphism of schemes. Let $X\stackrel{f^\circ}{\lra}\pi_0(X/S)\stackrel{\pi_0(f)}{\lra}S$ be its Stein factorisation (so that $\pi_0(f)$ is finite \'etale). Then there is a natural isomorphism
\[
\pi_0(f)_*\BQ_{\pi_0(X/S)}\stackrel{\sim}{\longrightarrow} \omega^0(f_*\BQ_X).
\]
\item \label{omega_0_smooth} The functor $\omega^0$ preserves geometrically smooth objects. More precisely, it sends $\DA^{\coh}_{\gsm}(S)$ to $\DA^0_{\sgsm}(S)$ and $\DA^\coh_{\gsm,c}(S)$ to $\DA^0_{\sgsm,c}(S)$. Moreover, for any $M\in \DA^{\coh}_{\gsm}(S)$ and any morphism $f:T\ra S$, the natural morphism $\alpha^0_f(M):f^*\omega^0 M\ra \omega^0 f^*M $ is an isomorphism.
\item \label{omega_0_pullback} The morphism $\alpha^0_f$ is invertible for $f$ smooth. 
\item \label{omega_0_compact} The functor $\omega^0$ preserves compact objects. More precisely, it sends $\DA^{\coh}_c(S)$ to $\DA^0_c(S)$.
\end{enumerate}
\end{prop}

\begin{remark}
These results were proved in \cite[\S 2]{Ayoub_Zucker} under the assumption that $S$ is quasi-projective over a field $k$ and $f$ is projective; they were also generalized, in a slightly different terminology, to the case of $S$ separated of finite type over a field in \cite[\S 3.1-2]{Vaish_EM}.
\end{remark}

\begin{proof}
It is easy to see from the definition of geometrically smooth motives and the fact that $\pi_0$ commutes with base change that point \ref{omega_0_smooth} follows from \ref{omega_0_calc}. 
We now notice that the end of the proof of \cite[Proposition~2.16]{Ayoub_Zucker} (starting at ``To complete the proof (...)''), which deduces \ref{omega_0_pullback} and \ref{omega_0_compact} in the situation of loc. cit. from \cite[Proposition~2.11]{Ayoub_Zucker},  applies verbatim and reduce Statements \ref{omega_0_smooth}-\ref{omega_0_compact} to the sole Statement \ref{omega_0_calc}.

To prove Statement \ref{omega_0_calc}, it is enough by the Yoneda lemma to establish that for all $N\in \DA^0(S)$, the natural map $\pi_0(f)_*\BQ\ra f_*\BQ_X$ induces an isomorphism
\[
\DA(S)(N,\pi_0(f)_*\BQ)\stackrel{\sim}{\longrightarrow}\DA(S)(N,f_*\BQ_X).
\]
By Proposition~\ref{prop:hom_cohom_twists}, we have $\DA^0(S)=\DA_0(S)$. It is thus enough to show that for all $e:U\ra S$ \'etale and $n\in \BZ$, we have an isomorphism
\[
\DA(S)(e_\sharp \BQ_U[-n],\pi_0(f)_*\BQ)\stackrel{\sim}{\longrightarrow}\DA(S)(e_\sharp \BQ_U[-n],f_*\BQ_X).
\]
By the $(e_{\sharp},e^*)$ adjunction, proper base change, and the fact that $\pi_0$ commutes with smooth base change, we see that we can assume $e=\id$. We are thus left to prove that for all $n\in \BZ$, we have
\[
\DA(\pi_0(X/S))(\BQ,\BQ[n])\stackrel{\sim}{\longrightarrow} \DA(X)(\BQ,\BQ[n])
\]
where the morphism is induced by pullback by $f^{\circ}$. The morphism $f^\circ$ is smooth proper with geometrically connected fibres, so this follows from Proposition~\ref{prop:mot_coh_0} \ref{mot_coh_0_pullback}. 
\end{proof}

Here are some corollaries of Proposition~\ref{prop:omega_0_props}.

\begin{cor}\label{coro:omega_0}
Let $S$ be a noetherian finite-dimensional scheme. 
\begin{enumerate}[label={\upshape(\roman*)}]
\item \label{hom_coh_morphisms} Let $M$ be in $\DA_\homo(S)$ and $N$ be in $\DA^\coh(S)$. Then the morphism $\delta^0(N)$ induces an isomorphism
\[
\DA(S)(M,\omega^0 N)\stackrel[\sim]{\delta^0(N)_*}{\longrightarrow} \DA(S)(M,N).
\]
\item \label{hom_coh_intersect} We have $\DA_{\homo}(S)\cap \DA^{\coh}(S)=\DA^0(S)$.
\item \label{omega_0_twists} For all $N\in \DA^{\coh}(S)$ we have $\omega^0 (N(-1))\simeq 0$. 

\item \label{omega_1_twists} For all $N\in \DA^\coh(S)$ and $d\geq 1$, we have
\[
\omega^1(N(-d))\simeq \left\{\begin{array}{c} (\omega^0 N)(-1),\ d=1\\ 0,\ d\geq 2.\end{array}\right. .
\]
\end{enumerate}
\end{cor}

\begin{proof}
We first prove \ref{hom_coh_morphisms}. It is enough to show the isomorphism for a generator of $\DA_{\homo}(S)$, namely $M=g_\sharp \BQ_X[n]$ for $g:X\ra S$ a smooth morphism and $n\in\BZ$. By naturality of the adjunction which underlies $\delta^0$, we have a commutative square
\[
\xymatrixcolsep{4pc}
\xymatrix{
\DA(S)(g_\sharp \BQ_X[n],\omega^0 N) \ar[r]^{\delta^0(N)_*} \ar[d]_{\sim} & \DA(S)(g_\sharp \BQ_X[n],N) \ar[d]^{\sim} \\
\DA(X)(\BQ_X[n],g^*\omega^0 N)\ar[r]^{\delta^0(N)_*} & \DA(X)(\BQ_X[n],g^* N).
}
\]
The first commutative triangle in Proposition~\ref{prop:omega_basics} \ref{omega^*} shows that we have a commutative square
\[
\xymatrixcolsep{4pc}
\xymatrix{
\DA(X)(\BQ_X[n],g^*\omega^0 N)\ar[r]^{\delta^0(N)_*} \ar[d]_{\alpha_g(N)} & \DA(X)(\BQ_X[n],g^* N) \ar@{=}[d] \\
\DA(X)(\BQ_X[n],\omega^0 g^*N) \ar[r]_{\delta^0(g^*N)_*} & \DA(X)(\BQ_X[n],g^* N).
}
\]
Since $g$ is smooth, the left vertical map is an isomorphism by Proposition \ref{prop:omega_0_props} \ref{omega_0_pullback}; the bottom map is an isomorphism because $\BQ_X[n]$ is a cohomological $0$-motive. Putting this together with the previous commutative square concludes the proof of \ref{hom_coh_morphisms}.

Statement \ref{hom_coh_intersect} follows directly from \ref{hom_coh_morphisms} applied to the identity map of an object in $\DA^{\coh}(S)\cap \DA_\homo(S)$.

To prove Statement \ref{omega_0_twists}, we must show that for all $M\in \DA^0(S)$, we have $\DA(S)(M,N(-1))=0$. Since $\DA^0(S)=\DA_0(S)$ by Proposition~\ref{prop:hom_cohom_twists} and $\DA_\homo(S)$ is stable by positive twists by Proposition~\ref{prop:n_monoidal}~\ref{hom_tensor}, the motive $M(1)$ is homological. By \ref{hom_coh_morphisms}, this implies that $\DA(S)(M(1),N)\simeq \DA(S)(M(1),\omega^0 N)$. In other words, we can assume that both $M$ and $N$ are $0$-motives. The statement to be proven is triangulated and commutes with infinite sums in $M$, so that we can assume that $M$ is a generator of the form $e_\sharp \BQ_U[n]$ for $e:U\ra S$ an \'etale morphism and $n\in \BZ$. Since this is a compact object, we can similarly assume that $N$ is a generator of $\DA^0(S)$, of the form $f_*\BQ_V[m]$ for $f:V\ra S$ a finite morphism. We then have \[\DA(S)(M,N(-1))\simeq \DA(U\times_S V)(\BQ,\BQ(-1)[m-n]).\] This group vanishes by Proposition~\ref{prop:mot_coh_<}.

By \ref{omega_0_twists}, we only need to establish \ref{omega_1_twists} in the case $d=1$. The motive $\omega^0(N)(-1)$ is in $\DA^1(S)$ by Proposition~\ref{prop:n_monoidal} \ref{coh_n_tensor}. Hence by the Yoneda lemma, it is enough to show that for all $M\in \DA^1(S)$, the map $\delta^0(N)$ induces an isomorphism
\[
\DA(S)(M,(\omega^0 N)(-1))\stackrel[\sim]{\delta^0(N)_*}{\longrightarrow} \DA(S)(M,N(-1)).
\] 
By Proposition~\ref{prop:hom_cohom_twists}, we have $\DA^1(S)=\DA_1(S)(-1)$. Write $M=M'(-1)$ with $M'\in \DA_1(S)$. In particular, $M'$ is an homological motive. We have a commutative square
\[
\xymatrix{
\DA(S)(M,(\omega^0 N)(-1)) \ar[r]^{\delta^0(N)_*}\ar[d]_{\sim} & \DA(S)(M,N(-1))\ar[d]^{\sim} \\
\DA(S)(M',\omega^0 N)\ar[r]_{\delta^0(N)_*} & \DA(S)(M',N)
}
\]
The bottom map is an isomorphism by \ref{hom_coh_morphisms}, and this concludes the proof for $d=1$.
\end{proof}

We now compute $\omega^0$ for some motives attached to commutative group schemes. 

\begin{prop}\label{prop:omega_0_gr}
  \begin{enumerate}[label={\upshape(\roman*)}]
  \item \label{omega_0_ab_lat} Let $G$ be an abelian scheme or a lattice over $S$; then $\omega^0(\Sigma^\infty G_\BQ(-1))\simeq 0$.
  \item \label{omega_0_tor} Let $T$ be a torus over $S$. Let $X_*(T)$ be the cocharacter lattice of $T$. Then \\ $\Sigma^\infty T_\BQ(-1)[-1])\simeq \Sigma^\infty X_*(T)_\BQ$ is in $\DA_{0,c}(S)$.
  \item \label{omega_0_M1} Let $\BM\in \CM_1(S)$ and $W_{-2}\BM$ be its toric part. Then $\omega^0(\Sigma^\infty\BM(-1))\simeq \Sigma^\infty X_*(W_{-2}\BM)_\BQ$.
  \end{enumerate}
\end{prop}
\begin{proof}
First of all, we note that the objects to which we wish to apply $\omega^0$ are in $\DA^1(S)\subset \DA^{\coh}(S)$ by Corollary~\ref{cor:Deligne_da_1} and Proposition~\ref{prop:hom_cohom_twists}. 

We first prove \ref{omega_0_ab_lat}. We first treat the case of an abelian scheme $A$. By the conservativity of the family of pullbacks to points \cite[Proposition 3.23]{Ayoub_Etale}, it is enough to show that for any $s\in S$, the restriction $s^*\omega^0 M\simeq 0$. We know from \cite[Theorem~3.3]{AHPL} that the motive $\Sigma^\infty A_\BQ$ is geometrically smooth. By Proposition~\ref{prop:omega_0_props}~\ref{omega_0_smooth} and Proposition~\ref{prop:pullback_complex}, we see that $s^*\omega^0 M\simeq \omega^0 s^*M\simeq \omega^0\Sigma^\infty (A_s)_\BQ(-1)$. We are thus reduced to the case where $S$ is the spectrum of a field $k$. We have to show that, for every $0$-motive $N$ over $\Spec(k)$, we have
\[
\DA(k)(N,\Sigma^\infty \Jac(C)_\BQ(-1)[n])=0.
  \]
The category $\DA_0(k)$ is generated, as a localising subcategory, by motives of the form $g_\sharp \BQ_L$ with $g:\Spec(L)\rightarrow \Spec(k)$ with $L/k$ finite \'etale. By adjunction, we are then reduced to the case $N=\BQ_k[-n]$ for some $n\in\BZ$.
  
We write $A$ as direct factor of the Jacobian of a smooth projective geometrically connected curve $f:C\ra \Spec(k)$ \cite[Theorem~11]{Katz_SpaceFill}. By Proposition~\ref{prop:motive_curve_field} and relative purity, we have 
\[
\BQ(-1)[-2]\oplus \Sigma^\infty \Jac(C)_\BQ(-1)[-2] \oplus \BQ \simeq f_*\BQ_C.
\]
We have  $\DA(k)(\BQ_k,\BQ_k(-1)[n])=0$ for all $n$ (Proposition~\ref{prop:mot_coh_<}). By adjunction, we have \[\DA(k)(\BQ_k,f_*\BQ_C[n])\simeq \DA(C)(\BQ_C,\BQ_C[n])\] which is isomorphic to $\BQ$ for $n=0$ and $0$ otherwise (Proposition~\ref{prop:mot_coh_0}). Similarly, we have $\DA(k)(\BQ_k,\BQ_k[n])$ is isomorphic to $\BQ$ for $n=0$ and $0$ otherwise. Putting everything together, we deduce that $\DA(k)(\BQ_k,\Sigma^\infty \Jac(C)_\BQ(-1)[n])=0$ for all $n$ as required.

We now turn to the lattice case. Again by an adjunction argument, we immediately reduce to show that, for all $n\in\BZ$, we have
\[\DA(S)(\BQ_S,\Sigma^\infty L_\BQ(-1)[n])=0.\]
If $S$ is geometrically unibranch, using Lemma~\ref{lemm:permutation_torus}, write $L_{\BQ}$ as a direct factor $f_*\BQ$ for $f$ finite \'etale, and we are done by adjunction and Proposition~\ref{prop:mot_coh_<}.

Unfortunately, if the base is not geometrically unibranch, it is not clear that $M$ is geometrically smooth, and we cannot directly reduce to the field case. However, iterating the construction of the normalisation, it is easy to see that $S$ admits a proper hypercovering $\pi_\bullet:S_\bullet\ra S$ with normal terms. By cohomological $h$-descent for $\DA(-)$ and Proposition~\ref{prop:pullback_complex}, we get a spectral sequence
\[
E_1^{p,q}=\DA(S_p)(\BQ_{S_p},\Sigma^\infty (L_{S_{p},\BQ}(-1))[q])\Rightarrow \DA(S)(\BQ_S,\Sigma^\infty (L_{\BQ}(-1))[p+q]).
  \]
By the geometrically unibranch case, the $E_1$ page of the spectral sequence vanishes completely. This ensures the convergence and finishes the proof of \ref{omega_0_ab_lat}.

We prove \ref{omega_0_tor}. Let $T$ be a torus. We have $\Sigma^\infty T_\BQ (-1)\simeq \Sigma^\infty X_*(T)_\BQ$ by Corollary~\ref{cor:motive_torus}. The motive $\Sigma^\infty X_*(T)_\BQ$ lies in $\DA_0(S)$: this can be tested pointwise by Proposition~\ref{prop:punctual_car}, and over a field a lattice is a direct factor of the motive of a finite \'etale morphism by Lemma~\ref{lemm:permutation_torus}. This concludes the proof. 

Finally, \ref{omega_0_M1} follows immediately from the two previous points by the d\'evissage of a Deligne $1$-motive along its weight filtration.
\end{proof}

\begin{cor}\label{cor:omega_0_Picard}
Assume $S$ regular. Let $f:X\ra S$ be a smooth projective Pic-smooth morphism of schemes. Then there is an isomorphism
\[
\omega^0(\Sigma^\infty \mathrm{P}(X/S)_\BQ(-1)[-2])\simeq \pi_0(f)_*\BQ
\]
\end{cor}
\begin{proof}
First, by Corollary~\ref{cor:Picard_smooth}, Proposition~\ref{prop:smooth_complex_compact} and Proposition~\ref{prop:hom_cohom_twists}, $\Sigma^\infty \mathrm{P}(X/S)_\BQ(-1)[-2]$ is in $\DA^1(S)$, and it makes sense to apply $\omega^0$. More precisely, Corollary~\ref{cor:Picard_smooth} together with Proposition~\ref{prop:omega_0_gr} shows that there is an isomorphism
\[
\omega^0(\Sigma^\infty \mathrm{P}(X/S)_\BQ(-1)[-2])\simeq \Sigma^\infty X_*(\Res_{\pi_0(f)}\Gm)_\BQ.
\] 
The cocharacter lattice of the Weil restriction $\Res_{\pi_0(f)}\Gm$ is the permutation lattice associated to $\pi_0(f)$; hence, $\Sigma^\infty X_*(\Res_{\pi_0(f)}\Gm)_\BQ\simeq \pi_0(f)_*\BQ$ as required.
\end{proof}

\subsection{The functors $\omega^n$ over a perfect field}
\label{sec:omega_field}
In this short section, we explain how, for $S$ the spectrum of a perfect field $k$, the functors $\omega^0$ and $\omega^1$ are related to the functors $L\pi_0$ and $\LAlb$ studied in \cite{BVK} and \cite{Ayoub_Barbieri-Viale}.

We need to connect our setup with the categories of effective motives with transferts over $k$. First, we define for every $n\in\BN$ the category $\DM^{(\eff)}_{n,(c)}(k)$ in a similar way as as $\DA_{n,(c)}(k)$, replacing $\DA(k)$ with $\DM^{(\eff)}(k)$ and $f_\sharp \BQ_X$ with $M^{(\eff),\tr}_k(X)$ for $f:X\ra \Spec(k)$ smooth. We also define the category $\DM_{\homo,(c)}(k)$ (resp. $\DM^\coh_{(c)}(k)$ in a similar way as $\DA_{\homo,(c)}(k)$ (resp. $\DA^\coh_{(c)}(k)$).

By construction of $\DM(k)$, there is an adjunction
\[
\Sigma^\infty_{\tr}:\DM^{\eff}(k)\rightleftarrows \DM(k):\Omega^\infty_{\tr}.
\]
\begin{lemma}
\label{lemm:DA_DM_n}
Let $k$ be a field and $n\in \BN$. The adjoint pairs $\Sigma^\infty_\tr \dashv \Omega^\infty_{\tr}$ and $a_\tr\dashv o^\tr$ restrict to equivalences of categories
\[
\DA_{n,(c)}(k) \stackrel[o^\tr]{a_\tr}{\rightleftarrows}\DM_{n,(c)}(k)\stackrel[\Sigma^\infty]{\Omega^\infty}{\rightleftarrows}   \DM^{\eff}_{n,(c)}(k),
\]
\[
\DA_{\homo,(c)}(k) \stackrel[o^\tr]{a_\tr}{\rightleftarrows} \DM_{\homo,(c)}(k)\stackrel[\Sigma^\infty]{\Omega^\infty}{\rightleftarrows} \DM^{\eff}_{(c)}(k),
\]
\[
\DA^\coh_{(c)}(k) \stackrel[o^\tr]{a_\tr}{\rightleftarrows} \DM^{\coh}_{(c)}(k),
\]
\[
\text{ and }\DA^n_{(c)}(k) \stackrel[o^\tr]{a_\tr}{\rightleftarrows} \DM^n_{(c)}(k).
\]
\end{lemma}
\begin{proof}
  The argument is essentially the same for the four series of equivalences; we only give the details for the first one. Recall that $a_\tr:\DM(S)\stackrel{\sim}{\ra} \DA(S)$ is an equivalence of categories for all $S$ geometrically unibranch \cite[Corollary~16.2.22]{Cisinski_Deglise_BluePreprint}, hence in particular for all $\Spec(k)$. Being equivalences of categories, $a_\tr$ and $o^\tr$ both commute with small sums and preserve compact objects. By construction of $a_{\tr}$, we have $a_\tr f_\sharp\simeq f_\sharp a_\tr$ for $f$ smooth. This implies that $a_\tr\dashv o^\tr$ restricts to an equivalence of categories between $\DA_{n,(c)}$ and $\DM_{n,(c)}$.

Let $\phi:\Spec(k^{\perf})\ra \Spec(k)$ be a perfect closure of $k$. The base change functors $\phi^{*}:\DA(k)\ra \DA(k^{\perf})$, $\phi^{*}:\DM(k)\ra \DM(k^{\perf})$ and $\phi^{*}:\DM^{\eff}(k)\ra \DM^{\eff}(k^{\perf})$ are all equivalences of categories: in the first case, this is the separation property of $\DA$, in the second case, this follows from the case of $\DA$ and the comparison isomorphism recalled above, and in the third case, we apply \cite[Corollary 4.13]{Suslin_imperfect}. Moreover, these equivalences commute with the functors in the two adjunctions of the statement (because they are equivalences of categories and commute with the left adjoints). We can thus assume that $k$ is a perfect field.

By Voevodsky's cancellation theorem~\cite{Voevodsky_Cancellation} which applies because $k$ is perfect, the functor $\Sigma_\tr^\infty:\DM^{\eff}\ra \DM(k)$ is fully faithful, so that it restricts to an equivalence of categories $\Sigma_{\tr}^\infty:\DM^{\eff}\ra \DM_{\hom}(k):\Omega_{\tr}^{\infty}$. We have $\Sigma_\tr^\infty f_\sharp\simeq f_\sharp \Sigma_\tr^\infty $ for $f:X\ra \Spec(k)$ smooth; this shows that $\DM_n(k)$ lies in the essential image of $\DM^{\eff}_n(k)$. Again, the equivalence of categories $\Sigma_{\tr}^\infty:\DM^{\eff}\ra \DM_{\hom}(k):\Omega_{\tr}^{\infty}$ preserves compact objects in both directions, hence we get an equivalence of categories between $\DM_{n,(c)}(k)$ and $\DM^{\eff}_{n,(c)}(k)$. This completes the proof.
\end{proof}

By \cite[Theorem 2.4.1]{Ayoub_Barbieri-Viale} specialized to the case of $\BQ$-coefficients, we have a functor
\[
L\pi_0:\DM^{\eff}(k) \ra \DM^{\eff}_0(k)
\] 
(respectively
\[
\LAlb:\DM^{\eff}(k) \ra \DM^{\eff}_1(k))
\] 
which is a left adjoint to the inclusion $\DM^{\eff}_0(k)\ra \DM^{\eff}(k)$ (resp. $\DM^{\eff}_1(k)\ra \DM^{\eff}(k)$) and restricts by \cite[Proposition 2.3.3]{Ayoub_Barbieri-Viale} (resp. \cite[Proposition 2.4.7]{Ayoub_Barbieri-Viale}) to a functor
\[
L\pi_0:\DM^{\eff}_c(k) \ra \DM^{\eff}_{0,c}(k)
\]
(resp.
\[
\LAlb:\DM^{\eff}_{c}(k)\ra \DM^{\eff}_{1,c}(k)).
\]
To be more precise, our notation differs from loc. cit. in the following way. The functor $L\pi_0$ (resp. $\LAlb$) in loc. cit. has as target category $D(\HI_{\leq 0}(k))$ (resp. $D(\HI_{\leq 1}(k))$), the derived category of the abelian category $\Sh(\Spec(k)_\et,\BQ)$ (resp. $\HI_{\leq 1}(k)$  of $1$-motivic sheaves \cite[Definition 1.1.20]{Ayoub_Barbieri-Viale}), which is equivalent by \cite[Lemma 2.3.1]{Ayoub_Barbieri-Viale} (resp.\cite[Theorem 2.4.1.(i)]{Ayoub_Barbieri-Viale}) to $\DM^{\eff}_0(k)$ (resp. $\DM^{\eff}_1(k)$), and the functor we call $L\pi_0$ (resp. $\LAlb$) is obtained by composing the functor of loc. cit. with this equivalence.

\begin{prop}\label{prop:omega_field}
Let $k$ be a perfect field. The functors $\omega^0$ and $\omega^1$ restrict to compact objects. Moreover, when restricting to compact objects, we have isomorphisms of functors
\[
\omega^0\simeq\BD_k o^\tr\Sigma^\infty_\tr L\pi_0 \Omega^\infty_\tr a_\tr \BD_k:\DA^{\coh}_c(k)\ra \DA^0_c(k)    
\]
and
\[
\omega^1\simeq\BD_k o^\tr\Sigma^\infty_\tr \LAlb \Omega^\infty_\tr a_\tr \BD_k:\DA^{\coh}_c(k)\ra \DA^1_c(k).    
\]
\end{prop}

\begin{proof}
By Proposition~\ref{prop:subcats_field}, the duality functor $\BD_k$ restricts to give anti-equivalences of categories \newline $\DA^\coh_c(k)^\op\simeq \DA_{\homo,c}(k)$ and $\DA^n_c(k)^\op\simeq \DA_{n,c}(k)$ for any $n\in\BN$. By Lemma~\ref{lemm:DA_DM_n}, this implies that the inclusion $\DA_{0,c}(k)\ra \DA_c(k)$ (resp. $\DA_{1,c}(k)\ra \DA_c(k)$) admits as right adjoint the composition
\[
\BD_k o^\tr\Sigma^\infty_\tr L\pi_0 \Omega^\infty_\tr a_\tr \BD_k:\DA^{\coh}_c(k)\ra \DA^0_c(k)    
\]
(resp.
\[
\BD_k o^\tr\Sigma^\infty_\tr \LAlb \Omega^\infty_\tr a_\tr \BD_k:\DA^{\coh}_c(k)\ra \DA^1_c(k)).    
\]
In the case $n=0$, we already know that the functor $\omega^0$ restricts to compact objects by Proposition~\ref{prop:omega_0_props} \ref{omega_0_compact}, so that this right adjoint and the restriction of $\omega^0$ (which we also denote by $\omega^0$) coincide. In the case $n=1$, we argue as follows. Write temporarily $\tilde{\omega}^1:=\BD_k o^\tr\Sigma^\infty_\tr \LAlb \Omega^\infty_\tr a_\tr \BD_k$. Let $M\in \DA^{\coh}_c(k)$. There is a morphism $\tilde{\omega}^1 M\ra M$ in $\DA^{\coh}_c(k)$, which by the adjunction property of $\omega^1$ factors through a morphism $\tilde{\omega}^1 M\ra \omega^1 M$ in $\DA^1(k)$. The category $\DA^1(k)$ is compactly generated, hence to show that this morphism is an isomorphism, it is enough to show that for every $N\in \DA^1_c(k)$, the induced morphism $\DA^1_c(k)(N,\tilde{\omega}^1 M)\ra \DA^1_c(k)(N,\omega^1 M)$ is an isomorphism. This follows from the adjunction properties of both functors. We deduce that $\omega^1$ restricts to compact objects, and that this restriction is related to $\LAlb$ by the formula above.
\end{proof}

Finally, we use another result of \cite{Ayoub_Barbieri-Viale} to show that the $\omega^n$'s for $n\geq 2$ are not well-behaved, at least over ``large'' fields.

\begin{prop}\label{prop:omega_>_nc}
Let $n\geq 2$ and $k$ be an algebraically closed field of infinite transcendence degree over $\BQ$, e.g. $k=\BC$. Then $\omega^n:\DA^\coh(k)\ra \DA^n(k)$ does not preserve compact objects.
\end{prop}
\begin{proof}
We prove this by contradiction. Assume that $\omega^n$ preserves compact objects and write again $\omega^n:\DA^\coh_c(k)\ra \DA^n_c(k)$ for the restriction. By Proposition~\ref{prop:subcats_field}, the duality functor $\BD_k$ restricts to anti-equivalences of categories $\DA^\coh_c(k)^\op\simeq \DA_{\homo,c}(k)$ and $\DA^n_c(k)^\op\simeq \DA_{n,c}(k)$. This implies that the composition $\BD_k \circ (\omega^n)^\op \circ \BD_k:\DA_{\homo,c}(k)\ra \DA_{n,c}(k)$ provides a left adjoint to the inclusion $\DA_{n,c}(k)\ra \DA_{\homo,c}(k)$.

By Lemma~\ref{lemm:DA_DM_n}, this also provides a left adjoint to $\DM^{\eff}_{n,c}(k)\ra \DM^{\eff}_c(k)$, which does not exists by \cite[\S 2.5]{Ayoub_Barbieri-Viale} (note that the assumption there is the existence of a left adjoint to $\DM^{\eff}_n(k)\ra \DM^{\eff}(k)$ but the proof only uses the existence of the adjoint on compact objects). This contradiction finishes the proof.
\end{proof}

\subsection{Computation and finiteness of the motivic Picard functor}
\label{sec:picard_sm_pr}

We can now compute $\omega^1$ in an important special case.

\begin{theo}\label{theo:omega_1_Pic_smooth} 
Let $f:X\ra S$ be a smooth projective Pic-smooth morphism, with $S$ regular excellent. The morphism $\Theta_f:\Sigma^\infty \mathrm{P}(X/S)(-1)[-2] \ra f_*\BQ_X$ of Section~\ref{sec:smooth_picard} induces an isomorphism
\[
\omega^1 f_* \BQ_X \simeq \Sigma^\infty \mathrm{P}(X/S)(-1)[-2].
\]
In particular, the motive $\omega^1f_*\BQ_X$ is compact.

\end{theo}

\begin{proof}
Assume $S$ is a regular scheme. First of all, the motive $\Sigma^\infty \mathrm{P}(X/S)$ lies in $\DA_{1,c}(S)$ by Corollary~\ref{cor:Picard_smooth}. By Proposition~\ref{prop:hom_cohom_twists}, this implies that $\Sigma^\infty (\mathrm{P}(X/S)\otimes \BQ)(-1)[-2]$ lies in $\DA^1_c(S)$. We deduce that $\Theta_f$ induces a morphism $\Sigma^\infty (\mathrm{P}(X/S)\otimes \BQ)(-1)[-2] \ra \omega^1 f_*\BQ_X$. It remains to show that this is an isomorphism. We have also observed that $(\Sigma^\infty \mathrm{P}(X/S))(-1)[-2]$ is compact, so this will also establish the last claim.

We first treat the case when $S$ is the spectrum of a perfect field $k$. The proof proceeds by reduction to a computation in the category of effective Voevodsky motives $\DM^{\eff}(k)$. By Proposition~\ref{prop:hom_cohom_twists}, the category $\DA^1(k)$ is compactly generated by motives of the form $g_\sharp\BQ_C(-1)$ for a smooth curve $g:C\ra k$. We thus have to show that for all such $g$ and all $n\in\BZ$, the map 
\[
\DA(k)(g_\sharp \BQ_C(-1)[-n],\Sigma^\infty \mathrm{P}(X/k)(-1)[-2])\stackrel{(\Theta_{f})_*}{\lra} \DA(k)(g_\sharp \BQ_C(-1)[-n], f_*\BQ_X)
\]
induced by $\Theta_f$ is an isomorphism (this turns out to hold for any smooth $C$, not only for curves, as the argument below shows).

First, using that $a^\tr \Theta_f=\Theta_f^\tr$ modulo a certain isomorphism (Proposition~\ref{prop:theta_tr}), this is equivalent to the morphism
\[
\DM(k)(g_\sharp \BQ^\tr_C,\Sigma^\infty_\tr \mathrm{P}^\tr(X/k)[n-2])\stackrel{(\Theta^\tr_{f})_*}{\lra} \DM(k)(g_\sharp \BQ^\tr_C, f_*\BQ^\tr_X(1)[n])
\]
being an isomorphism. By Lemma~\ref{lem:eff_proj_form}, we have a commutative diagram.
\[
\xymatrix{
\DM(k)(g_\sharp \BQ^\tr_C,\Sigma^\infty f_*\BQ^\tr_X(1)[n]) \ar[r] & \DM(k)(g_\sharp \BQ^\tr_C, f_*\BQ^\tr_X(1)[n]) \\
\DM^{\eff}(k)(g_\sharp \BQ^\tr_C,f_*\BQ^\tr_X(1)[n]) \ar[u]^{\Sigma^\infty}_{\sim} & \DM(k)(g_\sharp \BQ^\tr_C\otimes f_\sharp\BQ_X,\BQ^\tr_k(1)[n]) \ar[u]^{\Lambda}_{\sim} \\
\DM^{\eff}(k)(g_\sharp \BQ^\tr_C\otimes f_{\sharp} \BQ^\tr_X,\BQ^\tr_k(1)[n]) \ar[u]^{\Lambda^{\eff}}_{\sim} \ar@{=}[r] & \DM^{\eff}(g_\sharp \BQ^\tr_C\otimes f_{\sharp} \BQ^\tr_X,\BQ^\tr_k(1)[n]) \ar[u]^{\Sigma^\infty}_{\sim}.
}
\]
Using the alternative description of $\Theta_f^\tr$ from Proposition~\ref{prop:theta_tr} with $u_X^{\eff,\tr}$ and the fact that $u_X^{\eff,\tr}$ is an isomorphism, we see that we have to show that the top morphism in the previous diagram is an isomorphism. 

The maps induced by $\Sigma^\infty $ are isomorphisms because of the Cancellation theorem \cite{Voevodsky_Cancellation} (this is where we use the hypothesis $k$ perfect), hence the top morphism is an isomorphism. This concludes the proof in the case $k$ perfect.

We now turn to the case of $S=\Spec(k)$ with $k$ an arbitrary field. Let $k^\perf$ be a perfect closure of $k$ and $h:\Spec(k^\perf)\ra \Spec(k)$ be the canonical morphism. Write $T=\Spec(k^\perf)$. By Proposition~\ref{prop:base_change_theta} and applying $\omega^1$, we have a commutative diagram
\[
\xymatrix{
h^*\Sigma^\infty\mathrm{P}(X/S)(-1)[-2]\ar[r] \ar[d]_{V_h\circ R_h} & \omega^1(h^* f_* \BQ_X) \ar[d]^{\omega^1(\Ex^*_*)}\\
\Sigma^\infty \mathrm{P}(X_T/T)(-1)[-2] \ar[r]_(0.6){\Theta_{f'}} & \omega^1(f'_* \BQ_{X_T}).
}
\]
By Corollary~\ref{cor:Picard_bc}, the morphism $V_h$ is an isomorphism. Since $R_h$ is an isomorphism, we see that the left vertical map in the diagram is an isomorphism. By proper base change, the right vertical map is an isomorphism. We are reduced to prove that $\Theta_{f'}$ is an isomorphism, which follows from the perfect field case.

We now consider the general case. We can assume that $S$ is connected, and hence integral. The statement of the theorem is equivalent to the following claim: for all $M\in \DA^1(S)$, the map $\Theta_f$ induces an isomorphism
\[
\DA(S)(M,\Sigma^\infty \mathrm{P}(X/S)(-1)[-2])\stackrel{\sim}{\lra} \DA(S)(M, f_*\BQ_X).
\]
We first make a series of reformulations. By Proposition~\ref{prop:hom_cohom_twists} and the definition of $\DA_1(S)$, the category $\DA^1(S)$ is compactly generated by objects of the form $g_\sharp\BQ_C(-1)$ for a smooth curve $g:C\ra S$. We can thus state the theorem as follows: for every smooth curve $g:C\ra S$ and all $n\in\BZ$, the map 
\[
\DA(S)(g_\sharp \BQ_C(-1)[-n],\Sigma^\infty \mathrm{P}(X/S)(-1)[-2])\stackrel{\Theta_{f*}}{\lra} \DA(S)(g_\sharp \BQ_C(-1)[-n], f_*\BQ_X)
\]  
is an isomorphism. By adjunction, this is equivalent to the statement that the map  
\[
\DA(C)(\BQ_C(-1)[-n],g^*\Sigma^\infty \mathrm{P}(X/S)(-1)[-2])\xrightarrow{(g^*\Theta_f)_*} \DA(C)(\BQ_C(-1)[-n], g^*f_*\BQ_X)
\]  
is an isomorphism. Let $f':X_C\ra C$ be the pullback of $f$ along $g$. The morphism $f'$ is Pic-smooth by Lemma~\ref{lem:bc_Pic_smooth} and $C$ is regular. By Proposition~\ref{prop:base_change_theta} and the fact that $v^{\sm}_g$ is an isomorphism because $g$ is smooth (Lemma~\ref{lem:base_change_Pic}), the morphism $(g^*\Theta_f)_*$ above is an isomorphism if and only the morphism 
\[
\DA(C)(\BQ_C,\Sigma^\infty \mathrm{P}(X_C/C)[n-2])\stackrel{\Theta_{f'}}{\lra}\DA(C)(\BQ_C,f'_*\BQ_{X_C}(1)[n])
\]
is an isomorphism. By adjunction, the right-hand side is isomorphic to the motivic cohomology group $H^{n,1}_\CM(X_{C})$. For concision, let us introduce the ad hoc notation
\[
\mathrm{HP}^{n-2}(X/S):=\DA(S)(\BQ_S,(\Sigma^\infty \mathrm{P}(X/S))[n-2]).
\]
To conclude, since $f'$ still satisfies all the hypotheses of the theorem, we are reduced to prove that the map
\[
\mathrm{HP}^{n-2}(X/S) \ra H^{n,1}_{\CM}(X)
  \]
induced by $\Theta_{f}$ is an isomorphism for all $n\in\BZ$ and all $f:X\ra S $ as in the statement of the theorem.
  
As $S$ and $X$ are regular, we know from Proposition~\ref{prop:mot_coh_1} how to compute $H^{n,1}_{\CM}(X)$: it is zero for $n\neq 1,2$, and we have explicit isomorphisms relating it to $\CO^\times(X)_\BQ$ if $n=1$ (resp. $\Pic(X)_\BQ$ if $n=2$). The idea of the rest of the proof is to apply a localisation argument similar to the proof of Proposition~\ref{prop:mot_coh_1}. 
Let $j:U\ra S$ be a non-empty open set and $i:Z\ra S$ its reduced closed complement. Then by colocalisation, we get a morphism of long exact sequences
\[
\xymatrix@C=1em{
\ldots \ar[r] & \DA(Z)(\BQ_Z, i^!\Sigma^\infty \mathrm{P}(X/S)[n-2]) \ar[r] \ar[d]  & \mathrm{HP}^{n-2}(X/S) \ar[r] \ar[d] & \mathrm{HP}^{n-2}(X_U/U) \ar[r] \ar[d] & \ldots \\
\ldots \ar[r] & \DA(Z)(\BQ_Z,i^!(f_*\BQ_X(1)[n])) \ar[r] & H^{n,1}_\CM(X) \ar[r] & H^{n,1}_\CM(X_U) \ar[r] & \ldots
}
\] 
Since every closed subscheme of $S$ is excellent and reduced, hence has open non-empty regular locus, we can choose a stratification $Z=Z_0\supset Z_1\supset \ldots \supset Z_{d}=\emptyset$ in such a way that for all $k$, the scheme $Z_k\setminus Z_{k+1}$ is regular of codimension $d_k$ in $S$ and in such a way that $(Z\setminus Z_1)$ contains all points of codimension $1$ of $Z$ in $S$ (so that $d_k\geq 2$ for $k\geq 1$). Let $i_k: Z_k\setminus Z_{k+1}\ra S$ be the corresponding regular locally closed immersion.

By Corollary~\ref{cor:Picard_smooth}, the motive $\Sigma^\infty \mathrm{P}(X/S)(-1)$ is in $\DA^{1}_{\gsm}(S)$. By absolute purity in the form of Proposition~\ref{prop:sm_abs_purity}, for any $k$, we have $i_k^!\Sigma^\infty \mathrm{P}(X/S)\simeq i^*_k\mathrm{P}(X/S)(-d_k)[-2d_k]$. In particular, by Corollary~\ref{coro:omega_0} \ref{omega_0_twists}, we have $\omega^0(i_k^!\Sigma^\infty \mathrm{P}(X/S))\simeq 0$ for $k\geq 1$. This shows that by inductively applying absolute purity and colocalisation, we get a morphism of long exact sequences
\[
\xymatrix@C=1em{
\ldots \ar[r] & \DA(Z)(\BQ_{Z\setminus Z_1}, i_0^*\Sigma^\infty \mathrm{P}(X/S)(-1)[n-4]) \ar[r] \ar[d]  & \mathrm{HP}^{n-2}(X/S) \ar[r] \ar[d] & \mathrm{HP}^{n-2}(X_U/U) \ar[r] \ar[d] & \ldots \\
\ldots \ar[r] & H^{n-2,0}(X_{Z\setminus Z_1}) \ar[r] & H^{n,1}_\CM(X) \ar[r] & H^{n,1}_\CM(X_U) \ar[r] & \ldots
}
\] 
Write $Z'=Z\setminus Z_1$. The motive $i_0^*\Sigma^\infty \mathrm{P}(X/S)(-1)[n-4]$ lies in $\DA^{\coh}(Z')$, so that 
\[
\DA(Z')(\BQ_{Z'}, i_0^*\Sigma^\infty \mathrm{P}(X/S)(-1)[n-4])\simeq \DA(Z')(\BQ_{Z'},\omega^0(i_0^*\Sigma^\infty \mathrm{P}(X/S)(-1)[n-4])).
\]
Using Corollary~\ref{cor:Picard_smooth}, we apply Proposition~\ref{prop:omega_0_props} \ref{omega_0_smooth} to get an isomorphism  
\[\omega^0 (i^*\Sigma^\infty \mathrm{P}(X/S)(-1)[n-4])\simeq i^* \omega^0(\Sigma^\infty \mathrm{P}(X/S)(-1)[n-4]).\] 
By Corollary~\ref{cor:omega_0_Picard}, we then have 
\[\omega^0(\Sigma^\infty \mathrm{P}(X/S)(-1)[n-4])\simeq \pi_0(f)_*\BQ[n-2].\]
We deduce that 
\[
\DA(Z')(\BQ_{Z'},i^!\Sigma^\infty \mathrm{P}(X/S)[n-2])\simeq H^{n-2,0}(\pi_0(X_{Z'}/{Z'})).
\]
We rewrite this into the previous commutative diagram to get
\[
\xymatrix{
\ldots \ar[r] & H^{n-2,0}(\pi_0(X_{Z'}/Z')) \ar[r] \ar[d]_{(\pi_0)^*}  & \mathrm{HP}^{n-2}(X/S) \ar[r] \ar[d] & \mathrm{HP}^{n-2}(X_U/U) \ar[r] \ar[d] & \ldots \\
\ldots \ar[r] & H^{n-2,0}(X_{Z'}) \ar[r] & H^{n,1}_\CM(X) \ar[r] & H^{n,1}_\CM(X_U) \ar[r] & \ldots
}
\]
By Proposition~\ref{prop:mot_coh_0}, since $X_{Z'}$ and $\pi_0(X_{Z'}/Z')$ are both regular and have the same set of connected components, the map $(\pi_0)^*$ is an isomorphism for all $n$, and the groups $H^{n-2,0}(X_{Z'})$ vanish for $n\neq 2$. As a consequence, we see that the pullback map $\mathrm{HP}^{n-2}(S)\ra \mathrm{HP}^{n-2}(U)$ is an isomorphism for $n\neq 1,2$, and there is a commutative diagram with exact horizontal lines
\[
\xymatrix@C=1em{
0 \ar[r] & \mathrm{HP}^{-1}(X/S) \ar[r] \ar[d] & \mathrm{HP}^{-1}(X_U/U) \ar[r] \ar[d] & \BQ^{\pi_0(X_{Z'})} \ar[r] \ar@{=}[d] & \mathrm{HP}^{0}(X/S) \ar[r] \ar[d] & \mathrm{HP}^{0}(X_U/U) \ar[r] \ar[d] &  0\\
0 \ar[r] & H^{1,1}_\CM(X_S) \ar[r]  & H^{1,1}_\CM(X_U) \ar[r]  & \BQ^{\pi_0(X_{Z'})} \ar[r] & H^{2,1}_\CM(X_S) \ar[r]  & H^{2,1}_\CM(X_U) \ar[r]  &  0.\\
}
\]
We then pass to the limit over all non-empty closed subsets $Z$ and use continuity for $\DA$. For $n\neq 1,2$, we obtain that $\mathrm{HP}^{n-2}(S)\ra \mathrm{HP}^{n-2}(\kappa(S))$ is an isomorphism. By the field case and Proposition \ref{prop:mot_coh_1}, we have $\mathrm{HP}^{n-2}(\kappa(S))\simeq 0\simeq H^{n,1}_\CM(\kappa(s))\simeq H^{n,1}_{\CM}(S)$ for such $n$'s, and this concludes the proof for $n\neq 1,2$. For $n=1,2$, we obtain a commutative diagram
\[
\xymatrix@C=1em{
0 \ar[r] & \mathrm{HP}^{-1}(X/S) \ar[r] \ar[d] & \mathrm{HP}^{-1}(X_{\kappa(S)}/\kappa(S)) \ar[r] \ar[d] & \Pi \ar[r] \ar@{=}[d] & \mathrm{HP}^{0}(X/S) \ar[r] \ar[d] & \mathrm{HP}^{0}(X_{\kappa(S)}/\kappa(S)) \ar[r] \ar[d] &  0\\
0 \ar[r] & H^{1,1}_\CM(X_S) \ar[r]  & H^{1,1}_\CM(X_{\kappa(S)}) \ar[r]  & \Pi \ar[r] & H^{2,1}_\CM(X_S) \ar[r]  & H^{2,1}_\CM(X_{\kappa(S)}) \ar[r]  &  0\\
}
\]
with $\Pi$ a group which can be expressed in terms of the sheaf $\pi_0(X/S)$, but which we do not need to know explicitly. Applying the already established result in the field case (for the function field $\kappa(S)$), we see that the second and fifth vertical maps are isomorphisms. By the five lemma, we conclude that the first and fourth one are as well. This finishes the proof.
\end{proof}

The following lemma, which relates Grothendieck operations in the effective and non-effective settings, was used in the proof above.

\begin{lemma}\label{lem:eff_proj_form}
  Let $S$ be a field, $f:X\ra \Spec(k)$ a smooth $k$-variety, and $M,N\in \DM^{(\eff)}(k)$. There exists natural isomorphisms
\[
  \Lambda^{(\eff)}_{M,N}:\DM^{(\eff)}(k)(M\otimes f_{\sharp}\BQ,N)\simeq\DM^{(\eff)}(k)(M,f_* f^* N)
\]
such, for $M,N\in \DM^{\eff}(k)$, the diagram
\[
\xymatrix{
\DM(k)(\Sigma^\infty_{\tr}M,\Sigma^\infty_{\tr} f_*f^*N ) \ar[r] & \DM(k)(\Sigma^\infty_{\tr}M, f_*f^*\Sigma^\infty_{\tr}N) \\
\DM^{\eff}(k)(M,f_*f^* N) \ar[u]^{\Sigma^\infty}_{\sim} & \DM(k)(\Sigma^\infty_{\tr} M\otimes f_\sharp\BQ^{\tr}_X,\Sigma^\infty_{\tr} N) \ar[u]^{\Lambda}_{\sim} \\
\DM^{\eff}(k)(M\otimes f_{\sharp} \BQ^\tr_X,N) \ar[u]^{\Lambda^{\eff}}_{\sim} \ar@{=}[r] & \DM^{\eff}(k)(M\otimes f_{\sharp} \BQ^\tr_X,N) \ar[u]^{\Sigma^\infty}_{\sim}
}
\]
commutes.
\end{lemma}
\begin{proof}
Let $A,B\in \DM^{(\eff)}(k)$. We have smooth projection formula isomorphisms
  \[
\mathrm{sp}^{(\eff)}:f_\sharp (f^*A\otimes B)\stackrel{\sim}{\ra} A\otimes f_\sharp B.
\]
Moreover, there are natural isomorphisms $\Sigma^\infty_{\tr}f_{\sharp}\simeq f_{\sharp}\Sigma^\infty_{\tr}$ and $\Sigma^\infty_{\tr} f^*\simeq f^*\Sigma^\infty_{\tr}$, and modulo these natural isomorphisms, we have $\Sigma^\infty_{\tr} \mathrm{sp}^{\eff}=\mathrm{sp}$.

We now define $\Lambda$ as
\[
\DM^{(\eff)}(M\otimes f_{\sharp}\BQ,N)\stackrel{(\mathrm{sp}^{(\eff)})^{-1}}{\simeq} \DM^{(\eff)}(k)(f_\sharp f^*M,N)\simeq \DM^{(\eff)}(M,f_* f^*N)
  \]
where the second map is given by the two adjunctions $(f_\sharp,f^*)$ and $(f^*,f_*)$.

The commutation of the diagram then follows from $\Sigma^\infty \mathrm{sp}^{(\eff)}=\mathrm{sp}$.
\end{proof}

\begin{remark}
In view of the non-canonical decomposition of $P(X/S)$ from Corollary~\ref{cor:Picard_smooth}, Theorem \ref{theo:omega_1_Pic_smooth} can be interpreted as a $1$-motivic analogue of Deligne's decomposition theorem for smooth projective morphisms.  
\end{remark}

\begin{remark}
In the special case of $S=\Spec(k)$ with $k$ a perfect field, the theorem is closely related to computations of $\LAlb$ and $\RPic$ from \cite[\S 9]{BVK}. Let us sketch this connection. Let $f:X\ra \Spec(k)$ be a smooth projective variety. Then $X$ is automatically Pic-smooth, and the morphism 
\[
 (\Sigma^\infty \mathrm{P}(X/k))(-1)[-2] \ra \omega^1 f_* \BQ_X.
\]
induced by $\Theta_f$ is an isomorphism. By Proposition~\ref{prop:omega_field}, we have the following isomorphisms.

\begin{eqnarray*}
\omega^1 f_*\BQ_X & \simeq &  \BD_k o^\tr\Sigma^\infty_\tr \LAlb \Omega^\infty_\tr a_\tr \BD_k f_*\BQ_X \\
& \simeq & o^\tr \BD^\tr_k \Sigma^\infty_{\tr} \LAlb M^{\eff,\tr}_k(X) 
\end{eqnarray*}
where we have used the same arguments as in Section~\ref{sec:omega_field} to pass from $\DA$ to $\DM^{\eff}$. Moreover, by Lemma~\ref{lemm:duality_eff} below, we can write

\begin{eqnarray*}
 o^\tr \BD^\tr_k \Sigma^\infty_{\tr} \LAlb M^{\eff,\tr}_k(X)  & \simeq & o^\tr (\Sigma^\infty \Homint^{\eff}(\LAlb M^{\eff,\tr}_k(X), \BQ(1)))(-1) \\
& \simeq & o^\tr (\Sigma^\infty_\tr \RPic(X)) (-1)
\end{eqnarray*}
where $\RPic(X)$ is the motive introduced in \cite[Definition 8.3.1]{BVK} and we have used the duality between $\LAlb(X)$ and $\RPic(X)$ in \cite[\S 4.5]{BVK}. At this point, we have an isomorphism
\[
o^\tr (\Sigma^\infty_\tr \RPic(X))(-1)\simeq (\Sigma^\infty \mathrm{P}(X/k))(-1)[-2].
\]
We now apply $a^\tr$, use the isomorphism $a^\tr\Sigma^\infty\simeq \Sigma^\infty_\tr a^\tr$, Corollary~\ref{cor:Picard_tr}, and the cancellation theorem \cite{Voevodsky_Cancellation}: this yields an isomorphism
\[
\RPic(X)\simeq \mathrm{P}^\tr(X/k)[-2].
\]
We are now in position to connect with the results of \cite{BVK}: modulo this isomorphism, the distinguished triangles of Corollary~\ref{cor:Picard_tr} for $\mathrm{P}^\tr(X/k)$ give an alternative proof of \cite[Corollary 9.6.1]{BVK} in the special case where $X$ is smooth projective and we have $\BQ$-coefficients.
\end{remark}

\begin{lemma}
\label{lemm:duality_eff}
Let $k$ be a perfect field. We have for $M\in \DM^{\eff}_{n,c}(k)$ a natural isomorphism
\[
\BD^\tr_k \Sigma^\infty_\tr M\simeq (\Sigma^\infty_{\tr} \Homint^\eff(M,\BQ(n)))(-n).
\]
\end{lemma}
\begin{proof}
Let $N,M\in \DM^{\eff}_c(k)$. By adjunction, monoidality of $\Sigma^\infty_\tr$ and the cancellation theorem \cite{Voevodsky_Cancellation}, there is a sequence of natural isomorphisms
\begin{eqnarray*}
\DM^{\eff}(k)(N,\Omega^\infty_\tr \BD^\tr_k (\Sigma^\infty_\tr M(-n))) & \simeq & \DM(k)(\Sigma^\infty_\tr N,\BD^\tr_k(\Sigma^\infty_\tr M(-n))) \\
& \simeq & \DM(k)(\Sigma^\infty_\tr N\otimes \Sigma^\infty_\tr (M)(-n), \BQ) \\
& \simeq & \DM(k)(\Sigma^\infty_\tr (N\otimes M),\BQ(n)) \\
& \simeq & \DM^\eff(k)(N\otimes M,\BQ(n)) \\
& \simeq & \DM^\eff(k)(N,\Homint^\eff(M,\BQ(n)))
\end{eqnarray*}
which provides, by the Yoneda lemma, a natural isomorphism $\Omega^\infty_\tr \BD^\tr_k(M(-n))\simeq \Homint^{\eff}(M,\BQ(n))$. We apply $\Sigma^\infty_\tr$ to get an isomorphism $\Sigma^\infty_\tr\Omega^\infty_\tr \BD^\tr_k(M(-n))\simeq \Sigma^\infty\Homint^{\eff}(M,\BQ(n))$. 

Moreover, the motive $\BD^\tr_k(\Sigma^\infty_\tr M(-n))$ lies in $\DM_{\homo,c}(k)$ by Proposition~\ref{prop:hom_cohom_twists}, Lemma~\ref{lemm:DA_DM_n} and Proposition~\ref{prop:subcats_field}. Because of the cancellation theorem \cite{Voevodsky_Cancellation}, the counit $\Sigma^\infty_\tr \Omega^\infty_\tr \ra \id$ is an isomorphism on $\DM_{\homo}(k)$; hence $\Sigma^\infty_\tr \Omega^\tr \BD^\tr(\Sigma^\infty_\tr M(-n))\simeq \BD^\tr(\Sigma^\infty_\tr M(-n))\simeq (\BD^\tr \Sigma^\infty_\tr M)(n)$. Combining this with the previous paragraph completes the proof.
\end{proof}

In the special case of a relative curve, because the N\'eron-Severi rank is constant, we can remove the regularity hypothesis on the base. This yields a general computation of the motive of a smooth projective curve. Recall that for a smooth morphism $f:X\ra S$, we write $M_{S}(X)$ for the homological motive $f_{\sharp}\BQ_{X}$  (this notation is sometimes convenient because it does not refer to $f$).

\begin{cor}\label{coro:computation_curve}
  Let $f:C\ra S$ be a smooth projective curve.
  \begin{enumerate}[label={\upshape(\roman*)}]
  \item   The morphism 
\[
\Theta_f:\Sigma^\infty \mathrm{P}(C/S)(-1)[-2] \ra f_*\BQ_C
\]
is an isomorphism, and induces an isomorphism
\[
\Sigma^\infty \mathrm{P}(C/S)\simeq M_S(C).
\]
\item If $S$ is regular, we then have (non-canonical) isomorphisms
\[
  f_\sharp \BQ_C \simeq M_S(\pi_0(C/S))\oplus \Sigma^\infty \Jac(C/S)\oplus M_S(\pi_0(C/S))(1)[2]
\]
and
\[
  f_* \BQ_C \simeq M_S(\pi_0(C/S))\oplus \Sigma^\infty \Jac(C/S)(-1)[-2]\oplus M_S(\pi_0(C/S))(-1)[-2].
\]
\item If $f$ has geometrically connected fibres and a section $\sigma:S\ra C$, we have canonical isomorphisms
\[
  f_\sharp \BQ_C \simeq \BQ_S\oplus \Sigma^\infty \Jac(C/S)\oplus \BQ_S(1)[2]
\]
and
\[
    f_* \BQ_C \simeq \BQ_S\oplus \Sigma^\infty \Jac(C/S)(-1)[-2]\oplus \BQ_S(-1)[-2].
\]
  \end{enumerate}

\end{cor}
\begin{proof}
Let us show that $\Theta_f$ is an isomorphism. By~\cite[Proposition 3.24]{Ayoub_Etale}, it is enough to show that $s^*\Theta_f$ is an isomorphism for any $s\in S$. By Proposition \ref{prop:base_change_theta}  and Proposition \ref{prop:picard_curve}, we are then reduced to the case when $S$ is the spectrum of a field. The fact that $\Theta_f$ is then an isomorphism is a special case of Theorem \ref{theo:omega_1_Pic_smooth}. The claims in (i) and (ii) then follow from Corollary \ref{cor:Picard_smooth} and Proposition \ref{prop:picard_curve}.
  Let us assume further that $f$ has geometrically connected fibres and a section $\sigma:S\ra C$. The section $\sigma$ yields sections of the lattices $\BQ(\pi_0(C/S))$ and $\sNS^{\sm}_{C/S}\simeq \BQ(\pi_0(C/S))$ (Proposition~\ref{prop:picard_curve}), which can be used to produce splittings of the distinguished triangles computing $\Sigma^\infty P(C/S)(-1)[-2]$ in Proposition~\ref{cor:Picard_smooth}. From this and (i), we get the decompositions in (iii).
\end{proof}

As an application of the computation, we can now prove a fundamental finiteness result for $\omega^1$.

\begin{theo}\label{thm:finiteness_omega1}
Let $S$ be a noetherian finite-dimensional excellent scheme. Assume that $S$ admits resolution of singularities by alterations. Then the functor $\omega^1:\DA^\coh(S)\ra \DA^1(S)$ preserves compact objects.
\end{theo}
\begin{proof} We follow the argument of \cite[Proposition~2.14 (vii)]{Ayoub_Zucker} for the case of $\omega^0$, with minor changes.

By Corollary~\ref{cor:localisation_subcats} \ref{rad_sub} we can assume that $S$ is reduced. We prove the result by noetherian induction on $S$. Let $M$ be in $\DA^{\coh}_c(S)$. Since $M$ is compact and cohomological, Lemma~\ref{lem:subcats_comp}, Proposition~\ref{prop:subcats_field} and continuity implies that there exists a dense open set $V\subset S$ and a finite family $\{f_i\}_{i=1}^n$ of smooth projective morphisms $f_i:X_i\ra V$ such that $M_V$ lies in the triangulated subcategory generated by the motives $f_{i*}\BQ_{X_i}$. By Proposition~\ref{prop:generic_pic_smooth}, there exists an everywhere dense open subset $U\subset V$ such that $f_i\times_S U$ is Pic-smooth for every $i$. We can moreover assume that $U$ is regular. Write $j:U\ra S$ for the open immersion and $i:Z\ra S$ for the complementary reduced closed immersion. By Proposition~\ref{prop:permanence_coh}, because of the hypothesis of resolution of singularities by alterations for $S$, the colocalisation triangle 
\[
i_*i^! M \ra M \ra j_* j^* M\rap
\]
lies in $\DA^{\coh}(S)$. We apply $\omega^1$ and use Proposition~\ref{prop:omega_basics} \ref{omega_*} to obtain a distinguished triangle
\[
i_* \omega^1(i^! M)\ra \omega^{1} M \ra \omega^1(j_* j^* M)\rap.
\]
By induction, we know that $\omega^1(i^!M)$ is compact, so it is enough to show that $\omega^1(j_* j^* M)$ is as well. By Proposition~\ref{prop:omega_basics} \ref{omega_*}, we have an isomorphism $\omega^1(j_* j^* M)\simeq \omega^1(j_* \omega^1 j^* M)$. Put $N=j_*\omega^1(j^*M)$; we have to show that $\omega^1(N)$ is compact. The motive $j^* M$ lies in the triangulated subcategory generated by the motives $(f_i\times_S U)_*\BQ$ with $f_i\times_S U$ smooth projective Pic-smooth and $U$ regular; hence by Theorem~\ref{theo:omega_1_Pic_smooth}, $\omega^1(j^*M)$ is compact. This implies that $N$ is compact, with $j^*N\in \DA^1(U)$. In particular, we have $j_! j^*N\in \DA^1_c(S)$. Thus applying $\omega^1$ to the localisation triangle for $N$ and using Proposition~\ref{prop:omega_basics} \ref{omega_*} yield a distinguished triangle
\[
j_! j^* N\ra \omega^1 N\ra i_* \omega^1 i^* N\rap.
\]
By Proposition~\ref{prop:omega_basics} \ref{omega_key}, we have $i^*\omega^1(N)\simeq \omega^1(i^*N)$, which is compact by induction. This completes the proof that $\omega^1 N$ is compact, and the proof of the theorem.
\end{proof}

\section{Motivic $t$-structures}
\label{sec:t_structure}

We introduce the motivic $t$-structures on $\DA_1(S)$ and $\DA^1(S)$ and study how Deligne $1$-motives relates to its heart.

\subsection{Conservativity of realisations of $1$-motives}
\label{sec:conservativity}

As we have explained in the introduction, in our approach to the motivic $t$-structure for relative $1$-motives, the conservativity of realisation functors is a necessary first step to ensure uniqueness. Recall from \cite{Ayoub_Betti} that for $k$ a field of characteristic $0$ with a fixed complex embedding $\sigma:k\ra \mathbb{C}$ and $S$ scheme of finite type over $k$, there is a covariant Betti realisation functor
\[
R_{B,\sigma}:\DA(S)\ra D(S^{\an},\BQ)
\]
with target the derived category of sheaves of $\BQ$-vector spaces on the complex analytic space $S^{\an}$.

 Similarly, we fix a prime $\ell$, and let $S$ be a $\BZ[\frac{1}{\ell}]$-scheme. Let $D_c(S,\BQ_\ell)$ be subcategory of complexes with constructible cohomology in the derived category of $\BQ_\ell$-sheaves $S$ in the sense of Ekedahl \cite{Eke}.
By \cite[Section 9]{Ayoub_Etale}, there is a covariant $\ell$-adic realisation functor 
\[ R_\ell: \DA_c(S)\to D_c(S,\BQ_\ell).\] 
% Unfortunately, the ``unbounded'' $\ell$-adic realisation with source $\DA(S)$, presumably with a natural target category defined using the pro-\'etale topology of \cite{Bhatt_Scholze}, has not been constructed yet.

\begin{prop}\label{prop:real_conservative}
With the notations and hypotheses above, the functors $R_{B,\sigma}$ and $R_\ell$, restricted to either of $\DA_{0,c}(S)$, $\DA_{1,c}(S)$ or $\DA^{1}_{c}(S)$ are conservative.
\end{prop}

\begin{proof}
Since $\DA_{0,c}(S)\subset \DA_{1,c}(S)$ and $\DA^1_c(S)=\DA_{1,c}(S)(-1)$, it is enough to treat the case of $\DA_{1,c}(k)$. Artin's comparison theorem between Betti and $\ell$-adic cohomology, including in the relative setting \cite[Expos\'e XVI.4, Th\'eor\`eme 4.1]{SGA4_3}, implies that, for any $M\in\DA_{c}(S)$ for $S$ of finite type over $k$, any embedding $\sigma:k\hookrightarrow \BC$ and any prime number $\ell$, we have $R_{B,\sigma}(M)=0\Leftrightarrow R_{\ell}(M)=0$. It is thus enough to treat the $\ell$-adic case.

If $k$ is a perfect field of characteristic $p\neq \ell$, we have a $t$-exact equivalence of triangulated categories $\Sigma^\infty: D^b_c(\CM_1(k))\simeq\DA_{1,c}(k)$ by \cite{Orgogozo}. By Lemma~\ref{lem:t_conservative} below, we only have to check that the induced functor $R^\heartsuit_\ell$ from $\CM_1(k)$ to either $\BQ$ or $\BQ_l$-vector spaces is conservative. Using the weight filtration on Deligne $1$-motives, it is enough to show that if $\CM$ is a pure  object in $\CM_1(k)$ with trivial realisation, it is itself $0$. This follows from the computation of the realisation of such a motive in \cite[5.2]{AHPL}.

The $\ell$-adic realisation commutes with pullbacks along finite type morphisms \cite[Theoreme 9.7]{Ayoub_Etale}, hence by continuity for $\DA$ and for the constructible $\ell$-adic derived category it commutes with pullbacks along morphisms which can be written as cofiltered limits of morphisms with affine transition morphisms. In particular, $R_{\ell}$ commutes with pullbacks along the inclusion of a point in a scheme, and with pullbacks along arbitrary field extensions. Combining the two, we deduce that $R_{\ell}$ commutes with pullbacks along any morphism $i_{\bar{s}}:\bar{s}\ra S$ with $\bar{s}$ the spectrum of an algebraically closed field.

Let $M\in \DA_{1,c}(S)$. Assume that $R_\ell(M)=0$. By the previous paragraph, for any $i_{\bar{s}}:\bar{s}\ra S$ geometric point, we have $R_\ell(i_{\bar{s}}^* M)=0$. By the perfect field case above, we have $i_{\bar{s}}^*M=0$. By \cite[Lemma A.6]{AHPL}, the family of such pullback functors is conservative, and we conclude that $M=0$. This concludes the proof of conservativity.

\end{proof}

\begin{lemma}\label{lem:t_conservative}
Let $F:\CT\ra \CT'$ be a $t$-exact functor between triangulated categories equipped with $t$-structures. Assume that the $t$-structure on $\CT$ is bounded and that the induced functor $F^\heartsuit:\CT^\heartsuit\ra \CT'^\heartsuit$ is conservative. Then $F$ is conservative.
\end{lemma}
\begin{proof}
Let $M\in \CT$. Assume that $F(M)=0$. Then $H_n F(M)=F(H_n M)=0$ for all $n\in\BZ$. Since $F^\heartsuit$ is conservative, we have $H_n M=0$ for all $n\in\BZ$. Since the $t$-structure on $\CT$ is bounded, we deduce that $M=0$.
\end{proof}

\subsection{Construction of the t-structures}
\label{sec:generators}
 
We fix a (noetherian, finite-dimensional) base scheme $S$ for the rest of this section. We want to define $t$-structures by generators and relations. This is possible in the context of compactly generated triangulated categories.

\begin{prop}{\cite[Lemme 2.1.69, Proposition~2.1.70]{Ayoub_these_1}}
Let $\CT$ be a compactly generated triangulated category and $\CG$ be a family of compact objects in $\CT$. Define $\CT_{\geq 0}=\llangle\CG \rrangle_+$ and $\CT_{<0}$ as the right orthogonal of $\CG[\BN]$, i.e., the full subcategory of all objects $N$ with 
\[\forall n\in \BN,\ \forall G\in \CG,\ \Hom(G,N[-n])=0.\] 
Then $(\CT,\ \CT_{\geq 0},\ \CT_{<0})$ is a $t$-structure on $\CT$, which we denote by $t(\CG)$ and call the $t$-structure generated by $\CG$ on $\CT$. 
\end{prop}

We can now introduce our candidate generating families. The definition uses Deligne $1$-motives over a base: for definitions and notations, we refer to the first section of Appendix \ref{sec:app_deligne}.

\begin{defi}
\label{def:generators}
We define classes of objects in $\DA(S)$ as follows. We put
% \[
% \CCG_S=\left\{M_S(C),M^\bot_S(C)[-1]\ |\ \text{ $C/S$ smooth curve}\right\} ,
% \]
% \[
% \PG_S=\left\{e_\sharp M_S(C), e_\sharp M^\bot_S(C)[-1] | \text{ $e:U\ra S$ \'etale and $C/U$ smooth projective curve}\right\},
% \]
\[
\JG_S=\left\{ e_\sharp\Sigma^\infty (K\otimes\BQ)|\ e:U\ra S\text{ \'etale},K=\begin{array}{c} \text{ one of } \BZ,\ \Gm[-1]\text{, or }\Jac(C/U)[-1]\text{ for } \\ \text{$C/U$ smooth projective curve} \\ \text{ with geometrically connected fibres}\\ \text{ and a section }U\ra C\end{array}\right\},
\]
and
\[
\DG_S=\left\{e_\sharp\Sigma^\infty(\BM)|\ e:U\ra S\text{ \'etale },\ \BM\in \CM_1(U)\right\}.
\]
We call objects $\JG_S$ (resp. $\DG_S$) \emph{Jacobian generators} (resp. \emph{Deligne generators}).
\end{defi}

By construction, we have $\JG_{S}\subset \DG_{S}$. Jacobian generators are useful because we understand better relative motives of curves than of abelian schemes.

\begin{lemma}
\label{lem:gen_ops}
\begin{enumerate}[label={\upshape(\roman*)}]
\item \label{gen_pullbacks} Let $f:T\ra S$ be a morphism of schemes. Then we have $f^*\JG_S\subset \JG_T$ and $f^*\DG_S\subset \DG_T$.
\item \label{gen_!} Let $e:T\ra S$ be an \'etale morphism. Then $e_\sharp\JG_T\subset \JG_S$ and $e_\sharp \DG_T\subset \DG_S$. 
\end{enumerate}
\end{lemma}
\begin{proof}
Point \ref{gen_pullbacks} follows from the $\Ex^*_\sharp$ isomorphism and Corollary~\ref{coro:pullback_complex_DA}. Point \ref{gen_!} follows directly from the definition.
\end{proof}

We now come to a more difficult stability property.

\begin{prop}
\label{prop:generators_closed_imm}
Assume $S$ to be excellent and let $i:Z\ra S$ be a closed immersion. Then
\[
i_*\langle \JG_Z\rangle_{(+)}\subset \langle \JG_S\rangle_{(+)}.
\]
\end{prop}
\begin{proof}
Let $r:Z_{\red}\ra Z$ be the canonical closed immersion. Localisation implies that $\id\simeq r_* r^*$. Since $r^*$ preserves $\JG$ by Lemma~\ref{lem:gen_ops}, we see that it is enough to show the property for $i\circ r$. We can thus assume $Z$ reduced.

We proceed by induction on the dimension of $Z$. If $\dim(Z)=0$, because $Z$ is reduced, it is a disjoint union of closed points of $S$. Then $i_*$ is canonically the direct sum of the corresponding push-forwards, so we can assume that $Z$ is a single closed point $s\in S$. 

There are three different types of generators in $\JG_s$. Fix $e:V\ra s$ an \'etale morphism. Since $s$ is a point, $e$ is actually finite \'etale. By Lemma \ref{lem:deformation_theory} \ref{etale}, there exists an open neighbourhood $s\in U\stackrel{c}{\hookrightarrow} S$ and an \'etale morphism $\tilde{e}:\widetilde{V}\ra U$ extending $e$. We form the commutative diagram
\[
\xymatrix{
\widetilde{V}^\circ \ar[r]^{\tilde{\jmath}} \ar[d]_{\tilde{e}^\circ} & \widetilde{V} \ar[d]_{\tilde{e}} & V \ar[d]_{e} \ar[l]_{\tilde{\imath}} \\
U\setminus s \ar[r]^{\bar{\jmath}} & U & s \ar[l]_{\bar{\imath}}
}
\]
with cartesian squares.

We first consider the case of a generator $e_{\sharp}\BQ$. By localisation, we have a distinguished triangle
\[
\bar{\jmath}_!\bar{\jmath}^*\tilde{e}_\sharp \BQ \ra \tilde{e}_\sharp\BQ \ra \bar{\imath}_*\bar{\imath}^*\tilde{e}_\sharp\BQ \rap
\]
to which we apply $c_\sharp$ and then rewrite as
\[
(c \bar{\jmath})_\sharp\tilde{e}^\circ_\sharp \BQ\ra c_\sharp\tilde{e}_\sharp \BQ \ra i_* e_\sharp \BQ\rap.
\]
The motives $(c \bar{\jmath})_\sharp\tilde{e}^\circ_\sharp \BQ$ and $c_\sharp\tilde{e}_\sharp \BQ$ are in $\JG_S$, so this triangle shows that $i_* e_\sharp \BQ$ lies in $\langle \JG_S \rangle_+$.

The case of a generator of the form $e_{\sharp}\Sigma^\infty\Gm\otimes\BQ[-1]\simeq e_\sharp \BQ(1)$ (cf. Proposition~\ref{prop:Gm_Q1}) follows from essentially the same proof, twisting by $\BQ(1)$.

We now do the case of a generator of the form $e_{\sharp}\Sigma^\infty \Jac(C/V)_\BQ[-1]$ with $f:C\ra V$ a smooth projective curve with geometrically connected fibres and a distinguished section $\sigma:V\ra C$. We have an isomorphism
\[
i_* e_\sharp \Sigma^\infty \Jac(C/V)_\BQ\simeq (ie)_!  \Sigma^\infty \Jac(C/V)_\BQ\simeq (c\tilde{e}\tilde{\imath})_! \Sigma^\infty \Jac(C/V)_\BQ\simeq c_{\sharp}\tilde{e}_\sharp \tilde{\imath}_*  \Sigma^\infty \Jac(C/V)_\BQ
  \]
  which reduces us to show that $\tilde{\imath}_*\Sigma^\infty \Jac(C/V)$ lies in $\langle \JG_{\widetilde{V}} \rangle_+$. Since $V$ has finitely many points, a simple argument shows that, up to restricting $\widetilde{V}$ we can assume that $V$ and $\widetilde{V}$ are connected, with $V$ consisting of a single point $v$.
  
% We have $e_\sharp \Sigma^\infty\Jac(C/U)_\BQ\simeq e_* \Sigma^\infty\Jac(C/U)_\BQ\simeq \Sigma^\infty \Res_e \Jac(C/U)_\BQ$ by Lemma~\ref{lemm:Weil_pushforward}; the Weil restriction of the Jacobian of $C$ along $U/s$ is the Jacobian of $C$ considered as a smooth projective curve over $s$, hence we can assume $e=\id$ in what follows. Note that this destroys the geometric connectedness assumption on fibres, but this does not matter because of the following observation: for $e:U\ra S$ \'etale and $C\rightarrow U$ smooth projective curve, with non-necessarily connected fibres, we have by Lemma \ref{lemm:Weil_pushforward}
% \[e_\sharp \Sigma^\infty \Jac(C/U)\otimes\BQ[-1]  \simeq  (e\circ \pi_0(C/U))_{\sharp} \Sigma^\infty \Jac(C/\pi_0(C/U))\otimes\BQ[-1].\]

We use standard results from the deformation theory of curves, summarized in Lemma \ref{lem:deformation_theory} \ref{curve}. The outcome is that we have a pointed \'etale neighbourhood  $(d:\widetilde{V}'\ra \widetilde{V},v)$ of $(\widetilde{V},v)$ and a smooth projective curve $\tilde{f}:\breve{C}\ra \widetilde{V}'$ with geometrically connected fibers which extends $C$, together with a section $\breve{\sigma}$ (which extends $\sigma$). By the arguments of the previous paragraph, we can replace $\widetilde{V}$ by $\widetilde{V}'$ and assume that $\breve{C}$ and $\breve{\sigma}$ are defined over $\widetilde{V}$.

We form the following diagram of schemes with cartesian squares
\[
\xymatrix{
\breve{C}^0 \ar[r]^{\breve{\jmath}} \ar[d]_{\tilde{f}^\circ} & \breve{C} \ar[d]_{\tilde{f}} & C  \ar[l]_{\tilde{\imath}} \ar[d]  \\
\widetilde{V}^\circ \ar[r]^{\bar{\jmath}} & \widetilde{V}  & v \ar[l]_{\bar{\imath}}   \\
}
\]
We have a localisation triangle
\[
\bar{\jmath}_! \bar{\jmath}^*\Sigma^\infty \Jac(\breve{C}/\widetilde{V})_\BQ \ra \Sigma^\infty \Jac(\breve{C}/\widetilde{V})_\BQ\ra \bar{\imath}_! \bar{\imath}^*\Sigma^\infty \Jac(\breve{C}/\widetilde{V})_\BQ\rap  
\]
which we rewrite using Corollary~\ref{coro:pullback_complex_DA} to obtain
\[
\bar{\jmath}_{\sharp} \Sigma^{\infty}\Jac(\breve{C}^\circ/\widetilde{V}^{\circ})_\BQ \ra \Sigma^\infty\Jac(\breve{C}/\widetilde{V})_\BQ\ra i_*\Sigma^\infty\Jac(C/v)_\BQ\rap
\]
The first two terms of this complex are in $\JG_{\widetilde{V}}$, so this shows that $i_*\Jac(C/v)_\BQ$ is in $\langle \JG_{\widetilde{V}} \rangle_+$. This concludes the proof in the case $\dim(Z)=0$.

We now come to the induction step. Let $M\in \JG_Z$. Write for the moment $M=e_\sharp \Sigma^\infty G\otimes\BQ$ with $G$ one of the three possible types and $e:U\ra S$ an \'etale morphism.

Let $k:W\ra Z$ be a dense open irreducible subset such that $e_W$ is finite \'etale. Let $l:T\ra Z$ be the complementary reduced closed immersion; let further $k':W'\ra S$ be an open immersion with $W'\cap Z=W$ and $l':T'\ra S$ be the complementary reduced closed immersion. Write $m:W\ra W'$ and $n:T\ra T'$ for the induced closed immersions. 

We have a localisation triangle for $k,l$ to which we apply $i_!$ and get
\[
i_! k_!k^*M\ra i_*M \ra i_*l_* l^*M\rap 
\]
which can be rewritten as
\[
k'_!m_! k^*M\ra i_*M \ra (l'\circ n)_* l^*M\rap.
\]
By Lemma \ref{lem:gen_ops} \ref{gen_pullbacks} , we have $k^*M\in \JG_W$ and $l^*M\in  \JG_Z $. We have $\dim(T)<\dim(Z)$ so that by induction the third term of this triangle is in $\langle \JG_S\rangle_+$. Moreover $k'_!$ preserves $\langle \JG\rangle_+$ by Lemma \ref{lem:gen_ops} \ref{gen_pullbacks}. Together, this means that to show that $i_*M$ is in $\langle\JG_S \rangle_+$, we need only show that $m_! k^*M$ is in $\langle \JG_{W'}\rangle_+$. We are thus reduced to the case where $Z$ is irreducible (with generic point $\eta$) and $e$ is a finite \'etale morphism.

The rest of the induction step consists of applying the same type of spreading out and deformation arguments we used in the $\dim(Z)=0$ case to $G_\eta$. Since the three cases are similar and the case of $G=\Jac(C/S)$ with $f:C\ra S$ smooth projective curve is the most complicated, we only spell out that one. Using the same argument as for $\dim(Z)=0$ based on Lemma \ref{lem:deformation_theory} \ref{etale}, we can essentially assume $e=\id$ and $V=S$, which we do here for simplicity of notation.

By Lemma \ref{lem:deformation_theory} \ref{curve}, which applies to the non-closed point $\eta\in S$ as well, we can find a pointed \'etale neighbourhood $(e:W\ra S,x\ra \eta)$ of $(S,\eta)$ and a smooth projective curve $\tilde{f}:\breve{C}\ra W$ (with geometrically connected fibres and a section) which extends $C_\eta$.

Let $V=\overline{\{x\}}\subset W$ be the closure of $x$. By spreading-out, there exists an open neighbourhood $V^\circ\subset V$ of $x$ and a dense open subset $Z^\circ\subset Z$ such that $\tilde{f}$ induces an isomorphism $V^\circ\simeq Z^\circ$ (since it is an isomorphism above $\eta$). By localisation and the induction hypothesis, we can assume that $Z^\circ=Z$. We now have a smooth projective curve above an open set of $S$ (with geometrically connected fibres and a section)  which extends $f$, and we can then conclude by localisation as in the end of the proof of the $\dim(Z)=0$ case. This finishes the proof.
\end{proof}

\begin{lemma}\label{lem:deformation_theory}
Let $S$ be a scheme and $s\in S$.
  \begin{enumerate}[label={\upshape(\roman*)}]
  \item \label{etale}
Let $e:V\ra s$ be a \'etale morphism. There exists an open neighbourhood $s\in U\stackrel{c}{\hookrightarrow} S$ and an \'etale morphism $\tilde{e}:\widetilde{V}\ra U$ extending $e$.
  \item \label{curve} Assume $S$ is moreover excellent. Let $f:C\ra s$ be a smooth projective geometrically connected curve and a section $\sigma:s\ra S$. There exists a pointed \'etale neighbourhood $(c:W\ra S,s)$ of $(S,s)$ and a smooth projective curve $\tilde{f}:\breve{C}\ra W$ with geometrically connected fibers which extends $C$, together with a section $\breve{\sigma}:S\ra \breve{C}$ which extends $\sigma$.
  \end{enumerate}
\end{lemma}
\begin{proof}
Let us prove Statement \ref{etale}. The scheme $V$ is a disjoint union of spectra of separable field extensions of $\kappa(s)$, and to prove the statement it is enough to prove it for each connected component. The statement for each of those components follows from \cite[Expos\'e I, Proposition 8.1]{SGA1}.
  
Let us prove Statement \ref{curve}. By \cite[Expos\'e III, Corollaire 7.4]{SGA1}, the curve $C$ can be deformed to a smooth projective curve $\widehat{C}$ on the complete local ring $\Spec(\widehat{\CO}_{S,s})$. By smoothness of $\widehat{C}$ and Hensel's lemma, we can lift $\sigma$ to a section $\hat{\sigma}$ of $\widehat{C}$.

Recall that a functor $F:\Sch/S\ra\Set$ is called limit-preserving if for all filtered systems of $\CO_{S}$-algebras $(B_{\lambda})$ the natural map $\Colim_{\lambda}F(\Spec(B_{\lambda}))\ra F(\Lim_{\lambda}\Spec(B_{\lambda}))$ is a bijection. Consider the functor $\mathrm{Curv_*}:\Sch/S\ra \Set$ which to an $S$-scheme $T$ associates the set of pairs $(\CC\ra T,\sigma:T\ra \CC)$ where $\CC$ is a smooth projective relative curve and $\sigma$ is a section. We have $(\widehat{C},\hat{\sigma})\in \mathrm{Curv_*}(\Spec(\widehat{\CO}_{S,s}))$. The functor $\mathrm{Curv_*}$ is limit-preserving by \cite[Th\'eor\`eme 8.8.2.(i)-(ii)]{EGAIV_3}, \cite[Th\'eor\`eme 8.10.5.(xii)]{EGAIV_3} and \cite[Lemme 1.A.2]{Ayoub_rigide} (and its proof).

By Artin approximation (in the form of \cite[Corollary 2.2]{Artin_approximation}, which is known to hold over a general excellent scheme after the work of Conrad-De Jong \cite{Conrad_de_Jong}), there exists a pointed étale neighbourhood $(W,s)$ and $(\breve{C},\breve{\sigma})\in \mathrm{Curv_*}(W)$ which coincides with $(\widehat{C},\hat{\sigma})$ at the first order, i.e., which lifts the original pair $(C,\sigma)$. Finally, geometric connectedness of fibers for proper flat morphisms of finite presentation with geometrically reduced fibers is an open property \cite[Th\'eor\`eme 12.2.4.(vi)]{EGAIV_3}, and this implies that up to refining $W$ we can assume that $\breve{C}$ has geometrically connected fibers.
\end{proof}

The deformation theory argument in the proof of Proposition \ref{prop:generators_closed_imm} is the reason why we have introduced an arbitrary \'etale morphism in the definitions of $\DG$ and $\JG$, instead of say an open immersion. A simplification of the same proof yields the following $0$-motivic analogue.

\begin{lemma}\label{lem:generators_closed_imm_0}
Let $i:Z\ra S$ be a closed immersion. Then
\[
i_*\langle e_{\sharp}\BQ|\ e:U\ra Z\text{ \'etale }\rangle_{(+)}\subset \langle f_{\sharp}\BQ|\ f:V\ra S\text{ \'etale }\rangle_{(+)}.
\]  
\end{lemma}

We can now exhibit new generating families for $\DA_1(S)$ and $\DA^1(S)$.

\begin{prop}
\label{prop:generators_DA1}
Let $S$ be a noetherian finite-dimensional excellent scheme.
\begin{enumerate}[label={\upshape(\roman*)}]
\item \label{jg_dg} $\langle \JG_S\rangle_{(+)} = \langle \DG_S \rangle_{(+)}$.
\item \label{gen_1} We have \[\DA_{1,c}(S)= \langle \JG_S\rangle = \langle \DG_S\rangle\] and \[\DA_1(S) =\llangle \JG_S\rrangle = \llangle \DG_S\rrangle.\]
\item \label{gen^1} We have
\[\DA^1_c(S)=\langle \JG_S(-1)\rangle = \langle \DG_S (-1)\rangle\] and \[\DA_1(S) =\llangle \JG_S (-1)\rrangle = \llangle \DG_S(-1)\rrangle.\]
\end{enumerate}
\end{prop}

\begin{proof}
Let us prove Point~\ref{jg_dg}. Using Lemma~\ref{lem:gen_ops} and localisation, we can assume that $S$ is reduced. By definition, $\JG_{S}\subset\DG_{S}$, hence $\langle \JG_S \rangle_{(+)}\subset \langle \DG_S \rangle_{(+)}$. We prove the reverse inclusion by noetherian induction on $S$. Since $\langle \langle \CG \rangle_+\rangle = \langle \CG \rangle$ for any family $\CG$, it is enough to treat the $+$ version. Let $M$ be in $\DG_S$. By Proposition \ref{prop:generators_closed_imm}, Lemma \ref{lem:gen_ops} and localisation, to proceed with the induction, it is enough to show that there exists a non-empty open set $j:V\ra S$ such that $j^*M$ lies in $\langle \JG_V\rangle_+$. 

A lattice (resp. a torus) on a reduced scheme is generically a direct factor of a permutation lattice (resp. torus) by~\cite[Exp. X 6.2]{SGA3_2_old}, while an abelian scheme on $S$ is generically and up to isogeny a direct factor of the relative Jacobian of a smooth projective curve with a rational point by \cite[Theorem~11]{Katz_SpaceFill} applied at a generic point and a spreading out argument. This implies that for any $M\in\DG_S$, there exists a non-empty open $j:V\ra S$ such that $j^*M$ is a direct factor of a motive in $\JG_V$. This completes the proof of Point~\ref{jg_dg}.

For Point~\ref{gen_1}, it is enough to show that $\DA_{1,c}(S)= \langle \DG_S\rangle$. Over an arbitrary field $k$, we have that $\DA_{1,c}(k)$ is generated by motives of smooth projective curves by Proposition~\ref{prop:subcats_field}, and those lies in $\langle\DG_k\rangle$ by Proposition~\ref{prop:motive_curve_field}. In the other direction, it is enough to show that the image by $\Sigma^\infty$ of pure Deligne 1-motives over $k$ lie in $\DA_{1,c}(k)$; this is an easy case of Corollary \ref{cor:Deligne_da_1}. By continuity for both sides, we can apply noetherian induction, localisation and use the stability by $i_*$ of both sides (Proposition~\ref{prop:permanence_hom_n} for $\DA_{1,c}$, Proposition~\ref{prop:generators_closed_imm} for $\langle\DG\rangle$). This finishes the proof.
\end{proof}

Recall that $\DA^{1}(S)=\DA_{1}(S)(-1)$ by Proposition \ref{prop:hom_cohom_twists}. We come to the main definition of this paper. 

\begin{defi}\label{main_def_1}
Let $S$ be a noetherian finite-dimensional scheme. The \emph{motivic $t$-structure} $t_{\MM,1}(S)$ on $\DA_1(S)$ (resp. $t^1_{\MM}(S)$ on $\DA^1(S)$) is the $t$-structure $t(\DG_S)$ (resp. $t(\DG_S(-1))$). The heart of $t_{\MM,1}$ (resp. $t^1_{\MM}$) is the \emph{abelian category of $1$-motivic sheaves} $\MM_1(S)$ (resp. $\MM^1(S)$).

\end{defi}

The two abelian categories $\MM_1(S)$ and $\MM^1(S)$ are equivalent via Tate twists, but embedded differently in $\DA(S)$. From Proposition~\ref{prop:generators_DA1} we immediately get the following statement.

\begin{cor}\label{cor:t_gens} Assume that $S$ is excellent. Then
  $t_{\MM,1}=t(\JG_S)=t(\DG_S)$ (resp. $t^1_{\MM}=t(\JG_S(-1))=t(\DG_S(-1))$).
\end{cor}

We introduce a parallel definition for $0$-motives. Recall that $\DA_{0}(S)=\DA^{0}(S)$ by Proposition \ref{prop:hom_cohom_twists}, 

\begin{defi}\label{main_def_0}
The \emph{motivic $t$-structure} $t_{\MM,0}(S)=t^0_{\MM}(S)$ on $\DA_0(S)=\DA^0(S)$ is the $t$-structure generated by the family of objects of the form $e_\sharp \BQ$ with $e:T\ra S$ \'etale. The heart of $t^0_{\MM}$ is the \emph{abelian category of $0$-motivic sheaves} $\MM^0(S)$.
\end{defi}

\begin{remark}
The $t$-structure $t_{\MM,0}(S)$ is somewhat similar to the \emph{homotopy $t$-structure} on the whole of $\DA(S)$, which we define, following \cite[Definition 2.2.41]{Ayoub_these_1}, as the $t$-structure generated by the objects $f_\sharp \BQ(n)[n]$ for all $f:T\ra S$ smooth and $n\in \BZ$. It is likely that the homotopy $t$-structure restricts to $\DA_0(S)$ and that its restriction is $t_{\MM,0}(S)$.
\end{remark}

We now discuss some elementary exactness properties of Grothendieck operations with respect to the motivic $t$-structures.

\begin{prop}\label{prop:elementary_exactness}
The following properties hold for $t_{\MM,1}$, $t^1_{\MM}$ and $t_{\MM,0}$.
\begin{enumerate}[label={\upshape(\roman*)}]
\item \label{exact^*} Let $f$ be a morphism of schemes; then $f^*$ is $t$-positive.
\item \label{exact_!} Let $f$ be a quasi-finite separated morphism between excellent schemes; then $f_!$ is $t$-positive.
\item \label{exact_e^*} Let $e$ be an \'etale morphism; then $e^*$ is $t$-exact.
\item \label{exact_f_*} Let $f$ be a finite morphism between excellent schemes; then $f_*$ is $t$-exact.
\end{enumerate}
Let $\epsilon\in \{0,1\}$; the following properties hold for $t^\epsilon_{\MM}$.
\begin{enumerate}[resume, label={\upshape(\roman*)}]
\item \label{exact_omega_*} Let $f$ be a morphism of schemes; then $\omega^\epsilon f_*$ is $t$-negative.
\item \label{exact_omega^!} Let $f$ be a quasi-finite separated morphism of schemes between excellent schemes; then $\omega^\epsilon f^!$ is $t$-negative.
\end{enumerate}
\end{prop}

\begin{proof}
By Proposition~\ref{prop:permanence_hom_n} (resp. \ref{prop:permanence_coh_n}) and the very definition of $\omega^0$ and $\omega^1$, all the operations above are well-defined. We prove the proposition for $t_{\MM,1}$; the proof for the corresponding statements for $t^1_{\MM}$ is then obtained by twisting by $\BQ(-1)$, and the proof for $t^0_{\MM}$ is completely analoguous (using Lemma \ref{lem:generators_closed_imm_0}  instead of Proposition \ref{prop:generators_closed_imm})

 Let $f:S\ra T$ be any morphism of schemes (resp. a quasi-finite separated morphism between excellent schemes). Then $f^*$ (resp. $f_!$) commute with small sums since it is a left adjoint. By \cite[Lemme 2.1.78]{Ayoub_these_1}, to prove statements \ref{exact^*}, \ref{exact_!}, it remains to show that $f^* \DG_T\subset \DA_1(S)_{\geq 0}$ and that when $f$ is quasi-finite, $f_!\DG_S\subset \DA_1(S)_{\geq 0}$.

In the case of $f^*$, the result follows from Lemma \ref{lem:gen_ops} \ref{gen_pullbacks}.

For the case of $f_!$, we proceed in several steps. If $e$ is an \'etale morphism, we have $e_!\DG_S\subset \DG_T$ by definition. If $i$ is a closed immersion, we have $i_!\DG_S\subset \langle\DG_T\rangle_+$ by Proposition~ \ref{prop:generators_closed_imm} and Proposition~\ref{prop:generators_DA1}. Let $f$ be an arbitrary quasi-finite morphism. At this point, we know that for a open immersion $j$ (resp. a closed immersion $i$), the functors $j_!$ and $j^*$ (resp. the functors $i_!$ and $i^*$) are $t$-positive. This shows that to prove that an object $M$ is $t$-positive, one can proceed by localisation. A noetherian induction together with the \'etale case above then reduce us to the case where $f$ is finite surjective inseparable, and allows us further to restrict to an arbitrary dense open set of the base. Using continuity, this reduces us to the field case, where we can apply Lemma~\ref{lemm:hom_!_fields}.

Let $f$ be an \'etale morphism (resp. a finite morphism between excellent schemes). We have seen above that $f^*$ (resp. $f*\simeq f_!$) is $t$-positive. Moreover, since $e_!\simeq e_\sharp$ (resp. $f^*$) is $t$-positive, its right adjoint $e^*$ (resp. $f_*$) is $t$-negative. This proves \ref{exact_e^*} (resp. \ref{exact_f_*}).

Let $f:S\ra T$ be a morphism (resp. a quasi-finite separated morphism between excellent schemes). We have seen above that $f^*:\DA^1(T)\ra \DA^1(S)$ (resp. $f_!:\DA^1(S)\ra \DA^1(T)$) is $t$-positive, so its right adjoint $\omega^1 f_*$ (resp. $\omega^1 f_!$)  is $t$-negative. This proves \ref{exact_omega_*} (resp. \ref{exact_omega^!}).
\end{proof}

From the definition, we also get a partial result about the Betti and $\ell$-adic realisation functors. 

\begin{prop}\label{prop:elementary_realisations}
\begin{itemize}
  \item \label{RB_pos} Let $k$ be a field with a fixed complex embedding $\sigma$ and $S$ be a scheme of finite type over $k$. The functor $R_{B,\sigma}$, restricted to either $\DA_0(S)$, $\DA_1(S)$ or $\DA^1(S)$ is $t$-positive with respect to the motivic $t$-structure and the standard $t$-structure.
  \item \label{Rl_pos} Let $\ell$ be a prime, and let $S$ be a japanese $\BZ[\frac{1}{\ell}]$-scheme. The functor $R_{\ell}$, restricted to either $\DA_{0,c}(S)$, $\DA_{1,c}(S)$ or $\DA^1_c(S)$, sends compact $t_{\MM}$-positive objects to positive objects in the standard $t$-structure.
\end{itemize}
\end{prop}
\begin{proof}
  Because of the definition of the motivic $t$-structures above, and the structure of $t$-positive and compact objects in a generated $t$-structure, it is enough to show that the image of a compact generator is $t$-positive for the standard $t$-structure. The three cases being similar, let us treat the one of $\DA_1(S)$. Let $e:U\ra S$ be an \'etale morphism, $\BM=[L\ra G]\otimes\BQ\in \CM_1(U)$ and $M=e_!\Sigma^\infty_U\BM\in \DG_U$ (recall that $e_\sharp\simeq e_!$ as $e$ is \'etale). 

Write $R$ for either $R_B$ or $R_\ell$ (with the appropriate hypothesis on $S$). Then $R\BM\simeq e_!R(\Sigma^\infty_U\BM)$ with $e_!$ the corresponding Grothendieck operations on derived categories of sheaves (by \cite[Theoreme 3.19]{Ayoub_Betti} for $R=R_B$ and \cite[Theoreme 9.7]{Ayoub_Etale} for $R=R_\ell$). Since the functor $e_!$ is then $t$-exact for the standard $t$-structures, we only need to show that $R(\Sigma^\infty_U\BM)$ is $t$-positive. Let us show that it is in fact in the heart of the standard $t$-structure. We can show this separately for $\BM=[L\ra 0]\otimes\BQ$ and $\BM=[0\ra G]\otimes\BQ$, i.e., we need to compute $R(\Sigma^\infty L_{\BQ})$ and $R(\Sigma^\infty G_{\BQ}[-1])$. 

Note that because of the commutation of $R$ with the six operations, localisation and the $t$-exactness of $j_!j^*$ and $i_*i^*$ for the standard $t$-structures, we can always restrict to a non-empty open set of $U$ and argue by noetherian induction. We can then assume $U_\red$ to be normal (since $S$ is assumed japanese), and then write $L$ as a direct factor of $h_!\BQ_T$ for $h:T\ra U$ finite \'etale using Lemma \ref{lemm:permutation_torus}. Applying again the commutation of $R$ with $h_!$
 and the $t$-exactness of $h_!$ for the standard $t$-structures, we conclude that $R(\Sigma^\infty L_{\BQ})$ is in the heart.

In the case of $\Sigma^\infty G_{\BQ}[-1]$, our claim follows from the computation of the realisation of such a motive in \cite[Proposition 5.1.(2)]{AHPL} (for $R_B$) and \cite[5.2]{AHPL} (for $R_\ell$). This completes the proof.
\end{proof}

\begin{remark}
The $t$-exactness of pullbacks by arbitrary morphisms and of realisation functors has been proven in \cite{constructible_1-mot}.
\end{remark}  

There are simple connections between the $t$-structures for $0$ and $1$-motives.

\begin{prop}\label{t_struct_01}
Let $S$ be a noetherian finite-dimensional excellent scheme.
\begin{enumerate}[label={\upshape(\roman*)}]
\item \label{incl_01_exact} The inclusion of $\DA_0(S)$ into $\DA_1(S)$ is $t$-exact.
\item \label{restr_01} The $t$-structure $t_{\MM,1}(S)$ restricts to $\DA_0(S)$, and its restriction coincide with $t_{\MM,0}(S)$.
\end{enumerate}
\end{prop}

\begin{proof}
Let us prove Statement~\ref{incl_01_exact}. The inclusion functor commutes with small sums. The generators $e_\sharp \BQ$ ($e:U\ra S$ \'etale) of $t^0_{\MM}$ are also $t$-positive for $t_{\MM,1}$; this implies that the inclusion is $t$-positive. 

Let us now show the inclusion $\DA_1(S)$ is $t$-negative. Let $N\in \DA^0(S)_{\leq 0}$. We have to show that for every \'etale morphism $e:U\ra S$, $\BM=[L\ra G]\in\CM_1(U)$, and $n\in\BN^*$, we have $\DA(S)(e_\sharp (\Sigma^\infty \BM)[n],N)=0$. Using the $e_\sharp\dashv e^*$ adjunction and the fact that $e^*$ is $t$-negative (Proposition~\ref{prop:elementary_exactness}~\ref{exact_e^*}), we reduce to the case $e=\id$. We have a distinguished triangle
\[
\Sigma^\infty G_\BQ[-1] \ra \Sigma^\infty\BM\ra \Sigma^\infty L_\BQ  \rap.
\]
Let us first show that, for all $P\in \DA^0(S)$, we have $\DA(S)(\Sigma^\infty G_\BQ,P)=0$. Because $\Sigma^\infty G$ is compact, this vanishing statement can be checked on compact generators, so that we can assume that $P$ is of the form $a_*\BQ_X[m]$ for some $a:X\ra S$ finite and $m\in \BZ$. Using the $a^*\dashv a_*$ adjunction and Proposition~\ref{prop:pullback_complex}, we see that we can assume $a=\id$, so we have to show that $\DA(S)(\Sigma^\infty G_\BQ,\BQ[m])=0$. By~\cite[Theorem 3.3]{AHPL}, $\Sigma^\infty G_\BQ$ is a direct factor of $M_S(G)$, characterised as the $n$-eigenspace for the morphism induced by $[n]_G$ for any $n\neq 1$, and that $M_S(G)$ has also a direct factor $\BQ_S$, characterised as the $1$-th eigenspace for $[n]_G$. We have $\DA(S)(M_S(G),\BQ[m])\simeq H^{m,0}_\CM(G)$; since $\pi:G\ra S$ is smooth surjective with connected fibres, we deduce by Proposition~\ref{prop:mot_coh_0}~\ref{mot_coh_0_pullback} that $\pi^*:H^{m,0}_\CM(S)\ra H^{m,0}_\CM(G)$ is an isomorphism. Looking at the action of $[n]_G$, this shows that all the weight $0$ motivic cohomology of $G$ comes from the direct factor $\BQ_S$ of $M_S(G)$, and accordingly we deduce that $\DA(S)(\Sigma^\infty G_\BQ,\BQ[m])=0$ as claimed. This shows that $\DA(S)(\Sigma^\infty\BM[n],N)\simeq \DA(S)(\Sigma^\infty L_\BQ[n],N)$.

On the other hand, the motive $\Sigma^\infty L(-1)$ is in $\DA_0(S)$ and $t_{\MM,0}$-positive; this would be clear for $S$ normal since $L$ is then a direct factor of a permutation lattice, in general this can be checked by noetherian induction starting from a non-empty open set $V\subset U$ with $V_\red$ normal (possible since $U$ is excellent), using localisation, Proposition~\ref{prop:pullback_complex} and Proposition~\ref{prop:elementary_exactness}. Since by hypothetis $N$ is $t_{\MM,0}$-negative, we have $\DA(S)(\Sigma^\infty L_\BQ[n],N)=0$. This completes the proof that $\DA_0(S)\ra \DA_1(S)$ is $t$-negative, hence $t$-exact.

We now prove Statement~\ref{restr_01}. Write $ _0 \tau_{\geq 0}$ and $ _1\tau_{\geq 0}$ for the truncation functors of $t_{\MM,0}$ and $t_{\MM,1}$. We have to show that for every $M\in \DA_0(S)$, we have $ _1\tau_{\geq 0}M\in \DA_0(S)$ and $ _1\tau_{\geq 0}M\simeq\  _0\tau_{\geq 0}M$. But this follows immediately from the $t$-exactness of the inclusion, proved above.
\end{proof}

\begin{remark}
It is also likely that Proposition~\ref{t_struct_01} holds for $t^1_{\MM}(S)$; this seems to require more delicate vanishing results.
\end{remark}

\subsection{The $t$-structures over a field}
\label{t_struct_field}
In this short section, we compare our $t$-structures for homological $0$ and $1$-motives with the existing work on $t$-structures for $\DM^{\eff}_0(k)$ and $\DM^{\eff}_1(k)$ with $k$ a perfect field \cite{Orgogozo}~\cite{Ayoub_2-mot}, and we extend the results from these references to a possibly imperfect field.

For clarity, let us treat first the simpler case of $0$-motives. Let $k$ be a perfect field. We reformulate the treatement in \cite[\S 2]{Orgogozo}. There is a functor $\Sh_{\et}(k,\BQ)\ra \DM^{\eff}(k)$ (any sheaf of $\BQ$-vector spaces on the small \'etale site has a canonical extension as an \'etale sheaf with transfers on $\Sm/k$) which extends to a triangulated functor $D(\Sh_{\et}(k,\BQ))\ra \DM^{\eff}(k,\BQ)$. This factors through $\DM^{\eff}_0(k)$, and the resulting functor is an equivalence of categories $\CR^{\eff,0}_{\tr}:D(\Sh_{\et}(k))\simeq \DM^{\eff}_0(k)$.

 Another approach consists in first introducing the \emph{homotopy $t$-structure} on $\DM^{\eff}(k)$; this is the $t$-structure induced on $\DM^{\eff}(k)$ from the standard $t$-structure on $D(\Sh((\mathrm{Cor}/k)_\et,\BQ))$, but for our purposes it is best described as the $t$-structure on the triangulated category $\DM^{\eff}(k)$ compactly generated by the family of objects of the form $M^{\eff,\tr}_k(X)$ for all $X\in \Sm/k$ \cite[Proposition 3.3]{Ayoub_2-mot}. We claim that the homotopy $t$-structure restricts to $\DM^{\eff}_0(k)$, and that the restriction coincides with the $t$-structure generated by the family of objects of the form $M^{\eff,\tr}_k(Y)$ for all $Y/k$ finite \'etale. To do this, it suffices to show that the inclusion functor $\DM^{\eff}_0(k)\ra \DM^{\eff}(k)$ is $t$-exact for those two $t$-structures; it is $t$-positive because of the inclusion of generators, and $t$-negative because its left adjoint $L\pi_0$ is $t$-positive since $L\pi_0((M^{\eff,\tr}_k(X))\simeq M^{\eff,\tr}_k(\pi_0(X/k))$ for any $X/k$ smooth. 

It is easy to see that the $t$-structures on $\DM^{\eff}_0(k)$ introduced in the two previous paragraphs coincide. Moreover, through the equivalence of categories of Lemma~\ref{lemm:DA_DM_n}, we get an equivalence of categories $\CR^0:D(\Sh_{\et}(k,\BQ))\ra \DA_0(k)$, and this is a $t$-exact equivalence of triangulated categories when we equip $\DA_0(k)$ with $t_{\MM,0}$.

Finally, these $t$-structures on $\DM^{\eff}_0(k)$ and $\DA_0(k)$ restrict to compact objects; more precisely, there are equivalences of categories $D^b(\Sh_{\et,c}(k,\BQ))\simeq \DM^{\eff}_{0,c}(k)\simeq \DA_{0,c}(k)$ and the restriction of the $t$-structure coincides with the standard $t$-structure on the bounded derived category.

Let now $k$ be a general field and let $h:\Spec(k^{\perf})\ra \Spec(k)$ be a perfect closure. We have a commutative diagram
\[
\xymatrix{
D(\Sh_\et(k)) \ar[r]_{\CR^0} \ar[d]^\sim_{h^*} & \DA_0(k) \ar[d]^\sim_{h^*} \\
D(\Sh_\et(k^\perf)) \ar[r]_{\CR^0}^{\sim} & \DA_0(k^{\perf}) 
}
\]
where the bottom horizontal functor is an equivalence by the case of a perfect field, the left vertical functor is an equivalence because the \'etale sites of $k$ and $k^\perf$ are canonically isomorphic via $h$, and the right vertical functor is an equivalence by the separation property of $\DA(-)$ and Corollary~\ref{cor:localisation_subcats}~\ref{rad_sub}. Moreover, the functor $h^*:D(\Sh_\et(k))\ra D(\Sh_\et(k^\perf))$ is clearly $t$-exact, the functor $h^*:\DA_0(k)\ra \DA_0(k^\perf)$ is $t$-exact because it is a quasi-inverse of the $t$-exact functor $h_*$ (Proposition~\ref{prop:elementary_exactness}~\ref{exact_f_*}), and $\CR^0:D(\Sh_{\et}(k^\perf))\ra \DA_0(k^\perf)$ is $t$-exact by the perfect field case. This proves that the top arrow is also a $t$-exact equivalence of triangulated categories. There is a similar diagram in the compact case which we will not spell out. Let us summarise the results so far.

\begin{prop}\label{prop:0_motives_field}
  Let $k$ be a field. The $t$-structure $t_{\MM,0}$ restricts to compact objects, and we have equivalences of $t$-categories
\[
\CR^0:(D(\Sh_\et(k,\BQ)),\std)\stackrel{\sim}{\lra} (\DA_0(k),t_{\MM,0})
\]
\[
\CR^0:(D^b(\Sh_{\et,c}(k,\BQ)),\std)\stackrel{\sim}{\lra} (\DA_{0,c}(k),t_{\MM,0}).
\]
\end{prop}

We now turn to the case of $1$-motives. Assume again momentarily that $k$ is a perfect field. By \cite[Lemma 1.4.4]{BVK}, for any commutative locally of finite type $k$-group scheme $G$, the sheaf represented by $G$ on $\Sm/k$ has a canonical structure of \'etale sheaf with transfers. Write $G^\tr$ for this sheaf with transfers, with $\underline{o}^\tr G^\tr\simeq G$. 

Applying this construction at the level of complexes, Orgogozo defines in \cite[3.3.2]{Orgogozo} a functor which we will denote by
\[
\CR^{\eff,\tr}_1: \CM_1(k)\ra \DM^{\eff}_c(k).
\]
The category $\CM_1(k)$ is in this situation an abelian category \cite[Lemme 3.2.2]{Orgogozo} and this functor can in fact be extended to a functor
\[
\CR^{\eff,\tr}_1: D^b(\CM_1(k))\ra \DM^{\eff}_c(k).
\]
This functor factors through $\DM^{\eff}_{1,c}(k)$ (denoted as $ d_1 \DM^{\eff}_\gm(k)$ in loc. cit.) and the resulting functor is then an equivalence of categories \cite[Theorem 3.4.1]{Orgogozo}. In particular, this provides a $t$-structure on $\DM^{\eff}_{1,c}(k)$, which we will denote by $t^{\mathrm{Or}}_1(k)$. By the equivalence between $\DM^{\eff}_{1,c}(k)$ and $\DA_{1,c}(k)$, we get a $t$-structure on $\DA_{1,c}(k)$ which we also denote by $t^{\mathrm{Or}}_1(k)$. Moreover, by comparing $\CR^{\eff,\tr}_1$ with $\Sigma^\infty$, we get that the functor
\[
\Sigma^\infty: D^b(\CM_1(k))\ra \DA_{1,c}(k)
  \]
  is an equivalence of $t$-categories. By \cite[Proposition~3.3.3]{Orgogozo} and \cite[Proposition~3.2.4]{Orgogozo}, we have the following computation of morphisms groups in $\DA_{1,c}(k)$.

\begin{prop}\label{prop:morphisms_field}
Let $k$ be a field, $M_1,M_2\in\CM_1(k)$ and $n\in \BZ$. Then
\begin{eqnarray*}
\DA(k)(\Sigma^\infty M_1,\Sigma^\infty M_2[n]) & \simeq & \Ext^n_{\CM_1(k)}(M_1,M_2) \\
& \simeq & 0,\ n \neq 0,1.
\end{eqnarray*} 
\end{prop} 

We can now show the following basic result.

\begin{prop}\label{prop:1_motives_field}
Let $k$ be a field and $k^\perf$ a perfect closure. The $t$-structure $t_{\MM,1}$ restricts to compact objects, and we have an equivalence of $t$-categories
\[
\Sigma^{\infty}:(D^b(\CM_1(k^{\perf})),\std)\lra (\DA_{1,c}(k),t_{\MM,1}).
\]
\end{prop}

\begin{proof}
  We first assume that $k$ is perfect. Let us show that the $t$-structure $t_{\MM,1}(k)$ on $\DA_1(k)$ restricts to $\DA_{1,c}(k)$, and that its restriction is $t^{\mathrm{Or}}_1(k)$. For this, it is enough to show that if $M\in \DA_{1,c}(k)$ is $t^{\mathrm{Or}}_1(k)$-positive (resp. negative), it is $t_{\MM,1}(k)$-positive (resp. negative). Using the equivalence $\Sigma^\infty$, it is clearly enough to show this for $M=\Sigma^\infty(\BM)$ with $\BM\in \CM_1(k)$. By construction of $t_{\MM,1}(k)=t(\Sigma^\infty(\CM_1(k)))$, we see that $M$ is $t_{\MM,1}(k)$-positive. It remains to show that $M$ is $t_{\MM,1}(k)$-negative, i.e., that for all $\BN\in \CM_1(k)$ and $k>0$, we have
\[
\DA(k)(\Sigma^\infty \BN[k],\Sigma^\infty \BM)=0.
\]
This is a special case of Proposition~\ref{prop:morphisms_field}.

Let now $k$ be a general field and $h:\Spec(k^\perf)\ra \Spec(k)$ be a perfect closure. The functor $h^*:(\DA_1(k),t_{\MM,1})\ra (\DA_1(k^\perf),t_{\MM,1})$ is an equivalence of $t$-categories by the separation property of $\DA(-)$, Corollary~\ref{cor:localisation_subcats}~\ref{rad_sub}, and Proposition~\ref{prop:elementary_exactness}~\ref{exact_f_*}. It then follows from the perfect case above that $t_{\MM,1}(k)$ restricts to compact objects.
\end{proof}

\subsection{Deligne $1$-motives and the heart}
\label{sec:morphisms}

In this section, we compute certain morphism groups between objects in $\DA_1(S)$ and $\DA^1(S)$ and deduce various properties of the motivic $t$-structure. 

The following theorem shows the advantage of the Deligne generating family: it lies in the heart of the motivic $t$-structure.

\begin{theo} \label{theo:gen_in_heart}
Let $S$ be a noetherian finite-dimensional excellent scheme. We have $\DG_S\subset \MM_1(S)$ (resp. $\DG_S(-1)\subset \MM^1(S)$).
\end{theo}
\begin{proof}
The generators $\DG_{S}$ (resp. $\DG_{S}(-1)$) are $t$-positive by definition, it remains to show that they are $t$-negative.

Using the generating family $\JG_S$ (Corollary \ref{cor:t_gens}), this translates into the following vanishing statement. Let $S$ be a noetherian finite-dimensional scheme. Let $e:U\ra S$ be an \'etale morphism, and $N=e_{\sharp}\Sigma^\infty K\otimes\BQ$ be a Jacobian generator. Let $P=f_!\Sigma^\infty\BM\in \DG_S$ (i.e., $f:V\ra S$ \'etale, $\BM\in \CM_1(V)$). Then we show, for all $n<0$, that
\[
\DA(S)(N,P[n])=0\tag{$\CV_n(P)$}
\]
In the resp. case, $\BM$ is a successive extension of pure Deligne $1$-motives, so that we can assume that $\BM$ is pure.

By the $(e_{\sharp},e^*)$ adjunction and Proposition~\ref{prop:pullback_complex}, we can assume that $e=\id$. By localisation and Proposition~\ref{prop:pullback_complex}, we can assume that $S$ is reduced. By Zariski's main theorem, there exists a factorisation $f=\bar{f}\circ j$ with $\bar{f}:\overline{V}\ra S$ finite and $j:V\ra \overline{V}$ an everywhere dense open immersion; we can assume $\overline{V}$ is reduced as well. Combining this with the $(\bar{f}^*,\bar{f}_*)$ adjunction, and Proposition~\ref{prop:pullback_complex}, we see that we can assume $f=j$ is an everywhere dense open immersion. We write $i:Z\ra S$ for the complementary reduced closed immersion. 

We want to prove $(\CV_n(P))$ by induction on the dimension of $S$. In each case, to treat the case of $\dim(S)=0$, we reduce immediately to the case of $\Spec(k)$ for $k$ a field and apply Proposition~\ref{prop:morphisms_field}. We are thus left with the induction step.

First, we do a general reduction. Let $l:W\ra S$ an everywhere dense open immersion with $W\subset V$ and $k:Y\ra S$ the complementary reduced closed immersion. Then by localisation we have exact sequences
\[
\DA(S)(N,l_!l^* P[n])\ra \DA(S)(N,P[n]) \ra \DA(N,k_* k^* P[n])
\]
and in both cases the right term vanishes for $n<0$ by adjunction and the induction hypothesis (since $\dim(Z)<\dim(S)$). This means we can replace $P$ with 
\[
l_! l^*P\simeq l_! l^* j_!\Sigma^\infty \BM\simeq (W\ra S)_!(W\ra V)^* \Sigma^\infty \BM\simeq (W\ra S)_!\Sigma^\infty \BM_W
\]
where we have used the $\Ex_!^*$ isomorphism and Corollary~\ref{coro:pullback_complex_DA}. In other words, we can replace the dense open subscheme $V$ by any smaller dense open $W$.

There are three types of Deligne generators and three types of Jacobian generators, which lead to a distinction in nine cases. To lighten the notation, we index them by weights: for instance, the case where $\BM=[L\ra 0]$ and $K=\Gm$ will be labelled $(0,-2)$.

\underline{Cases $(0,*)$:}

Let $\BM$ be $[L\ra 0]\otimes\BQ$ with $L$ a lattice on $V$.

Replacing $V$ by a smaller open, we can assume $V$ to be normal (since $V$ is reduced and excellent). This allows us by Lemma~\ref{lemm:permutation_torus} to write $\Sigma^\infty\BM$ as a direct factor of $e_* \BQ$ for a finite \'etale morphism $e:T\ra V$. Applying Zariski's main theorem to the morphism $j\circ e:T\ra S$ and adjunction, we reduce to the case $P=\BQ_V$. The motive $P$ then extends to a motive on $S$, namely $\BQ_S$. By localisation, we have an exact sequence
\[
\DA(S)(N,i_*\BQ[n-1])\ra \DA(S)(N,j_!\BQ[n])\ra \DA(S)(N,\BQ[n])
\]
and the left term vanishes for $n<0$ by adjunction and induction on the dimension. This means we can assume $V=S$.

If we are in case $(0,0)$ (resp. $(0,-2)$), then we have $N=\BQ_S$ (resp. $N=\BQ_S(1)$). By adjunction and  Proposition~\ref{prop:mot_coh_0} \ref{mot_coh_0_<} (resp. Proposition~\ref{prop:mot_coh_<} ), we get $\DA(S)(N,\BQ[n])=0$ for $n<0$.

It remains to treat the case $(0,-1)$. Let $C\rightarrow S$ be a smooth projective with geometrically connected fibres and a section $\sigma$. We have $N=\Sigma^\infty \Jac(C/S)_\BQ[-1]$, which by Corollary~\ref{coro:computation_curve} is a direct factor of $M_S(C)[-1]$. By adjunction, we thus have that $\DA(S)(N,\BQ_S[n])$ is a direct factor of $\DA(C)(\BQ_C,\BQ_C[n+1])$. For $n<-1$, this group vanishes by Proposition~\ref{prop:mot_coh_0}\ref{mot_coh_0_<}. For $n=0$, we apply Proposition~\ref{prop:mot_coh_0} \ref{mot_coh_0_0} and get $\BQ^{\pi_0(S)}\simeq \DA(S)(\BQ_S,\BQ_S)\rightarrow \DA(C)(\BQ_C,\BQ_C)\simeq \BQ^{\pi_0(C)}$. The map on $\pi_0$ is an isomorphism since $C$ has geometrically connected fibres. This shows that the constribution of the direct factor $\Sigma^\infty \Jac(C/S)_\BQ[-1]$ is $0$, and proves the case $n=-1$. This finishes the treatment of the cases $(0,*)$.

\underline{Cases $(-2,*)$:}

Let now $\BM$ be of the form $[0\ra T]\otimes\BQ$ with $T$ a torus on $V$. As in the proof for a lattice, we can replace the dense open $V$ by a smaller dense open normal subscheme, thus to a permutation torus using Lemma \ref{lemm:permutation_torus}, then finally to $T=\Gm$. Then $\Sigma^\infty\BM \simeq \BQ_V(1)$ extends to a motive on $S$, namely $\BQ_S(1)$. By localisation, we have an exact sequence
\[
\DA(S)(N,i_*\BQ(1)[n-1])\ra \DA(S)(N,j_!\BQ(1)[n])\ra \DA(S)(N,\BQ(1)[n])
\]
and the left term vanishes for $n<0$ by adjunction and induction. This means we can assume $V=S$.

If we are in case $(0,0)$ (resp. $(0,-2)$), then we have $N=\BQ_S$ (resp. $N=\BQ_S(1)$). By adjunction and Proposition~\ref{prop:mot_coh_1}~\ref{mot_coh_1_leq} (resp. Proposition~\ref{prop:mot_coh_0} \ref{mot_coh_0_<}), we get $\DA(S)(N,\BQ(1)[n])=0$ for $n<0$.

It remains to treat the case $(0,-1)$. Let $C\rightarrow S$ be a smooth projective with geometrically connected fibres and a section $\sigma$. We have $N=\Sigma^\infty \Jac(C/S)_\BQ[-1]$, which by Corollary~\ref{coro:computation_curve} is a direct factor of $M_S(C)[-1]$. By adjunction, we thus have that $\DA(S)(N,\BQ_S(1)[n])$ is a direct factor of $\DA(C)(\BQ_C,\BQ_C(1)[n+1])$. For all $n<0$, this group vanishes by Proposition~\ref{prop:mot_coh_1}\ref{mot_coh_1_leq}. 

\underline{Cases $(-1,*)$:}

Let $\BM$ finally be of the form $[0\ra A]\otimes\BQ$ with $A$ an abelian scheme on $V$. As in the two previous cases, we can replace the dense open $V$ by any smaller dense open. Using \cite[Theorem~11]{Katz_SpaceFill} and continuity, this lets us assume that there exists a smooth projective curve $f:D\ra V $ with geometrically connected fibres together with a section $s:V\ra D$ such that the $\Sigma^\infty [0\ra A]$ is a direct factor of $\Sigma^\infty [0\ra \Jac(D/V)]$. In the following, we replace $A$ by $\Jac(D/V)$.

Unlike in the two previous cases, we cannot ensure that the curve $D$ extends to a smooth projective curve over $S$, so we have to work around this. From Corollary \ref{coro:computation_curve}, we have an isomorphism $f_\sharp \BQ_D \simeq \BQ_V\oplus \Sigma^\infty \Jac(D/V)_\BQ\oplus \BQ_V(1)[2]$; hence $\Sigma^\infty \BM\simeq \Sigma^\infty \Jac(D/V)_\BQ[-1]$ is a direct factor of $f_\sharp\BQ_D[-1]$. By relative purity, we have $f_\sharp\BQ_D[-1]\simeq f_!\BQ_D(1)[1]$.

We apply Nagata's theorem \cite{Nagata_comp} \cite{Conrad_Nagata} to compactify $f$ over $S$: there exists an open immersion $\bar{\jmath}:D\ra \overline{D}$ and a proper morphism $\bar{f}:\overline{D}\ra S$ with $j\circ f=\bar{f}\circ\bar{\jmath}$. Write $\bar{\imath}:Y\ra \overline{D}$ for the complementary closed immersion; note that because $f$ was proper over $V$, we can choose the compactification $\overline{D}$ so that $Y$ lies entirely over $Z$, and we have a commutative diagram with cartesian squares.
\[
\xymatrix{
  D \ar[d]_{f} \ar[r]_{\bar{\jmath}} & \overline{D} \ar[d]_{\bar{f}} & Y \ar[l]^{\bar{\imath}} \ar[d]\\
  V \ar[r]_{j} & S & Z \ar[l]^{i}
  }
\]
This implies that $j_! f_!\simeq \bar{f}_!\bar{\jmath}_!\simeq \bar{f}_*\bar{\jmath}_!$; hence $j_! f_!\BQ_D(1)[1]\simeq \bar{f}_*\bar{\jmath}_!\BQ_D(1)[1]$. The motive $\bar{\jmath}_!\BQ_D(1)[1]$ extends to a motive on $\overline{D}$, namely $\BQ_{\overline{D}}(1)[1]$. By localisation, we have an exact sequence
\[
  \xymatrix@C=1em{
    \DA(\overline{D})(\bar{f}^*N,\bar{\imath}_*\BQ(1)[n]) \ar[r] & \DA(\overline{D})(\bar{f}^*N,\bar{\jmath}_!\BQ(1)[n+1])  \ar[r] & \DA(\overline{D})(\bar{f}^*N,\BQ(1)[n+1]) \ar[d] \\
    & &  \DA(\overline{D})(\bar{f}^*N,\bar{\imath}_*\BQ(1)[n+1])
    }
  \]
  The left term is isomorphic to $\DA(Y)((\bar{f}\bar{\imath})^*N, \BQ(1)[n])$. Since $(\bar{f}\bar{\imath})^*N$ is a Jacobian generator on $Y$, the vanishing of this group for $n<0$ was proved in Case $(-2,*)$. Similarly, for $n<-1$, the right term vanishes by Cases $(-2,*)$. We can thus assume $n=-1$ in the end of the proof, so that we are looking at the exact sequence
  \[0\ra \DA(\overline{D})(\bar{f}^*N,\bar{\jmath}_!\BQ(1))  \ra \DA(\overline{D})(\bar{f}^*N,\BQ(1)) \ra \DA(\overline{D})(\bar{f}^*N,\bar{\imath}_*\BQ(1))\]
  and we need to show that the direct factor of leftmost term corresponding to the direct factor $\Sigma^\infty \Jac(D/V)_\BQ[-1]$ of $f_{\sharp}\BQ_{D}[-1]$ vanishes. In fact, in cases $(-1,0)$ and $(-1,1)$, the whole of the leftmost term vanishes, as we'll see below.
  
  If we are in case $(-1,0)$, we have $N=\BQ_S$, hence $\bar{f}^*N=\BQ_{\overline{D}}$, and the group $\DA(\overline{D})(\BQ,\BQ(1))$ vanishes by Proposition~\ref{prop:mot_coh_1}\ref{mot_coh_1_leq}. This concludes the proof for $(-1,0)$.

  If we are in case $(-1,-1)$, we have $N=\Sigma^\infty \Jac(C/S)\otimes\BQ[-1]$ with $C\rightarrow S$ a smooth projective with geometrically connected fibres and a section $\sigma$. Hence $\bar{f}^* N\simeq \Sigma^\infty \Jac(C\times_S \overline{D}/\overline{D})\otimes\BQ[-1])$. The morphism group $\DA(\overline{D})(\Sigma^\infty \Jac(C\times_S \overline{D}/\overline{D})\otimes\BQ[-1],\BQ(1))$ vanishes by Lemma~\ref{lem:mor_ab_tor} below; this concludes the proof for $(-1,-1)$.
  
  If we are in case $(-1,2)$, we have $N=\BQ_S(1)$, hence $\bar{f}^*N=\BQ_{\overline{D}}(1)$. We have \[\BQ^{\pi_0(\overline{D})}\simeq\DA(\overline{D})(\BQ_{\overline{D}}(1),\BQ(1)) \ra \DA(\overline{D})(\BQ_{\overline{D}}(1),\bar{\imath}_*\BQ(1))\simeq \BQ^{\pi_0(Y)}\] by Proposition~\ref{prop:mot_coh_0}~\ref{mot_coh_0_0}, hence
  \[
\DA(\overline{D})(\bar{f}^*N,\bar{\jmath}_!\BQ(1))\simeq \Ker(\BQ^{\pi_0(\overline{D})}\rightarrow \BQ^{\pi_0(Y)}).
    \]
    On the other hand, we have, by the same argument
    \[
      \DA(S)(N,j_!\BQ(1))\simeq \Ker(\BQ^{\pi_0(S)}\rightarrow \BQ^{\pi_0(Z)}).
    \]
    Since $Y\simeq \overline{D}\times_S Z$, we have $\pi_0(Y)\simeq \pi_0(\overline{D})\times_{\pi_0(S)}\pi_0(Z)$ (in fact, since $S$ is normal and $f$ has geometrically connected fibers, Zariski's connectedness theorem implies that $\bar{f}$ has geometrically connected fibers and $\pi_{0}(\overline{D})\simeq \pi_{0}(S)$, but we do not need this). This implies that the map
    \[
\DA(\overline{D})(\bar{f}^*N,\bar{\jmath}_!\BQ(1))\ra \DA(S)(N,j_!\BQ(1))
    \]
    is an isomorphism. By looking at the direct factor decomposition of $f_\sharp \BQ_D$, we conclude that
    \[
\DA(S)(N,j_!\Sigma^\infty \Jac(D/V)_\BQ[-2])=0
\]
which finishes the proof of the case $(-1,-2)$.
\end{proof}

\begin{lemma}\label{lem:mor_ab_tor}
  Let $S$ be a noetherian finite-dimensional scheme. Let $A$ be an abelian scheme and $T$ be a torus. Let $n\leq 0$. Then
  \[
\DA(S)(\Sigma^\infty A\otimes\BQ,\Sigma^\infty (T\otimes\BQ)[n])=0
    \]
  \end{lemma}
  \begin{proof}
Let $\pi_\bullet:S_\bullet\ra S$ be a $h$-hypercovering. Write $A_p$ (resp. $T_p$) for $A\times_S S_p$ (resp. $T\times_S S_p$). We have a descent spectral sequence
\[
 E_1^{p,q}=\DA(S_p)(\Sigma^\infty A_p\otimes\BQ,\Sigma^\infty T_p\otimes\BQ[q])\Rightarrow \DA(S)(\Sigma^\infty A\otimes\BQ,\Sigma^\infty T\otimes\BQ[p+q]).
\]
Let $\theta:A\ra S$ be the structure morphism of $A$. The motive $\Sigma^\infty A\otimes\BQ$ is a direct factor of $\theta_\sharp\BQ_A$, functorially in $S$. We see that
\[\DA(S_p)(\Sigma^\infty A_p\otimes\BQ,\Sigma^\infty (T_p\otimes\BQ)[n])\]
is a direct factor of   
\[\DA(A_{S_p})(\BQ_{A_p},\Sigma^\infty (T_p\times_{S_p} A_p)\otimes\BQ[n]).\]
We apply the previous spectral sequence for $\pi_\bullet$ the Cech covering associated to an \'etale covering trivializing $T$; since $\DA(A_{S_p})(\BQ_{A_p},\BQ(1)[1][n])$ vanishes for $n< 0$ by Proposition~\ref{prop:mot_coh_1}~\ref{mot_coh_1_leq}, the corresponding spectral sequence converges and we have $\DA(S)(\Sigma^\infty A\otimes\BQ,\Sigma^\infty T\otimes\BQ[n])=0$ for $n<0$.

It remains to treat the case $n=0$ with $T=\Gm$. We apply the previous spectral sequence for $\pi_\bullet$ the Cech covering associated to an affine Zariski cover of $S$; by the previous paragraph, this spectral sequence converges, so that we are reduced to the case when $S=\Spec(R)$ is affine. By a continuity argument, we reduce to the case where $R$ is of finite type of a Dedekind ring, in particular satisfying resolution of singularities by alterations. Then $S$ admits an $h$-hypercovering $\pi:S_\bullet\ra S$ with regular terms. By the descent spectral sequence, which again converges by the previous paragraph, it is then enough to show the result for $S$ regular, $n=0$ and $T=\Gm$.

Again, we write $\Sigma^\infty A\otimes\BQ$ as a direct factor of $\theta_\sharp \BQ_A$. Since $S$ and $A$ are regular, Proposition~\ref{prop:mot_coh_1}~\ref{mot_coh_1_1} implies that 
\[
\DA(S)(\BQ,\BQ(1)[1])\simeq \CO^\times(S)\otimes\BQ
  \]
  and
  \[
\DA(A)(\BQ,\BQ(1)[1])\simeq \CO^\times(A)\otimes\BQ.
    \]
Since the induced morphism $\theta$ is proper with geometrically connected fibres, the map $\CO^\times(S)\otimes\BQ\ra \CO^\times(A)\otimes\BQ$ is an isomorphism. This implies that 
\[
\DA(S)(\Sigma^\infty A\otimes\BQ,\BQ(1)[1])=0
\]
and concludes the proof.
\end{proof}

\begin{cor}\label{cor:motive_group_heart}
Let $S$ be a noetherian finite-dimensional excellent scheme. Let $G$ be a smooth commutative group scheme with connected fibres. Then the motive $\Sigma^\infty G_{\BQ}[-1]$ lies in $\MM_1(S)$.  
\end{cor}
\begin{proof}
By noetherian induction and localisation, we can assume that $S$ is reduced and it is enough to show that there exists $j:U\ra S$ a dense open immersion such that $j_!\Sigma^\infty G_U[-1]$ is in $\MM_1(S)$. We can also assume $S$ to be irreducible; let $\eta$ be its generic point and $h:\eta^{\perf}\ra \eta$ a perfect closure. Then $G_{\bar{\eta}}$ is a smooth commutative connected algebraic group over a perfect field, hence there exists an exact sequence
\[
0\ra U\ra G_{\eta^\perf}\ra H\ra 0
\] 
with $U$ unipotent connected and $H$ a semi-abelian variety. Since $\eta^\perf$ is perfect, the motive $\Sigma^\infty U\otimes\BQ$ is trivial (apply \cite[Lemma 7.4.5]{AEH} to a composition series of $U$), thus the morphism $h^*\Sigma^\infty G_{\BQ}\simeq \Sigma^\infty G_{\bar{\eta},\BQ}\ra \Sigma^\infty H_{\BQ}$ is an isomorphism.

By Lemma~\ref{lem:semiab_insep}, there is an abelian variety $H'$ over $\eta$ such that \[\Sigma^\infty(H'_{\eta^{\perf}}\otimes\BQ)\simeq \Sigma^\infty (H\otimes\BQ).\] Appplying the separation property of $\DA(-)$, we get an isomorphism $\Sigma^\infty G_{\eta,\BQ}\simeq \Sigma^\infty H'_{\BQ}$. By continuity, we can arrange for such an isomorphism to hold over a dense open set $U$ of $S$. We then have $j_!\Sigma^\infty G_{U,\BQ}[-1]\simeq j_!\Sigma^\infty H'_{\BQ}[-1]$ and this last motive is in $\MM_1(S)$ by Theorem~\ref{theo:gen_in_heart}.
\end{proof}

\begin{lemma}\label{lem:semiab_insep}
Let $l/k$ be a purely inseparable field extension, and $H$ a semi-abelian variety over $l$. Then there exists a semi-abelian variety $H'$ over $k$ such that $\Sigma^\infty((H'_{l})\otimes\BQ)\simeq \Sigma^\infty(H\otimes\BQ)$ in $\DA(l)$.
\end{lemma}

\begin{proof}
  We can clearly assume $\charact(k)=p>0$. By \cite[Lemma 3.10]{Brion_isogeny}, there exists a (smooth) commutative algebraic group $G'$ over $k$ and an epimorphism $f:H\ra G'_l$ such that $\Ker(f)$ is infinitesimal (in particular, killed by a power of $p$). By \cite[Corollary 2.13]{Brion_isogeny}, which applies over any field of positive characteristic, there exists an epimorphism of commutative algebraic $k$-groups $g:G'\ra H'\times U'$ with $H'$ semi-abelian, $U'$ split unipotent and $\Ker(g)$ finite (in particular, killed after tensoring by $\BQ$). We deduce that
  \[
H\otimes \BQ\simeq (H'\times_k U')_l\otimes\BQ.
    \]
    But the motive $\Sigma^\infty U'_l\otimes\BQ$ is trivial since $U'$ is split unipotent (apply \cite[Lemma 7.4.5]{AEH} to a composition series). We deduce that
    \[
\Sigma^\infty (H\otimes \BQ)\simeq \Sigma^\infty (H'_l\otimes\BQ).
      \]
\end{proof}

In the course of the proof of Theorem \ref{theo:gen_in_heart} (in the lattice case), we also established the vanishing statements necessary to prove the following lemma.
\begin{lemma}\label{lemm:genn0_in_heart}
Let $S$ be a noetherian finite-dimensional scheme. Let $e:U\ra S$ be an \'etale morphism. Then $e_\sharp \BQ\in \MM_0(S)$.
\end{lemma} 

Using the same strategy as in the proof of the abelian scheme case (reduction to Jacobian, extension of the curve), one can also prove the following related vanishing result.

\begin{prop}\label{prop:partial_left_adjoint}
Let $S$ be a noetherian finite-dimensional scheme. Let $e:U\ra S$ be an \'etale morphism and $A/U$ be an abelian scheme. Then for all $n\in\BZ$, we have
\[
\DA(S)(\BQ,e_\sharp \Sigma^\infty A(-1)[n])=0.
\]
\end{prop}

We deduce an additional compatibility relation between the motivic $t$-structures on $0$ and $1$-motives.
\begin{cor}\label{cor:omega_1_exact}
Let $S$ be a noetherian finite-dimensional excellent scheme. The functor 
\[\omega^0:(\DA^1(S),t^1_{\MM})\ra (\DA^0(S),t^0_{\MM})\]
 is $t$-exact.  
\end{cor}

\begin{proof}
The functor $\omega^0:\DA^1(S)\ra \DA^0(S)$, defined as the restriction of $\omega^0$ to $\DA^1(S)$, is the right adjoint to the inclusion $\DA^0(S)\ra\DA^1(S)$. This inclusion is $t$-positive by looking at generators, which implies that its right adjoint $\omega^0$ is $t$-negative. 

It remains to show $\omega^0$ is $t$-positive. By Lemma~\ref{lemm:omega_sums}, $\omega^0$ commutes with small sums. It is thus enough to show that a family of compact generators of $\DA^1(S)$ is sent to $t$-positive objects. By Proposition~\ref{prop:generators_DA1}, $\DA^1(S)$ is compactly generated by $\DG_S(-1)$. Let $e:U\ra S$ be an \'etale morphism and $\BM=[L\ra G]\in \CM_1(U)$, and let $e_\sharp (\Sigma^\infty \BM) (-1)\in \DG_S(-1)$. We also write $A$ (resp. $T$) for the abelian (resp. torus) part of $G$. We have to be careful because $\omega^0$ and $e_\sharp\simeq e_!$ do not commute in general and we cannot apply directly Proposition~\ref{prop:omega_0_gr}~\ref{omega_0_M1}. However, we have distinguished triangles
\[
e_\sharp\Sigma^\infty T (-1) \ra e_\sharp\Sigma^\infty \BM (-1) \ra e_\sharp\Sigma^\infty W_{\geq -1}\BM (-1) \rap
\]
and
\[
e_\sharp\Sigma^\infty A(-1) \ra e_\sharp\Sigma^\infty W_{\geq -1} \BM(-1)\ra e_\sharp\Sigma^\infty L (-1) \rap.
\]
The motive $(e_\sharp \Sigma^\infty L)(-1)$ is in $\DA_0(S)(-1)$, which by Corollary~\ref{coro:omega_0}~\ref{omega_0_twists} implies that its $\omega^0$ vanishes. We will show that we have $\omega^0(e_\sharp \Sigma^\infty A(-1))\simeq 0$. Using the generating family of $\DA_{0}(S)$, we have to show that, for all $f:W\ra S$ \'etale and all $n\in\BZ$, we have $\DA(S)(f_\sharp\BQ[-n], e_\sharp \Sigma^\infty A(-1))=0$. By adjunction, the $\Ex^*_\sharp$ isomorphism and Proposition~\ref{prop:pullback_complex}, we can assume $f=\id$ and apply Proposition~\ref{prop:partial_left_adjoint}.

All together, this means that $\omega^0(e_{\sharp}\Sigma^\infty \BM(-1))\simeq \omega^0(e_\sharp\Sigma^\infty T (-1))\simeq e_{\sharp} \Sigma^\infty X_*(T)$ (Proposition~\ref{prop:omega_0_gr}~\ref{omega_0_tor}) which is $t$-positive for $t_{\MM,0}(S)$. This completes the proof. 
\end{proof}

Notice that at this point we do not know if the motivic $t$-structures restricts to compact objects. A weaker result in that direction is the following result.

\begin{cor}\label{cor:compact_bounded}
Let $S$ be a noetherian finite-dimensional excellent scheme. Any compact object in either $\DA_0(S)$, $\DA^1(S)$ or $\DA_1(S)$ is bounded for the motivic $t$-structure, i.e., it has only finitely many non-zero homology objects. 
\end{cor}
\begin{proof}
The argument is the same for the three categories, let us explain it for $\DA_1(S)$. Let $M\in \DA_{1,c}(S)$. Since $\DA_{1,c}(S)=\langle \DG_S\rangle$, the motive $M$ is obtained by a finite number of steps from objects of $\DG_S$ by by taking cones of morphisms, shifts and direct factors. Since $\DG_S$ lies in the heart (Theorem \ref{theo:gen_in_heart}), this implies immediately that $M$ is bounded.
\end{proof}

The proof of the following result is also very similar to the proof of Theorem \ref{theo:gen_in_heart}, hence we include it here.

\begin{prop}\label{prop:t_non_degenerate}
Let $S$ be a noetherian finite-dimensional excellent scheme. The $t$-structures $t_{\MM,0}(S)$, $t_{\MM,1}(S)$ and $t^1_{\MM}(S)$ are non-degenerate. 
\end{prop}

\begin{proof}
Since $t^1_{\MM}=t_{\MM,1}(-1)$, it is enough to treat the cases of $t_{\MM,0}$ and $t_{\MM,1}$. These $t$-structures are defined as generated $t$-structures. By \cite[Proposition 2.1.73]{Ayoub_these_2}, to show that a $t$-structure of the form $t(\CG)$ on a triangulated category $\CT$ for a family of compact objects $\CG$ is non-degenerate, it is enough to check that $\CT=\llangle \CG \rrangle$ and that for $A\in\CG$, there exists an integer $d_A\geq 0$ such that for all $B\in \CG$, $\Hom(A,B[n])=0$ for $n\geq d_A$. 

Let us check these conditions for $t_{\MM,0}$, using the generating family $\CG_0=\{e_\sharp \BQ| e:U\ra S\text{ \'etale}\}$. By definition, we have $\DA_0(S)=\llangle \CG_0\rrangle$. Let $e:U\ra S$ and $h:V\ra S$ be \'etale morphisms. We will prove that
\[
\forall n> \dim(S),\ \DA(S)(e_\sharp \BQ,h_\sharp\BQ[n])=0.
\]
 By the $(e_{\sharp},e^*)$ adjunction, we can assume $e=\id$. Using Zariski's main theorem, we compactify $h$ into $h=\bar{h}\circ j$ with $j:V\ra \overline{V}$ a dense open immersion and $\bar{h}:\overline{V}\ra S$ a finite morphism. Using the $(\bar{h}^*,\bar{h}_*)$ adjunction, we see that we can assume $h=j$ a dense open immersion. Notice that through these reductions, the dimension of the base does not increase. Write $i:Z\ra S$ for the complementary closed immersion to $j$. We have $\dim(Z)\leq \dim(S)-1$. By localisation and adjunction, we have an exact sequence
\[
\DA(Z)(\BQ,\BQ[n-1])\ra \DA(S)(\BQ, j_!\BQ[n])\ra \DA(S)(\BQ,\BQ[n])
\]
The two outer group vanish because of Proposition~\ref{mot_coh_gg} (noticing that $n-1>\dim(S)-1\geq\dim(Z)$), and this completes the proof that $t_{\MM,0}$ is non-degenerate.

We now look at the case of $t_{MM,1}$. Again by \cite[Proposition 2.1.73]{Ayoub_these_2} applied to the generating family $\JG_S$, it suffices to prove that
\[
\forall M,N \in \JG_S,\forall n>\dim(S)+4,\ \DA(S)(M,N[n])=0.
\]

Let us first concentrate on $M$. We have an \'etale morphism $e:U\rightarrow S$ and $M$ has the form $e_\sharp \Sigma^\infty (K\otimes\BQ)$ with $K$ one of $\BZ$, $\Gm[-1]$ or $\Jac(C/U)$ with $C$ a smooth projective curve with geometrically connected fibres and a section. We have $\dim(U)\leq \dim(S)$, hence by adjunction and Lemma~\ref{lem:gen_ops}, we can assume $e=\id$. Then, using Proposition~\ref{prop:Gm_Q1} and Corollary~\ref{coro:computation_curve}, in every case, we can write $M$ as a direct factor of the motive $M_S(C')[\epsilon]$ with $C'/S$ a smooth curve and $\epsilon\in\{0,-1\}$. By adjunction again, and taking into account that $\dim(C')\leq \dim(S)+1$, we are reduced to showing
\[
\forall N \in \JG_S,\forall n>\dim(S)+3,\ \DA(S)(\BQ_{S},N[n])=0.
\]
We now go into the case distinction for $N$. Let $p:V\ra S$ be an \'etale morphism. The motive $N$ is of one of the following forms: $p_\sharp \BQ$, $p_\sharp \BQ(1)$ or $p_\sharp\Jac(X/V)[-1]$ for $\pi:X\ra V$ a smooth projective curve with geometrically connected fibres and a section. By Zariski's main theorem, localisation and adjunction, we can assume $e=j$ is an open immersion (this does not change the dimension). In the first two cases, we apply the same argument as for $t_{\MM,0}$: by localisation, we can assume $p=\id$ and then apply Proposition~\ref{mot_coh_gg}. Let us focus on the Jacobian case. We write $\Jac(X/V)[-1]$ as direct factor of $M_S(X)[-1]$ by Corollary~\ref{coro:computation_curve}, then compactify $j\circ \pi=\bar{\pi}\circ \bar{\jmath}$ with $\bar{\jmath}:X\ra \overline{X}$ dense open immersion and $\bar{\pi}:\overline{X}\ra S$ a proper morphism using Zariski's main theorem. Writing $\bar{\imath}:Z\ra \overline{X}$ for the complementary closed immersion to $\bar{\jmath}$ and using localisation and relative purity, we have an exact sequence
\[
\DA(Z)(\BQ_Z, \BQ(1)_{Z}[n])\ra \DA(S)(\BQ_S, j_\sharp M_S(X)[-1]) \ra \DA(\overline{X})(\BQ_{\overline{X}}, \BQ_{\overline{X}}(1)[n+1]).
\] 
We have $\dim(Z),\dim(\overline{X})\leq \dim(S)+1$, hence the two outer groups vanish for $n>\dim(S)+3$ by Proposition~\ref{mot_coh_gg}. This completes the proof that $t_{\MM,1}$ is non-degenerate.
\end{proof}

Finally, we compute more precisely the morphisms between Deligne $1$-motives over a regular base.

\begin{theo}\label{theo:deligne_full}
Let $S$ be a regular excellent scheme, $\BM_1,\BM_2\in\CM_1(S)$ and $n\in\BZ$. Then
\[
\DA(S)(\Sigma^\infty\BM_1,\Sigma^\infty \BM_2[n])\simeq \left\{ \begin{array}{c}  0,\ n<0 \\ \CM_1(S)(\BM_1,\BM_2),\ n=0 \\ 0,\ n\geq 3. \end{array}\right.
\]
In particular, the functor $\Sigma^\infty:\CM_1(S)\ra \MM_1(S)$ is fully faithful.
\end{theo}
\begin{proof}
  By considering the connected components, we reduce to the case where $S$ is irreducible. The idea of the proof is that in the range we are considering, i.e., for $n\neq 1,2$, everything happens at the generic point $\eta$ of $S$. Let $j:U\ra S$ be an open immersion with $U \neq \emptyset$. The restriction functor $j^*:\CM_1(S)\ra \CM_1(U)$ is fully faithful by Proposition~\ref{prop:smooth_1_mot}. Moreover the category $\CM_1(\eta)$ is the $2$-colimit of the $\CM_1(U)$ for $U$ running through all non-empty open sets of $S$ by Proposition~\ref{prop:continuity_1-mot}. This implies that $\CM_1(S)(\BM_1,\BM_2)\simeq \CM_1(\eta)(\eta^* \BM_1,\eta^*\BM_2)$.

On the $\DA(-)$ side, by continuity and Proposition~\ref{prop:pullback_complex}, we have that
  \[\DA(\eta)(\eta^*\Sigma^\infty \BM_1,\eta^*\Sigma^\infty \BM_2[n])\simeq \Colim_{U \neq \emptyset} \DA(U)(j^* \Sigma^\infty \BM_1, j^*\Sigma^\infty \BM_2[n]).\] Furthermore, by Proposition~\ref{prop:morphisms_field}, we have an isomorphism 
\[\DA(\eta)(\Sigma^\infty \eta^* \BM_1,\Sigma^\infty \eta^* \BM_2[n])\simeq \Ext^n_{\CM_1(\eta)}(\BM_1,\BM_2)\stackrel{n\neq 0,1}{\simeq} 0.\]

Putting everything together, we see that the statement of the proposition follows from the claim that $j^*:\DA(S)(\Sigma^\infty \BM_1,\Sigma^\infty \BM_2[n])\ra \DA(U)(j^*\Sigma^\infty \BM_1, j^* \Sigma^\infty \BM_2[n])$ is bijective for $n\neq 1,2$. Write $i:Z\ra S$ for the reduced complementary closed immersion of $U$ in $S$. Consider the localisation exact sequence
\[
\xymatrix{
\ldots \ar[r] & \DA(Z)(i^*\Sigma^\infty \BM_1,i^!\Sigma^\infty \BM_2[n]) \ar[r] & \DA(S)(\Sigma^\infty \BM_1,\Sigma^\infty \BM_2[n]) \ar[d]^{j^*} \\ 
\ldots & \DA(Z)(i^*\Sigma^\infty \BM_1,i^!\Sigma^\infty \BM_2[n+1])\ar[l] & \DA(U)(j^*\Sigma^\infty \BM_1, j^*\Sigma^\infty \BM_2[n]) \ar[l] \\ 
}
\]
We have to prove the vanishing of $\DA(Z)(i^*\Sigma^\infty \BM_1,i^!\Sigma^\infty \BM_2[n+1])$ for $n\neq 2$. By Proposition~\ref{coro:pullback_complex_DA}, we have $i^*\Sigma^\infty \BM_1\simeq \Sigma^\infty \BM_{1,Z}$. Using that any closed subscheme of $S$ is excellent and reduced and thus has a non-empty open regular locus, we can stratify $Z$ by regular constructible subschemes and applying further localisations, we can reduce to the case where $Z$ is also regular of some codimension $1+e$ with $e\geq 0$. By absolute purity in the form of Proposition~\ref{prop:sm_abs_purity}, which applies by~Corollary~\ref{cor:Deligne_da_1}, we then have $i^!\Sigma^\infty \BM_2[n+1]\simeq i^*\Sigma^\infty \BM_2(-1-e)[n-1-2e]\simeq \Sigma^\infty \BM_{2,Z}(-1-e)[n-1-2e]$. We know, again from Corollary~\ref{cor:Deligne_da_1}, that the motive $\Sigma^\infty \BM_{1,Z}(-1)$ lies in $\DA^1(S)$, hence we have an isomorphism
\begin{align*}
  \DA(Z)(\Sigma^\infty \BM_{1,Z}, \Sigma^\infty\BM_{2,Z}(-1-e)[n-1-2e]) \\
\simeq  \DA(Z)(\Sigma^\infty \BM_{1,Z}(-1), \omega^1(\Sigma^\infty\BM_{2,Z}(-2-e)[n-1-2e])). 
\end{align*}  
The motive $\Sigma^\infty\BM_{2,Z}(-1)$ is cohomological, so by Corollary~\ref{coro:omega_0} the group on the right hand side vanishes unless $e=0$. If $e=0$, we have further $\omega^1(\Sigma^\infty\BM_{2,Z}(-1)(-1))\simeq \omega^0(\Sigma^\infty\BM_{2,Z}(-1))(-1)$. This motive was computed in Proposition~\ref{prop:omega_0_gr}~\ref{omega_0_M1} and we get \[\omega^0(\Sigma^\infty\BM_{2,Z}(-1))(-1)\simeq \Sigma^\infty X_*(W_{-2}\BM_{2,Z})(-1).\] 

To sum up, we are reduced to showing that for $S$ regular, $\BM\in\CM_1(S)$ and $L$ lattice over $S$, the morphism group $\DA(S)(\Sigma^\infty\BM,\Sigma^\infty L_\BQ[n-1])$ is $0$ for $n\neq 2$. Since $S$ is normal, the motive $\Sigma^\infty L_\BQ$ is a direct factor of $e_* \BQ$ for $e:T\ra S$ finite \'etale (Lemma~\ref{lemm:permutation_torus}). By adjunction, we are then reduced to the case $L=\BZ$. Write $\BM=[N\ra G]$ with $N$ a lattice and $G$ a semi-abelian scheme. We have a distinguished triangle 
\[\Sigma^\infty [0\ra G]\ra \Sigma^\infty \BM\ra \Sigma^\infty [N\ra 0]\rap\]
which shows that we can treat separately the cases $\BM =[N\ra 0]$ and $\BM = [0\ra G]$.

In the case $\BM=[N\ra 0]$, we again write $N$ as a direct factor of a permutation lattice, which implies that $\Sigma^\infty \BM$ is a direct factor of $e'_\sharp\BQ$ with $e':T'\ra S$ finite \'etale. By adjunction, we are then reduced to a computation of weight zero motivic cohomology on a regular scheme, which vanishes for $n\neq 2$ by Propositions~\ref{prop:mot_coh_<} and~\ref{prop:mot_coh_0}.

In the second case, we have $\Sigma^\infty\BM =\Sigma^\infty G_\BQ[-1]$, which by \cite[Theorem 3.3]{AHPL} is a direct factor of $M_S(G)$. We are then done using the $((G\ra S)_\sharp, (G\ra S)^*)$ adjunction and Propositions~\ref{prop:mot_coh_<} and~\ref{prop:mot_coh_0}.
\end{proof}

\appendix

\section{Deligne $1$-motives}
\label{sec:app_deligne}

% TODO:
% - Torsion 1-motives
% - Examples via generalized Jacobians/Picard functors, BVS
% - Finite push-forward via Weil restriction ? K/k-image/trace for primary extensions ? Should be there at least for 1-motives with torsion/iso motives.
% - Polarisations
% - Cartier duality
% - Biextensions and Cartier duality (\`a la BVK 4.1) 
% - Multilinear morphisms à la Bertolin

We gather necessary results on Deligne $1$-motives \cite[\S 10]{Deligne_HodgeIII} over general base schemes which we could not find in the literature. Useful references besides Deligne's original work are \cite{Jossen_thesis}, \cite[Appendix C]{BVK}. 

\subsection{Definitions}

\begin{defi}\label{def:gr_schemes}
Let $S$ be a scheme. We say that a commutative group scheme $G/S$ is
\begin{enumerate}[label={\upshape(\roman*)}]
\item \emph{discrete} if it is \'etale locally constant finitely generated.
\item a \emph{lattice} if it is discrete and torsion free.
\end{enumerate}
A lattice of the form $f_{*}\BZ$ for $f$ a finite \'etale morphism is called a \emph{permutation lattice}.
\end{defi}

Before we come to the definition of Deligne $1$-motives, let us discuss a recurrent technical point about lattices and tori over general schemes. In general, it is not the case that a discrete group scheme is isotrivial in the \'etale topology. However, we have the following useful lemma.

\begin{lemma}\label{lemm:permutation_torus}
Let $S$ be a locally noetherian, geometrically unibranch scheme. Let $L$ be a lattice over $S$ (resp. $T$ be a torus over $S$). 
\begin{enumerate}[label={\upshape(\roman*)}]
\item $L$ (resp. $T$) is isotrivial, i.e., it becomes split after passing to a finite \'etale cover of $S$.   
\item The sheaf $L\otimes\BQ\in \Sh(\Sm/S)$ (resp. $T\otimes\BQ\in \Sh(\Sm/S)$ is a direct factor of the sheaf $f_*\BQ$ (resp. $f_*(\Gm\otimes\BQ)$) for $f:V\ra S$ a finite \'etale cover.
\end{enumerate}
\end{lemma}
\begin{proof}
Point (i) for lattices follows from the discussion in \cite[Exp. X 6.2]{SGA3_2_old}. For tori, it is precisely \cite[Exp. X Th\'eor\`eme 5.16]{SGA3_2_old}. 

We now prove Point (ii). Let $L$ be a lattice over $S$. By (i), we can find a finite \'etale cover $g:V\ra S$ such that $g^*L$ is split, say $g^*L\simeq \BZ^r$. Because $g$ is finite \'etale, $L$ becomes a direct factor of $g_*g^*L$ after inverting $\deg(f)$ by a transfer argument. We thus have that $L\otimes\BQ$ is a direct factor of $g_*g^*L\otimes\BQ\simeq g_*\BQ^r$. Write $f:V^{\coprod r}\ra $ for the coproduct  of $r$ copies of $g$. Then $g_*\BQ^{\coprod r}\simeq f_*\BQ$. This concludes the proof of (ii) for lattices. The case of tori follows by the same argument.
\end{proof}

\begin{defi}\label{def:1-mot}
  Let $S$ be a scheme. A $2$-term complex of commutative $S$-group schemes
  \[ M=[L\longrightarrow G]
  \]
is called a \emph{Deligne $1$-motive} over $S$ if $L$ is a lattice and $G$ is a semi-abelian scheme. A morphism of Deligne $1$-motives is a morphism of complexes of group schemes, or equivalently a morphism of complexes of the associated representable sheaves on $(\Sm/S)_\et$. We denote by $\CM_1(S,\BZ)$ the category of Deligne $1$-motives. It is an additive category, with biproducts induced by fibre products of $S$-group schemes.

A Deligne $1$-motive $M=[L\ra G]$ comes with a $3$-term functorial weight filtration, defined as
\[W_{-2}M=[0\longrightarrow T]\]
\[W_{-1}M=[0\longrightarrow G]\]
\[W_{0}M=M.\]

\end{defi}

\begin{notation}
Let $f:[L\rightarrow G]\rightarrow [L'\rightarrow G']$ be a morphism of Deligne $1$-motives. We use the notation $f_L$, $f_G$, $f_A$, $f_T$ for the induced maps $\Gr^W_0f:L\ra L'$, $W_{-1} f:G\ra G'$, $\Gr^W_{-1} f:A\ra A'$, $\Gr^W_{-2} f:T\ra T'$.
\end{notation}

\begin{defi}

Let $f:S'\rightarrow S$ be any morphism of schemes. Then pullback of $S$-group schemes along $f$ induces an additive functor
\[
f^*:\CM_1(S,\BZ)\rightarrow \CM_1(S',\BZ).
\]
\end{defi}

We are not so much interested in $1$-motives per se but rather in the objects they define in the derived category of sheaves with rational coefficients.

\begin{lemma}
  \label{lemma:qiso_Deligne}
 Any morphism in $\CM_1(S,\BZ)$ which induces a quasi-isomorphism of complexes of abelian sheaves on $(\Sm/S)_{\et}$ is an isomorphism.
\end{lemma}
\begin{proof}
  Let $f=(f_L,f_G):[L_1\longrightarrow G_1]\rightarrow
  [L_2\longrightarrow G_2]$ be a quasi-isomorphism of complexes of
  sheaves. By a diagram chase, this is equivalent to $\Ker(f_L)\simeq
  \Ker(f_G)$ and $\Coker(f_L)\simeq \Coker(f_G)$. Since $\Ker(f_L)$ is
  locally constant finitely generated free and $\Ker(f_G)$ is a group
  scheme whose identity component is semi-abelian and with finite
 $\pi_0$, they must be both trivial. Similarly, $\Coker(f_L)$ is
  discrete and $\Coker(f_G)$ has connected fibres, so they must be
  both trivial. Hence $f$ is an isomorphism.
\end{proof}
We can consequently think of $\CM_1(S,\BZ)$ as a full subcategory of
$D(\Cpl(\Sh((Sm/S)_{\et},\BZ)))$.

\begin{defi}
Let $S$ be a noetherian scheme. We write $\CM_1(S)$ for the idempotent completion of the $\BQ$-linear category $\CM_1(S,\BZ)\otimes \BQ$.  We say that a morphism in $\CM_1(S)$ is integral if it comes from $\CM_1(S,\BZ)$. For $f:S'\rightarrow S$ morphism of schemes, we still write $f^*$ for the induced additive functor $\CM(S)\ra \CM(S')$.
\end{defi}

By the above, we can and do view $\CM_1(S)$ as a full subcategory of $D(\Cpl(\Sh((\Sm/S)_\et,\BQ)))$. In practice, the idempotent completion in the definition does not affect anything that we do in this paper, and we will allow ourselves statements of the form ``Let $\BM=[L\ra G]\otimes\BQ$ be an object in $\CM_1(S)$'' without spelling out the immediate reduction to that case. If $k$ is a field, the category $\CM_{1}(k,\BZ)\otimes\BQ$ is idempotent complete since it is an abelian category \cite[Lemme 3.2.2]{Orgogozo}.

\subsection{Continuity and smoothness}

We think of Deligne $1$-motives as "$1$-motivic local systems" over the base
$S$. The results in this section, which have classical analogues for local
systems and lisse $\ell$-adic sheaves, justify in part this intuition.

We start with a lemma about discrete group schemes.

\begin{lemma}
\label{lemm:continuity_discrete}
Let $S$ be a locally noetherian japanese scheme, $\eta$ its scheme of generic points. Then the category of discrete group schemes on $\eta$ is the 2-colimit of the categories of discrete group schemes on dense open subschemes of $S$. The same statement holds for the category of lattices.
\end{lemma}
\begin{proof}
The statement is equivalent to the following results.
\begin{enumerate}[label={\upshape(\roman*)}]
\item For $L/\eta$ discrete group scheme, there exists $U\subset S$ dense open and $L'/U$ discrete such that $L\simeq \eta^*L'$. Moreover, if $L$ is a lattice, one can choose $L'$ to be a lattice as well.
\item For $U\subset S$ dense open, $L,L'/U$ discrete, we have
\[
\Hom(\eta^*L,\eta^*L')\simeq \Colim_{V\subset U}\Hom((V\ra U)^*L,(V\ra U)^*L').
\]
\end{enumerate}
We first make some reductions which apply both to (i) and (ii). By the topologically invariance of the \'etale site, we can assume $S$ to be reduced. Since $S$ is locally noetherian japanese and reduced, the normal locus of $S$ is open and non-empty \cite[Proposition~6.13.2]{EGAIV_2}. So any small enough open set $U$ in $S$ is normal, and in particular geometrically unibranch. By the discussion in \cite[Exp. X 6.2]{SGA3_2_old}, discrete group schemes on geometrically unibranch schemes are split by finite \'etale covers. Moreover, for any small enough open set $U$, the set of connected components (open by local noetherianness) of $U$ and of $\eta$ coincide. We can thus reduce to the case where $\eta$ is connected (i.e., $S$ irreducible).

We prove (i). Since $\eta$ itself is normal, there is a finite \'etale Galois cover $\tilde{\eta}/\eta$ such that $L_{\tilde{\eta}}$ is constant. In other words, $L$ corresponds to a representation $\rho$ of $\Gal(\tilde{\eta}/\eta)$ on a finitely generated abelian group $F$. By \cite[Th\'eor\`eme 8.8.2, Th\'eor\`eme 8.10.5]{EGAIV_3} and \cite[Th\'eor\`eme 17.7.8]{EGAIV_4} there exists a $U\subset S$ dense open and $\widetilde{U}/U$ finite \'etale such that $\widetilde{U}\times_U \eta\simeq \tilde{\eta}$. Up to shrinking $U$, one can assume $U$ to be normal. By \cite[Th\'eor\`eme 8.8.2]{EGAIV_3} applied to the finite group $\Gal(\tilde{\eta}/\eta)$, up to shrinking $U$ one can assume that $\Aut(\widetilde{U}/U)\simeq \Gal(\tilde{\eta}/\eta)$ ( in particular $\widetilde{U}/U$ is Galois). Then the representation of $\Gal(\widetilde{U}/U)$ on $F$ corresponding to $\rho$ via this isomorphism defines a discrete group scheme $L'/U$ such that $L\simeq \eta^*L'$ as required. The addendum about lattices follows from the construction, i.e., $L'$ is a lattice if $L$ is.

We now prove (ii). Let $U\subset S$ be a dense open subset, $L,L'/U$ discrete group schemes. We can shrink $U$ and assume it is normal. Let $\tilde{U}/U$ be a finite \'etale Galois covering trivializing $L$ and $L'$. We thus get two finitely generated abelian groups $F,F'$ with representations $\rho,\rho'$ of $\Gal(\tilde{U}/U)$. Let $\tilde{\eta}:=\tilde{U}\times_U \eta$. Then $\tilde{\eta}/\eta$ is Galois with  $G:=\Gal(\tilde{U}/U)\simeq \Gal(\tilde{\eta}/\eta)$. Then the system in the right-hand side of (ii) is constant and both sides of (ii) are in bijection with $\Hom_{G}(\rho,\rho')$. This concludes the proof.
\end{proof}

\begin{remark}
It is not clear to the author how to extend this result to a more general continuity result for discrete group schemes on a projective limit of schemes with affine transition morphisms.
\end{remark}

We deduce from this a continuity result for Deligne 1-motives.

\begin{prop}
\label{prop:continuity_1-mot}
  Let $S$ be a locally noetherian japanese scheme, $\eta$ its scheme of generic points. Then the category $\CM_1(\eta,\BZ)$ (resp. $\CM_1(\eta)$) is the 2-colimit of the categories $\CM_1(U,\BZ)$ (resp. $\CM_1(U)$) for all dense open subsets $U\subset S$.
\end{prop}
\begin{proof}
The case of $\CM_1(-)$ follows directly from the one of $\CM_1(-,\BZ)$. We have to show that 
  \begin{enumerate}[label={\upshape(\roman*)}]
  \item for all $M\in \CM_1(\eta,\BZ)$, there exists $U\subset S$ dense open and
    $M'\in \CM_1(U,\BZ)$ such that $M\simeq \eta^*M'$, and that
  \item for all $U\subset S$ dense open and all $M,N\in\CM_1(U,\BZ)$:
    \[
    \CM_1(\eta,\BZ)(\eta^*M,\eta^*N)\simeq \Colim_{V\subset
      U}\CM_1(V,\BZ)((V\ra U)^*M,(V\ra U)^*N).
    \]
  \end{enumerate}
 We prove (i). Let $M=[L\rightarrow
  G]\in\CM_1(\eta,\BZ)$ with the extension $0\rightarrow T\rightarrow
  G\rightarrow A\rightarrow 0$.

  By \cite[Th\'eor\`eme 8.8.2.(ii), Scholie 8.8.3, Th\'eor\`eme
  8.10.5.(xii)]{EGAIV_3} and \cite[Proposition~17.7.8]{EGAIV_4}, we
  can find an $U\subset S$ and a smooth group scheme $G'/U$ such
  that $G\simeq G'\times_{U}\eta$. Recall that an abelian scheme is by
  definition a smooth proper group scheme with connected fibres, hence by \cite[Th\'eor\`eme
  8.8.2.(ii), Scholie 8.8.3, Th\'eor\`eme 8.10.5.(xii)]{EGAIV_3} and
  \cite[Proposition~17.7.8]{EGAIV_4}, we can shrink $U$ and find an
  abelian scheme $A'/U$ such that $A\simeq
  A'\times_{U}\eta$. By Lemma \ref{lemm:continuity_discrete} and the duality between lattices and tori, we can shrink $U$ and assume that there exists a lattice $L'$ and a torus $T'$ over $U$ such that $L\simeq L'\times_U\eta$ and $T\simeq T'\times_U\eta$. 

We have spread out the pure pieces of $M$, it remains to glue them together. By \cite[Th\'eor\`eme 8.8.2.(i)]{EGAIV_3}, up to shrinking $U$, we have morphisms $T'\ra G'\ra A'$ which restrict to the extension defining $G$. By a standard argument based on \cite[Th\'eor\`eme 8.10.5]{EGAIV_3}, up to shrinking $U$, this is in fact an exact sequence of group schemes. Finally, we have to spread out the morphism $L\ra G$. This can be done by the same Galois descent argument as in the end of the proof of Lemma \ref{lemm:continuity_discrete}.

Let us now prove (ii). In $\CM_1(-,\BZ)$, the components of a morphism are morphisms of (group) schemes. It is enough to spread them out one by one because the resulting diagram will commute by schematic density of $\eta$ in $S$. We have treated morphisms of lattices in Lemma \ref{lemm:continuity_discrete}. The case of morphisms of semi-abelian schemes (which are in particular of finite presentation) is a direct application of \cite[Th\'eor\`eme 8.8.2.(i)]{EGAIV_3}.
\end{proof}

When the base scheme is moreover reduced or even normal, we can say more.
\begin{prop}
\label{prop:smooth_1_mot}
  Let $S$ be a locally noetherian japanese scheme, $i:\eta\rightarrow S$ its scheme of generic
  points.
  \begin{enumerate}[label={\upshape(\roman*)}]
  \item Suppose $S$ reduced. Then the pullback functor $\eta^*:\CM_1(S,\BZ)\ra \CM_1(\eta,\BZ)$ (resp. $\eta^*:\CM_1(S)\rightarrow \CM_1(\eta)$) is faithful.
  \item Suppose moreover that $S$ is noetherian and normal. Then $\eta^*$ is fully faithful.
  \end{enumerate}
\end{prop}

\begin{proof}
Let us prove (i). By Proposition~\ref{prop:continuity_1-mot} this is equivalent to the faithfulness of the functor $j^*$ for all $j:U\rightarrow V$ dense
 open immersions. It is enough to show faithfullness of $j^*$ separately for morphisms of discrete group schemes and semi-abelian schemes, and in both cases it follows from the "reduced to separated" uniqueness criterion \cite[Lemme 7.2.2.1]{EGAI}.

  We now prove (ii). By Proposition~\ref{prop:continuity_1-mot}, it is enough to prove fullness for the functor $j^*$ for all dense open immersions $j:U\rightarrow V$. Let
$M=[L\stackrel{u}{\rightarrow}G]$, $M'=[L'\stackrel{u'}{\rightarrow}G']\in\CM_1(V,\BZ)$ and $f_U=(f^L_U,f^G_U):j^*M\rightarrow j^*M'$. First, using the normality of $V$ and \cite[Expos\'e I Corollaire 10.3]{SGA1}, the morphism $f^L_U$ extends uniquely to a morphism $f^L:L\rightarrow L'$. Second, using the normality of $V$ and \cite[Proposition~2.7]{Faltings_Chai}, the morphism $f^G_U$ extends uniquely to a morphism $f^G:G\rightarrow G'$. The uniqueness ensures that $(f^L,f^G)$ is a morphism $M\rightarrow M'$ which extends $f_U$.
\end{proof}

\subsection{Pushforward and Weil restriction}
\label{sec:Weil_Deligne}
Let $g:S'\rightarrow S$ be a finite \'etale morphism. We want to define a pushforward functor $g_*:\CM_1(S')\ra\CM_1(S)$ using Weil restriction of scalars. Recall the following definition.

\begin{defi}\label{def:Weil}
  Let $g:S'\rightarrow S$ be a morphism of schemes and $X/S'$ be a $S'$-scheme. The \emph{Weil restriction} $\Res_gX$ is the presheaf of sets on $\Sch/S$ defined for any $S$-scheme $U$ by
\[
\Res_gX(U)=X(U\times_S S').
\]
\end{defi}

If $X/S'$ is a commutative group scheme (or more generally an fppf sheaf of abelian groups on $\Sch/S$), then $\Res_gX$ is naturally an fppf sheaf of abelian groups on $\Sch/S$. Moreover, the construction of $\Res_g$ is functorial and compatible with base change on $S$. We summarise results on the representability of $\Res_g X$ from the litterature.

\begin{prop}\label{prop:Weil}
Let $g:S'\rightarrow S$ be a morphism of schemes and $X/S'$ be a $S'$-scheme.
  \begin{enumerate}[label={\upshape(\roman*)}]
  \item \cite[Theorem~1.5]{Olsson_restriction} Assume that $g$ is proper flat of finite presentation. Then $\Res_gX$ is representable by an algebraic space (note that we will only need the case $g$ finite flat, which is presumably easier, but we could not find a reference).
 \item \cite[7.6/5]{BLR} Assume that $g$ is finite flat. If $X$ is smooth (resp. of finite presentation) then $\Res_gX$ (which exists at least as an algebraic space by (i)) is smooth (resp. of finite presentation).
  \item \cite[7.6/5]{BLR} Assume that $g$ is finite \'etale. If $X$ is proper then $\Res_gX$ (which exists at least as an algebraic space by (i)) is proper.
  \item \cite[7.6/2]{BLR} Let $h:X\rightarrow Y$ be a closed immersion of $S'$-schemes. Then $\Res_gh:\Res_gX\rightarrow \Res_gY$ is a closed immersion of presheaves. As a corollary, if $X/S$ if affine, then $\Res_gX$ is representable by an affine scheme.
\end{enumerate}
\end{prop}

We now use the results above to analyse Weil restriction of pure $1$-motives. We are spared from having to consider algebraic spaces by the following result.

\begin{prop}\label{prop:weil_pieces}
  Let $g:S'\rightarrow S$ be a finite \'etale morphism. Let $T/S'$ be a torus (resp. $L/S'$ be a lattice, $A/S'$ an abelian scheme). Then $\Res_g T$ is a torus (resp. $\Res_g L$ is a lattice, $\Res_g A$ is an abelian scheme).
\end{prop}
\begin{proof}
  By Proposition~\ref{prop:Weil} (iv), we know that $\Res_g T$ is representable by a affine $S'$-group scheme. To show that it is a torus, because Weil restriction is compatible with base change, it is enough to show this \'etale locally on $S$, so that we can assume that $S'=\cup_{i=1}^n S\ra S$, and then $\Res_g T=\prod_{i=1}^n T$ is clearly a torus.

A lattice is in particular an \'etale group scheme over $S$, and \'etale morphisms (including not of finite type) satisfy effective descent in the \'etale topology (see e.g. the reference \cite[Theorem 5.19]{Rydh_descent} which proves a much stronger result) hence to show that $\Res_g L$ is a lattice it is enough to show this \'etale locally on $S$, so we can once more assume that $S'=\cup_{i=1}^n S\ra S$, and then $\Res_g L=\prod_{i=1}^n L$ is again clearly a lattice.
  
  By Proposition~\ref{prop:Weil} (i)-(iii), we know that $\Res_gA$ is representable by a proper smooth algebraic group space over $S$. By \cite[Theorem~1.9]{Faltings_Chai}, this implies that $\Res_g A$ is an abelian scheme. 
\end{proof}

Now we tackle the case of semi-abelian schemes.

\begin{lemma}\label{lemm:weil_exact}
Let $g:S'\ra S$ be a morphism of schemes. 
\begin{enumerate}[label={\upshape(\roman*)}]
\item \label{left_exact} When restricted to fppf sheaves of abelian groups, the functor $\Res_g$ is left exact.
\item \label{smooth_surj} Assume that $g$ is finite flat. Let $f:G\ra H$ be a smooth and surjective morphism between commutative groups schemes of finite presentation. Then the morphism of algebraic group spaces $\Res_g f:\Res_g G\ra \Res_g H$ is smooth and surjective.
\item \label{weil_exact} Assume $g$ is finite flat. Let $0\ra G'\stackrel{i}{\ra} G\stackrel{p}{\ra} G''\ra 0$ be an exact sequence of smooth commutative $S$-group schemes with $G\ra G''$ flat (and hence smooth). The sequence \[0\ra \Res_g G'\ra \Res_g G\ra \Res_g G''\ra 0\] is exact.  
\end{enumerate}
\end{lemma}

\begin{proof}
Point \ref{left_exact} is clear from the definition. We turn to point \ref{smooth_surj}. The fact that $\Res_g f$ is smooth follows from the infinitesimal criterion of smoothness (and does not require that we are working with group schemes). The surjectivity can be tested pointwise on $S$, so that by compability of $\Res_g$ with base change we can assume that $S$ is the spectrum of a field $k$. Surjectivity is a geometric property, so that we can assume $k$ to be algebraically closed as well. We then have to check the surjectivity of the induced map $\Res_g G(k)=G(S')\ra \Res_g H(k)=H(S')$ on $k$-points. Since $S'/k$ is finite flat, it is a product of finite local algebras. Surjectivity then follows from the surjectivity of $f$, the fact that $k$ is algebraically closed, and the formal smoothness of $f$. Note that if $g$ is finite \'etale, we do not need $f$ smooth.

For \ref{weil_exact}, it is enough to check that $\Res_g G'$ is the scheme-theoretic kernel of $\Res_g p$ and that $\Res_g p$ is an fppf morphism. The first assertion follows from \ref{left_exact}, and the second from \ref{smooth_surj}.
\end{proof}

\begin{prop}\label{prop:weil_sab}
  Let $g:S'\rightarrow S$ be finite \'etale and $G/S$ be a semi-abelian scheme. Then $\Res_g G$ is an semi-abelian scheme.
\end{prop}
\begin{proof}
The result follows directly from Proposition~\ref{prop:weil_pieces} and Lemma \ref{lemm:weil_exact} \ref{weil_exact}.
\end{proof}

\begin{defi}\label{defi:pushforward_deligne}
  Let $g:S'\rightarrow S$ be a finite \'etale morphism. We define the Weil restriction of a Deligne $1$-motive $M=[L\stackrel{u}{\ra} G]\in \CM^\BZ_1(S')$ as $\Res_g M = [\Res_g L\stackrel{\Res_g u}{\ra} \Res_g G]$ which is in $\CM_1(S)$ by Propositions \ref{prop:weil_pieces} and \ref{prop:weil_sab}. This induces a functor
\[
g_*:\CM^{\BZ}_1(S')\rightarrow \CM^{\BZ}_1(S).
\]
\end{defi}

\section{Motivic cohomology in degrees $(*,\leq 1)$}
\label{sec:app_mot_coh}

We gather here some computations of rational motivic cohomology groups which are used at various places in this paper. Most of the following is present, explicitely or implicitely, in \cite[\S 11]{Ayoub_Etale} and in the K-theoretic interpretation of rational motivic cohomology provided by the comparison with Beilinson motives \cite[\S 14]{Cisinski_Deglise_BluePreprint}.

\begin{notation}
Let $S$ be a noetherian finite-dimensional scheme. For $p,q\in\BZ$, we write $H^{p,q}_\CM(S):= \DA(S)(\BQ_S,\BQ_S(q)[p])$.
\end{notation}

\begin{prop} \cite[Proposition~11.1 (b)]{Ayoub_Etale}
\label{prop:mot_coh_<}
Let $S$ be a noetherian finite-dimensional scheme. For all $w<0$ and $n\in \BZ$, we have $H^{n,w}_\CM(S)\simeq 0$.
\end{prop}

\begin{prop}\label{mot_coh_gg}
Let $S$ be a noetherian finite-dimensional quasi-excellent scheme (respectively, a regular and finite-dimensional scheme). For all $i\in \BN$ and $n>\dim(S)+2i$ (resp. $n>2i$), we have $H^{n,i}_\CM(S)\simeq 0$.
\end{prop}
\begin{proof}  
The group $H^{n+2i,i}_\CM(S)\simeq \DA(S)(\BQ[n],\BQ(i)[2i])$ is a direct factor of $\DA(S)(\BQ[n],\sum_{i\in \BZ}\BQ(i)[2i])$. By Theorem~\cite[16.2.18]{Cisinski_Deglise_BluePreprint}, this group is isomorphic to $\DM_{\Bei}(S)(\BQ[n],\sum_{i\in\BZ}\BQ(i)[2i])$ where $\DM_{\Bei}(S)$ is the triangulated category of Beilinson motives. By Corollary~\cite[14.2.17]{Cisinski_Deglise_BluePreprint}, we have $\bigoplus_{i\in\BZ} \BQ(i)[2i]\simeq \mathrm{KGL}_{\BQ,S}$ where the last object is the $\BQ$-localisation of the motivic spectrum $\mathrm{KGL}_S$. This implies that 
\[
\DM_{\Bei}(S)(\BQ[n],\mathrm{KGL}_{\BQ,S})\simeq \SH(S)(\Sigma^n \Sigma^\infty_T (S_+),\mathrm{KGL})\otimes\BQ.
\] 
By \cite[Th\'eor\`eme 2.20]{Cisinski_KGL}, this last group is isomorphic to $\mathrm{KH}_n(S)\otimes\BQ$, where $\mathrm{KH}$ is homotopy-invariant $K$-theory. The negative homotopy-invariant $K$-theory of a regular scheme vanishes, and this implies the resp. case. Finally, by the main step in the proof of \cite[Theorem 3.5]{Shane_vanishing}, under our hypotheses on $S$ (including quasi-excellent), the group $\mathrm{KH}_n(S)\otimes\BQ$ vanishes for $n<-\dim(S)$. This completes the proof.
\end{proof}

\begin{remark}
For the cases $i=0,1$, it is likely that there is a non-$K$-theoretic proof, combining results below on $H^{n,0}_\CM$, $H^{n,1}_\CM$ of regular schemes with an ingenious use of resolution of singularities by alterations as in the proof of \cite[Theorem 3.5]{Shane_vanishing}. 
\end{remark}

Let $S$ be a scheme. Then we have $D(\Sm/S)(\BQ_S,\BQ_S)\simeq \BQ^{\pi_0(S)}$ (with $\pi_0(S)$ the set of connected components of $S$). This provides a morphism
\[
\nu^{0,0}:\BQ^{\pi_0(S)}\simeq D(\Sm/S)(\BQ_S,\BQ_S)\lra \DA(S)(\BQ_S,\BQ_S)=H^{0,0}_\CM(S).
\]
More generally, we have for all $n\in\BZ$ a morphism
\[
\nu^{n,0}:D(\Sm/S)(\BQ_S,\BQ_S[n]) \lra H^{n,0}_\CM(S)
\]
coming from the construction of $\DA^{(\eff)}$. 

Let $f:T\ra S$ be any morphism of schemes. Then it is easy to see that the diagram
\[
  \xymatrix{
D(\Sm/S)(\BQ_S,\BQ_S[n]) \ar[r]_(.65){\nu^{n,0}_T} \ar[d]_{f^*} & H^{n,0}_\CM(S) \ar[d]_{f^*}\\
D(\Sm/T)(\BQ_T,\BQ_T[n]) \ar[r]_(.65){\nu^{n,0}_S} & H^{n,0}_\CM(T)
    }
  \]
is commutative. We will use this fact without comment in the proof below.

% \begin{lemma}
%   Suppose we have a commutative diagram of schemes
%   \[
% \xymatrix{    
%     T \ar[r] \ar[d] & Y \ar[d] \\
%     X \ar[r] & S.
%     }
%   \]
% \begin{enumerate}[label={\upshape(\roman*)}]
%   \item The complex
%   \[
% 0\ra \BQ^{\pi_0(S)} \ra \BQ^{\pi_0(X)}\oplus \BQ^{\pi_0(Y)} \ra \BQ^{\pi_0(T)}
% \]
% is exact if and only if the map $\pi_0(X)\cup_{\pi_0(T)}\pi_0(Y)\ra \pi_0(S)$ is an isomorphism.
% \item The condition in $(i)$ holds in the following cases.
%   \begin{enumerate}
%   \item 
%   \end{enumerate}
%   \end{enumerate}
% \end{lemma}
% \begin{proof}
% Point $(i)$ is immediate. To see that Point $(ii)$ holds **  
% \end{proof}

\begin{prop}
\label{prop:mot_coh_0}\
\begin{enumerate}[label={\upshape(\roman*)}]
\item \label{mot_coh_0_<} For all $n<0$, we have $H^{n,0}_\CM(S)\simeq 0$.
\item \label{mot_coh_0_0} The morphism $u^{0,0}$ induces an isomorphism $H^{0,0}_\CM(S)\simeq \BQ^{\pi_0(S)}$.
\item \label{mot_coh_0_geq} Assume $S$ regular. For all $n>0$, we have $H^{n,0}_\CM(S)\simeq 0$.
\item \label{mot_coh_0_pullback} Let $f:T\ra S$ be a smooth surjective morphism with geometrically connected fibres. Then for all $n\in \BZ$, we have $f^*:H^{n,0}_\CM(S)\stackrel{\sim}{\longrightarrow} H^{n,0}_\CM(T)$.
\end{enumerate}

\end{prop}

\begin{proof}
  Statement \ref{mot_coh_0_<} and \ref{mot_coh_0_0} are proved in \cite[Proposition~11.1 (a)]{Ayoub_Etale}.

Let us prove Statement \ref{mot_coh_0_geq}. Fix $n>0$. We can assume that $S$ is connected with generic point $\eta$. By the argument at the beginning of the proof of \cite[Corollaire 11.4]{Ayoub_Etale}, combining absolute purity and localisation with the vanishing of negative motivic cohomology (Proposition \ref{prop:mot_coh_<}), one can deduce that for any dense open set $U$ in $S$, the restriction map $H^{n,0}_\CM(S)\ra H^{n,0}_\CM(U)$ is injective. By the continuity property of \cite[Proposition~3.20]{Ayoub_Etale}, we deduce that the restriction map $H^{n,0}_\CM(S)\ra H^{n,0}_\CM(\eta)$ is injective. So we are reduced to the case where $S$ is the spectrum of a field $k$.

 By separation, we can assume that $k$ is perfect. By \cite[Corollary~16.2.22]{Cisinski_Deglise_BluePreprint}, we reduce to compute $\DM(k,\BQ)(\BQ_k,\BQ_k[n])$. By the cancellation theorem \cite{Voevodsky_Cancellation}, we reduce to compute $\DM^\eff(k,\BQ)(\BQ_k,\BQ_k[n])$. Since the sheaf with transfers $\BQ_k$ is both cofibrant and $\BA^1$-local, this coincides with the same Hom group computed in the derived category of \'etale sheaves with transfers over $\Sm/S$, which vanishes. This concludes the proof of \ref{mot_coh_0_geq}.

 Let us prove Statement \ref{mot_coh_0_pullback}. By Mayer-Vietoris, we can assume $S$ to be affine. By a limit argument using the continuity property of $\DA$, we can then assume that $S$ is of finite type over a Dedekind ring. Using \cite[Corollary~5.15]{De_Jong_altcurves} applied to the irreducible components of the normalisation of $S$ and then iterating, we build a proper hypercovering $\pi_\bullet:\widetilde{S}_\bullet\ra S$ with all $\widetilde{S}_n$ regular. We pullback $\pi_\bullet$ to obtain a proper hypercovering $\pi'_\bullet:\widetilde{T}_\bullet\ra T$. Since $f$ is smooth, all $\widetilde{T}_n$ are regular as well. By cohomological descent for the h-topology~\cite[Theorem 14.3.4]{Cisinski_Deglise_BluePreprint}, we have $\BQ_S\simeq \pi_{\bullet*}\BQ_{\widetilde{S}_\bullet}$ and $\BQ_T\simeq \pi'_{\bullet*}\BQ_{\widetilde{T}_\bullet}$. We deduce that $H^{n,0}_\CM(S)\simeq \DA(\widetilde{S}_\bullet)(\BQ_{\widetilde{S}_\bullet},\BQ_{\widetilde{S}_\bullet}[n])$ and $H^{n,0}_\CM(T)\simeq \DA(\widetilde{T}_\bullet)(\BQ_{\widetilde{T}_\bullet},\BQ_{\widetilde{T}_\bullet}[n])$. By \ref{mot_coh_0_<}, \ref{mot_coh_0_0} and \ref{mot_coh_0_geq}, we have for every $k,m\in \BZ$ that $\DA(\widetilde{S}_k)(\BQ_{\widetilde{S}_k},\BQ_{\widetilde{S}_k}[m])$ is isomorphic to $\BQ^{\pi_0(\widetilde{S}_k)}$ if $m=0$ and $0$ otherwise; a similar formula holds for $\widetilde{T}$.
 
 A morphism of topological spaces which is open and has connected fibres induces an isomorphism on sets of connected components. The map $f$ and all its pullbacks are smooth with geometrically connected fibres, hence are open with connected fibres. This implies that the map $f$ and its pullbacks induce isomorphisms $\pi_0(S_k)\simeq \pi_0(T_k)$ on sets of connected components. This implies the result. 
\end{proof}

Let $S$ be a scheme. We have $D(\Sm/S)(\BQ_S,\Gm\otimes\BQ)\simeq H^0(S_\et,\Gm\otimes\BQ)\simeq \CO^\times(S)\otimes\BQ$ and $D(\Sm/S)(\BQ_S,\Gm\otimes\BQ[1])\simeq H^1(S_\et,\Gm\otimes\BQ)\simeq \Pic(S)\otimes\BQ$. Combining these isomorphisms with Proposition~\ref{prop:Gm_Q1}, this induces morphisms
\[
\nu^{1,1}:\CO^\times(S)\lra H^{1,1}_\CM(S)
\]
and
\[
\nu^{2,1}:\Pic(S)_\BQ \lra H^{2,1}_\CM(S).
\]
More generally, for any $n\in\BZ$, we have an induced morphism
\[
\nu^{n,1}:D(\Sm/S)(\BQ_S,\Gm[n-1])\ra H^{n,1}_\CM(S).
\]
\begin{prop}
\label{prop:mot_coh_1}\
  \begin{enumerate}[label={\upshape(\roman*)}]
  \item \label{mot_coh_1_leq} For all $n\leq 0$, we have $H^{n,1}_\CM(S)\simeq 0$.
  \item \label{mot_coh_1_1} Assume $S$ regular. The morphism $\nu^{1,1}$ induces an isomorphism $H^{1,1}_\CM(S)\simeq \CO^\times(S)_\BQ$.
  \item \label{mot_coh_1_2} Assume $S$ regular. The morphism $\nu^{2,1}$ induces an isomorphism $H^{2,1}_\CM(S)\simeq \Pic(S)_\BQ$.
  \item \label{mot_coh_1_3} Assume $S$ regular. For all $n\neq 1,2$, we have $H^{n,1}_\CM(S)\simeq 0$. We have also \[D(\Sm/S)(\BQ_S,\Gm\otimes\BQ[n-1])\simeq 0,\] so that the morphism $\nu^{n,1}$ is an isomorphism.
  \end{enumerate}
\end{prop}
\begin{proof}
Statement \ref{mot_coh_1_leq} for $S$ regular and a weaker version of \ref{mot_coh_1_1} (without specifying the isomorphism) are proved in \cite[Corollaire 11.4]{Ayoub_Etale}. 

To pass from \ref{mot_coh_1_leq} for $S$ regular to a general $S$, we apply resolution of singularities by alterations and cohomological $h$-descent for a proper regular hypercovering (which induces a descent spectral sequence for $H^{n,1}(-)$). To be more precise, one has to reduce to a situation where one can apply De Jong's theorem, e.g. $S$ is of finite type over a Dedekind ring: for this, one uses Mayer-Vietoris to first reduce to $S$ affine, and then uses continuity. The argument is the same as in the proof of Lemma~\ref{lem:mor_ab_tor}, so we do not spell out the details.

We revisit and make more precise the argument in \cite[Corollaire 11.4]{Ayoub_Etale} to establish \ref{mot_coh_1_1}, \ref{mot_coh_1_2} and \ref{mot_coh_1_3}. 

Let us first treat the case where $S$ is the spectrum of a field. In that case, for $n\neq 1$, both the source and target of $\nu^{n,1}$ are $0$, so the only interesting case is $n=1$. We have to show that the map
\[
\nu^{1,1}_k:k^\times\otimes\BQ \ra H^{1,1}_\CM(k)
\]
is an isomorphism. By the definition of $\nu^{1,1}$, we have to show that the map
\[
k^\times\otimes\BQ\simeq \DA^\eff(k)(\BQ,\Gm\otimes\BQ) \ra \DA(k)(\BQ,\Sigma^\infty (\Gm\otimes\BQ))
\]
induced by $\Sigma^\infty$ is an isomorphism. 

Let $k^\perf$ be a perfect closure of $k$ and $h:\Spec(k^\perf)\ra \Spec(k)$ be the canonical morphism. In the diagram 
\[
\xymatrix{
\DA^\eff(k)(\BQ,\Gm\otimes\BQ) \ar[r] \ar[d]_{h^*} & \DA(k)(\BQ,\Sigma^\infty (\Gm\otimes\BQ)) \ar[d]_{h^*} & \\
\DA^\eff(k^\perf)(\BQ,\Gm\otimes\BQ) \ar[r]_{(R_h)_*} & \DA(k^\perf)(\BQ,h^*\Sigma^\infty (\Gm\otimes\BQ)) \ar[r]^\sim & \DA(k)(\BQ,\Sigma^\infty (\Gm\otimes\BQ))
}
\]
the left square commutes because of the natural isomorphism $h^*\Sigma^\infty\simeq \Sigma^\infty h^*$. The left vertical arrow is an isomorphism because $k^\times\otimes\BQ\simeq (k^\perf)^\times\otimes\BQ$ (any element of $k^\perf$ has a power in $k$), and the right vertical arrow is an isomorphism by separation for $\DA$.

We are now reduced to the case when $k$ is perfect. Then we can follow a familiar pattern: comparison with $\DM(k)$ using \cite[Theorem 2.8, Proposition 2.10]{AHPL}, then with $\DM^{\eff}(k)$ using Voevodsky's cancellation theorem (this is where we need $k$ perfect), and finally the classical computation of weight one effective motivic cohomology \cite[Lecture 4]{MVW}. 

We now do the general case. We can assume $S$ connected, hence integral. By a continuity argument, one can reduce to the case where $S$ is of finite type over a Dedekind ring, and in particular excellent. Let $j:U\ra S$ be a non-empty open set and $Z$ its closed complement. Since $S$ is excellent, we can stratify $Z=Z_0\supset Z_1\supset \ldots \supset Z_{k}=\emptyset$ in such a way that for all $i$, the scheme $(Z_i\setminus Z_{i+1})_\red$ is regular and in such a way that $(Z\setminus Z_1)$ contains all points of codimension $1$ of $Z$ in $S$. Then by applying inductively localisation, absolute purity (for the regular pair $(S,(Z_i\setminus Z_{i+1})_\red)$) and the vanishing result of Proposition~\ref{prop:mot_coh_0} \ref{mot_coh_0_<} and \ref{mot_coh_0_0} we see that
\begin{itemize}
\item the map $u^{0,0}:\BQ^{\pi_0(Z\setminus Z_1)}\ra H^{0,0}_\CM(Z\setminus Z_1)$ is an isomorphism,
\item the pullback map $H^{n,1}_\CM(S)\ra H^{n,1}_\CM(U)$ is an isomorphism for $\neq 1,2$, and
\item there is a short exact sequence
\[
0\ra H^{1,1}_\CM(S)\ra H^{1,1}_\CM(U)\ra H^{0,0}_\CM(Z\setminus Z_1)\ra H^{2,1}_\CM(S)\ra H^{2,1}_\CM(U)\ra 0.
\]
\end{itemize}
Putting this together with the localisation sequence for $\CO^\times$ and $\Pic$, we get a diagram
\[
\xymatrix@C=1em{
0 \ar[r] & \CO^\times_S\otimes\BQ \ar[r] \ar[d]_{\nu^{1,1}_S} & \CO^{\times}_U\otimes\BQ \ar[r]^<<<{\mathrm{val}} \ar[d]_{\nu^{1,1}_U} \ar @{} [rd] |{\mathbf{(A)}} & \BQ^{\pi_0(Z\setminus Z_1)}\simeq \bigoplus_{z\in Z^{(1)}}\BQ\cdot z \ar[r] \ar[d]_{\nu^{0,0}}^{\sim} \ar @{} [rd] |{\mathbf{(B)}}& \Pic(S)\otimes\BQ \ar[r] \ar[d]_{\nu^{2,1}_S} &  \Pic(U)\otimes\BQ \ar[r] \ar[d]_{\nu^{2,1}_U} & 0 \\
0 \ar[r] & H^{1,1}_\CM(S) \ar[r] &  H^{1,1}_\CM(U) \ar[r] & H^{0,0}_\CM(Z\setminus Z_1) \ar[r] & H^{2,1}_\CM(S) \ar[r] & H^{2,1}_\CM(U) \ar[r] & 0.
}
\]
We claim that the diagram above is commutative. For the two outer squares, this follows from the commutation of $u_S$ with pullbacks in Proposition \ref{prop:Gm_Q1}. 

 For the commutation of diagrams (A) and (B), we have to do more work, since one arrow is defined explicitly using valuations and line bundle attached to a divisor while the other is defined via the absolute purity isomorphism. Instead of giving a long explicit computation, we prefer to see it as a special case of D\'eglise's machinery of ``residual Riemann-Roch formulas'' in \cite[4.2.1, 5.5.1]{Deglise_RR}; namely, take the diagram (4.2.1 b) in loc. cit. with $\BE$ being algebraic $K$-theory tensor $\BQ$, $\BF$ being motivic cohomology with rational coefficients, the morphism $\phi$ being the Chern character, and then use that $\CO^\times(S)_\BQ\subset K_1(S)\otimes\BQ$ (resp. $\Pic(S)_\BQ\subset K_0(S)\otimes\BQ$ for $S$ regular), and that the Chern character maps coincide with the maps $\nu^{n,1}$ modulo this identification.

Passing to the limit in the previous commutative diagram over all non-empty open sets, using continuity both for motivic cohomology and for the \'etale cohomology of $\Gm$, we get a commutative diagram
\[
\xymatrix@C=1em{
0 \ar[r] & \CO^\times_S\otimes\BQ \ar[r] \ar[d]_{\nu^{1,1}_S} & \kappa(S)^\times\otimes\BQ \ar[r]^{\mathrm{val}} \ar[d]_{\nu^{1,1}_U} & \bigoplus_{z\in S^{(1)}}\BQ\cdot z \ar[r] \ar@{=}[d] & \Pic(S)\otimes\BQ \ar[r] \ar[d]_{\nu^{2,1}_S} &\Pic(\kappa(S)) \ar[r] \ar[d]_{\nu^{2,1}_U} & 0 \\
0 \ar[r] & H^{1,1}_\CM(S) \ar[r] &  H^{1,1}_\CM(\kappa(S)) \ar[r] & \bigoplus_{z\in S^{(1)}}\BQ \cdot z \ar[r] &  H^{2,1}_\CM(S) \ar[r] & H^{2,1}_\CM(\kappa(S)) \ar[r] & 0.
}
\]

Using the case of a base field treated above, we see that

\begin{itemize}
\item the group $H^{n,1}_\CM(S)$ vanishes for $n\neq 1,2$, and
\item there is a short exact sequence
\[
0\ra H^{1,1}_\CM(S)\ra \kappa(S)^\times\otimes\BQ\stackrel{\mathrm{val}}{\ra} \bigoplus_{z\in S^{(1)}} \BQ \cdot z\ra H^{2,1}_\CM(S)\ra 0.
\]
\end{itemize}
Using the normality (resp. regularity) of $S$, this implies $H^{1,1}_\CM(S)\simeq \CO(S)^\times_\BQ$ and $H^{2,1}_\CM(S)\simeq \Pic(S)_\BQ$ and finishes the proof.
\end{proof}

%%%%%%%%%%%%%%%%%%%%%%%%%%%%%%%%%%%%%%

%% Counter-example 
We finish by giving an example which shows that even for weight zero motivic cohomology on normal (but not regular) schemes, the result can differ from \'etale cohomology.

\begin{prop}
Let $S$ be a normal excellent surface with a unique singular point. Let $\pi:\tilde{S}\rightarrow S$ be a resolution of singularities of $S$, with $D=\pi^{-1}(p)$ simple normal crossing divisor in $\tilde{S}$. Let $\Gamma=(V,E)$ be the dual graph of $D$. Then
\[
H^{n,0}_\CM(S)\simeq\left\{\begin{array}{c} \BQ,\ n=0\\H^1(\Gamma,\BQ),\ n=2\\0,n\neq 0,2\end{array}\right.
\]
while on the other hand
\[
D(\Sm/S)(\BQ_S,\BQ_S[n])\simeq \left\{\begin{array}{c}\BQ,\ n=0\\0,n\neq 0\end{array}\right. .
\]
\end{prop}

\begin{proof}
The last statement comes from the fact that the \'etale cohomology of a normal scheme with $\BQ$-coefficients is trivial. So we concentrate on the first. For $n\leq 0$, the result follows from \ref{prop:mot_coh_0}, so we assume $n>0$.

We have the cartesian diagram of schemes:
\[
\xymatrix{
U\ar@{=}[d] \ar[r]^{\tilde{\jmath}} & \tilde{S} \ar[d]^{\pi} & D \ar[l]_{\tilde{\imath}} \ar[d]^{\pi_p}\\
U \ar[r]^{j} & S & p \ar[l]_{i} 
}
\]
Localisation yields the long exact sequence:
\[
\xymatrix@C=1em{
\DA(S)(\BQ_S,\BQ_S[n-1])\ar[r] & \DA(U)(\BQ_U,\BQ_U[n-1]) \ar[r] & \DA(p)(\BQ_p,i^!\BQ_S[n])\ar[r] & \DA(S)(\BQ_S,\BQ_S[n])\ar[d] \\
&  & & \DA(U)(\BQ_U,\BQ_U[n]).
}
\]
By Proposition~\ref{prop:mot_coh_0} (using that $U$ is regular, and that $\pi_{0}(S)\simeq \pi_{0}(U)$), this yields an isomorphism $\DA(p)(\BQ_p,i^!\BQ_S[n])\simeq \DA(S)(\BQ_S,\BQ_S[n])$.

Write $ \{D_v\}_{v\in V}$ for the set of irreducible components of $D$ and $p_e$ for the intersection point $D_v\cap D_{v'}$ for $e=vv'\in E$. We set $Z=\bigcup_{e\in E}\{p_e\}$ and $\mathring{D}=D\setminus Z$. Write $k:\mathring{D}\rightarrow D$, $l:Z\rightarrow D$. Localisation gives a distinguished triangle
\[
l_*(\tilde{\imath}\circ l)^!\BQ_{\tilde{S}}\rightarrow \tilde{\imath}^!\BQ_{\tilde{S}}\rightarrow k_*(\tilde{\imath}\circ k)^!\BQ_{\tilde{S}}\stackrel{+}{\rightarrow}.
\]
By the relative purity theorem for $\DA$ (see \cite[1.6.1]{Ayoub_these_1} and \cite[Corollaire 3.10]{Ayoub_Etale}) applied to the regular immersions $\tilde{\imath}\circ l$ and $\tilde{\imath}\circ k$, this triangle takes the form:
\[
l_*\BQ_{Z}(-2)[-4]\rightarrow \tilde{\imath}^!\BQ_{\tilde{S}}\rightarrow k_*\BQ_{\mathring{D}}(-1)[-2]\stackrel{+}{\rightarrow}.
\]
So we get the exact sequence:
\[
 \DA(Z)(\BQ_Z,\BQ_Z(-1)[n-2])\rightarrow \DA(D)(\BQ_D,\tilde{\imath}^!\BQ_{\tilde{S}}[n])\rightarrow  \DA(\mathring{D})(\BQ_{\mathring{D}},\BQ_{\mathring{D}}(-2)[n-4])
\]
By Proposition~\ref{prop:mot_coh_<}, the groups on the left and on the right are zero for all $n\in \BZ$, so we conclude that $\DA(D)(\BQ_D,\tilde{\imath}^!\BQ_{\tilde{S}}[n])=0$ for all $n\in \BZ$.

 Now, the fact that $\pi_U$ is an isomorphism, colocalisation and base change for immersions (see \cite[1.4.6]{Ayoub_these_1}) implies that $\Cone(i^!\BQ_S\rightarrow \pi_{p,*}\tilde{\imath}^!\BQ_{\tilde{S}})\simeq \Cone(\BQ_p\rightarrow \pi_{p,*}\BQ_D)$. Combining with the previous result, we get that for all $n\in\BZ$:
\[
\DA(S)(\BQ_S,\BQ_S[n])\simeq \DA(p)(\BQ_p,\Cone(\BQ_p\rightarrow \pi_{p,*}\BQ_D)[n-1]))\simeq \DA(p)(\BQ_p,\pi_{p,*}\BQ_D[n-1])
\]
(where the last isomorphism follows as $n>1$). 

Using Cech descent for closed covers and Proposition~\ref{prop:mot_coh_<}, it is then easy to see that this last group is isomorphic to $\BQ$ if $n=0$ (note that $\Gamma$ is connected by normality of $S$), isomorphic to $H^1(\Gamma,\BQ)$ if $n=1$, and $0$ otherwise.

\end{proof}

\bibliography{bibthese}
\bibliographystyle{alpha}

%\printbibliography

\end{document}